\documentclass{amsart}
\usepackage[utf8]{inputenc}
\usepackage[margin=1in]{geometry}
\usepackage{hyperref}
\hypersetup{
    colorlinks=true,
    citecolor=blue,
    linkcolor=blue,
    filecolor=magenta,      
    urlcolor=cyan,
    pdftitle={Kh Notes},
    pdfpagemode=FullScreen,
    }

\usepackage{amsmath,amssymb,amsthm, mathrsfs}
\usepackage{dsfont}
\usepackage[dvipsnames]{xcolor}
\usepackage{tikz}
    \usetikzlibrary{cd}
\usepackage{tikz-cd}
\allowdisplaybreaks
\usepackage{standalone}
\usepackage{comment}
\usepackage{wasysym} 
\usepackage{mathtools} 
\usepackage{bbm} 
\usepackage{stmaryrd} 
\usepackage{enumitem}

\theoremstyle{definition}
\newtheorem{theorem}{{Theorem}}[subsection]
\newtheorem{lemma}[theorem]{{Lemma}}
\newtheorem{proposition}[theorem]{{Proposition}}
\newtheorem{definition}[theorem]{{Definition}}
\newtheorem{corollary}[theorem]{{Corollary}}
\newtheorem{example}[theorem]{{Example}}
\newtheorem{exercise}[theorem]{{Exercise}}
\newtheorem{conjecture}[theorem]{{Conjecture}}
\newtheorem{remark}[theorem]{{Remark}}
\newtheorem{claim}[theorem]{{Claim}}

\newtheorem{algorithm}[theorem]{{Algorithm}}
\newtheorem{notation}[theorem]{{Notation}}
\newtheorem{question}[theorem]{{Question}}
\newtheorem{fact}[theorem]{{Fact}}

\newtheorem{aside}[theorem]{{Aside}}
\newtheorem{warning}[theorem]{{Warning}}
\newtheorem{warmup}[theorem]{{Warmup}}

\newcommand{\bea}{\begin{enumerate}[label=(\alph*)]}
\newcommand{\ee}{\end{enumerate}}

\newcommand{\C}{\mathbb{C}}

\newcommand{\Q}{\mathbb{Q}}
\newcommand{\R}{\mathbb{R}}
\newcommand{\Z}{\mathbb{Z}}
\newcommand{\N}{\mathbb{N}} 
\newcommand{\F}{\mathbb{F}} 
\newcommand{\CA}{\mathcal{A}} 
\newcommand{\essl}{\mathfrak{sl}}

\newcommand{\Sym}{\mathrm{Sym}} 

\newcommand{\Hom}{\mathrm{Hom}} 
\newcommand{\trace}{\mathrm{Tr}} 
\newcommand{\Kar}{\mathrm{Kar}} 

\newcommand{\map}[1]{\xrightarrow{#1}}
\newcommand{\inv}{^{-1}}
\newcommand{\st}{\ | \ }
\newcommand{\into}{\hookrightarrow}
\newcommand{\onto}{\twoheadrightarrow}

\newcommand{\abuts}{\rightrightarrows} 
\DeclareMathOperator{\rk}{rank} 
\renewcommand{\emptyset}{\varnothing} 

\newcommand{\diffeo}{\cong} 
\newcommand{\homeo}{\simeq} 
\newcommand{\Top}{\mathrm{Top}}
\newcommand{\catLink}{\textbf{Link}} 
\newcommand{\catTang}{\textbf{Tang}} 
\newcommand{\catDiag}{\textbf{LinkDiag}} 
\newcommand{\catTangDiag}{\textbf{TangDiag}} 
\newcommand{\concgroup}{\mathscr{C}} 

\newcommand{\KBSM}{\mathrm{KBSM}} 
\newcommand{\Kh}{\mathrm{Kh}}
\newcommand{\Kg}{\mathrm{Kg}} 
\newcommand{\Khred}{\widetilde{\Kh}} 
\newcommand{\KhR}{\mathrm{KhR}} 
\newcommand{\BNfun}{\functor_{\text{BN}}} 
\newcommand{\CAKh}{\mathrm{CAKh}} 
\newcommand{\AKh}{\mathrm{AKh}}

\newcommand{\Khsymp}{\Kh_{\mathrm{symp}}}
\newcommand{\BN}{\mathrm{BN}}
\newcommand{\BNred}{\widetilde{\BN}} 
\newcommand{\BNcat}{\mathcal{BN}} 
\newcommand{\CKh}{\mathrm{CKh}}
\newcommand{\CKhred}{\widetilde{\CKh}} 
\newcommand{\CLee}{\mathrm{CLee}} 
\newcommand{\Lee}{\mathrm{Lee}} 
\newcommand{\bolda}{\mathbf{a}} 
\newcommand{\boldb}{\mathbf{b}} 
\newcommand{\fraks}{\mathfrak{s}} 
\newcommand{\TL}{\mathrm{TL}} 
\newcommand{\TLcat}{\mathcal{TL}} 
\newcommand{\WW}{\mathbb{W}} 
\newcommand{\VV}{\mathbb{V}} 

\newcommand{\CS}{\mathcal{S}} 
\newcommand{\CL}{\mathcal{L}} 
\newcommand{\lasagna}{\CS_0^2} 
\newcommand{\lasfilling}{\mathcal{F}} 
\newcommand{\cabledKhR}{\underline{\KhR}} 

\newcommand{\TT}{\mathbb{T}} 
\newcommand{\CF}{\mathrm{CF}}
\newcommand{\CFhat}{\widehat{\CF}}
\newcommand{\HFhat}{\widehat{\mathrm{HF}}}

\newcommand{\HD}{\mathcal{H}} 

\newcommand{\Tot}{\mathrm{Tot}} 
\newcommand{\gr}{\mathrm{gr}} 

\newcommand{\Ob}{\mathrm{Ob}} 
\newcommand{\id}{\mathrm{id}}
\newcommand{\CC}{\mathcal{C}} 
\newcommand{\SA}{\mathscr{A}} 
\newcommand{\One}{\mathbbm{1}} 
\newcommand{\functor}{\mathcal{F}} 
\newcommand{\cone}{\mathrm{Cone}} 

\newcommand{\colim}{\mathrm{colim}}
\newcommand{\im}{\mathrm{im}}

\newcommand{\Zmod}{\text{Mod}_\Z} 
\newcommand{\ggmod}{gg\text{Mod}_\Z} 
\newcommand{\ggVect}{gg\text{Vect}} 
\newcommand{\Mat}{\mathrm{Mat}} 
\newcommand{\Kom}{\mathrm{Kom}} 

\newcommand{\CX}{\mathcal{X}} 
\newcommand{\XKh}{\CX_{\Kh}} 
\newcommand{\moduli}{\mathcal{M}} 
\newcommand{\Sq}{\mathrm{Sq}}
\newcommand{\cubecat}{\underline{2}} 
\newcommand{\Burn}{\mathcal{B}} 
\newcommand{\Topcat}{\mathbf{Top}}
\newcommand{\SC}{\mathscr{C}} 
\newcommand{\SChat}{\hat\SC} 
\newcommand{\BE}{\mathbb{E}} 
\newcommand{\Conf}{\mathrm{Conf}} 

\newcommand{\Wh}{\mathrm{Wh}} 
\newcommand{\writhe}{\mathrm{wr}} 

\newcommand{\tb}{\mathrm{tb}}
\newcommand{\rot}{\mathrm{rot}}
\newcommand{\maxtb}{\overline{\tb}} 
\newcommand{\xistd}{\xi_{std}} 
\newcommand{\xirot}{\xi_{rot}} 
\newcommand{\transknot}{\mathcal{T}} 
\newcommand{\selflinking}{sl} 
\newcommand{\fix}{{\mathrm{fix}}} 
\newcommand{\idbraid}{\mathbbm{1}} 

\newcommand{\crossing}{
    \begin{tikzpicture}[scale=.125, baseline=-.67ex]
        \draw (-1,1) -- (1,-1);
        \filldraw[white] (0,0) circle (.4cm);
        \draw (-1,-1) -- (1,1);
    \end{tikzpicture}
}
\newcommand{\vertres}{
    \begin{tikzpicture}[scale=.125, baseline=-.67ex]
        \draw (-1,-1) .. controls (0,0) .. (-1,1);
        \draw (1,-1) .. controls (0,0) .. (1,1);
    \end{tikzpicture}
}
\newcommand{\horizres}{
    \begin{tikzpicture}[scale=.125, baseline=-.67ex]
        \begin{scope}[rotate=90]
            \draw (-1,-1) .. controls (0,0) .. (-1,1);
            \draw (1,-1) .. controls (0,0) .. (1,1);
        \end{scope}
    \end{tikzpicture}
}
\newcommand{\poscrossing}{
    \begin{tikzpicture}[scale=.25, baseline=-1ex]
        \draw[-Stealth] (1,-1) -- (-1,1);
        \filldraw[white] (0,0) circle (.4cm);
        \draw[-Stealth] (-1,-1) -- (1,1);
    \end{tikzpicture}
}
\newcommand{\negcrossing}{
    \begin{tikzpicture}[xscale=-1, scale=.25, baseline=-1ex]
        \draw[-Stealth] (1,-1) -- (-1,1);
        \filldraw[white] (0,0) circle (.4cm);
        \draw[-Stealth] (-1,-1) -- (1,1);
    \end{tikzpicture}
}

\newcommand{\mzcolor}{WildStrawberry} 
\newcommand{\mz}[1]{{\color{\mzcolor}#1}}  
\newcommand{\toadd}[1]{{\color{Melon}(add: #1)}}
\newcommand{\note}[1]{{\color{MidnightBlue}#1}} 
\newcommand{\alert}[1]{{\color{red}#1}} 

\title{Notes on Khovanov homology}
\author{Melissa Zhang}
\date{Fall 2024}

\begin{document}

\begin{abstract}
These are expository lecture notes from a graduate topics course taught by the author on Khovanov homology and related invariants. 
Major topics include the Jones polynomial, Khovanov homology, Bar-Natan's cobordism category, applications of Khovanov homology, some spectral sequences, Khovanov stable homotopy type, and skein lasagna modules. 
Topological and algebraic exposition are sprinkled throughout as needed.
\end{abstract}

\maketitle

\tableofcontents

\section{Introduction}

These are lecture notes from a graduate topics course (MAT280) on \emph{Khovanov homology and related invariants} I taught at UC Davis in Fall Quarter of 2024. 

\subsubsection*{Goal}
The goal of the course was to expose a diverse group of algebra, topology, and combinatorics graduate students to some of the most impactful ideas from Khovanov homology. 
As such, we covered some requisite material quite quickly, opting for more intuitive explanations of auxiliary concepts, rather than formal treatments.

\subsubsection*{Audience}
Students taking the course were required to be familiar with homological algebra (MAT250) but not necessarily with low-dimensional topology. 
As such, I included some exposition on topological concepts when needed, but at othe times also discussed  advanced topology topics without exposition. 

\subsubsection*{Tone}
The casual tone of these notes are similar to the style of my handwritten notes from previous courses, which received positive feedback from students. 
Throughout, you may find colored text:
\begin{itemize}
    \item \note{Additional mathematical commentary is written in this color.}
    \item \alert{Warnings} look like this.
    \item Occasional other comments from me are written in \mz{this color}. (These include organizational notes, general encouragement, life advice, etc.)
    \item \toadd{This indicates that I would like to add something in a future version of these notes.}
\end{itemize}

\subsubsection{Evolution}
There are likely still numerous typos in these notes. 
If you find any and would like to let me know, please email me.
I will post the most recent version somewhere on my personal website:
\url{https://www.melissa-zhang.com}.

\subsection*{Acknowledgements}
I would like to thank the students, both registered and those just sitting in, in my MAT280 course. 
Your persistent engagement, enthusiasm in lectures, and pertinent questions helped me clarify my explanations in these notes. 
I would also in particular like to thank Jake Quinn and Daniel Qin for finding many typos in the notes, and Ian Sullivan for teaching the final lecture of the course while I was traveling.

\section{Knots and Topology}

\subsection{Knots and links}


A knot is sometimes defined as a smooth embedding $S^1 \into S^3$. 

Notice that we can 
\begin{itemize}
    \item reparametrize the embedding, preserving the image setwise
    \item perform an \emph{ambient isotopy} on the knot (`isotop' the knot)
\end{itemize}

\begin{definition}
Let $K, K': S^1 \into S^3$ be two smooth embeddings. 
We say $K$ and $K'$ are (smoothly, ambiently) \emph{isotopic} if there exists a smooth family of diffeomorphisms 
\[
    \{\varphi_t: S^3 \to S^3\}_{t \in [0,1]}
\]
such that $\varphi_0 = \id$ and $\varphi(1) \circ K = K'$. 
\end{definition}

\begin{remark}
Convince yourself of the following:
\begin{enumerate}
    \item ``Ambiently isotopic'' is an equivalence relation. The equivalence classes are called \emph{isotopy classes} of knots.
    \item In other words, an ambient isotopy smoothly morphs the embedding $K$ into the embedding $K'$, through a family of diffeomorphisms of $S^3$.
    \item An order-preserving reparametrization is an ambient isotopy.
\end{enumerate}
\end{remark}

\begin{remark}
\label{rmk:colloquial-knot}
    In practice, when I say \emph{knot}, I'm probably referring to either (1) the image of the knot in $S^3$ or (2) an entire isotopy class.  
\end{remark}

\begin{example}
    `The' \emph{unknot} $U$ is the equivalence class of the standard embedding $S^1 \into \R^2 \into \R^3 \into S^3$.
\end{example}

\begin{definition}
Any diffeomorphic copy of $S^1$ can have two possible orientations. 
The \emph{orientation} of a knot is given by the direction $\frac{dK}{d\theta}$.\footnote{This is terrible notation and should not ever appear again, because of how we use the term `knot'; see Remark \ref{rmk:colloquial-knot}.}
\end{definition}

\begin{definition}
The \emph{reverse} $K^r$ (also denoted by $\bar K$) of a knot $K$ is 
the\footnote{By using the article `the', I'm using the term `knot' in the sense of Remark \ref{rmk:colloquial-knot} (2).}
knot obtained by precomposing $K$ with an order-reversing diffeomorphism $\rho: [0,1] \to [0,1]$. 

Let $\tau$ be an orientation-reversing diffeomorphism of $S^3$. 
The \emph{mirror} of a knot $K$ is the knot $m(K) = \tau \circ K$. 
\end{definition}

\begin{remark}
    An \emph{unoriented knot} is just the knot after you forget about the orientation. You can think of this as the union of the isotopy classes of $K$ and $\bar K$. 
\end{remark}

\begin{exercise}
First observe that the orientation of $S^3$ does not reverse under isotopy. 

Prove that 
\begin{itemize}
    \item In general, $m(K)$ is not necessarily isotopic to $K$, even if we treat them as unoriented knots.  \note{You can prove this via a (counter)example. Proving this directly is hard. Use the Jones polynomial, introduced in Section \ref{sec:jones-polyn}.}
    \item In general, $K^r$ is not necessarily isotopic to $K$. \note{There is no general algorithm for determining when $K \not\sim K^r$! For this problem, do a search online and see if you can an example in the literature, as well as the relevant terms for describing this kind of symmetry.}
\end{itemize}

Then, find examples of knots $K$ that do happen to satisfy the following, and exhibit an explicit isotopy:
\begin{itemize}
    \item $K \sim K^r$
    \item $K \sim m(K)$
\end{itemize}
\note{Diagrammatically exhibit the isotopy by applying Reidemeister moves; see Section \ref{sec:link-diagrams}. If  there's an obvious part of the isotopy that's easy to describe but takes a lot of Reidemeister moves, you can just note what you're doing between two pictures.}
\end{exercise}

\begin{exercise}
\alert{(Important)}
A \emph{link} with $\ell \in \Z_{\geq 0}$ components is a smooth embedding 
\[
    L: \coprod_{i=1}^\ell S^1 \into S^3
\]
The link with zero components is called the \emph{empty link}. 
The $i$th \emph{component} is the embedding of the $i$th copy of $S^1$ into $S^3$.

All of the definitions above can be generalized to links.
Write down definitions for the following:
\begin{enumerate}
    \item oriented link $L = (L,o)$ \note{If we want to be explicit about the orientation of a link, we use the letter $o$ for the orientation information.}
    \item unoriented (isotopy class of) $L$
    \item $m(L)$, the mirror of a link $L$
    \item $L^r$, the reverse of an oriented link $(L,o)$
\end{enumerate}

How many orientations does a link with $\ell$ components have?
\end{exercise}

\begin{remark}
    There are also other categories of links, such as \emph{topological links} and \emph{wild links}. We won't talk about these until, maybe, far later in the course.
\end{remark}

\subsection{Link diagrams and Reidemeister moves}
\label{sec:link-diagrams}

We will now start abusing notation and language without comment, per Remark \ref{rmk:colloquial-knot}. 

We've so far been implicitly using link diagrams to draw links. 
More formally, a \emph{link diagram} is a compact projection of a \note{representative of a} link $L$ onto the $xy$-plane such that the only intersections are transverse double points.  

For example, here are some bad singularities:

\begin{center}
\includegraphics[width=\textwidth]{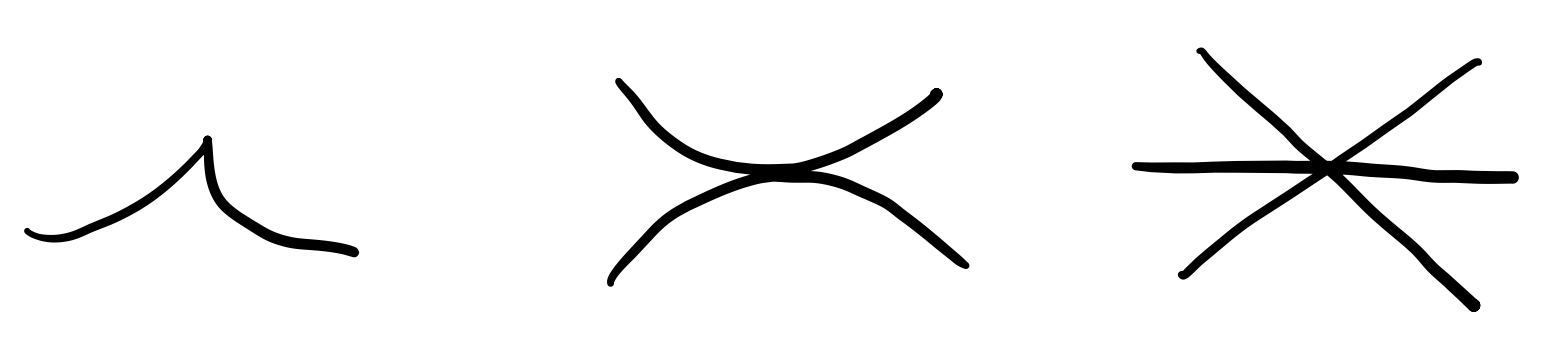}
\end{center}

\begin{theorem}[Reidemeister, 1930s]
If $D$ and $D'$ are two diagrams of the same link, then they are related by a finite sequence of the following moves:

\begin{enumerate}
\item[(R1)] \includegraphics[width=.8\textwidth]{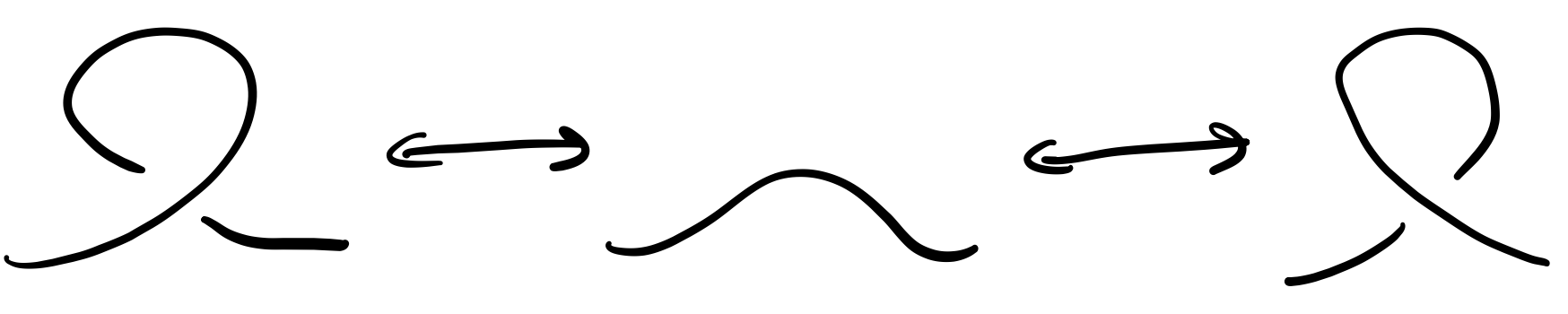}
\item[(R2)] \includegraphics[width=.8\textwidth]{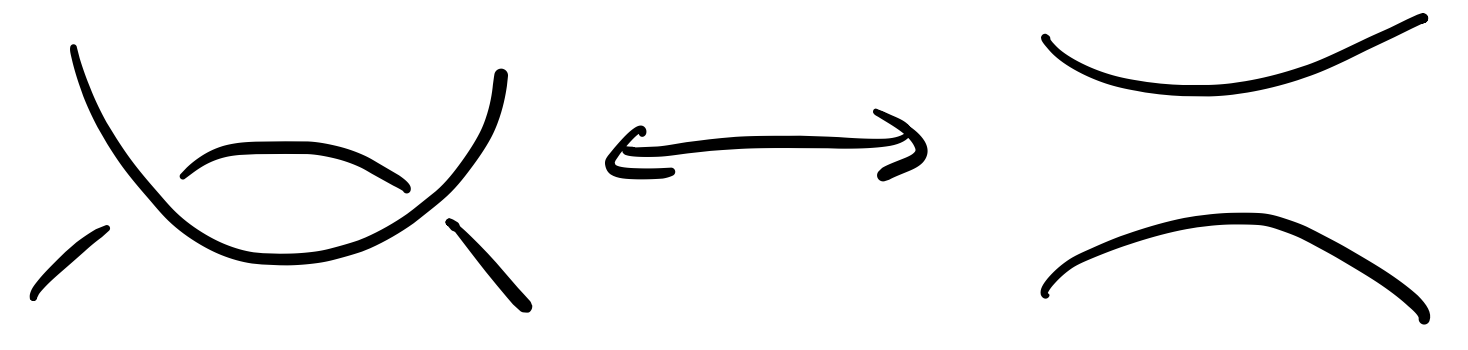}
\item[(R3)] \includegraphics[width=.8\textwidth]{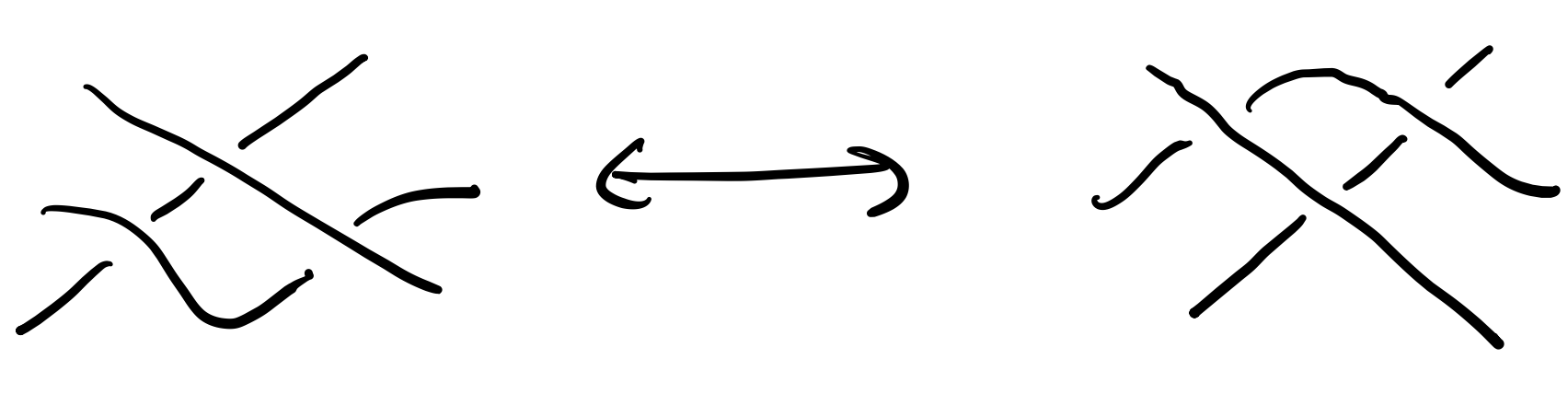}
\end{enumerate}

\end{theorem}

\begin{remark}
\begin{enumerate}
    \item These are local pictures. You can rotate them.
    \item Note that there are technically many cases within each `move', if you consider all the possible orientations of the interacting strands in these local pictures. 
    This is important when producing invariants for oriented links. 
    \item If $D$ and $D'$ are diagrams for the same link, then we will sometimes write $D \sim D'$. In other words, ``there exists a sequence of Reidemeister moves between'' is an equivalence relations on the set of link diagrams.
\end{enumerate}    
\end{remark}

\begin{exercise}
(Unimportant)
    Explain why we don't need to include the following local move:
    \begin{center}
    \includegraphics[width=.8\textwidth]{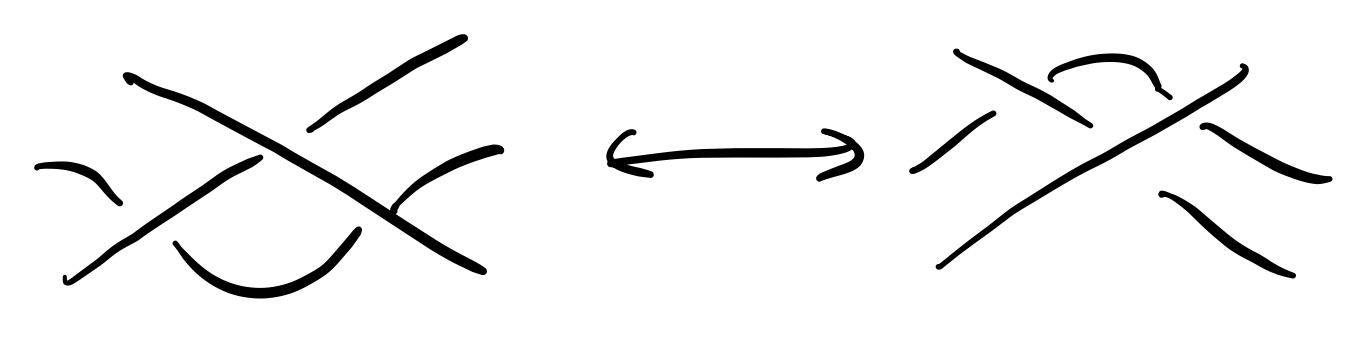}
    \end{center}
\end{exercise}

\begin{definition}
   A \emph{link invariant} $F$ valued in a category $\CC$ is a machine that 
   \begin{itemize}
       \item takes in a link diagram $D$
       \item and outputs $F(D) \in \Ob(\CC)$
   \end{itemize}
   such that 
   \[
            D \sim D' \qquad \implies \qquad F(D) = F(D').
   \]
\end{definition}

\begin{remark}
    Note that a link invariant is not necessarily a functor. We will define the categories $\catLink$ and $\catDiag$, and define \emph{functorial} link invariants later.
\end{remark}

\subsection{Jones polynomial via skein relation}
\label{sec:jones-polyn}

Khovanov homology is a categorical lift of the Jones polynomial, so we will focus on this invariant first.

\begin{remark}
    The first polynomial invariant of links was the Alexander polynomial, introduced in the 1920s. We will not discuss the Alexander polynomial until needed later.
\end{remark}

The Jones polynomial was discovered by Vaughan Jones in the 1980s, and arose from his work in statistical mechanics \cite{Jones-polynomial}. 
\note{We will not study the original definition of his invariant, but you are welcome to look into it if it interests you. Perhaps a final project idea?}

\begin{definition}
\label{def:jones-polyn} 
\textbf{(/ Theorem / Algorithm / Conventions)} 
The Jones polynomial is a link invariant valued in Laurent polynomials $\Z[q,q\inv]$ that is uniquely determined by the following recursion:
\begin{itemize}
\item base case: $J( \fullmoon ) = 1$
\item skein relation: \alert{$q^{-2} J(L_+) - q ^2 J(L_-) = (q\inv-q) J(L_0)$}
where 
\end{itemize}
\begin{center}
    \includegraphics[width=\textwidth]{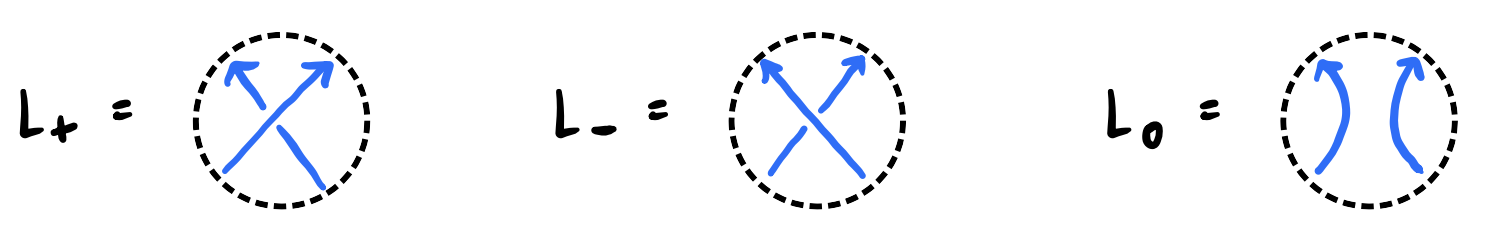}
\end{center}
\end{definition}

\begin{remark}
    The diagrams for $L_+$, $L_-$, and $L_0$ above show a \emph{positive crossing}, a \emph{negative crossing}, and an \emph{oriented resolution}, respectively. 
    \note{The singularity at the crossing can be resolved in two ways; $L_0$ is the only resolution that maintains the orientation of all the strands.}
\end{remark}

\begin{remark}
In class, we essentially used Jones' conventions; these produce a Laurent polynomial in the variable $\sqrt{t}$:
\begin{itemize}
    \item base case: $V(\fullmoon) = 1$
    \item skein relation: {$t\inv V(L_+) - t V(L_-) = (t^{1/2}-t^{-1/2}) V(L_0)$}.
\end{itemize}
If you look at a database of Jones polynomials, you'll likely find this convention used.

\note{For the purposes of this course, we will stick with Khovanov's original conventions for the first few weeks of class. }
\end{remark}

\begin{example}
Here we use the skein relation to compute the Jones polynomial (with the conventions set in Definition \ref{def:jones-polyn}) of the \emph{unlink of two components}, $\fullmoon \fullmoon$. 

\begin{center}
    \includegraphics[width=\textwidth]{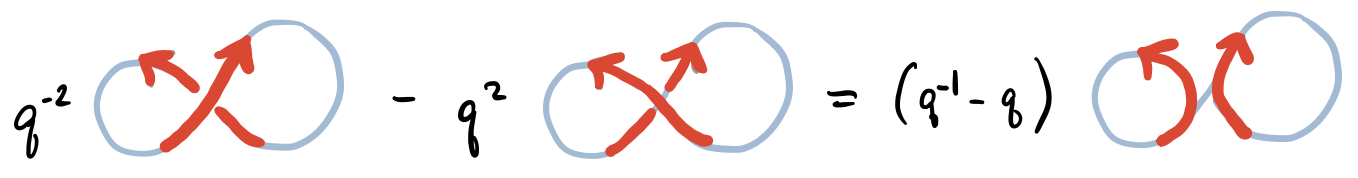}
\end{center}

\begin{align*}
    q^{-2} \cdot 1 - q^2 \cdot 1 
        &= (q\inv - q) J(\fullmoon \fullmoon) \\
    J(\fullmoon \fullmoon) &= \frac{ q^2-q^{-2} }{q - q\inv} = q + q\inv.
\end{align*}

\end{example}

\begin{exercise}
    Compute the Jones polynomials of the following links:
    \begin{enumerate}
        \item (positively linked) Hopf link

        \item right-handed trefoil
    \end{enumerate}
\end{exercise}

\begin{exercise}
\alert{(Important)}
\begin{enumerate}
    \item Prove that for any link $L$, $J(L) = J(L^r)$. 
    \item Let $K_1$ and $K_2$ be knots, and let $K_1 \# K_2$ denote their \textit{connected sum}\footnote{I have not provided a precise definition here. This is a good time to Google the term yourself to see some examples of the connected sum operation.}: 
    \begin{center}
    \includegraphics[width=.8\textwidth]{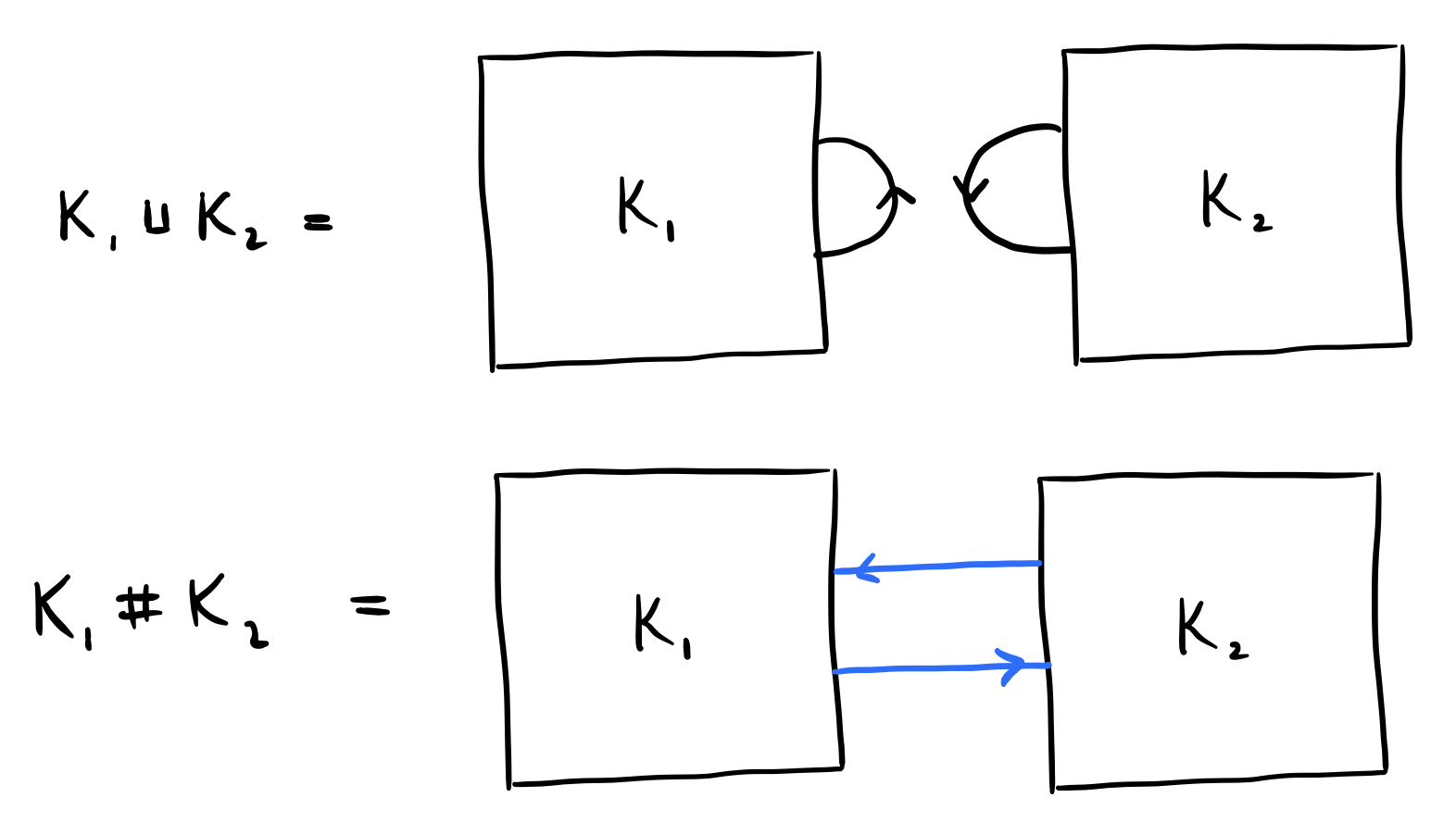}
    \end{center}
    Prove that $J(K_1 \# K_2) = J(K_1)J(K_2)$.
    What is $J(K_1 \sqcup K_2)$?
\end{enumerate}
\end{exercise}

\begin{exercise}
\begin{enumerate}[label=(\alph*)]
    \item Show why the Jones polynomial is invariant under an R1 move. 
    \item Prove that for any link $L$, $J(\fullmoon L) = (q+q\inv) J(L)$.
\end{enumerate}
\end{exercise}

\begin{question}(Open)
Does the Jones polynomial \emph{detect} the unknot $U$? 
In other words, if $J(D) = 1$, is $D$ necessarily a diagram for the unknot?
\end{question}

\subsection{Jones polynomial from Kauffman bracket}

In this section, we generally follow \cite{BN-Kh}, certainly with the same conventions.
\note{However, I may word things differently or rotate some pictures, for our future benefit.}

\begin{definition}[\cite{Kauffman-bracket}]
\label{def:kauffman-bracket}
    The \emph{Kauffman bracket} $\langle \cdot \rangle$ is defined by the recursion
    \begin{itemize}
        \item $\langle \emptyset  \rangle = 1$
        \item $\langle \fullmoon L \rangle = (q+q\inv) \langle L \rangle$
        \item $\langle \crossing \rangle = \langle \vertres \rangle - q\langle \horizres \rangle$
    \end{itemize}
\note{Note that these local pictures are \emph{unoriented}, unlike those in the skein relation for the Jones polynomial.}
\end{definition}

\begin{remark}
Each crossing has two possible smoothings, which we will name the $0$-resolution and the $1$-resolution, as shown:
\[
    \vertres \xleftarrow{0} \crossing \map{1} \horizres.
\]
\begin{itemize}
\item The $0$-resolution is the one that you would naturally draw if you started at an over-strand and drew a smile. It gets no coefficient in the Kauffman bracket. 
\item The $1$-resolution is the one that you would draw if you started at the over-strand and drew a frown. It gets a coefficient of $-q$ in the Kauffman bracket.
\end{itemize}
Note that I draw my smileys like this:
\begin{tikzpicture}[scale=.2, baseline=0ex]
    \node at (-1,1) {$\bullet$};
    \node at (1,1) {$\bullet$};
    \draw[->] (-1,0) arc (180:360:1cm);
\end{tikzpicture}
\end{remark}

The Kauffman bracket is \emph{not} a link invariant. To see this, compare $\langle \fullmoon \rangle$ and 
        $\langle 
            \begin{tikzpicture}[scale=.125, baseline=-.67ex]
                \draw (-1,-1) arc (270:90:1);
                \draw (1,1) arc (90:-90:1);
                \draw (-1,1) -- (1,-1);
                \filldraw[white] (0,0) circle (.4cm);
                \draw (-1,-1) -- (1,1);
            \end{tikzpicture}
        \rangle$.
It is, however, a \emph{framed} invariant. \mz{May talk more about this later.}

To fix this, we have to take into account the \emph{writhe} of a diagram, which measures how `twisty' the choice of diagram is. 
\includegraphics[width=.25in]{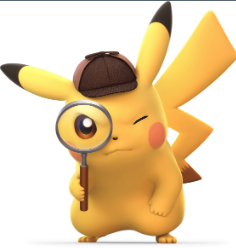}

\begin{definition}
Let $D$ be an oriented link diagram. 
\begin{itemize}
\item Let $n$ be the number of crossings in $D$.
\item Let $n_+$ be the number of positive crossings, and let $n_-$ be the number of negative crossings.
\end{itemize}
The \emph{writhe} of $D$ is $\writhe(D) = n_+ - n_-$. 
\end{definition}

\begin{definition}
    The \emph{(unnormalized)} Jones polynomial is defined by 
    \[
        \hat J(L) = (-1)^{n_-} q^{n_+-2n_-} \langle L \rangle
    \]
    \note{Note that we are treating $L$ as both the link diagram and the link it represents.}
    Equivalently, $\hat J(L) = (-q)^{-n_-} \cdot q^{\writhe(L)} \langle L \rangle$.
\end{definition}

\begin{remark}
\label{rmk:oriented-kauffman-bracket}
Alternatively, we can build the overall shifts into the bracket, using the following recursion:
\begin{itemize}
    \item $\langle \emptyset \rangle_o = 1$
    \item $\langle \fullmoon L \rangle_o = (q+q\inv) \langle L \rangle_o$ 
        \note{Recall that the two orientations on the unknot are isotopic.}
    \item (positive crossing) $\langle \poscrossing \rangle_o 
        = q \langle \vertres \rangle_o - q^2 \langle \horizres \rangle_o$
    \item (negative crossing) $\langle \negcrossing \rangle_o 
        = q^{-2} \langle \horizres \rangle_o - q\inv \langle \vertres \rangle_o$
\end{itemize}
\alert{The notation `$\langle \cdot \rangle_o$' is not standard, and is used here only to distinguish it from the  bracket used in \cite{BN-Kh}.}
\end{remark}

\begin{exercise}
\begin{enumerate}
    \item How does the writhe of a diagram change under the Reidemeister moves? 
    \item How does the Kauffman bracket $\langle \cdot \rangle$ of a diagram change under the Reidmeister moves?
    \item Verify that the bracket $\langle \cdot \rangle_o$ in Remark \ref{rmk:oriented-kauffman-bracket} computes $\hat J$, i.e. $\hat J(L) = \langle L \rangle_o$. 
\end{enumerate}
\end{exercise}

\begin{example}
\label{eg:hopf-kauffman-computation}
Let $H$ denote (oriented) Hopf link shown below, where $n = 2$, $n_+ = 2$, and $n_- = 0$. 

First, we compute the Kauffman bracket of the shown diagram:

\begin{center}
    \includegraphics[width=3in]{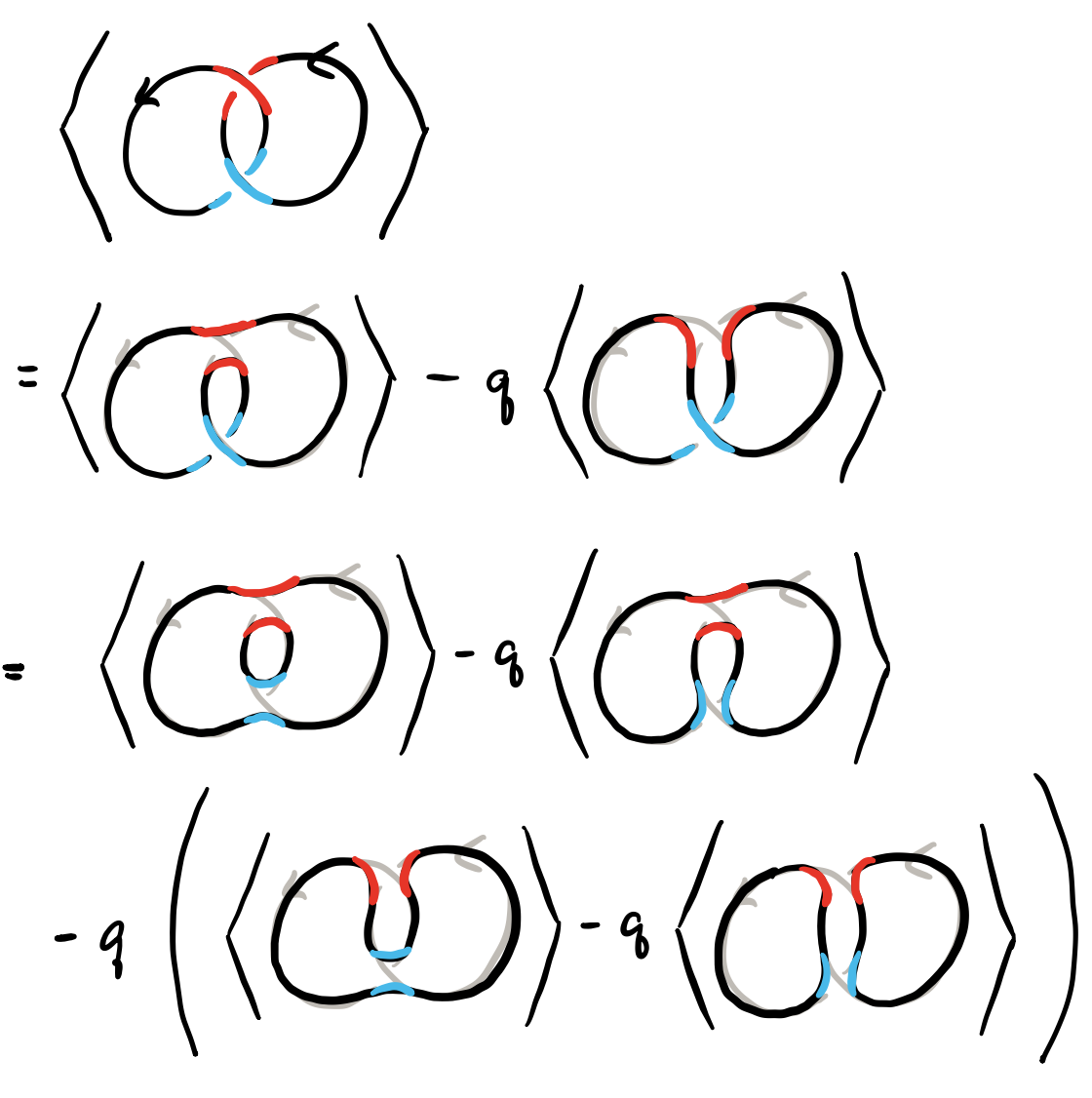}
\end{center}
\begin{align*}
    &= (q+q\inv)^2  -2q(q+q\inv) + q^2(q+q\inv)^2 \\
    &= q^4 + q^2 + 1 + q^{-2}.
\end{align*}

Therefore 
\begin{align*}
    \hat J(H) &= (-1)^{n_-}q^{n_+ - 2n_-} \langle H \rangle \\
    &= q^2(q^4 + q^2 + 1 + q^{-2}) = q^6 + q^4 + q^2 + 1.
\end{align*}

The normalized Jones polynomial, using these conventions, would be $\frac{\hat J(H)}{q+q\inv} = q^5 + q$. 
\end{example}

\subsection{Jones polynomial via the Khovanov bracket}

Khovanov homology $\Kh(\cdot)$ is a $(\Z \oplus \Z)$-graded (co)homology theory whose graded Euler characteristic recovers the Jones polynomial. The extra grading is usually called the \emph{quantum grading} or the \emph{internal grading}.\footnote{Sometimes I may say `degree' instead of grading. 
    I will likely pontificate on the terms \emph{grading} and \emph{degree} later in the course.}

\begin{definition}
    Let $C = (\bigoplus_{i,j \in \Z} C^{i,j}, d)$ be a bigraded chain complex of $\Z$-modules where, 
    for each $j$, $\bigoplus_{i \in \Z} C^{i,j}$ is finite rank. 

    The \emph{graded Euler characteristic} of $C$ is 
    \[
        \chi_q(C) = 
        \sum_{i,j \in \Z} (-1)^i q^j \rk H^{i,j}(C).
    \]
\end{definition}

Recall that the Euler characteristic of a CW complex can be computed using any CW decomposition; in particular, you do not need to compute the differentials in the CW chain complex to determine the Euler characteristic. 

Similarly, we do not need to know the differential $d$ to compute $\chi_q(C)$.
So, for now, we will define Bar-Natan's \emph{Khovanov bracket}, which lift the recursion in Definition \ref{def:kauffman-bracket} 
to the level of bigraded chain complexes:

\begin{definition}
    The \emph{Khovanov bracket} is defined by the axioms
    \begin{itemize}
        \item $\llbracket \emptyset \rrbracket 
        = \left ( 0 \to \underline{\Z} \to 0 \right )$
        \item $\llbracket \fullmoon L \rrbracket = V \otimes \llbracket L \rrbracket$ where $V = \Z v_+ \oplus \Z v_-$ of graded dimension $q + q\inv$. \mz{We will discuss the actual definition later.}
        \item $\llbracket \crossing \rrbracket
        = \Tot \left ( 0 \to 
            \underline{\llbracket \vertres \rrbracket} \map{d} 
            \llbracket \horizres \rrbracket\{1\} \to 0
        \right )$, where $\{1\}$ means `$q$-grading shift of 1' (see Notation \ref{nota:shift-functors}).
    \end{itemize}
    The terms of each chain complex at homological grading 0 are underlined.
    The \emph{totalization} functor flattens a multi-dimensional complex into a one-dimensional chain complex. \note{Since we are not considering the differential right now, we will not carefully define this here.} \alert{If the tensor product of chain complexes looks unfamiliar to you, this is something you should look up.} 
\end{definition}

\begin{theorem}[Khovanov, as interpreted by Bar-Natan]
Given a link diagram $L$ with $n_\pm$ crossings of sign $\pm$ respectively, the associated Khovanov chain complex is given by
\[
    \CKh(L) = \llbracket L \rrbracket[n_-]\{n_+ - 2n_-\}. 
\]
The square brackets indicate a shift in the homological grading (see Notation \ref{nota:shift-functors}).
Then
\[
    \hat J(L) = \chi_q(\CKh(L)) =  \chi_q(\Kh(L)).
\]
\end{theorem}

\begin{notation}
\label{nota:shift-functors}
We set our conventions for the grading shift functors $[n]$ and $\{m\}$ to agree with those appearing in \cite{BN-Kh}.
Let $A = \bigoplus A^{\bullet,\bullet}$ be a bigraded $\Z$-module. 
Then $A[n]\{m\}$ is the bigraded $\Z$-module where
\[
    \left ( A[n]\{m\} \right )^{i,j} = A^{i-n,j-m}.
\]
\note{Visually, if $A$ is plotted on the the $\Z \oplus \Z$ bigrading lattice, $A[n]\{m\}$ is obtained by grabbing $A$ and moving it by the vector $n e_1 + m e_2$. 
}
\end{notation}

\begin{example}
\label{eg:hopf-resolutions-cube}
We now reorganize the computation in Example \ref{eg:hopf-kauffman-computation}, as a primer for our formal introduction to Khovanov homology in the next section. 

Let $H$ denote both the following diagram as well as the underlying oriented Hopf link:
\begin{center}
    \includegraphics[width=1in]{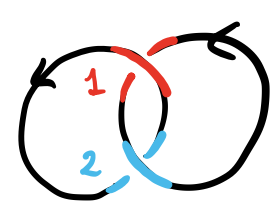}
\end{center}

To compute $\llbracket H \rrbracket$, we draw the \emph{cube of complete resolutions}:
\begin{center}
    \includegraphics[width=4in]{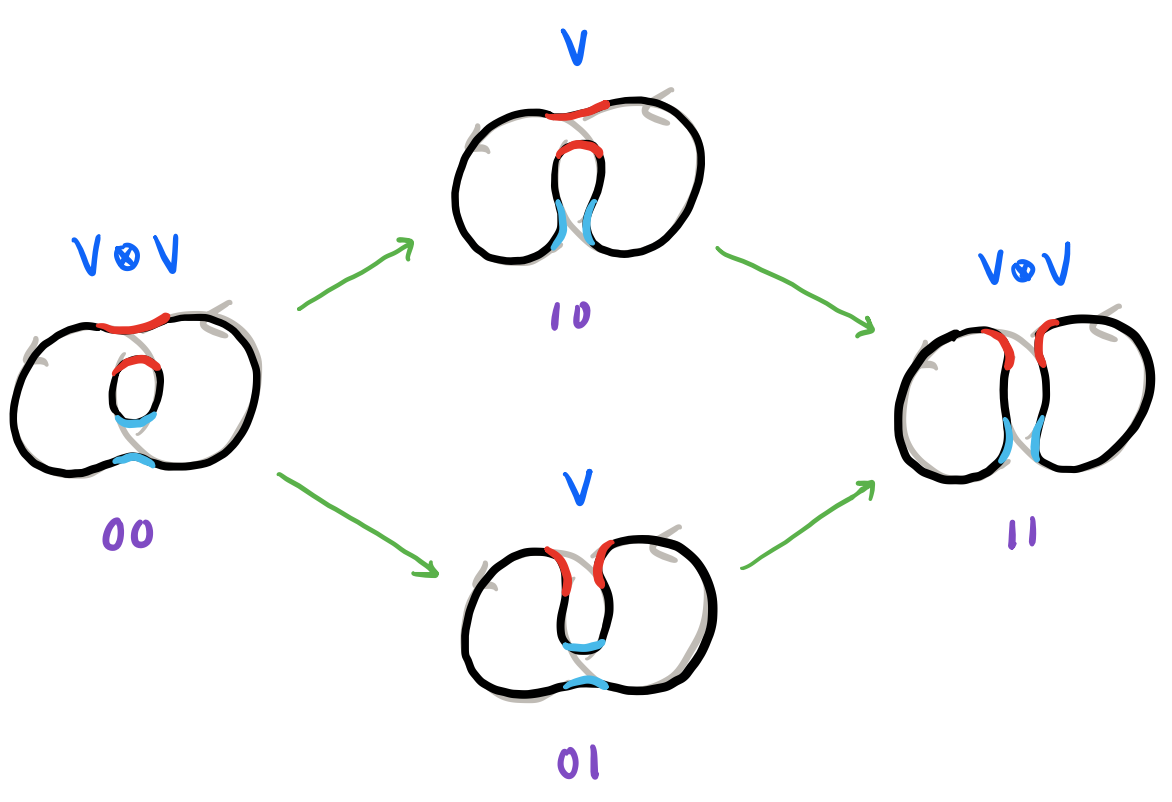}
\end{center}
We associate $V^{\otimes c}$ to each complete resolution containing $c$ closed components.
Each complete resolution corresponds to a binary string $u \in \{0,1\}^2$, and we perform a quantum grading shift of $|u|$ on associated $\Z$-module, where $|u|$ is the number of 1's appearing in the bitstring $u$. 

This yields the Khovanov bracket
\[
    \llbracket H \rrbracket 
    = \left (
        \underline{V \otimes V} 
        \to 
        V \oplus V \{1\} 
        \to
        V \otimes V\{2\}
    \right ).
\]
(Recall that the underline indicates the chain group at homological grading 0.)

The Khovanov chain complex (without specifying differentials) is therefore
\begin{align*}
    \CKh(H) &= \left (
        \underline{V \otimes V} 
        \to 
        V \oplus V \{1\} 
        \to
        V \otimes V\{2\}
    \right )[n_-]\{n_+ - 2n_-\}\\
    &= \left (
        \underline{V \otimes V} 
        \to 
        V \oplus V \{1\} 
        \to
        V \otimes V\{2\}
    \right )[0]\{2\}\\
    &= \left (
        \underline{V \otimes V} \{2\}
        \to 
        V \oplus V \{3\} 
        \to
        V \otimes V\{4\}
    \right ). \\
\end{align*}

Schematically, can visualize the bigraded chain groups as follows, where each \textbullet\  represents a copy of $\Z$: 

\begin{center}
\begin{tabular}{c || c | c | c || c }
&  & & & $\chi$: \\
\hline
$\gr_q=6$ & & & \textbullet  & $0 - 0 + 1 = 1$ \\
$\gr_q=4$ & \textbullet  & \textbullet  \textbullet & \textbullet \textbullet & $1 - 2 + 2 = 1$ \\
$\gr_q=2$ & \textbullet  \textbullet  & \textbullet  \textbullet & \textbullet & $2-2+1 = 1$ \\
$\gr_q=0$ & \textbullet  & & & $1-0+0=1$ \\
\hline
& $\gr_h = 0$ & $\gr_h = 1$ & $\gr_h = 2$ & \\ 
\end{tabular}
\end{center}
Here $\gr_h$ and $\gr_q$ denote the homological and quantum grading, respectively.

We conclude that, indeed, $\chi_q(\CKh(H)) = q^6 + q^4 + q^2 + 1 = \hat J(H)$. 
\end{example}

\begin{example}
    For a very comprehensive computation of the Jones polynomial of the trefoil using the Kauffman bracket, see Equation (1) of \cite{BN-Kh}. 
\end{example}

\section{Khovanov homology}

\subsection{The Khovanov chain complex $\CKh(D)$}

Given an oriented link diagram $D$ representing a link $L$, the \emph{Khovanov chain complex} $\CKh(D)$ is a
$(\Z\oplus \Z)$-graded chain complex of abelian groups. 
We define this chain complex throughout this section. 
Throughout, we will use the term \emph{bigraded} in place of `$(\Z\oplus \Z)$-graded'.

\subsubsection{Cube of resolutions}

\note{The cube of resolutions will set up the topological story that the chain complex tells. Afterwards, we will replace each part of the cube of resolution with algebraic objects, to fully define the Khovanov chain complex.}

Let $n$ be the number of crossings in the diagram $D$, and pick an ordering for the crossings. Let $c_i$ denote the $i$th crossing.

The $n$-dimensional \emph{binary cube} $\{0,1\}^n$ is the poset of \emph{binary strings} (a.k.a. \emph{bitstrings}) of length $n$, with partial order given by $0 \prec 1$ and lexicographic order. 

\begin{notation}
Recall some terms and notations used when discussing partially ordered sets:
    \begin{itemize}
        \item If $u\prec v$, then we say $u$ \emph{precedes} $v$ (and $v$ \emph{succeeds} $u$).
        \item If $u \prec v$ and they differ at exactly one bit, then we say $u$ is an \emph{immediate predecessor} of $v$ (and $v$ is an \emph{immediate successor} of $u$). In this case, we write $u \prec_1 v$. 
    \end{itemize}
We may think of a poset as a directed graph, with and edge $u\to v$ if $u \prec_1 v$. 
\end{notation}

\begin{notation}
\label{nota:edge-star-notation}
For convenience, if $u \prec_1 v$ and they differ at the $i$th bit, then we use the following length-$n$ string in $\{0,1,*\}^n$ to denote the edge $u \to v$:
\[
    u_1u_2\cdots u_{i-1} * u_{i+1} \cdots u_n
    = v_1v_2\cdots v_{i-1} * v_{i+1} \cdots v_n
\]
\end{notation}

For $u \in \{0,1\}^n$, let $D_u$ denote the complete resolution of the diagram $D$ where crossing $c_i$ is smoothed according to the bit $u_i$:
\[
    \vertres \xleftarrow{0} \crossing \map{1} \horizres
\]

If $u \prec_1 v$, then $D_u$ and $D_v$ differ only in the neighborhood of one crossing, called the \emph{active crossing} of the edge $u \prec_1 v$. All other crossings are \emph{passive}. \

Also, $u \prec_1 v$ represents either 
\begin{itemize}
    \item a \emph{merge} of two \emph{circles} (i.e.\ closed components) of $D_u$ into one circle in $D_v$ or
    \item a \emph{split} of one circle of $D_u$ into two in $D_v$.
\end{itemize} 
Any of these circles are called \emph{active circles} of the edge $u \prec_1 v$. 
All other circles are \emph{passive}; they look the same in both $D_u$ and $D_v$.

Finally, when assigning gradings, we will need the \emph{Hamming weight} of bitstrings:
\[
    |u| = \sum_i u_i.
\]

\subsubsection{Chain groups, distinguished generators, gradings}

\emph{We now describe the \emph{distinguished generators} of the Khovanov chain complex.}

Our chain groups will be bigraded by $\gr = (\gr_h, \gr_q)$:
\begin{itemize}
    \item The first grading is called the \emph{homological} grading, denoted by $\gr_h$. Its shift functor is denoted by $[\cdot]$. 
    \item The second grading is called the \emph{quantum} (or \emph{internal}) grading, denoted by $\gr_q$. Its shift functor is denoted by $\{\cdot \}$. 
\end{itemize}

Let $V$ denote the bigraded  $\Z$-module $\Z v_+ \oplus \Z v_-$ 
with generators $v_\pm$ in bigradings $(0,\pm 1)$.

Let $|D_u|$ denote the number of circles in the resolution $D_u$.
The chain group lying above vertex $u$ of the cube is 
$V^{\otimes |D_u|}[|u|]\{|u|\}$ generated by the $2^n$ length-$n$ pure tensors 
\[
    \{ v_\pm \otimes \cdots \otimes v_\pm \}.
\]
These are called the \emph{distinguished} (i.e.\ chosen) generators of the chain group at this vertex.

\begin{remark}
The bigrading $\gr$ for a distinguished generator in $V^{\otimes k}$ is determined by the bigrading on $V$:
\[
    \gr(a \otimes v_\pm) = \gr(a) + \gr(v_\pm). 
\]
For example, $\gr(v_+ \otimes v_- \otimes v_-) = (0,-1)$. 
\end{remark}

\begin{remark}
Note that when we actually write down a chain complex, we implicitly homologically shift the chain groups. 
For example, the following chain complex is acyclic\footnote{i.e.\ homology vanishes}:
\[
    \underline{V} \map{1} V.
\]
If we view the chain complex as a graded module $C$, and the differential as an endomorphism $d$, then 
\begin{itemize}
    \item $C = V \oplus V[1]$ and 
    \item $d = \begin{bmatrix} 0 & 0 \\ 1 & 0 \end{bmatrix}$
\end{itemize}
This \emph{mapping cone} point of view will become useful later in the course.
\note{Now is a good time to review/read about \emph{mapping cones} in homological algebra. See \S\ref{sec:mapping-cones} for a review, and also some comments about the conventions we are using.}
\end{remark}

\subsubsection{Differentials}

To each edge of the cube $u\to v$, we assign a map according to whether the edge represents a merging of two circles or the splitting of one circle.
\note{We casually call these the `edge maps'. }

\begin{itemize}
\item If $u \to v$ represents a merge, the map on tensor components corresponding to active circles is given by:
\begin{align*}
    m: V \otimes V &\to V \\
    v_+ \otimes v_+ &\mapsto v_+ \\
    v_+ \otimes v_- , v_- \otimes v_+ &\mapsto v_- \\
    v_- \otimes v_- &\mapsto 0
\end{align*}

\item If $u \to v$ represents a split, the map on active components is given by:
\begin{align*}
    \Delta: V &\to V \otimes V \\
    v_+ &\mapsto v_+ \otimes v_- + v_- \otimes v_+ \\
    v_- &\mapsto v_- \otimes v_-
\end{align*}
\end{itemize}
The map is identity on the passive components of the tensor product. 

\begin{example}
Here is a diagram for the right-handed trefoil, with a choice of ordering on the three crossings (in blue):
\begin{center}
    \includegraphics[height=1in]{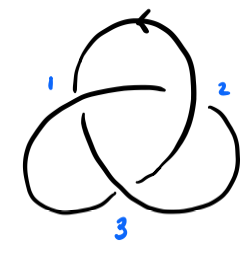}
\end{center}
In the cube of resolutions, the edge corresponding to $101 \to 111$ is a merging of two circles into one, with the active crossing circled in dotted cyan:
\begin{center}
    \includegraphics[height=1.2in]{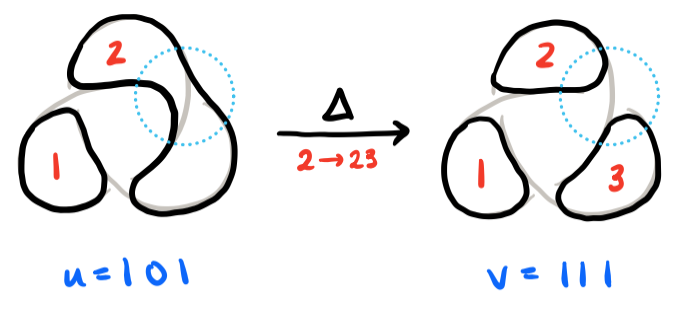}
\end{center}
We have chosen an ordering on the set of circles (in red) at each resolution so that we can identify the copies of $V$ in the tensor product at each vertex.
For example, the distinguished generator $v_- \otimes v_+$ at resolution $D_u$ labels the smaller circle $v_-$ and the larger $v_+$. 

The linear map $d_{uv}$ is given by the bundle of arrows shown below, where we use shorthand notation for compactness (e.g. $v_{+++} := v_+ \otimes v_+ \otimes v_+$):

\begin{center}
\begin{tikzpicture}
    \node (++) at (0,2) {\color{blue}$v_{++}$};
    \node (+-) at (-1,1) {\color{blue}$v_{+-}$};
    \node (-+) at (1,1) {\color{RedOrange}$v_{-+}$};
    \node (--) at (0,0) {\color{RedOrange}$v_{--}$};
\begin{scope}[xshift=4cm]
    \node (+++) at (0,3) {\color{blue}$v_{+++}$};
    \node (++-) at (-1,2) {\color{blue}$v_{++-}$};
    \node (+-+) at (0,2) {\color{blue}$v_{+-+}$};
    \node (-++) at (1,2) {\color{RedOrange}$v_{-++}$};
    \node (+--) at (-1,1) {\color{blue}$v_{+--}$};
    \node (-+-) at (0,1) {\color{RedOrange}$v_{-+-}$};
    \node (--+) at (1,1) {\color{RedOrange}$v_{--+}$};
    \node (---) at (0,0) {\color{RedOrange}$v_{---}$};
    \draw[->] (++) to [bend left] (++-);
    \draw[->] (++) to [bend left] (+-+);
    \draw[->] (-+) to [bend left] (-+-);
    \draw[->] (-+) to [bend left] (--+);
    \draw[->] (+-) to [bend right] (+--);
    \draw[->] (--) to [bend right] (---);
\end{scope}
\end{tikzpicture}
\end{center}
Observe that $d_{uv}$ is identity on the passive circle labeled `1' in both $D_u$ and $D_v$. 

\end{example}

\begin{exercise} 
\label{ex:edge-map-properties}
\alert{(Important)}
\bea
\item Verify that the merge and split maps, as written, decrease the quantum grading by $1$. 

So, as part of the differential, we modify the merge and split maps to be $\gr_q$-preserving maps 
    \[
        d_{uv}: V \otimes V \to V\{1\} \qquad 
        \text{or}
        \qquad 
        d_{uv}: V \to V\otimes V\{1\},
    \]
depending on whether the edge $u\to v$ corresponds to a merge or a split, respectively.

\item Verify that along each 2D face of the binary cube, the edge maps commute. 

Therefore, in order to get a chain complex (where $d^2 =0$), we will need to add some signs so that the faces instead \emph{anticommute}.

\item For an edge $u \to v$ in the cube with active crossing $c_i$, associate the following sign:
    \[
        s_{uv} = (-1)^{\sum_{j=1}^{i-1} u_i}.
  \]
In other words, $s_{uv}$ measures the parity of the number of 1's appearing before the $*$ in the label given to the edge in the binary cube (see Notation \ref{nota:edge-star-notation}).

Verify that, for any face of the binary cube, and odd number of the four edges bounding that face will have sign assignment $-1$.

\note{This is not the only possible sign assignment.}

\ee
\end{exercise}

\subsubsection{Global grading shifts and homology}

To compute Khovanov homology of an oriented link $L$ using diagram $D$:
\begin{enumerate}
    \item Draw the cube of resolutions. 
    \item Associate modules $V^{\otimes |D_u|}$ to each vertex $u$ of the cube.
    \item Associate linear maps $s_{uv}d_{uv}$ to each edge $u\to v$.
    \item Flatten the complex, by taking direct sums along Hamming weights; the resulting complex is the Khovanov bracket, $\llbracket D \rrbracket$.
    \item Add in the global bigrading shift $\{-n_-\}[n_+ - 2n_-]$ to get the \emph{Khovanov chain complex}:
        \[
            \CKh(D) = \llbracket D \rrbracket \{-n_-\}[n_+ - 2n_-].
        \]
    \item Take homology to get \emph{Khovanov homology}:
        \[
            \Kh(L) = H^*(\CKh(D)).
        \]
\end{enumerate}

The Khovanov chain complex $\CKh(D)$ is bigraded, and its differential $d_{Kh}$ is a bidegree $(1,0)$ endomorphism. Therefore $\CKh(D)$ is really the direct sum of many chain complexes, one for each quantum grading:
\[
    \CKh(D) = \bigoplus_{j \in \Z} \CKh^{\bullet,j}(D).
\]
The homology is bigraded by the homological and quantum gradings (indexed below by $i$ and $j$, respectively):
\[
    \Kh(L) = \bigoplus_{i,j \in \Z} \Kh^{i,j}(L).
\]

\begin{example}
We now finally compute the Khovanov homology of the Hopf link $H$ from Examples \ref{eg:hopf-kauffman-computation} and \ref{eg:hopf-resolutions-cube}. 

In order to uniquely identify the generators at the different resolutions in the cube, we label the resolutions $u,v,w,x$ and use these letters to denote the distinguished generators at each resolution:
\begin{center}
    \includegraphics[width=\textwidth]{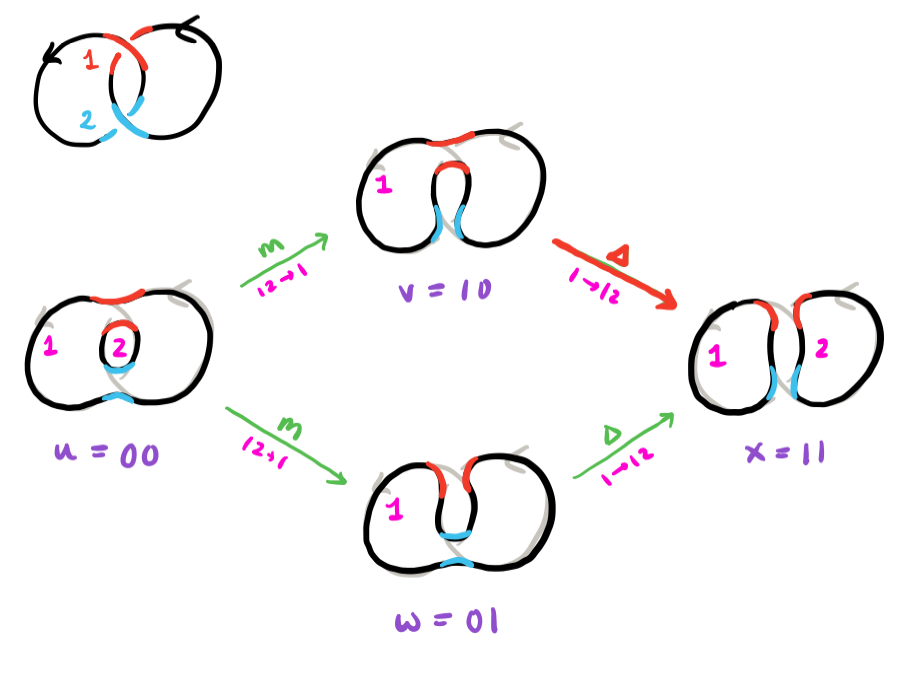}
\end{center}
\begin{itemize}
    \item The vertices of the binary cube are labeled in {\color{violet}purple}.
    \item The ordering of the tensor factors (i.e.\ circles) at each resolution is shown in {\color{magenta}pink}. Below each arrow, the {\color{magenta}pink} text indicates the active circles in the source and target of that edge.
    \item The only edge with sign assignment $-1$ is the split map shown in bold {\color{red}red}. 
\end{itemize}
In shorthand, the Khovanov chain complex $\CKh(H)$ (i.e.\ with global shifts incorporated) is the following chain complex (direct sums are taken vertically):
\begin{center}
    \includegraphics[width=\textwidth]{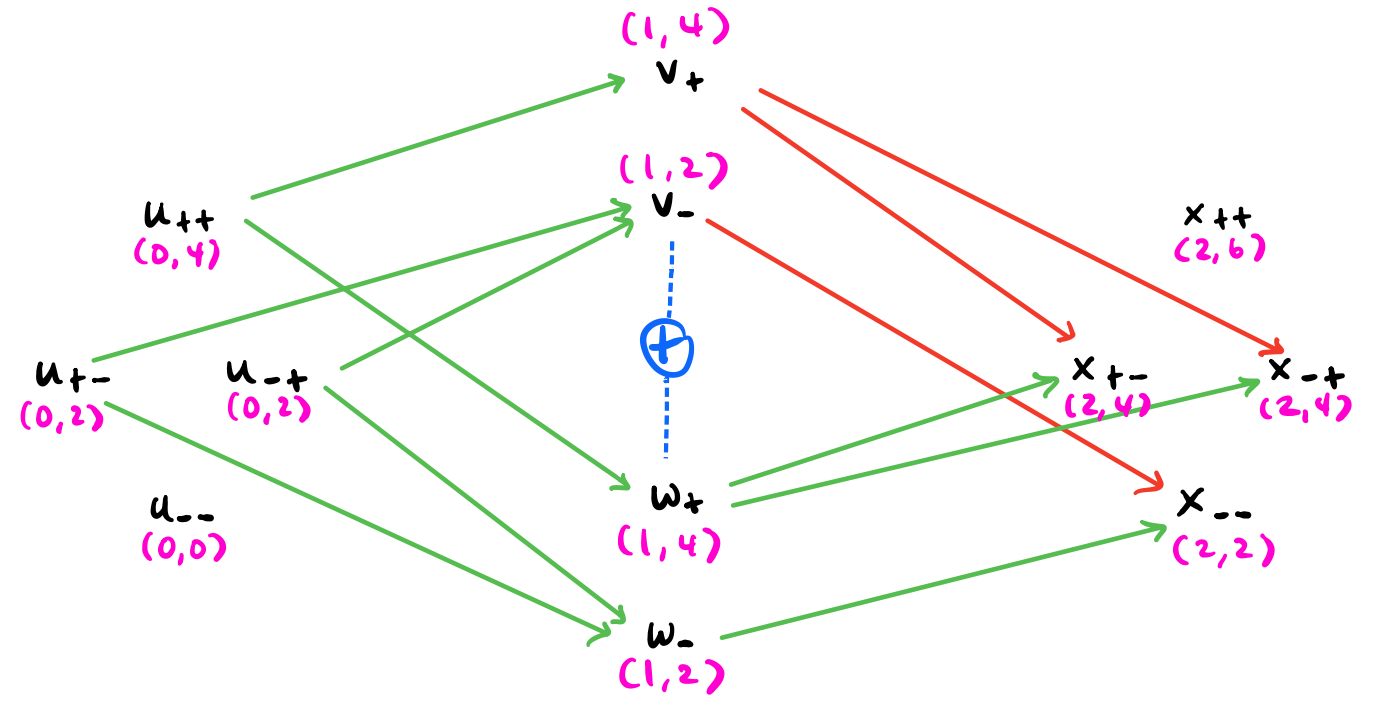}
\end{center}
The bigradings of the distinguished generators of $\CKh(H)$ are shown in {\color{magenta}pink}.

The reader may now verify that the Khovanov homology of the positively-linked Hopf link $H$ is the bigraded $\Z$-module 
\renewcommand{\arraystretch}{1.5}
\begin{center}
\begin{tabular}{|c|c|c|c|}
    \hline
    $\gr_q = 6$  &  &  &  $\Z$ \\ \hline
    $\gr_q = 4$  &  &  &  $\Z$ \\ \hline
    $\gr_q = 2$  & $\Z$ &  & \\ \hline
    $\gr_q = 0$  & $\Z$ &  & \\ \hline 
       & $\gr_h = 0$ & $\gr_h = 1$ & $\gr_h = 2$ \\ \hline
\end{tabular}
\end{center}
generated by the homology classes 
\begin{center}
\begin{tabular}{|c|c|c|c|}
    \hline
    $\gr_q = 6$  &  &  &  $[x_{++}]$ \\ \hline
    $\gr_q = 4$  &  &  &  $[x_{+-}] = [x_{-+}]$ \\ \hline
    $\gr_q = 2$  & $[u_{+-} - u_{-+}]$ &  & \\ \hline
    $\gr_q = 0$  & $[u_{--}]$ &  & \\ \hline 
       & $\gr_h = 0$ & $\gr_h = 1$ & $\gr_h = 2$ \\ \hline
\end{tabular}
\end{center}
\end{example}

\begin{exercise}
The diagrams $D$ and $D'$ below both represent the unknot, $U$.
\begin{center}
    \includegraphics[height=.8in]{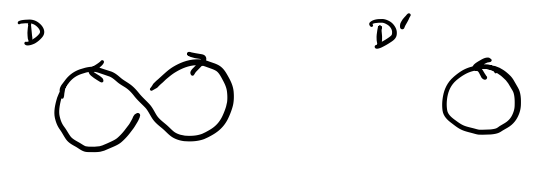}
\end{center}
\bea
\item  
Compute the Khovanov chain complex $\CKh$ for both, and then compute homology to verify that they indeed agree.

\item Can you prove that Khovanov homology is invariant under the following Reidemeister 1 move?
\begin{center}
    \includegraphics[height=.8in]{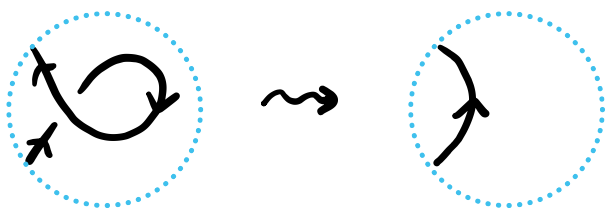}
\end{center}
\note{Bar-Natan's introductory paper does cover this, but try first to use intuition from your solution to part (a) to find the appropriate chain homotopy equivalence.}
\ee
\end{exercise}

\subsection{Algebraic aside: Mapping cones}
\label{sec:mapping-cones}

In a category of chain complexes graded cohomologically, let $f: A \to B$ be a chain map between complexes 
$A = (\bigoplus_i A^i, d_A^i)$ and $B = (\bigoplus_i B^i, d_B^i)$:
\begin{center}
    \includegraphics[width=4in]{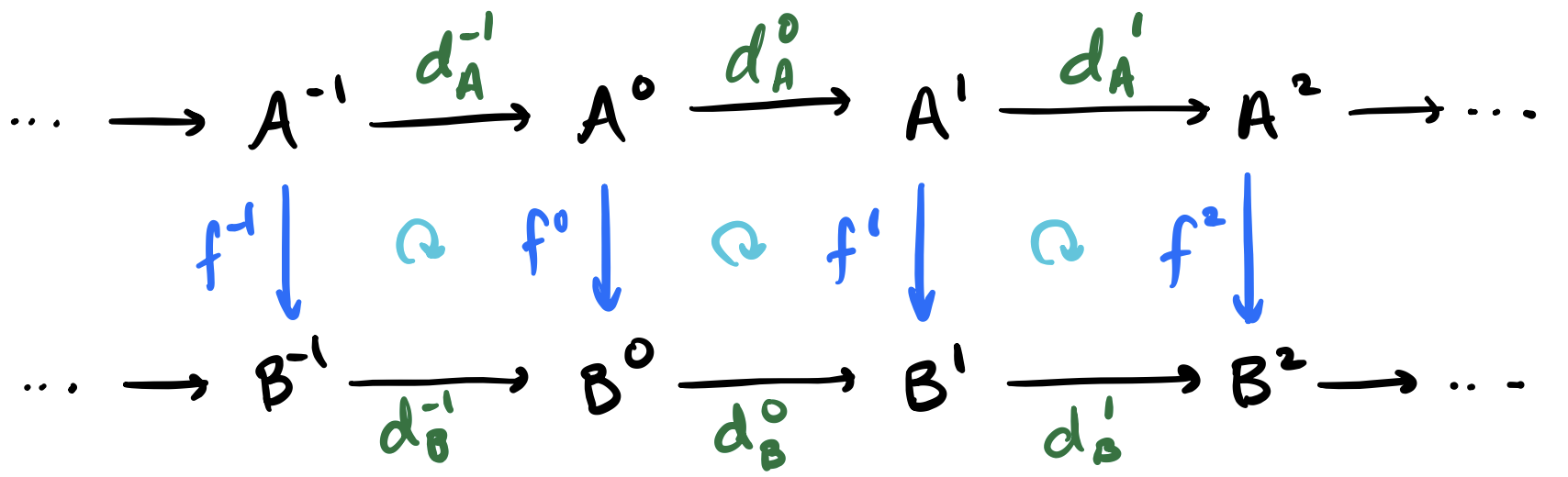}
\end{center}

The \emph{mapping cone} of $f$, denoted by  $C(f)$ or $\cone(f)$, is defined as 
\[
    C(f) := \left (
    \bigoplus_i (A^{i+1} \oplus B^i) ,  \quad
    \bigoplus_i 
        \begin{pmatrix}
            -d_A^{i+1} & 0 \\
            f^{i+1} & d_B^i
        \end{pmatrix}
    \right)
\]
and can be visualized as a collapse of previous diagram:
\begin{center}
    \includegraphics[width=4in]{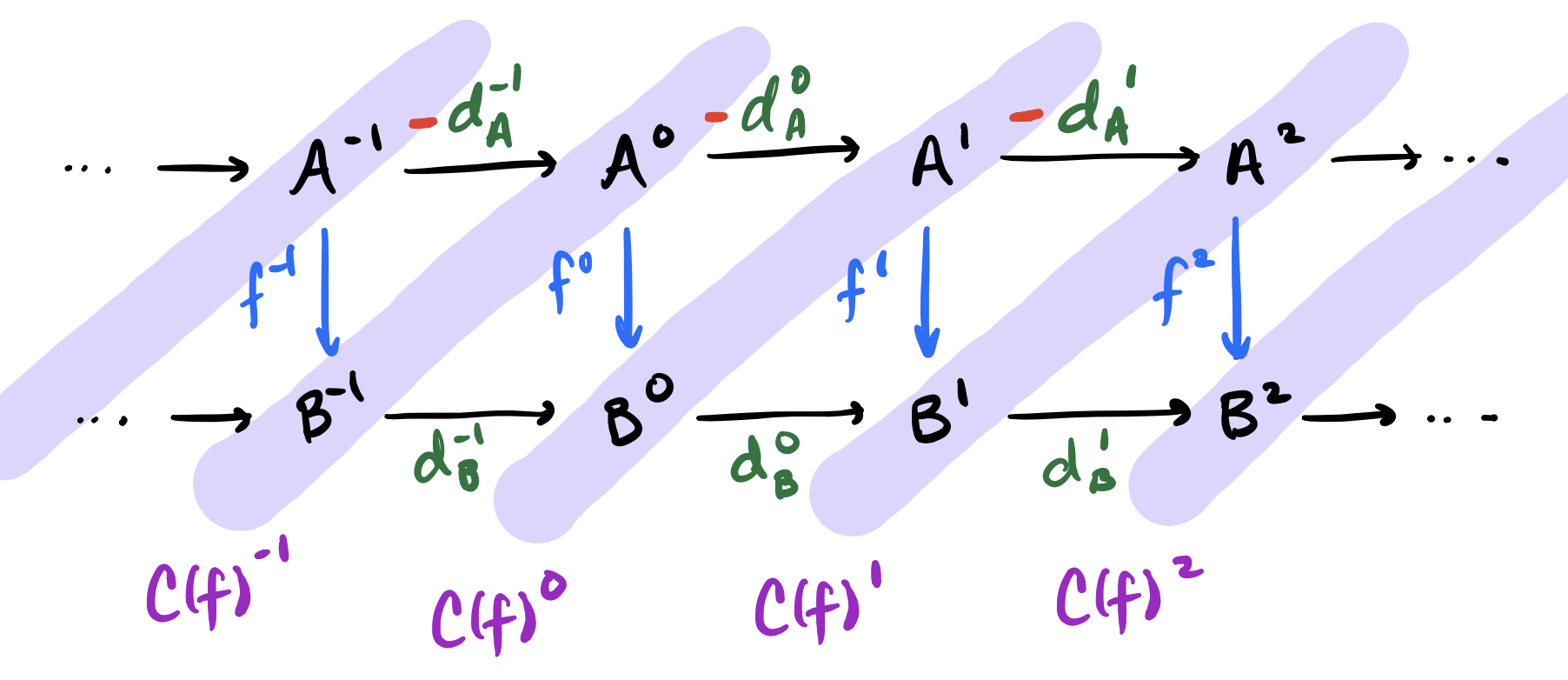}
\end{center}
In the previous diagram, the squares commuted (because $f$ is a chain map). In order to get $d_{C(f)}^2 = 0$, we instead need the squares to \emph{anticommute}, and so we made the choice to negate all the instances of the $d_A^i$ differentials. 

\begin{warning}
Recall that the shift functor $[n]$ in Bar-Natan's conventions \cite{BN-Kh} is defined by 
\[
    (A[n])^i = A^{i-n}
\]
so that the quantum dimension of  $(A[n]^i)$ is  $q^n$ times the quantum dimension of $A^i$.
In these conventions, we may write 
\[
    C(f) = \left (
    A[-1] \oplus B,  
        \begin{pmatrix}
            d_{A[-1]} & 0 \\
            f[-1] & d_B
        \end{pmatrix}
    \right)
\]
where we use the convention that $d_{A[-1]} = -d_A$.

This is not how you'd usually see the definition of a cone  in a cochain complex category, so I encourage you to instead carefully understand the cone construction without memorizing the notation.
We will, however, continue using these conventions for the sake of having useful notation.
\end{warning}

In practice, we usually just write the cone as 
\[
    C(f) = \left ( A \map{f} B \right ),
\]
with the parentheses indicating that we are thinking of a single complex, i.e.\ the whole cone, rather than a map between complexes. 
\note{We can think of $B$ as `shifted by 0 homologically', and instead write 
\[
    C(f) = \left ( A \map{f} \underline{B} \right )
\]
to be more clear that $A$ is the one being shifted left.}

Notice that $B \into C(f)$ as a subcomplex, and $A[-1] \cong C(f)/ B$ is the quotient. 
Thus there is a short exact sequence 
\[
    0 \to B \map{i} C(f) \map{p} A[-1] \to 0
\]
where $i$ is inclusion and $p$ is the quotient map.
This gives rise to a long exact sequence on homology
\[
\cdots \to H^j(B) \map{i^*} H^j(C(f)) \map{p^*} H^{j+1}(A)
\map{f^*} H^{j+1}(B) \map{i^*} H^{j+1}(C(f)) \to \cdots
\]
where the connecting map $f^*$ is induced by $f$. 

\begin{remark}
    Once again, because of our conventions, we have a somewhat nonstandard-looking exact triangle:
    \[
        A \to B \to C(f) \to A[-1].
    \]
    But all is well; this is just a notational difference from what you might be used to.
\end{remark}

\subsection{An underlying TQFT}

The algorithm/definition in the previous section is concrete, but you might be wondering where these $m$ and $\Delta$ maps come from. The answer lies in the fact that morphisms in the category of bigraded $\Z$-modules are images of cobordisms under a functor from a more topologically defined category. In this section, we describe an\footnote{We will see variations later, but focus first on the one that's easiest to work with.} underlying TQFT that determines the maps $m$ and $\Delta$ from the previous section.

\subsubsection{TQFTs}
QFTs (quantum field theories) and TQFTs (topological quantum field theories), have a rich history in mathematical physics. 
The purposes of this course, we will use the following simplified (i.e.\ vague) definition, adapted from \cite{Atiyah-tqft}:

\begin{definition}(Vague\footnote{We are more interested in specific TQFTs, so will not emphasize the details here.})
Let $R$ be a commutative ground ring. 
An $(n+1)$-dimensional \emph{TQFT} is a functor $Z$ from a category of closed $n$-dimensional manifolds and  $(n+1)$-dimensional cobordisms\footnote{up to some equivalence relation} between them to a category of finitely generated (see Remark \ref{rmk:tqft-pivotal-category}) $R$-modules, such that 
\begin{itemize}
    \item $Z$ is \emph{multiplicative}: $Z(Y \sqcup Y') = Z(Y) \otimes Z(Y')$. 
    \item $Z$ is \emph{involutory}: If $\bar Y$ is $Y$ with the opposite orientation, then $Z(\bar Y) = Z(Y)^*$ (the dual module). 
\end{itemize}
We also naturally would like $Z$ to send the identity cobordism $Y \times I: Y \to Y$ to the identity map. 
\end{definition}

Functoriality implies that if $C_{01}: Y_0 \to Y_1$ and $C_{12}: Y_1 \to Y_2$ are cobordisms, then
\[
    Z(C_{12}\circ C_{01}) = Z(C_{12}) \circ Z(C_{01}): Z(Y_0) \to Z(Y_2).
\]

\note{Make sure you know what \emph{category} and \emph{functor} mean. Things will get confusing from here on out if these terms aren't clear.}

\begin{remark}
    To learn more about physical origins of the term \emph{topological quantum field theory}, take a look at the \href{https://ncatlab.org/nlab/show/topological+quantum+field+theory}{ncatlab page}.
   
    Atiyah's \textit{Topology quantum field theories} \cite{Atiyah-tqft} provides a set of precise axioms for $(n+1)$-TQFTs and surveys some prominent examples. 
\end{remark}

\begin{remark}
\label{rmk:tqft-pivotal-category}
Since any (orientable) manifold only has two orientations, the category of cobordisms is \emph{pivotal}, i.e.\ $A \cong (A^*)^*$ for any object $A$. 
\begin{enumerate}
\item Thus the target category of the functor needs to be pivotal as well; this is why we require \emph{finitely generated} $R$-modules. 

\item Since our categories are pivotal we can identify any cobordism $C: K \to L$  with various `bent' versions of $C$, shown in the schematics below:
\begin{center}
    \includegraphics[width=3in]{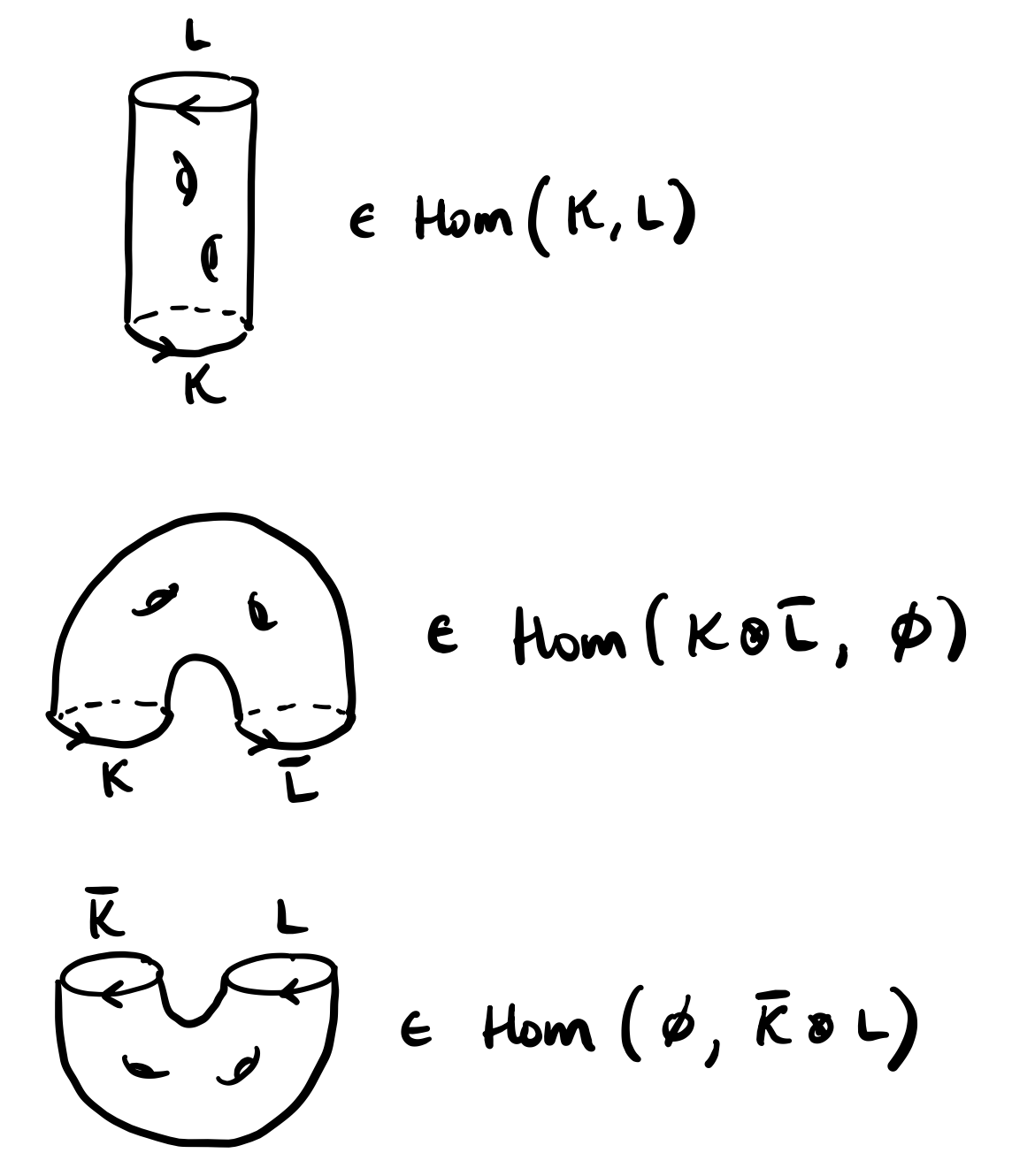}
\end{center}
We thus get equivalences
\[
    \Hom(K, L) \cong \Hom(K \otimes \bar L, \emptyset) \cong \Hom(\emptyset, \bar K \otimes L).
\]

\item We can view the equivalence 
\[
    \Hom(Z(K) \otimes Z(\bar L), Z(\emptyset)) 
    \cong 
    \Hom(Z(K), Z(L))
\]
as an instance of the
Tensor-Hom Adjunction:
\[
    \Hom(M \otimes N,P) \cong \Hom(M, \Hom(N,P))
\]
by setting $M = Z(K)$, $N = Z(\bar L) = Z(L)^*$, and $P = Z(\emptyset) = R$. 
(Note that $\Hom(N,P) = \Hom(Z(L)^*, R) = (Z(L)^*)^* \cong Z(L)$.)

\end{enumerate}
\end{remark}

\subsubsection{Bar-Natan's dotted cobordism category and TQFT}
In this section, we will focus on one particular TQFT, which we will denote $\BNfun$. This material is adapted from \cite{BN-tangles}. 

First, we need to understand the source category.

\begin{definition}
\label{defn:preadditive-category}
    A small category $\CC$ is \emph{preadditive} (a.k.a.\ \emph{$\Zmod$-enriched})if for $X,Y \in \Ob(\CC)$, $\Hom_\CC(X,Y)$ is an abelian group (i.e.\ a $\Z$-module) under composition, and this composition is bilinear (under the action of $\Z$). 
\end{definition}

\begin{remark}
\label{rmk:additive-category}
A category is \emph{additive} if additionally, we have finite coproducts. We will boost up to additive categories later in \S \ref{sec:BN-Kom-categories}.
\end{remark}

\begin{definition}
\label{def:TLcat-0}
    The preadditive category $\TLcat_0$ is defined as follows:
\begin{itemize}
    \item Objects are closed 1-manifolds with finitely many components embedded smoothly in the plane $\R^2$; we call these \emph{planar circles}. \note{Notice that we are \emph{not} identifying isotopic embeddings, but we \emph{do} ignore parametrization.}
    \item Morphisms are finite sums of \emph{dotted cobordisms}, or smooth surfaces embedded in $\R^2 \times I$,
    \begin{itemize}
        \item with boundary only in $\R^2 \times \partial I$,
        \item possibly decorated with a finite number of dots,
        \item up to boundary-preserving isotopy, and
        \item subject to the following relations:        
    \end{itemize}
    \begin{center}
        \begin{equation}
        \label{eq:dotted-BN-relations}`
        \includegraphics[width=4in]{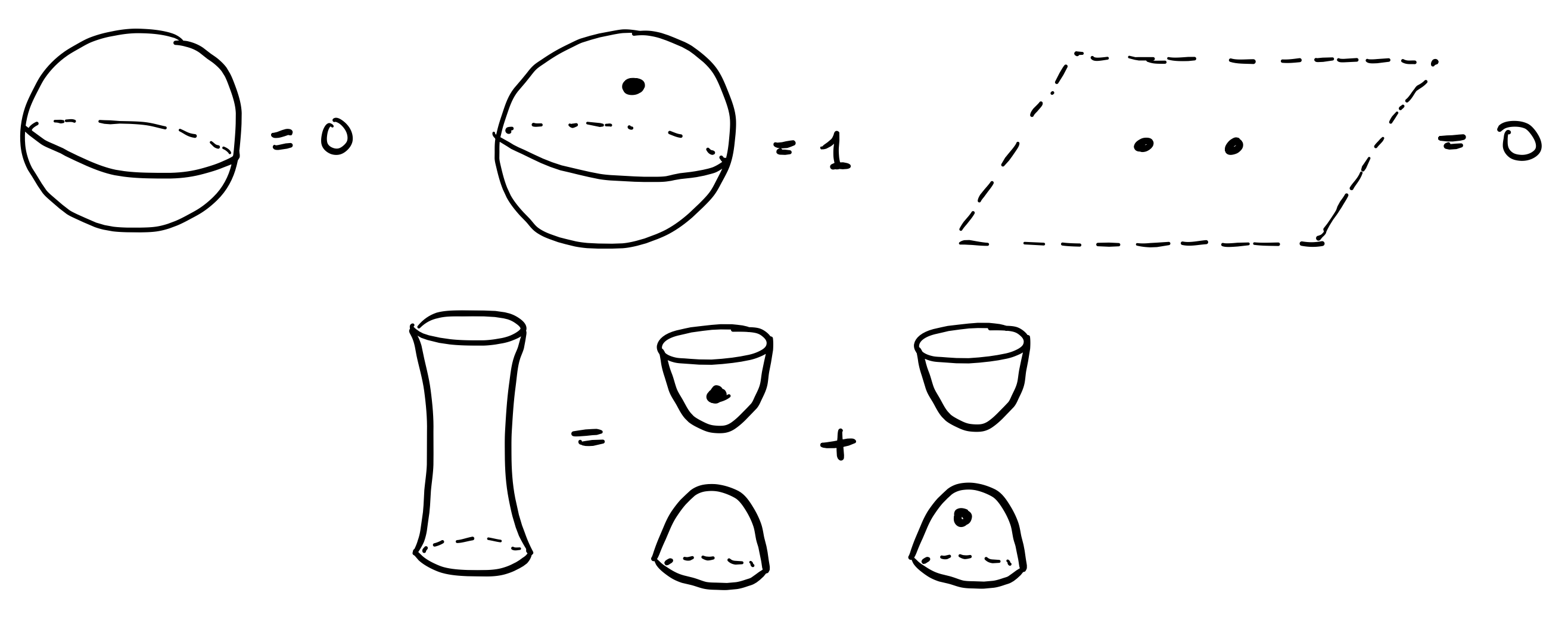}
        \end{equation}
    \end{center}
    The last relation is often referred to (gruesomely) as the \emph{neck-cutting} relation.
\end{itemize} 
\end{definition}

The target category will be bigraded $\Z$-modules, which we will denote as $\ggmod$ when we want to emphasize the bigrading, or just $\Zmod$ for short. The most important module for us will be the one associated to a single circle in the plane, which we previously called $V$. 
We will now rename this module and view it as $V \cong \CA\{1\}$, where 
\[
    \CA = \Z[X]/(X^2)
\]
with $\deg_{q}(X) = -2$. 

\begin{definition}
\label{def:BN-functor}
    The functor $\BNfun: \TLcat_0 \to \ggmod$ assigns
    \begin{itemize}
        \item to each collection of $k$ planar circles the module $(\CA\{1\})^{\otimes k}$ (with $\emptyset \leadsto \Z$)
        \item to each dotted cobordism a linear map determined by the following assignments:
        \begin{center}
            \includegraphics[width=2.5in]{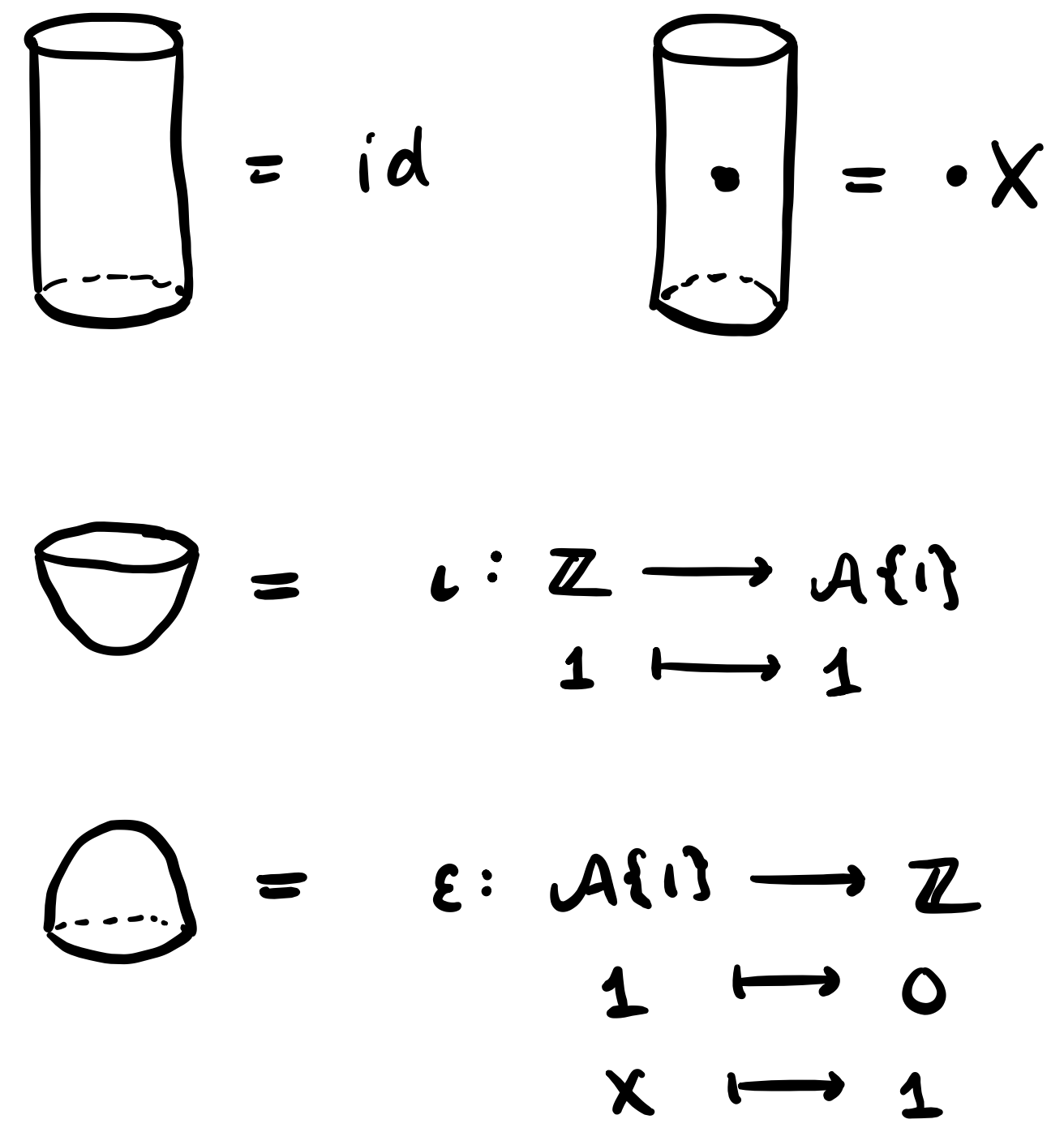}
        \end{center}
    \end{itemize}
\end{definition}

\begin{remark}
    Notice that Definition \ref{def:BN-functor} is really a Definition-Theorem, since one needs to check that it really is a functor. 
    \mz{What does one need to check to make sure that $\BNfun$ really is a functor?}
\end{remark}

Let $C_1$ and $C_2$ be cobordisms $\emptyset \to P$, where $P$ is a collection of planar circles. Then there is a $\Z$-valued pairing 
\[
    \langle C_1, C_2 \rangle = \BNfun( C_1 \cup_\partial \bar C_2)  \qquad (= \BNfun(\bar C_1 \cup_\partial C_2) )
\]
because $C_1 \cup_{\partial} \bar C_2$ is a closed 2-manifold, and the relations in $\TLcat_0$ allow us to evaluate closed surfaces.

In the following exercise, we will use this pairing to recover the linear maps $m$ and $\Delta$ from the previous section.

\begin{exercise}
\textit{(Not eligible for HW submission; we will complete these in lecture.)}
    \bea
    \item In Definition \ref{def:BN-functor}, why did we not need to specify how $\BNfun$ assigns higher-genus cobordisms?
    \item We can identify the distinguished $\Z$-module generators $1$ and $X$ in $\CA\{1\}$ with a cup and a dotted cup, respectively. Why does this make sense? 
    \item What is the dual cobordism to the cup? ...the dotted cup?
    \item Using this basis for $\CA\{1\}$ and the pure tensor basis for $\CA\{1\}^{\otimes 2}$, determine the matrix (i.e. chart) associated to the merge and split cobordisms by using the pairing.
    \ee
\end{exercise}

\begin{aside}
(Computing linear maps using pairings)

Here is the linear algebra analogue to the method we are using above to compute the maps $m$ and $\Delta$. 

Let $M: \R^n \to \R^n$ be a matrix with respect to the standard basis vectors $\{e_i\}$. 
We can view $M$ as a pairing, by setting 
\[
    \langle v,w \rangle_M = w^\top M v \in \R.
\]
The entries of the matrix are determined by the pairing on basis vectors:
\[
    M_{ij} = e_i^\top M e_j
\]
(Recall $i$ is the row and $j$ is the column.)
This is because $M_{ij}$ is the coefficient of $e_i$ in the image vector $M e_j$.

In our setting, we are actually taking one additional step, which is to identify the column vectors $v$ and $w$ with the linear maps $\R \to \R^n$ that they represent, as $n \times 1$ matrices. (The `dual' vector $w^\top$ is a row vector that represents a linear map $\R^n \to \R$.)

So a `closed surface' in our setting corresponds to the composition of maps 
\[
    \R \map{v} \R^n \map{M} \R^n \map{w^\top} \R,
\]
and setting $v$ and $w$ to be basis vectors allows us to compute the entries of $M$.

(Note that any linear map $\R \to \R$ is necessarily of the form $\cdot r$ for some $r \in \R$. We are implicitly using $\Hom_\R(\R,\R) \cong \R.$)
\end{aside}

Observe that even though we defined the target of $\BNfun$ to be $\ggmod$, everything is happening at homological grading 0 at the moment; we don't yet have chain complexes! 
However, the quantum degree of morphisms can be determined in the source category $\TLcat_0$, as you'll discover in the next exercise.

Using Morse theory, we can show that every cobordism $C$ between finite collections of planar circles $P \to P'$ can be isotoped so that the critical points occur at distinct times $t \in I$. We can slice up the cobordism into pieces
\[
    C = C_m \circ C_{m-1} \circ \cdots \circ C_2 \circ C_1
\]
where each $C_i$ is a disjoint union of some identity cylinders and one of the following four \emph{elementary cobordisms}:
\begin{itemize}
    \item cup ($\iota$)
    \item cap ($\varepsilon$)
    \item merge ($m$)
    \item split ($\Delta$)
\end{itemize}

\begin{exercise}
\alert{(Important)}
Prove that for any cobordism $C: P \to P'$, the bidegree of the associated linear map is 
\[
    \gr(\BNfun(C)) = (0, \chi(C)),
\]
where $\chi(C)$ is the Euler characteristic of the surface $C$.
\end{exercise}

\begin{remark}
    Bar-Natan's more general cobordism category does not include dots as decorations. 
    The objects are the same as in $\TLcat_0$, but morphisms are subject to different relations:
    \begin{center}
        \includegraphics[width=3in]{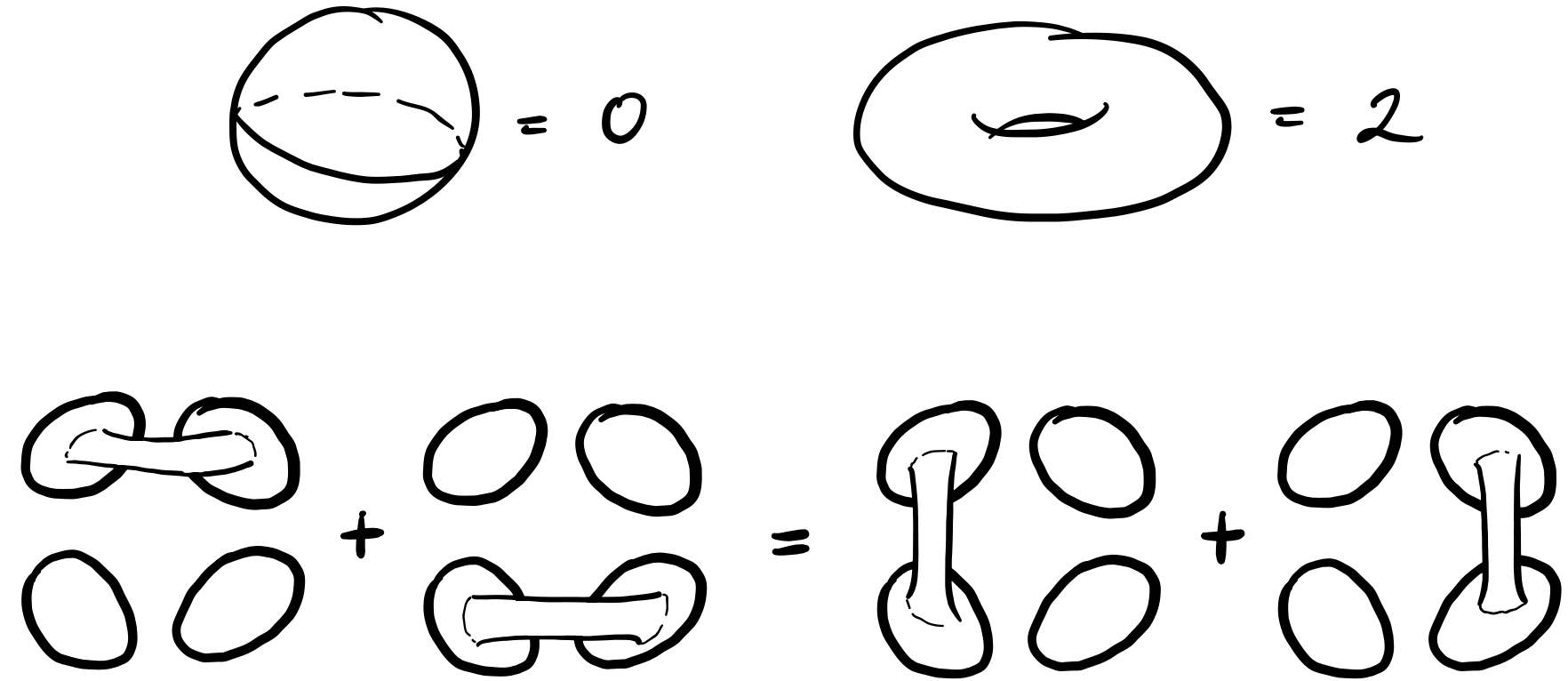}
    \end{center}
    These relations are called the $S$ (sphere), $T$ (torus), and $4Tu$ (four tubes) relations.

    We will sometimes work with this category, but for now have chosen to start with the dotted category because elementary calculations are easier there.
\end{remark}

\begin{exercise}
\bea
\item Use the $T$ and $4Tu$ relations to show that the genus-2 orientable surface evaluates to 0:
    \begin{center}
        \includegraphics[height=1in]{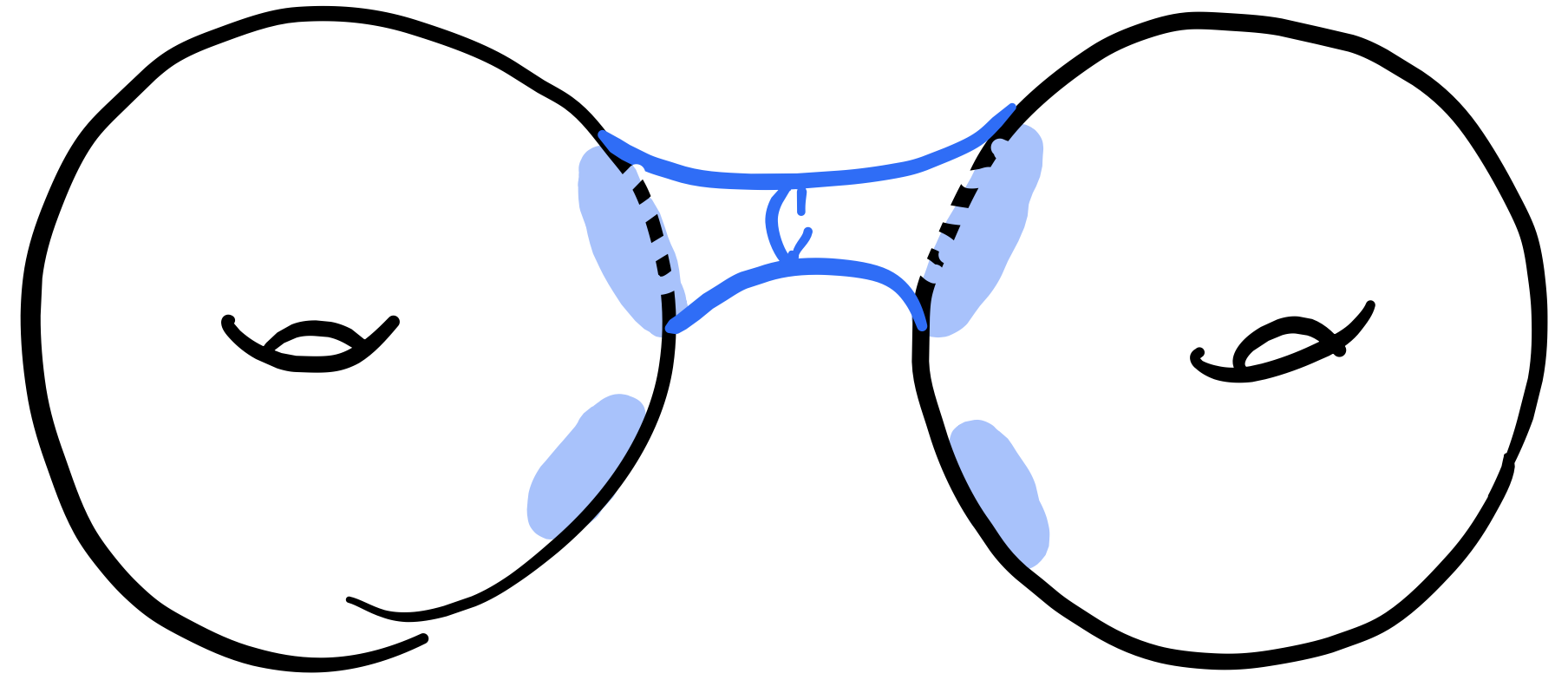}
    \end{center}
\item Use $4Tu$ to show the relation below:
    \begin{center}
    \includegraphics[height=1in]{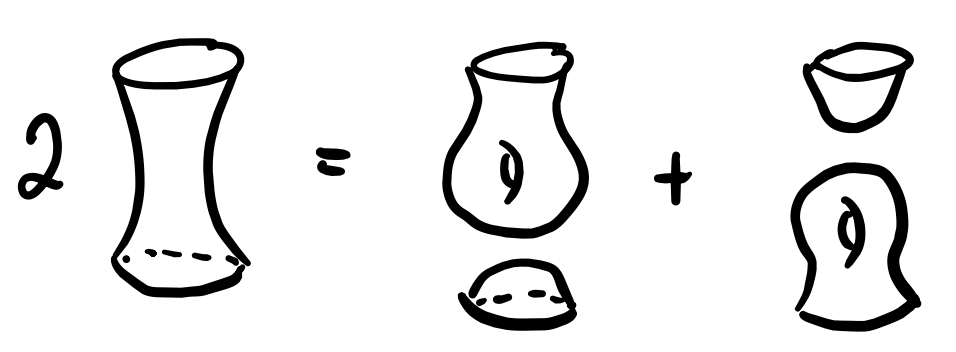}
    \end{center}
Use this to explain why, morally speaking, `dot' = `half of a handle'. 
\ee
\end{exercise}

\subsubsection{(Aside) Frobenius algebras and (1+1)-TQFTs}

We now discuss why we switched from using the $\Z$-module
\[
    V = \Z v_+ \oplus \Z v_-
\]
to the underlying $\Z$-module of the $\Z$-algebra
\[
    \CA = \Z[X]/(X^2).
\]
See \cite{Kh-frob-ext} for a reference.

\begin{definition}
A \emph{Frobenius system} is the data $(R, A, \varepsilon, \Delta)$ consisting of 
    \begin{itemize}
        \item a commutative ground ring $R$; 
        \item an $R$-algebra $A$; in particular:
            \begin{itemize}
                \item There is a \emph{unit} or inclusion map $\iota: R \to A$  that sends $1 \mapsto 1$.
                \item $A$ has a \emph{multiplication} map $m: A \otimes_R A \to A$.
            \end{itemize}
        \item a \emph{comultiplication map} $\Delta: A \to A \otimes_R A$ that is both coassociative and cocommutative; and
        \item an $R$-module \emph{counit} map $\varepsilon: A \to R$ such that 
            \[
                (\varepsilon \otimes \id) \circ \Delta = \id.
            \]
    \end{itemize}
\end{definition}

The algebra $A$ is a \emph{Frobenius algebra}; it is simultaneously both an algebra and a coalgebra, and the following relation holds:
\[
    (\id_A \otimes m) \circ (\Delta \otimes \id_A) 
    = \Delta \circ m
    = (m \otimes \id_A) \circ (\id_A \otimes \Delta).
\]

\begin{remark}
    There are many equivalent definitions for the term `Frobenius algebra'. The definition we used above is the most topological:
    \begin{center}
        \includegraphics[width=3in]{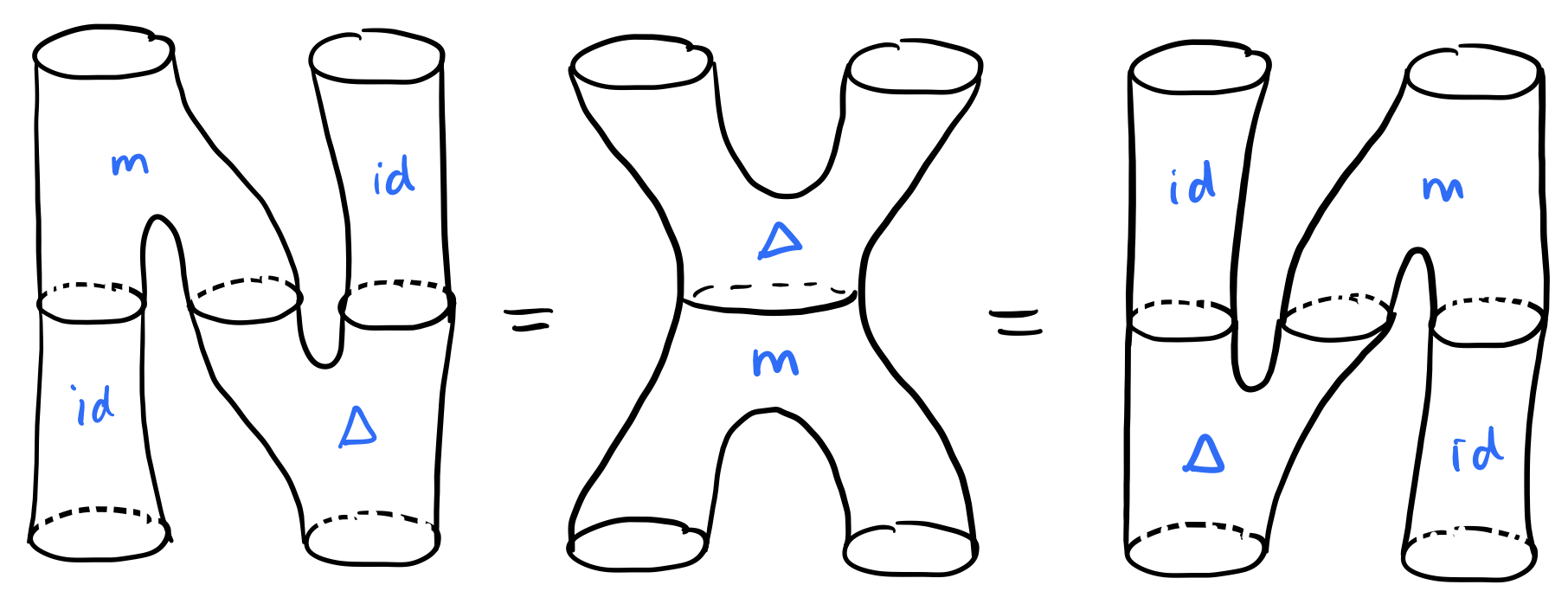}
    \end{center}
\end{remark}

As you might have guessed, there is a strong relationship between Frobenius algebras and TQFTs.

In our case, $\CA = \Z[X]/(X^2)$ is a rank 2 Frobenius extension of the ground ring $\Z$. 
Rank 2 Frobenius systems yield $(1+1)$-dimensional TQFTs, via the identifications below:

\renewcommand{\arraystretch}{1.4}
\begin{center}
\begin{tabular}{| c | c | }
    \hline
    the $(1+1)$-TQFT sends this... & ... to this in the Frobenius system \\
    \hline
    $\emptyset$ & $R$  \\
    $S^1$ & $A$  \\
    cup  & $\iota$ \\
    merge & $m$  \\
    split & $\Delta$   \\
    cap & $\varepsilon$ \\
    \hline
\end{tabular}
\end{center}

By modifying the Frobenius algebra, we can get many more flavors of Khovanov homology. For example, if we instead use $R = \Z$, $A = \Z[X]/(X^2-1)$, we would build a version of Khovanov homology that is not quantum-graded (because $X^2 = 1$ is not a graded equation). 
This version of Khovanov homology is called \emph{Lee homology}, and will be discussed later in this course when we talk about topological applications of Khovanov homology.

\begin{warning}
Do not use the category $\TLcat_0$ from Definition \ref{def:TLcat-0} as the source of the TQFT functor for Lee homology! The category $\TLcat_0$ was specifically tailored to the version of Khovanov homology over $\CA = \Z[X]/(X^2)$, where `two dots = 0'. 
\end{warning}

\subsection{Bar-Natan's tangle categories}

We now return back to dotted cobordisms, but develop the theory for not just planar circles, but also planar tangles. 

Just as a link is an embedding of a finite number of circles in $\R^3$, a \emph{tangle} is a proper embedding of a finite number of circles \underline{and arcs} in a 3-ball $B^3$. 
A \emph{planar tangle} is a tangle embedded in a 2-disk $D^2$. 

Since we will be building complicated categories out of tangles, we want to be very concrete with our definition of tangle categories, and will instead use the following definition.

\begin{definition}
An \emph{$(n,n)$-tangle} is a 1-manifold with $2n$ boundary components (a.k.a.\ endpoints),
properly embedded in the thickened square $[0,1] \times [0,1] \times (-\frac{1}{2}, \frac{1}{2})$,
with $2n$ endpoints located at 
\begin{equation}
\label{eq:tangle-endpoints}
    \left \{ \left (\frac{i}{n+1}, 0, 0 \right ) \right \}_{i=1}^{n}
\ \cup \ \ 
\left \{ \left (\frac{i}{n+1}, 1, 0 \right ) \right \}_{i=1}^{n}.
\end{equation}
\end{definition}

An \emph{$(n,n)$-tangle diagram} is a projection of an $(n,n)$-tangle to the unit square $[0,1] \times [0,1]$ where all singular points are transverse intersections (just as in link diagrams).

A \emph{planar} $(n,n)$-tangle is an $(n,n)$-tangle embedded in the square $[0,1] \times [0,1] \times \{0\}$. 
\note{In other words, a planar tangle is a crossingless projection of a (quite untangled) tangle.} 

\begin{center}
    \includegraphics[width=3in]{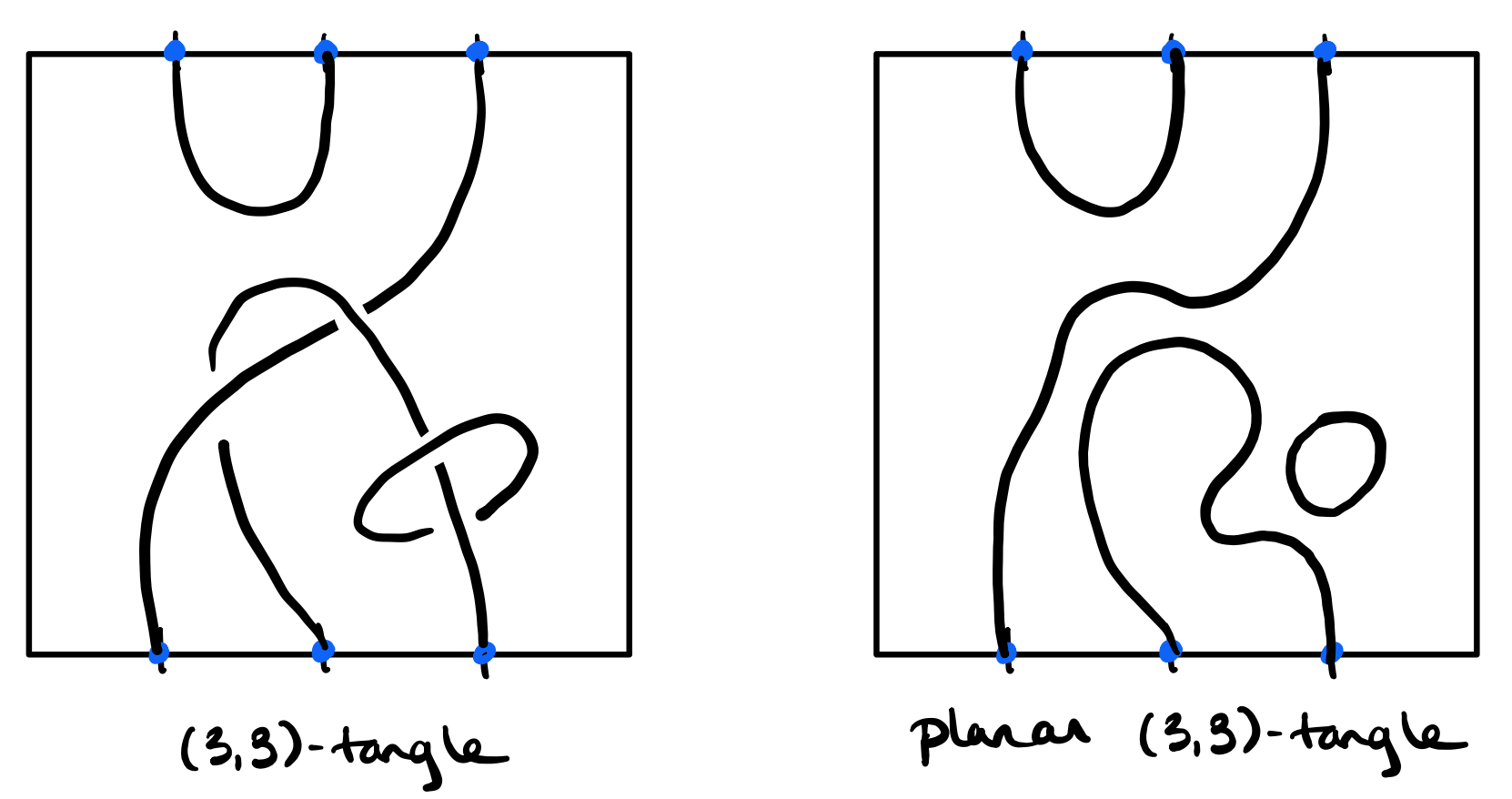}
\end{center}

\begin{remark}
    A planar $(n,n)$-tangle with no closed components is a \emph{crossingless matching}. These are the \note{Catalan($n$) many} generators of the Temperley-Lieb algebra $\TL_n(\delta)$ over a field $\Bbbk$, whose composition $\otimes$ is given by stacking squares vertically:
        \begin{center}
        \includegraphics[height=1.5in]{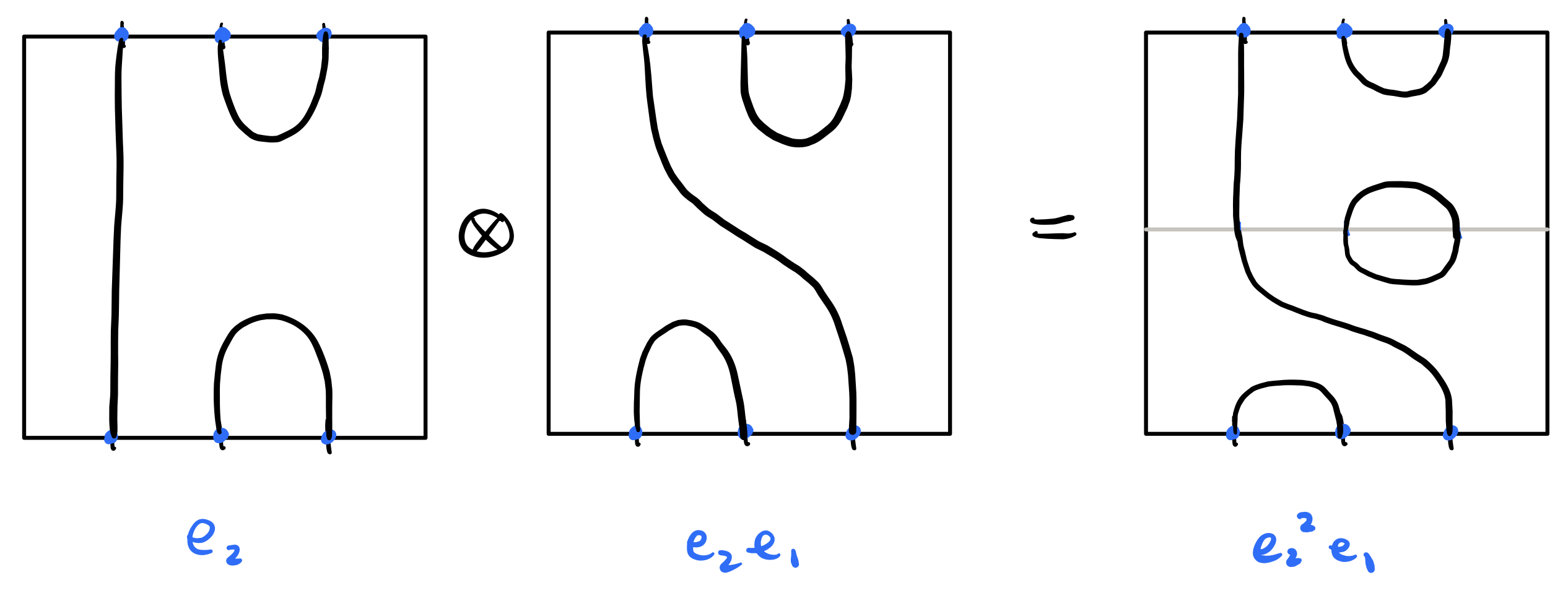}
        \end{center}
    Closed components evaluate to a nonzero element $\delta \in \Bbbk$:
        \begin{center}
        \includegraphics[height=1.5in]{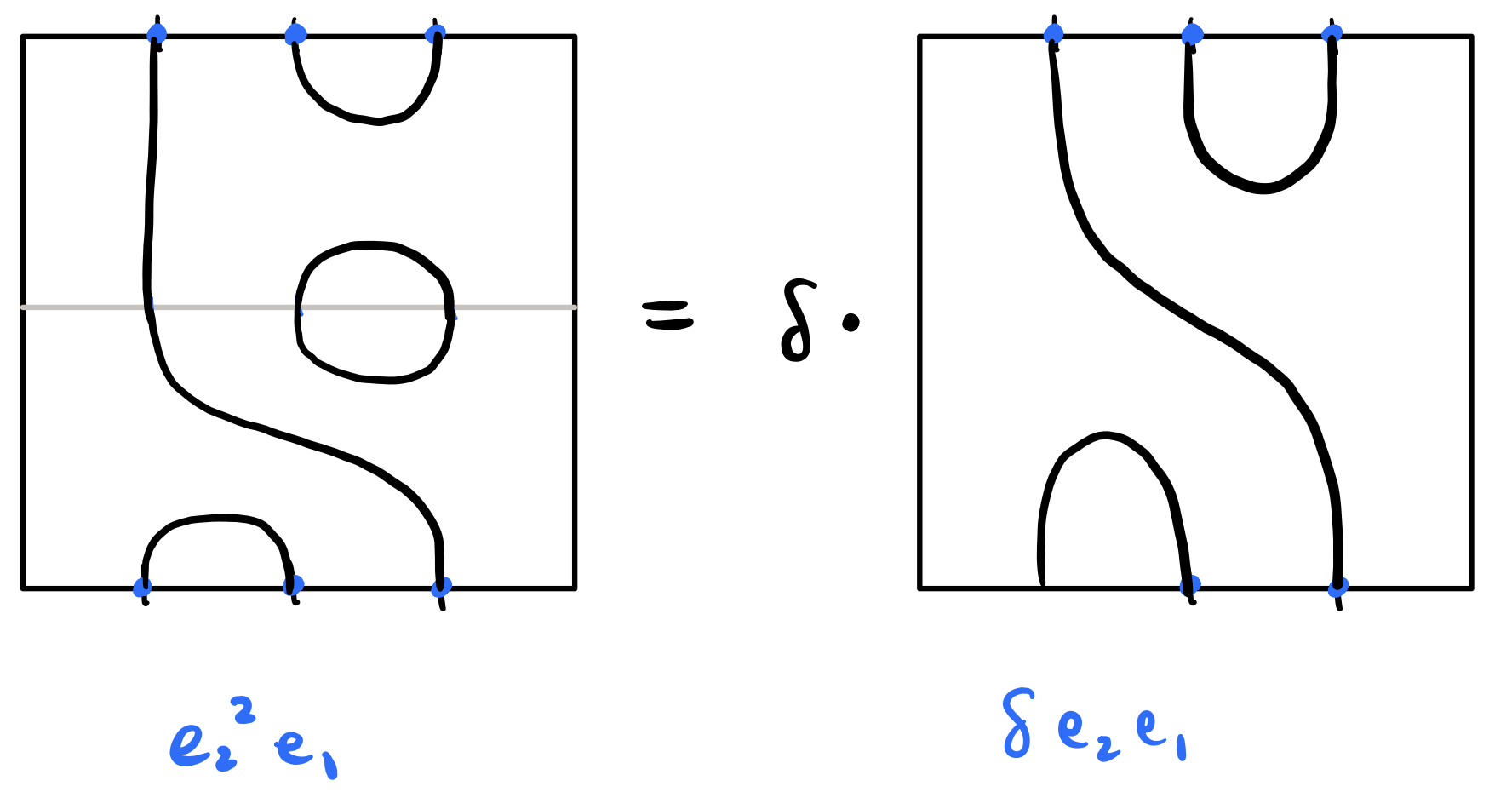}
        \end{center}
    \note{We are not thinking about the monoidal structure of these diagrams just yet. But perhaps you can see the analogue between this (0+1)-dimensional cobordism category and the (1+1)-dimensional cobordism categories we're working with.}
\end{remark}

\begin{definition}[cf.\ Definition \ref{def:TLcat-0}]
\label{def:TLcat-n}
The preadditive category $\TLcat_n$ is defined as follows:
\begin{itemize}
    \item Objects are planar $(n,n)$-tangles with finitely many components, embedded smoothly in the square $I^2$ with endpoints at the $2n$ points specified in \eqref{eq:tangle-endpoints}, denoted by $\mathbf{p}$.
    \item Morphisms are finite sums of dotted cobordisms properly embedded in $I^2 \times [0,1]$,  
    \begin{itemize}
        \item with vertical boundary (i.e.\ on $\partial I^2 \times [0,1]$) consisting only of $2n$ vertical line segments $\mathbf{p} \times [0,1]$;
        \item possibly decorated with a finite number of dots;
        \item up to boundary-preserving isotopy;
        \item subject to the same local relations \eqref{eq:dotted-BN-relations} as in Definition \ref{def:TLcat-0}.
    \end{itemize}
\end{itemize}
\end{definition}

So far, we've only defined a TQFT that can `evaluate' closed components to a $\Z$-module. Since tangles in $\TLcat_n$ will inevitably have non-closed components, i.e.\ arcs, we will not be specifying a functor to $\Zmod$ just yet. Instead, we will bring our algebraic tools to the topological categories $\TLcat_n$ and do as much homological algebra as we can before passing through any TQFT.
\note{Later on in the course, however, we may discuss how to define a 2-functor from the 2-category of ($n$ points, $(n,n)$-tangles, $(n,n)$-tangle cobordisms) to an appropriately rich algebraic 2-category.}

In particular, we can still treat $\TLcat_n$ as a (bi)graded category, where the quantum degree of a dotted cobordism $C$ is \emph{defined} as 
\[
    \deg_q(C) := \chi(C) - n = \chi(C) - \frac{1}{2}(\text{\# vertical boundary components})
\]
(and the homological degree is $\deg_h(C) = 0$).

\subsection{Adding crossings : Boosting to chain complexes}
\label{sec:BN-Kom-categories}

The Kauffman bracket showed us how to take a tangle with 
crossings and express it in terms of planar tangles; 
the Khovanov bracket tells us that a tangle with crossings is really just a chain complex of planar tangles. 
In this section, we will boost our categories $\TLcat_n$ to categories of chain complexes, so that we can capture tangles in general, not just planar tangles.

Throughout this section, we will be working in the more general setting of tangles; recall that $\TLcat_0$ is just a special case of the categories $\TLcat_n$. Also recall that these are all preadditive categories, by construction.

\begin{definition}[\cite{BN-tangles}, Definition 3.2]
Let $\CC$ be a preadditive category. The \emph{additive closure} of $\CC$, denoted by $\Mat(\CC)$, is the additive category defined as follows:
\begin{itemize}
    \item Objects are (formal) direct sums of objects of $\CC$. We can represent these as column vectors whose entries are objects of $\CC$.
    \item Morphisms are matrices of morphisms in $\CC$, which are added, multiplied, and applied to the objects just as matrices are added, multiplied, and applied to vectors.
\end{itemize}
\end{definition}

An additive category is preadditive, by definition (Remark \ref{rmk:additive-category}). We can build a category of chain complexes over any preadditive category:

\begin{definition}[\cite{BN-tangles}, Definition 3.3]
Let $\CC$ be a preadditive category. The \emph{$\star$ category of chain complexes} over $\CC$ for $\star \in \{b, -, +\}$, denoted by $\Kom^\star(\CC)$,  is defined as follows:
\begin{itemize}
    \item Objects are \{ finite length, bounded above, bounded below \}\footnote{If the notation feels counterintuitive, just remember that `bounded above' complexes are `supported mostly in negative degrees'.}
    chain complexes of objects and morphisms in $\CC$.
    \item Morphisms are chain maps between complexes.
\end{itemize}

\end{definition}

\begin{remark}
    For an example of a category of chain complexes over a preadditive but not additive category, take a look at $\Kom(\mathcal{M}_{\Z[G]})$ in \cite[Remark 2.11]{ILM-rational-unknotting}.
\end{remark}

\begin{exercise}
\mz{compute this composition of morphisms in Mat}
\end{exercise}

We are now ready to define the \emph{Bar-Natan categories}, which are the categories we land in just before applying a TQFT to an algebraic category like chain complex over $\Z$-modules.

\begin{definition}
    The Bar-Natan category $\BNcat^\star_n$ is $\Kom^\star(\Mat(\TLcat_n))$, for $\star \in \{ b, -, +\}$. 
    \note{In practice, I will mostly drop the $\star$ from the notation. The bounded category is a full subcategory of both the $-$ and $+$ categories. We will only be working with bounded complexes at the beginning of the course, so the distinction won't matter.}
\end{definition}

These categories are rich enough to capture the entirety of the information in Khovanov homology, without ever passing to rings and modules.
For example, we can think of Khovanov homology as a functor from links and cobordisms to the Bar-Natan category $\BNcat_0$. We may then choose a TQFT (or, equivalently, a Frobenius system) to apply to the resulting invariant, to obtain many different flavors of Khovanov homology. 
Moreover, we now have a ``Khovanov homology for tangles.''

\begin{remark}
Since $\SA = \Mat(\TLcat_n)$ is an abelian category, the \emph{homotopy category} $K^\star(\SA)$ of chain complexes over $\SA$ (for $\star \in \{b, -, +\}$) is a \emph{triangulated category}. Khovanov homology can be thought of a functor to a triangulated category $K^b(\Mat(\TLcat_0))$, in which case the quasi-isomorphism class of $\Kh(L)$ is the link invariant. The Kauffman bracket skein relation gives exact triangles in this category.
\note{For exercises involving long exact sequences in Khovanov homology, see \cite{Turner-galaxy}.}
\end{remark}

To demonstrate the use of these Bar-Natan categories, we will use these categories to prove the invariance of Khovanov homology under some Reidemeister moves. You may find the proofs using the undotted ($S$, $T$, $4Tu$) theory in \cite{BN-tangles}. We will instead use the dotted theory as part of our demonstration. But first, we need to introduce two important lemmas that are immensely helpful with both by-hand and computer-assisted computations.

\subsection{Computational tools}

Bar-Natan's categories are also incredibly useful in computing Khovanov homology (by computer), because they allow for a `divide-and-conquer' approach using tangles. 
Implemented algorithms typically scan a link diagram, simplifying the homological data at each step of the filtration of the diagram. 
The main tools used for simplification are \emph{delooping} and \emph{abstract Gaussian elimination}, which we discuss below. This section follows \cite{BN-fast-kh}.

\begin{remark}
These algorithms have been immensely important to solving problems in low-dimensional topology. For example, Lisa Piccirillo's proof that the Conway knot is not slice \cite{Piccirillo-conway-knot} involves computing the $s$ invariant (see \S \ref{sec:rasmussen-s}) from the Khovanov homology of a knot with \textit{a lot} (\note{like around 40; I haven't counted carefully}) of crossings, which is effectively impossible by naive computation.
\end{remark}

Let us first write down a concrete formula for the degree of a morphism in $\TLcat_n$.

\begin{warmup}
As a sanity check, let's answer some warmup questions. See Notation \ref{nota:shift-functors} for our conventions on shift functors.
\bea
\item Let $A$ be a graded $R$-module, with shift functor $[\cdot]$. 
The set of degree-preserving maps $A[i] \to A[j]$ correspond to elements of $\Hom^k(A,A)$ for some $k \in \Z$. What is $k$? \note{Answer: $k = i-j$}
\item Suppose $\varphi \in \Hom^0(A,A)$. What is the degree of the map induced by $\varphi$ from $A[i] \to A[j]$? \note{Answer: $j-i$}
\item Now suppose $\psi \in \Hom^\ell(A,A)$. What is the degree of the map induced by $\psi$ from $A[i] \to A[j]$? \note{Answer: $j-i + \deg(\psi)$}
\ee
\end{warmup}

\begin{lemma}
\label{lem:cobordism-degree-with-shifts}
Let $F$ be a (possibly dotted) cobordism between planar tangles $T, T' \in \TLcat_n$. 
The degree of the morphism 
\[
    F: T\{i\} \to T'\{j\}
\]
is
\begin{equation}
\label{eq:cobordism-degree-with-shifts}
    \deg_q(F) = j - i + \chi(F) - n
\end{equation}
(where $n$ is $\frac{1}{2}$ the number of tangle endpoints, or vertical boundary components).
\end{lemma}
\note{Verify that this makes sense to you!}

From now on, we will treat diagrams and cobordism drawings as the same things as the objects and morphisms they represent in $\BNcat$. 

\begin{lemma}[\cite{BN-fast-kh}, Lemma 4.1]
(Delooping)
$\fullmoon$ is chain homotopy equivalent to $\emptyset \{1\} \oplus \emptyset \{-1\}$ via the chain homotopy equivalences
\begin{center}
    \includegraphics[width=3in]{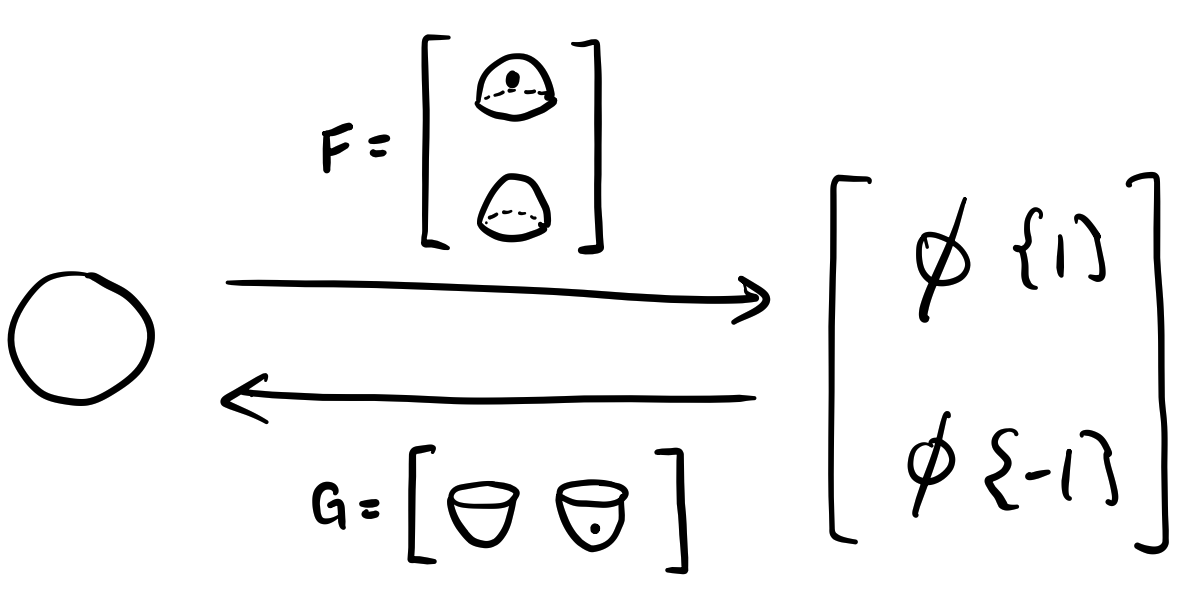}
\end{center}
\begin{proof}
We leave it to the reader to check that $F$ and $G$ are indeed degree-preserving (use Lemma \ref{lem:cobordism-degree-with-shifts}).
It suffices to check that $G \circ F \simeq \id_{\fullmoon}$ and $F \circ G \simeq \id_{\emptyset \{1\} \oplus \emptyset \{-1\}}$. 
Indeed,
\begin{center}
    \includegraphics[width=3.5in]{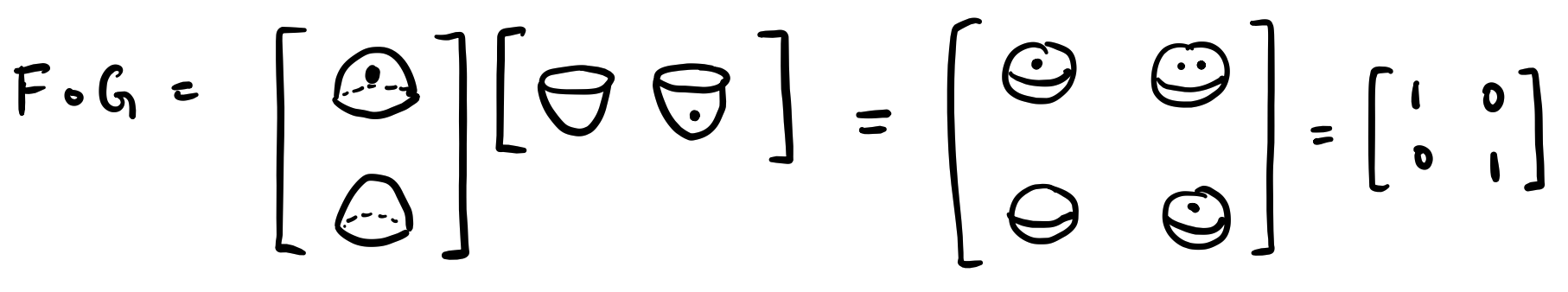}
\end{center}
by the sphere, dotted sphere, and two dots relations, and 
\begin{center}
    \includegraphics[width=3.5in]{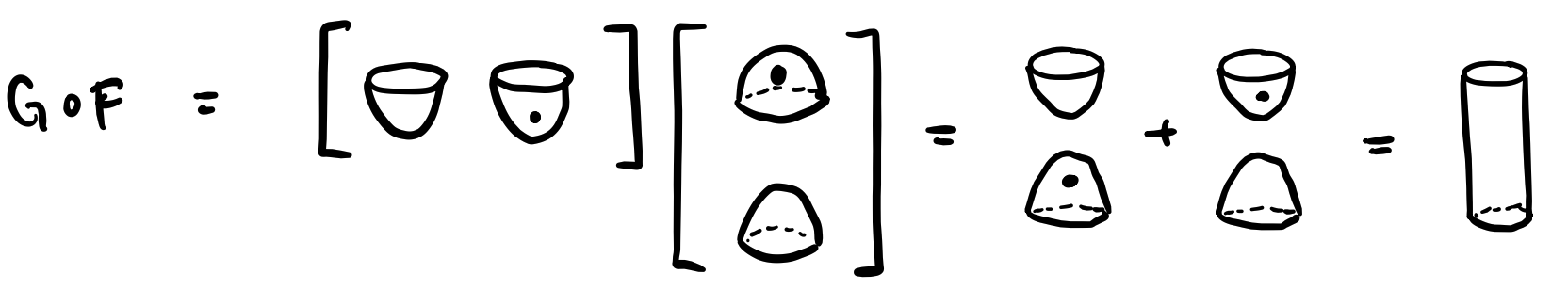}
\end{center}
by the neck-cutting relation.
(No nontrivial homotopies were needed; these compositions are identity ``on the nose''.)
\end{proof}
\end{lemma}

Delooping essentially allows us to replace any flat tangle containing a closed component with two (quantum-shifted) copies of that flat tangle with that circle removed. 

Our next tool is an abstract version of Gaussian elimination.
\note{You can ponder for yourself why this `is' Gaussian elimination. Just remember that, in $\Q$, any nonzero number is a unit. Row operations are just changes of basis on the target of the linear map. Similarly, column operations are just changes of the basis on the source of the linear map.}

\begin{lemma}[\cite{BN-fast-kh}, Lemma 4.2]
(Gaussian elimination)
Let $\CC$ be a pre-additive  category. Suppose in $\Kom(\Mat(\CC))$ there is a chain complex segment
\[
    \cdots 
    \map{
        \begin{pmatrix}
            \alpha \\ \beta
        \end{pmatrix}
    }
    \begin{bmatrix}
        A \\ B 
    \end{bmatrix}
    \map{
        \begin{pmatrix}
            e & g \\ f & h
        \end{pmatrix}
    }
    \begin{bmatrix}
        C \\ D
    \end{bmatrix}
    \map{
        \begin{pmatrix}
            \gamma & \delta
        \end{pmatrix}
    }
    \cdots
\]
where $e$ is an isomorphism. 
Then the chain complex is homotopy equivalent to 
\[
    \cdots 
    \map{
        \begin{pmatrix}
            \beta
        \end{pmatrix}
    }
    \begin{bmatrix}
        B 
    \end{bmatrix}
    \map{
        \begin{pmatrix}
            h-fe\inv g
        \end{pmatrix}
    }
    \begin{bmatrix}
        D
    \end{bmatrix}
    \map{
        \begin{pmatrix}
            \delta
        \end{pmatrix}
    }
    \cdots
\]

\note{Note that 
\begin{itemize}
    \item $A$, $B$, $C$, $D$ are objects in $\Mat(\CC)$
    \item $e,f,g,h,\alpha,\beta, \gamma, \delta$ are morphisms in $\Mat(\CC)$, i.e.\ these are matrices.
\end{itemize}
}

\end{lemma}

The main idea is that, by row and column operations, we are able to choose a basis so that the first chain complex is the direct sum of the second chain complex and an acyclic complex 
\[
    0 \to A \map{e} C \to 0.
\]
See Bar-Natan's proof for full details. 

\begin{corollary}
\label{cor:cancellation-lemma}
(Cancellation lemma)\footnote{See \cite[Lemma 4.1]{Baldwin-Plam-tight}, which directs you to \cite[Lemma 5.1]{Rasmussen-thesis}, which directs you to \cite{Floer-infinite-morse}.}
Suppose $(C, d)$ is a chain complex freely generated by a distinguished set of generators $\mathcal{G}$, and we draw it using dots and arrows.
For $x,y \in \mathcal{G}$, let $d(x,y)$ denote the coefficient of $y$ in $d(x)$. 

Suppose there is an isomorphism arrow $a \map{e} c$ between distinguished basis elements $a,c \in \mathcal{G}$, i.e.\ the coefficient of the arrow is a unit in the ground ring $R$.
Then $(C,d)$ is chain homotopy equivalent to a `smaller' chain complex $(C', d')$ where $C'$ is generated by $\mathcal{G} - \{a,c\}$, and for any $b \in \mathcal{G}$,
\[
    d'(b) = d(b) - d(b,c)d(a).
\]

\begin{center}
    \includegraphics[height=1in]{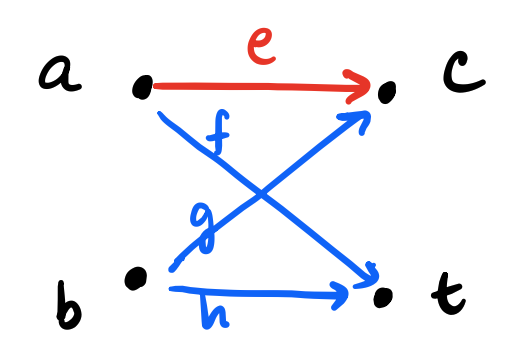}
\end{center}

\end{corollary}

The new arrows ($f e\inv g$ in the figure above) are called \emph{zigzags} for obvious reasons.

\begin{remark}
Cancellation provides an especially fast way to compute Khovanov homology over $\F_2$. Here is what my calculation for the Khovanov homology of the Hopf link looks like:
\begin{center}
    \includegraphics[width=2in]{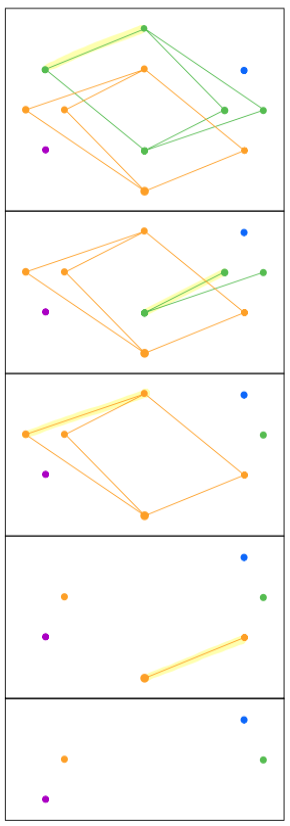}
\end{center}
\end{remark}

\subsection{Reidemeister invariance of Khovanov homology}

Bar-Natan's Reidemeister invariance proofs use the more general $S, T, 4Tu$ categories, and gives explicit homotopies where needed.

Here, we will stay with the $\TLcat_n$ categories that we defined and prove invariance under some (easy) Reidemeister moves, and make use of some facts from homological algebra.

\begin{definition}
    Let $\CC, \CC'$ be chain complexes. 
    Let $C_n$ denote the $n$-th chain group of $\CC$, and let $C_n'$ be defined analogously.
\begin{itemize}
    \item $\CC'$ is a \emph{subcomplex} of $\CC$ (written $\CC' \subseteq \CC$) if each $C_n'$ is a submodule of $C_n$ and the differential on $\CC'$ is the restriction of the differential on $\CC$. 
    \item If $\CC' \subseteq \CC$, then the \emph{quotient complex} $\CC/\CC'$ has chain groups $C_n/C'_n$, and the differential is induced by the differential on $\CC$. 
\end{itemize}
\end{definition}

\begin{lemma}
\label{lem:null-sub-and-quotient-complex}
Let 
\[
    0 \to \CC' \to \CC \to \CC'' \to 0
\]
be an exact sequence of chain complexes, i.e.\ $\CC'' \cong \CC/\CC'$. 
\begin{enumerate}
    \item If $\CC' \simeq 0$, then $\CC \simeq \CC''$.
    \item If $\CC'' \simeq 0$, then $\CC \simeq \CC'$.
\end{enumerate}
\end{lemma}

We now show R1 invariance of any flavor of Khovanov homology that factors through the dotted cobordism categories we defined.

\begin{example}
Consider the Reidemeister move involving a twist with a negative crossing:
\begin{center}
    \includegraphics[width=1.5in]{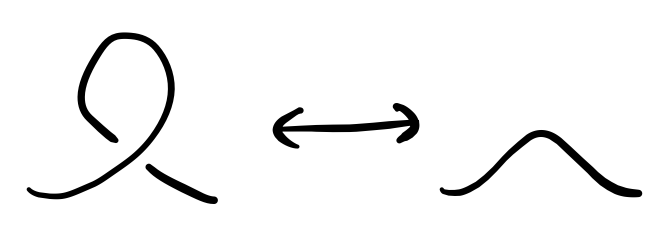}
\end{center}
\note{There is only one strand in this local picture, so this crossing will be negative regardless of how you orient the strand.}

Using the Khovanov bracket, we resolve the crossing to obtain a two-term chain complex in $\Kom(\Mat(\TLcat_1))$ representing the twist, and deloop the resolution on the right:
\begin{center}
    \includegraphics[width=3.5in]{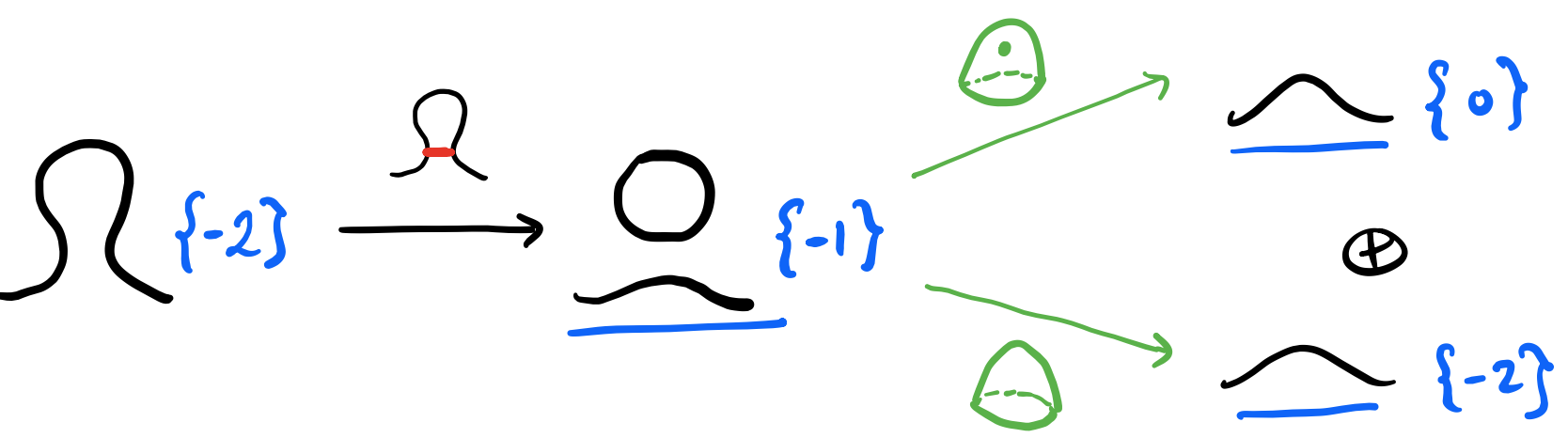}
\end{center}

After delooping, our homotopic (actually isomorphic) complex is:
\begin{center}
    \includegraphics[width=3.5in]{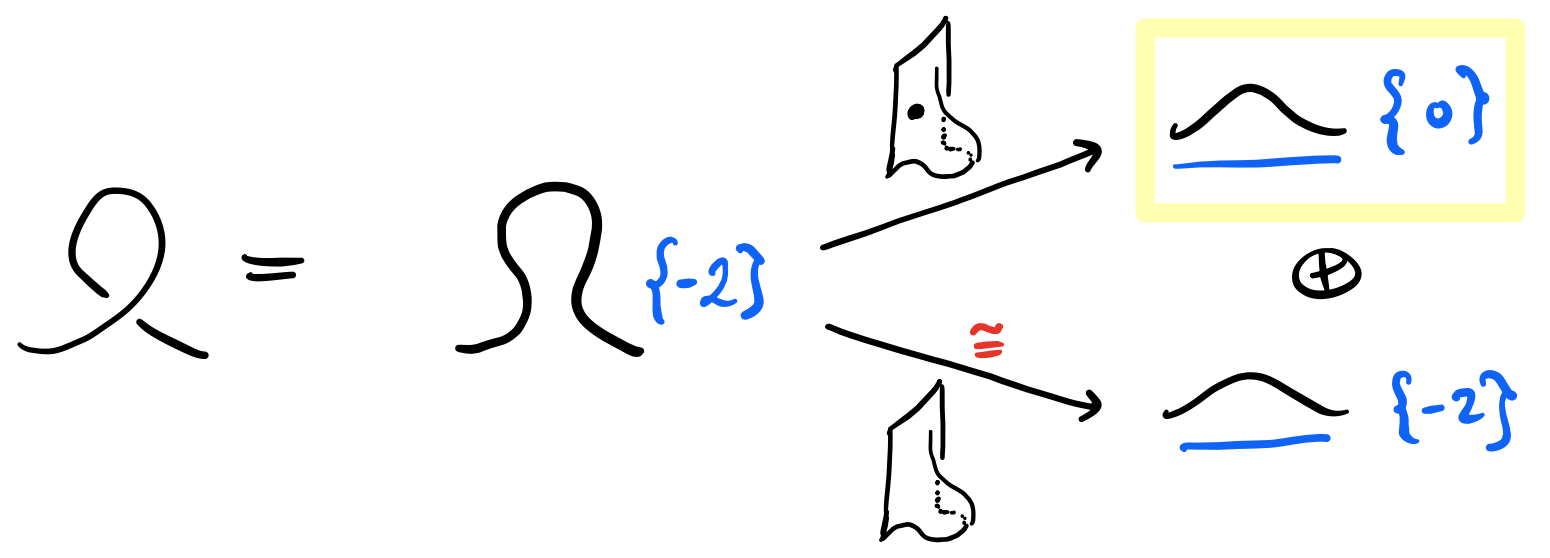}
\end{center}
Notice that the downward arrow is an isomorphism in $\Mat(\TLcat_1)$. 
The quotient of the highlighted subcomplex is therefore nullhomotopic. 

Using Lemma \ref{lem:null-sub-and-quotient-complex}, we conclude that the complex representing the negative R1 twist is chain homotopy equivalent to the highlighted complex.
\end{example}

\begin{remark}
If you close up the $(1,1)$-tangle in the example above, then we can see exactly why the downward arrow is an isomorphism (red below), and the upward arrow (green below) is not:
\begin{center}
    \includegraphics[width=3.5in]{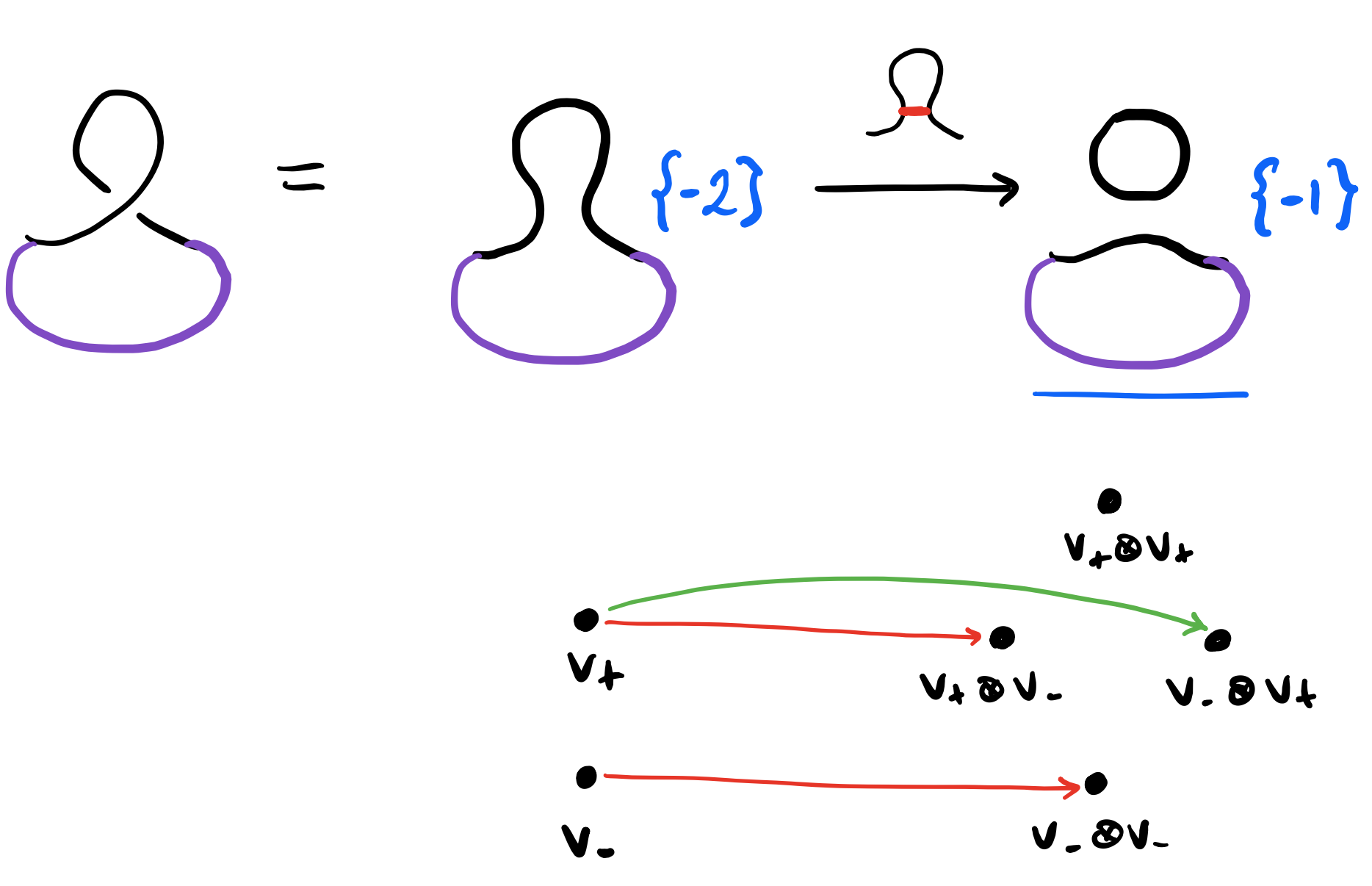}
\end{center}
\end{remark}

\begin{notation}
    We have been using shorthand for merges and splits by drawing the descending manifold of the index-1 critical point in the saddle, on a diagram of the domain diagram:
    \begin{center}
        \includegraphics[width=.7in]{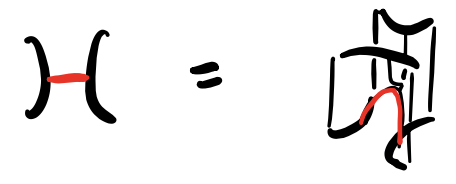}
    \end{center}
    It is sometimes also useful to have shorthand for other types of elementary morphisms:
    \begin{center}
        \includegraphics[width=2in]{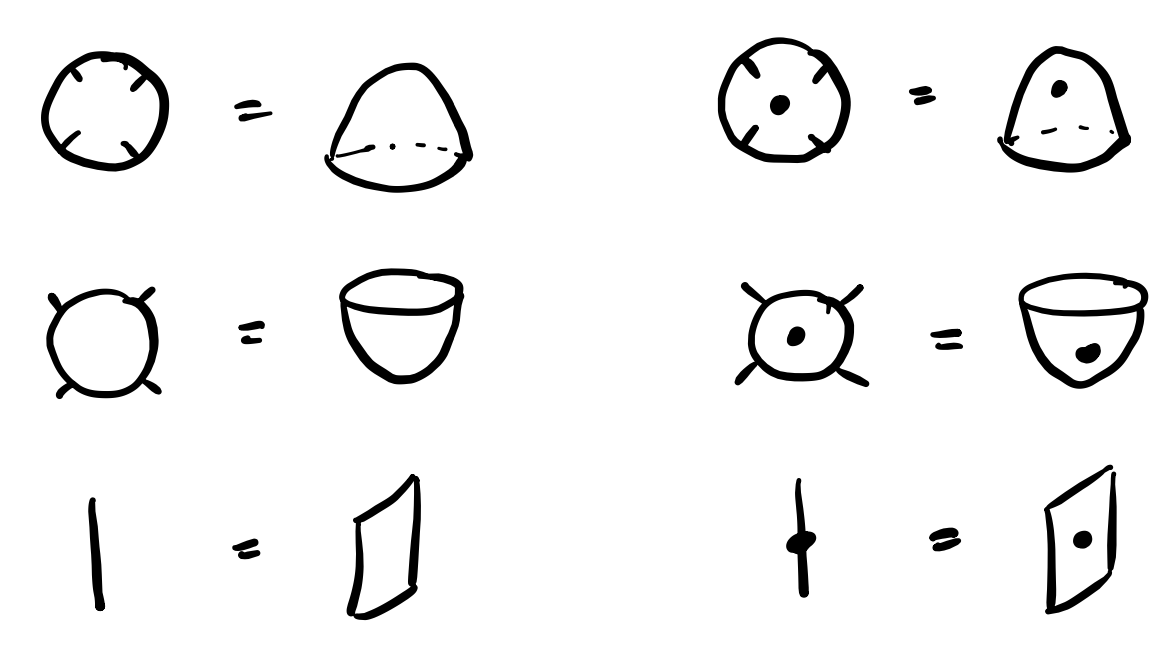}
    \end{center}
\end{notation}

\begin{example}
Here is a proof of invariance under the `braidlike' R2 move, corresponding to the diagram 
\begin{center}
    \includegraphics[width=1in]{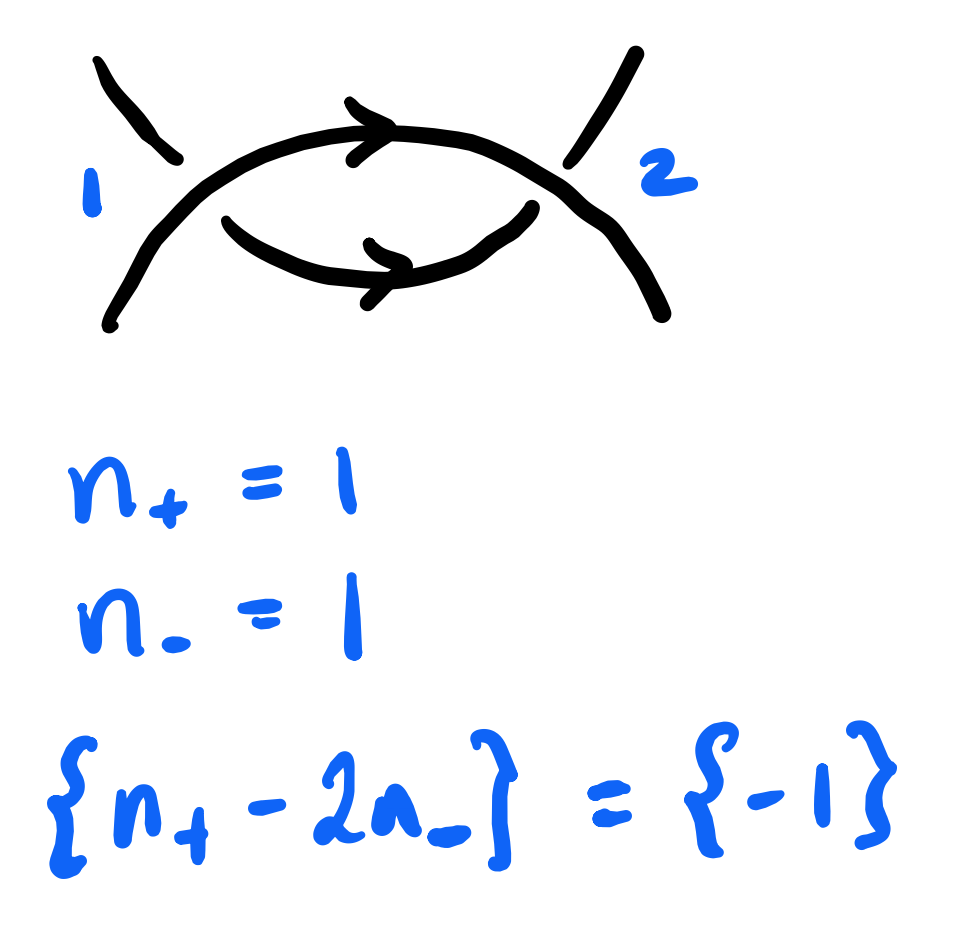}
\end{center}
Below is the complex representing the $(2,2)$-tangle above, with delooping maps draw in orange. The purple maps are compositions of green and orange maps.
\begin{center}
    \includegraphics[width=4in]{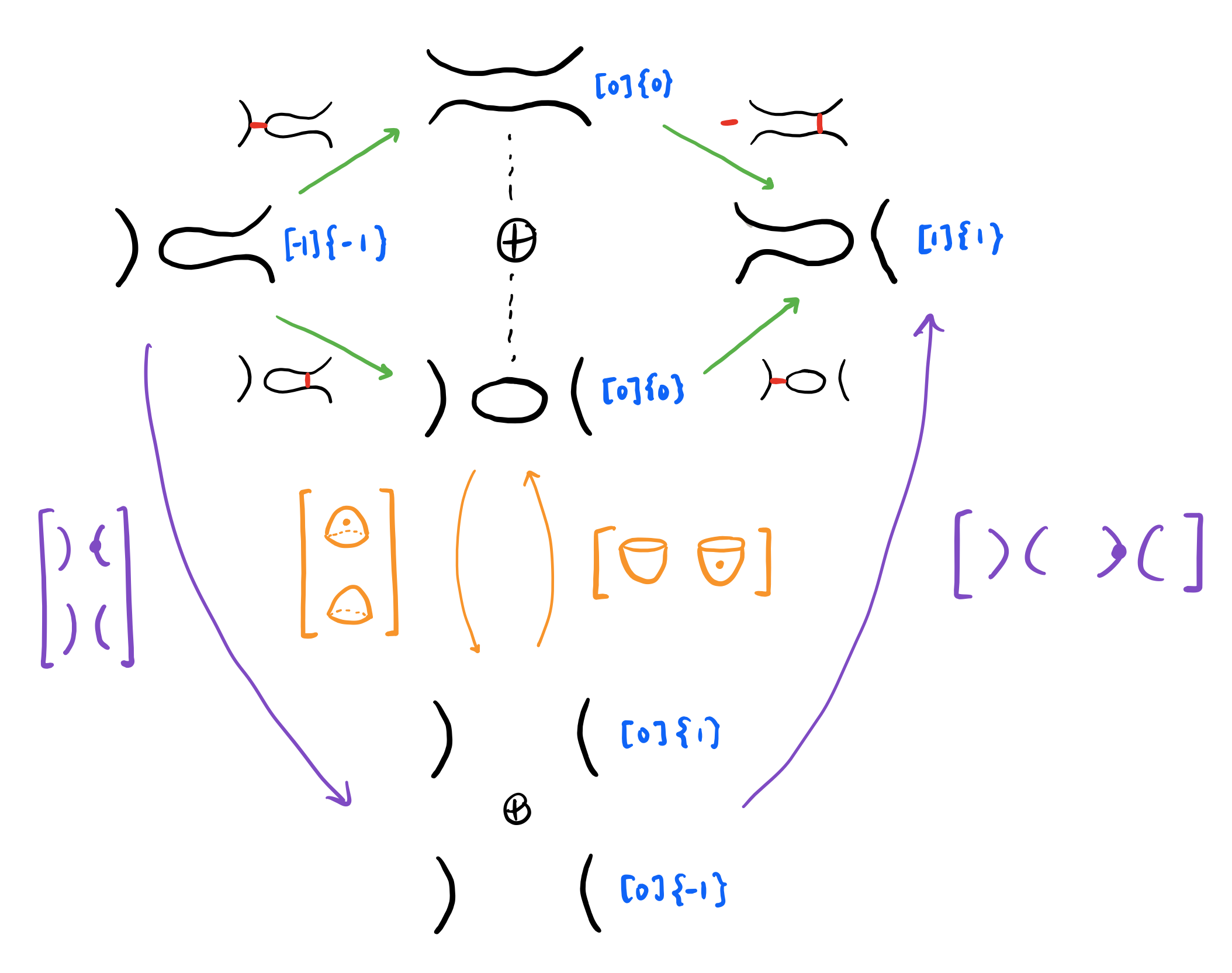}
\end{center}

After delooping, our complex looks like this:
\begin{center}
    \includegraphics[width=4in]{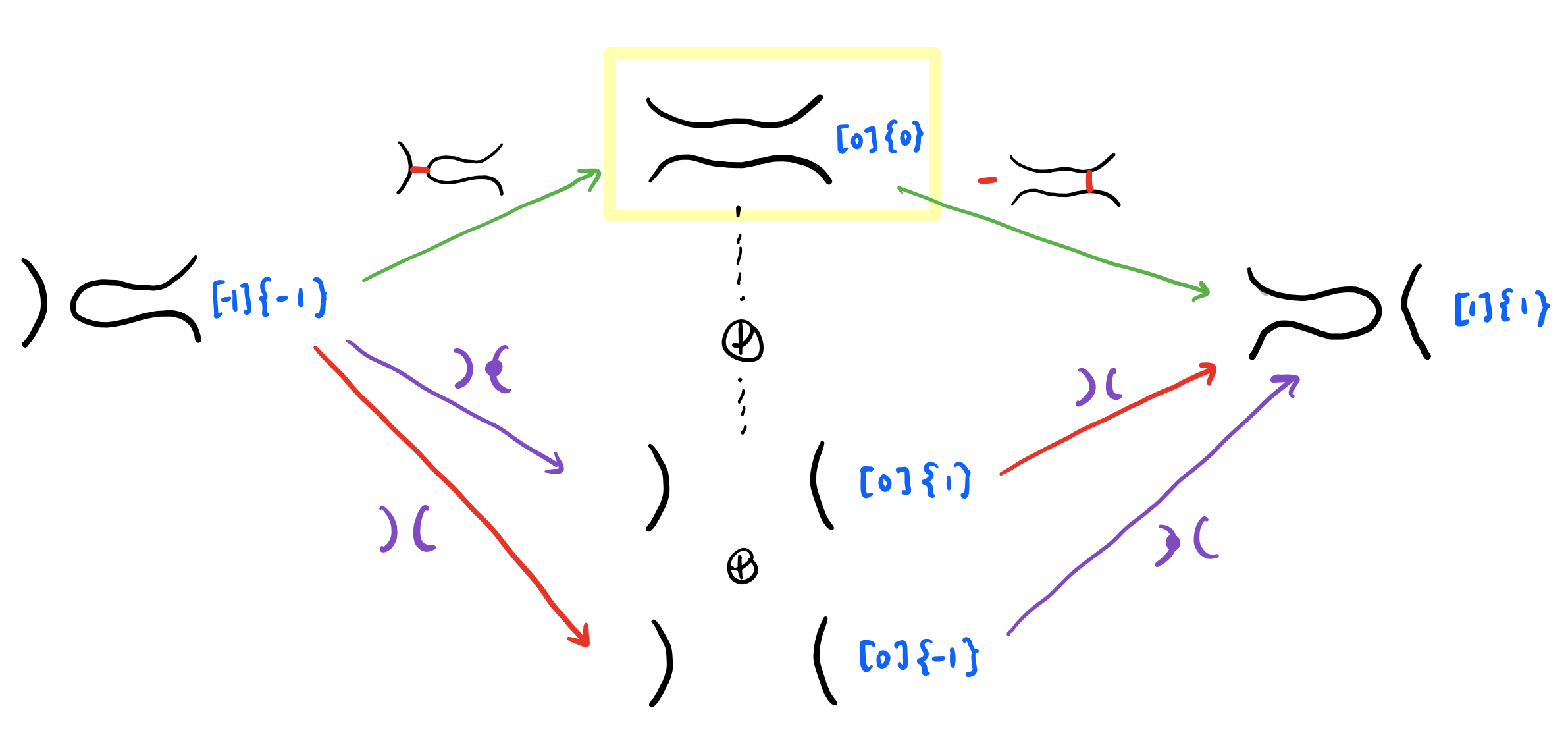}
\end{center}
Notice that the red arrows are isomorphisms. 
After performing Gaussian elimination on these red arrows, we obtain the homotopy equivalent complex highlighted in yellow.
\end{example}

R3 is always a more complicated move to deal with, because
\begin{itemize}
    \item there are 3 crossings, so one needs to compare two cubes with eight vertices each, and
    \item there are three strands, and therefore many possible orientations
\end{itemize}
Nevertheless, Bar-Natan's proof fits beautifully on just one page; see \cite[Figure 9]{BN-tangles} to learn what a `monkey saddle' is.

\begin{exercise}
\label{ex:reid-invariance}
\begin{enumerate}
    \item Prove invariance under the other R1 move, where the twist has a positive crossing.
    \item Prove invariance under the other R2 move, where the strands are antiparallel.
\end{enumerate}
\end{exercise}

\subsection{(Projective) Functoriality}
\label{sec:projective-functoriality}

We are now equipped to fully define a functor that allows us to study links and cobordisms by replacing them with  chain complexes and chain maps. 

Once again, we begin by carefully defining our domain categories.

\begin{definition}
The category $\catLink$ is defined as follows:
\begin{itemize}
    \item Objects are smooth links in $\R^3$\footnote{or $S^3$, if preferred, but see Remark \ref{rmk:catLink-in-S3}}
    \item Morphisms are cobordisms between links, modulo isotopy rel boundary.
\end{itemize}
\end{definition}

At first glance, $\catLink$ would be the category we would want to define our functor out of. 
However, recall that we don't actually compute Khovanov homology directly from links; we actually use link diagrams. 
So, we need to define an intermediate \emph{diagrammatic category} that is equivalent to $\catLink$.

\begin{definition}
The category $\catDiag$ is defined as follows:
\begin{itemize}
    \item Objects are smooth link diagrams drawn in $\R^2$ \note{again, NOT up to isotopy!}
    \item Morphisms are \emph{movies} between link diagrams, modulo \emph{movie moves}. 
\end{itemize}
\end{definition}

\emph{Movies} and \emph{movie moves} need to be discussed carefully, analogously to how we defined \emph{link diagrams} and \emph{Reidemeister moves}.

\begin{definition}
A \emph{movie} is a finite composition of the following `movie clips':
    \begin{itemize}
        \item planar isotopy
        \item Reidemeister moves
        \item Morse moves: birth of a circle, death of a circle, merging of two circles, splitting of one circle into two
    \end{itemize}
\end{definition}

If $F \subset \R^3 \times I$ is a cobordism from $L_0 \subset \R^3 \times \{0\}$ to $L_1 \subset \R^3 \times \{1\}$, then an associated movie $M$ can be thought of as a collection of movie `frames' $\{M_t \st t \in [0,1]\}$ where $M_t$ is a diagrammatic projection of the `slice' of the cobordism at time $t$, $F \cap (\R^3 \times \{t\})$. 
Indeed, the four Morse moves correspond to cup, cap, merge, and split cobordisms, respectively. 

Just as we required link diagrams to only have transverse intersections, and for crossings to not be on top of each other, our definition for `movie' ensures that our \emph{critical frames} (i.e.\ frames where the projection $\R^3 \times \{t\} \to \R^2$ is not a link diagram) are isolated. Reidemeister moves and Morse moves are basically `before and after' pictures of the process of passing through a critical frame. 

In the category $\catLink$, morphisms are considered up to isotopy rel boundary. Just as Reidemeister proved that any diagram isotopy can be described as a finite composition of 3 local Reidemeister moves (and planar isotopy), Carter and Saito showed that any isotopy rel boundary of a cobordism can be captured as a finite composition of 15 local movie moves (and  time-preserving isotopy), which are now known as the \emph{Carter-Saito movie moves} \cite{Carter-Saito-movie}. 

Here is a figure ripped from \cite{Graham-markings}:
\begin{center}
    \includegraphics[width=4in]{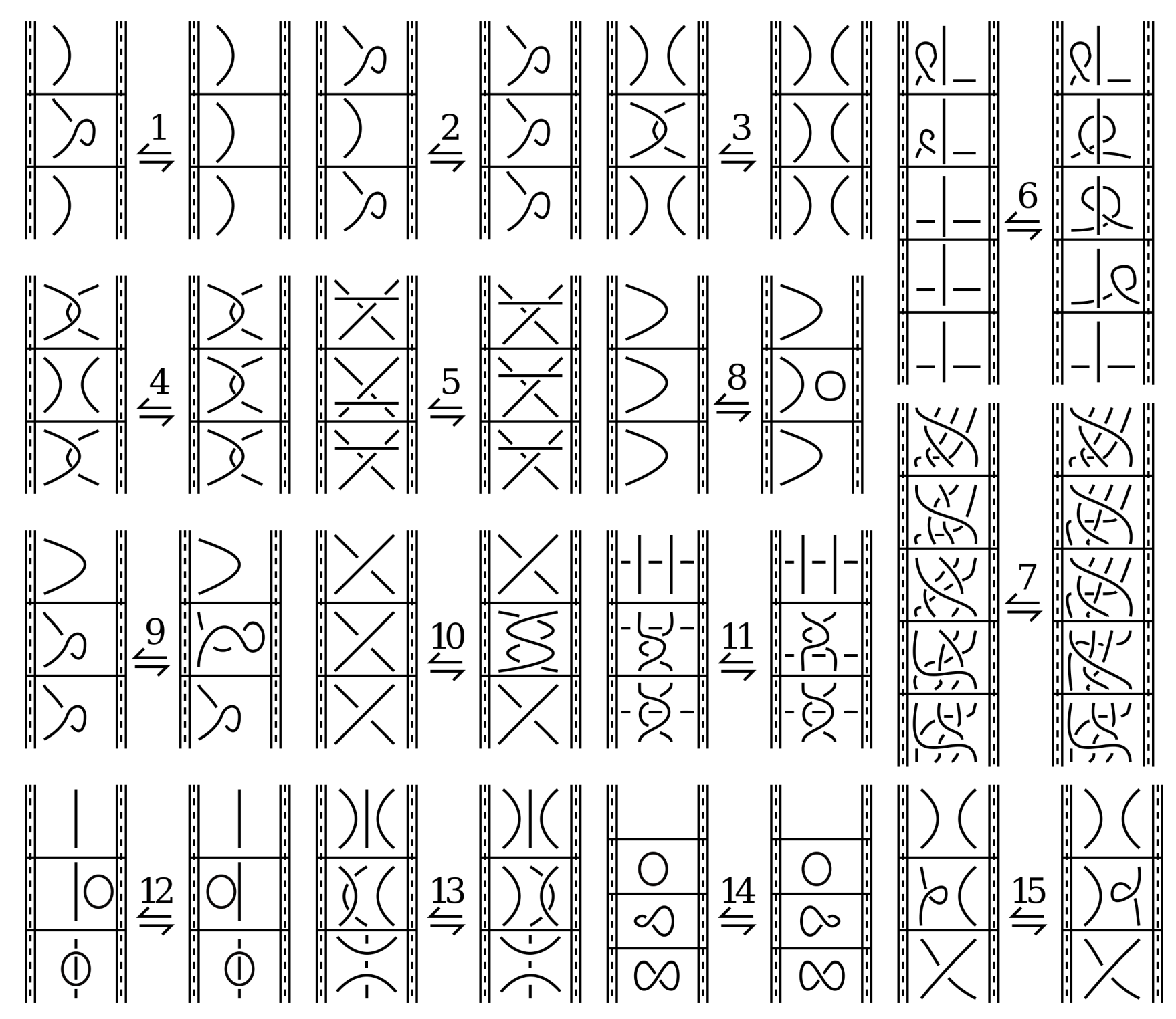}
\end{center}

Similar definitions can also be made for tangles:

\begin{definition}
The category $\catTang_n$ is defined as follows:
\begin{itemize}
    \item Objects are smooth $(n,n)$-tangles in $[0,1]^2$ with boundary at the $2n$ points $\mathbf{p}$ as in Definition \ref{def:TLcat-n}.
    \item Morphisms are tangle cobordisms whose vertical boundary (i.e.\ the boundary in $\partial [0,1]^2 \times I$) consists of the $2n$ line segments $\mathbf{p} \times I$, up to isotopy rel boundary.
\end{itemize}
\end{definition}

The Carter-Saito movie moves are local, and so the definition of the diagrammatic category is essentially the same:

\begin{definition}
The category $\catTangDiag_n$ is defined as follows:
\begin{itemize}
    \item Objects are smooth $(n,n)$-tangle diagrams drawn in $[0,1]^2$.
    \item Morphisms are \emph{movies} between tangle diagrams that preserve the boundary points $\mathbf{p}$, modulo \emph{movie moves}. 
\end{itemize}
\end{definition}

The punchline is that the diagrammatic categories are sufficient for capturing all the information in the topological categories.

\begin{theorem}
    There is an equivalence of categories between $\catLink$ and $\catDiag$ (and similarly between $\catTang_n$ and $\catTangDiag_n$).
\end{theorem}

We can attribute this theorem to the combined work of Reidemeister and Carter--Saito \mz{but I should check this}. The proof is beyond the scope of this course, but we will make use of these equivalences all the time.

\begin{remark}
\label{rmk:catLink-in-S3}
    If you want to work with links in $S^3$ instead, there are more diagrammatic moves to check. In particular, we need to include the \emph{sweep-around} move, which equates  the sweep-around movie (swinging a strand of the diagram around the 2-sphere through the point at infinity) with the identity movie. See \cite{MWW-lasagna}.
\end{remark}

We now know what it means for a link invariant to be \emph{functorial}: it is a functor from the category $\catLink$ to its target category.  

For our purposes, we will define the \emph{Khovanov functor} 
\[
    \functor_{\Kh}: \catDiag \to \ggmod
\]
to be the composition of the functors
\[
    \catDiag \to \Kom(\Mat(\TLcat_0)) \to \ggmod.
\]

The first functor is the unfortunately unnamed (projective) functor that Bar-Natan gives us in \cite{BN-tangles}\footnote{We can't call it $\functor_{BN}$ because `Bar-Natan homology' means something else, which we will discuss later.}
The second is the TQFT associated to the Frobenius system $(\Z, \CA = \Z[X]/(X^2), \iota, m, \Delta, \varepsilon)$ where the morphisms are given by 

    \begin{align*}
        \iota : \Z &\to \CA \\
                1 &\mapsto 1
    \end{align*}
    \begin{align*}
        m : \CA \otimes \CA &\to \CA \\
                1 \otimes 1 &\mapsto 1 \\
                X \otimes 1, 1 \otimes X &\mapsto X \\
                X \otimes X &\mapsto 0
    \end{align*}
    \begin{align*}
        \Delta : \CA &\to \CA \otimes \CA \\
                1 &\mapsto X \otimes 1 + 1 \otimes X \\
                X &\mapsto X \otimes X
    \end{align*}
    \begin{align*}
        \varepsilon : \CA &\to \Z \\
                1 &\mapsto 0\\
                X &\mapsto 1.
    \end{align*}

\begin{example}
Khovanov homology over $\F_2$ is functorial; it defines a functor
\[
    \catDiag \map{\functor_{\Kh}} \ggVect_{\F}
\]
which, when precomposed with the equivalence
\[
    \catLink \map{\cong} \catDiag,
\]
defines a functor from the category of links to the category of bigraded vector spaces. 
\end{example}

Jacobsson \cite{Jacobsson-functoriality} and Bar-Natan \cite[Section 4.3]{BN-tangles}
each showed that Khovanov homology over $\Z$ is \emph{projectively} functorial, meaning that it's functorial up to a sign. 
In other words, suppose you take two movies $M$ and $M'$ representing isotopic (rel boundary) cobordisms $C$ and $C'$. Then the corresponding morphisms in $\ggmod$ satisfy
\[
    \Kh(M') = \pm \Kh(M).
\]

This sign discrepancy can be fixed, and it has been by a whole host of authors \cite{CaprauFix, CMW09, Bla10, San21, Vog20, ETW, Beliakova-Hogancamp-Putyra-functoriality}. 
However, projective functoriality is enough for many, many important applications of Khovanov homology, which we will see in the next section.

\begin{exercise}
\alert{(Highly recommended)}
The following is a version of Movie Move 14 (`MM14'):
\begin{center}
    \includegraphics[width=3in]{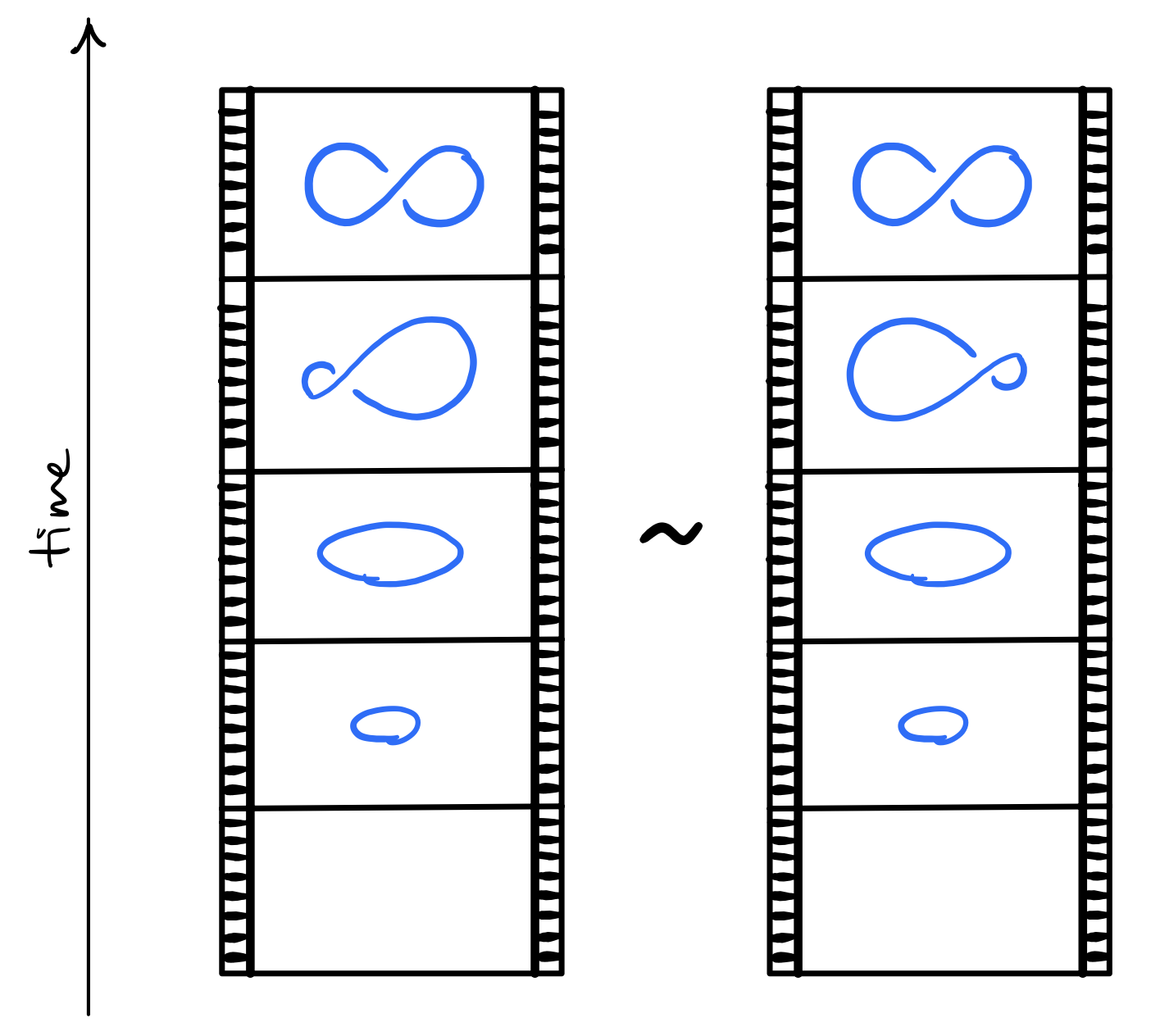}
\end{center}
Compute the induced map on Khovanov homology for each movie, by using the Reidemeister 1 chain maps given in \cite[Figure 5]{BN-tangles}. Show that the resulting morphisms have opposite signs.

\note{You do not need to pass through the TQFT, but you are welcome to; the TQFT is a true (not projective) functor.}

\alert{Beware}: Bar-Natan's cobordisms flow \emph{downward} with time. You can tell by looking at the domain and target objects of the cobordisms.

\note{You can also quickly convince yourself that these movies, when read \emph{backwards}, actually \emph{do} yield morphisms with the same sign.}
\end{exercise}


\mz{Maybe add section on general planar algebras, operads}

\section{Applications of Khovanov homology}

As a functorial invariant for links in the 3-sphere and cobordisms in the 4-ball, Khovanov homology is a natural tool to study the relationship between links in the context of the surfaces they bound. In this section, we survey some of these applications. We start with some additional background on surfaces properly embedded in $B^4$. Once again, everything is smooth.

\subsection{Surfaces in $B^4$}

We have already discussed cobordisms between links in $S^3$. In this section, we will use $F$ to denote a cobordism (because they're 2-dimensional, like `faces'). 
We reserve $C$ for \emph{concordances}, which are cobordisms that are diffeomorphic to cylinders. 
\note{I.e.\ without the context of their embedding, they \emph{are} cylinders.}
We will use the more precise term \emph{annulus} instead of `cylinder'. 

\begin{definition}
Let $K_0, K_1$ be \textbf{knots} in $S^3$. A \emph{concordance} $C$ from $K_0$ to $K_1$ is an oriented cobordism such that $C \cong S^1 \times I$. 

Equivalently, $C: K_0 \to K_1$ is a concordance if it is a (smooth, oriented) connected cobordism with $\chi(C) = 0$.

If such a $C$ exists, then we say $K_0$ and $K_1$ are \emph{concordant}: $K_0 \sim K_1$. This is an equivalence relation, and the equivalence classes are called \emph{concordance classes}. 
\end{definition}

In fact, we can turn the set of knots into a group (!) by modding out by concordance:

\begin{definition}
    The \emph{smooth knot concordance group} $\concgroup$ is the group where 
    \begin{itemize}
        \item the elements are the concordance classes knots in $S^3$;
        \item the binary operation is induced by $\#$ (connected sum);
        \item the identity element is the class of \emph{slice knots}, or knots that are concordant to the unknot;
        \item inverses are given by mirroring.
    \end{itemize}
\end{definition}

\begin{remark}
If you take a knife to a $B^4$ and cut off a slice, the cut you make is a slice disk. \note{I haven't checked if this is the historical origin of the text.}

In general, a surface $F$ properly embedded in $B^4$ whose boundary is $\partial F = K$ is called a \emph{slice surface} for $K$. 
\end{remark}

\begin{remark}
    We can equivalently say that a knot $K \subset S^3$ is \emph{slice} if it bounds a disk $D$ in $B^4$: if we arrange $D$ so that it is in Morse position with respect to the radial function on $B^4$, then the boundary of a neighborhood of its lowest 0-handle is an unknot. 
    \note{In other words, an annulus is just a punctured disk.}

    In this case, we say that $D$ is a \emph{slice disk} for $K$.
\end{remark}

Again, this is actually a definition-theorem, and we will prove the theorem after seeing a quick example.

\begin{example}
\label{eg:square-knot-slice}
    Consider the right-handed trefoil $K$:
    \begin{center}
        \includegraphics[width=1in]{images/RH-trefoil-diagram.png}
    \end{center}
    
    If we claim that the left-handed trefoil $m(K)$ represents the inverse concordance class, then we must show that $K \# m(K)$ is concordant to the unknot $U$.
    
    It suffices to show that $K \# m(K)$ bounds a disk $D$ embedded in $B^4$. 
    You can see a projection of a slice disk in $S^3$ in the following picture:

    \begin{center}
        \includegraphics[width=2in]{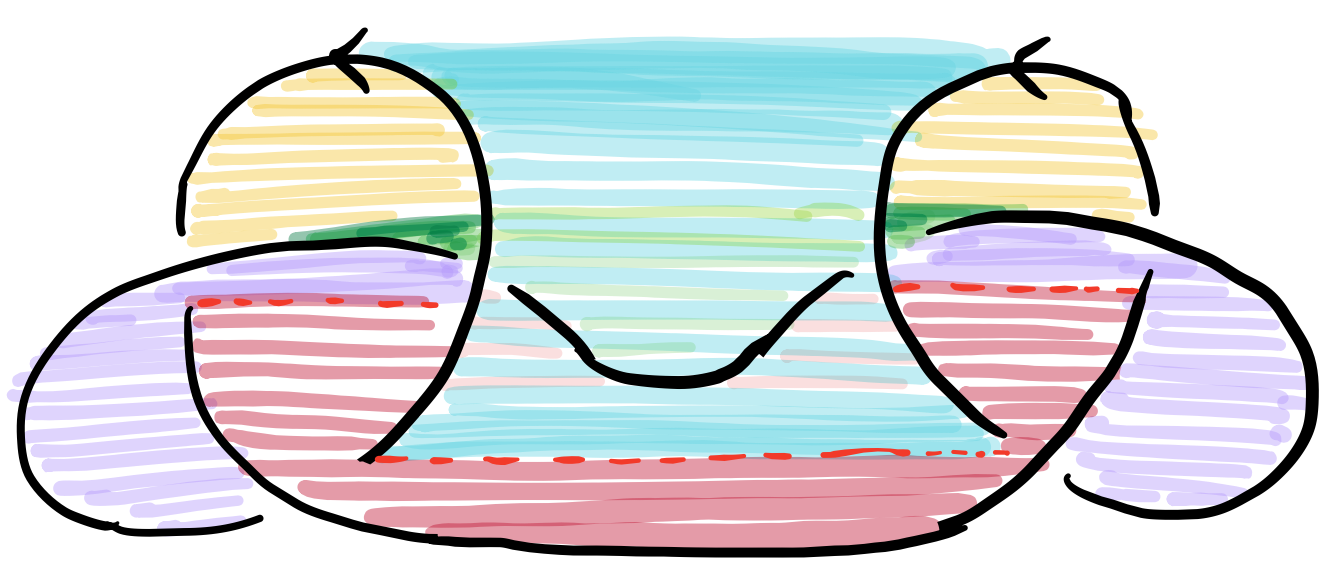}
    \end{center}

    The disk is immersed in $S^3$, and the only intersections are of the following form, called \emph{ribbon intersections}:

    \begin{center}
        \includegraphics[width=1in]{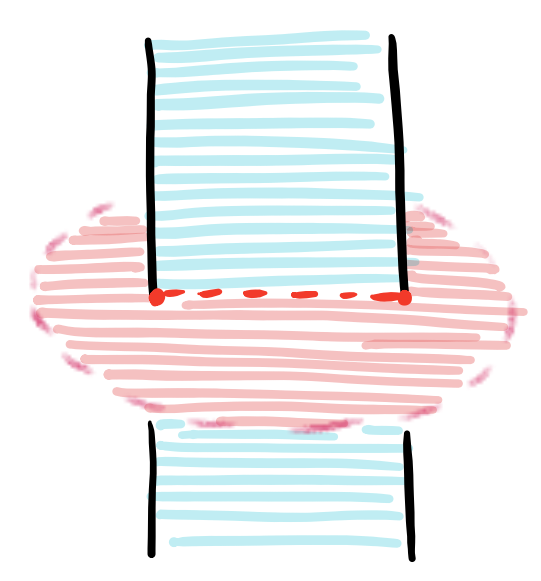}
    \end{center}

    The interior of the horizontal sheet can be pushed deeper into $B^4$, so that the slice disk in $B^4$ has no self-intersection. 
    
\end{example}

To verify that the concordance group really is a group, we need to check that 
\begin{enumerate}
\item $\#$ is well-defined on equivalence classes
\item $\#$ is associative
\item the equivalence class $[U]$ acts as the identity 
\item for any $K$, $K \# m(K)$ is slice.
\end{enumerate}
You can convince yourself that the binary operation $\#$ on the set of knots is both associative and commutative; if you want to think about this in more detail, see \cite{Adams-knot-book}. 
Once we show that $\#$ is well-defined, showing that $[U]$ is the identity element is also easy. 

It remains to show that $\#$ is well-defined, and that $K \# m(K)$ is slice. Both proofs use standard topological arguments.

\begin{claim}
    $\#$ is a well-defined binary operation on the set of concordance classes. 
\begin{proof}
To see that $\#$ is a well-defined binary operation on concordance classes, consider knots $K, K', J \subset S^3$, where $K \sim K'$. Then there is some concordance $C: K \to K'$, a cobordism in $\S^3 \times [0,1]$. 
Pick basepoints $p \in K$, $p' \in K'$ and isotope $K'$ so that $p' = p$ (as points in $S^3$). 
Pick an arc $\gamma: [0,1] \into C$ such that $\gamma(0) = p$ and $\gamma(1) = p'$. 

Perform a boundary-preserving ambient isotopy to `straighten out' $\gamma$; that is, to arrange so that $\gamma(t) \in S^3 \times \{t\}$. (This is possible because the codimension of $\gamma$ in $S^3 \times [0,1]$ is 3, and codimension 3 submanifolds can always be unknotted.)
We will now assume that $C$ is in such a position so that $\gamma$ is of the form $p \times [0,1]$. 

\begin{center}
    \includegraphics[width=3in]{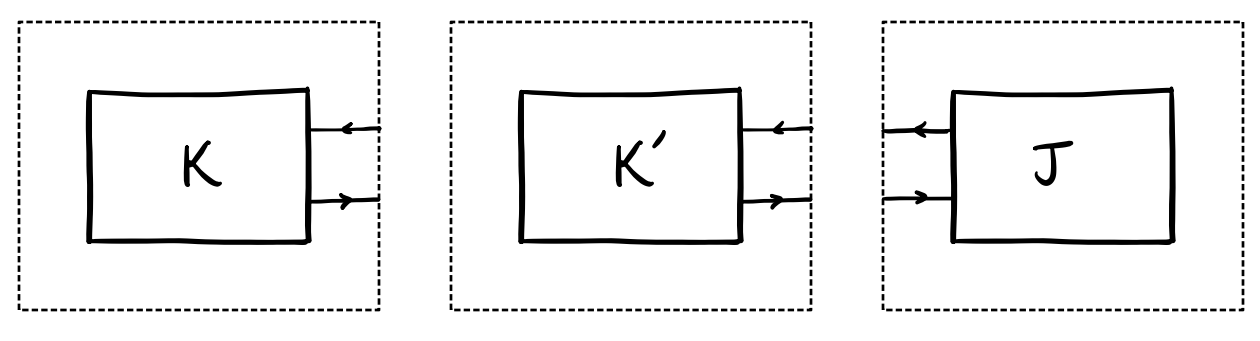}
\end{center}

Now delete a small neighborhood of $\gamma$ (i.e. $\nu(p) \times [0,1] \cap C$) in $C$. 
Pick a point $q \in J$ and delete it, forming a $(1,1)$ tangle $J - \nu(q)$ whose closure is $J$. Shrink this tangle so that it fits within $\nu(p)$.

Finally, glue $C - (\nu(p) \times [0,1])$ with $(J - \nu(q)) \times [0,1]$ to form a concordance from $K \# J$ to $K'\# J$. 
\begin{center}
    \includegraphics[width=3in]{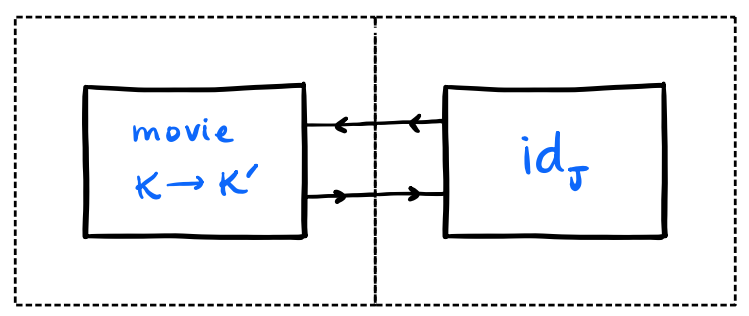}
\end{center}

\end{proof}
\end{claim}

\begin{aside}
\textit{How much should one write for such a proof?}
You might notice that, for example, the last sentence of the preceding proof is not super duper precise. 
However, the figure helps you understand the notation, and also the underlying concept is quite simple. Also, we were careful to define the category $\catDiag$ so that we can make diagrammatic arguments.

As a human who uses language, I don't have a perfect answer for `how much detail to show' -- this comes from getting to know the common vocabulary, techniques, facts, and tricks used in the community you are writing for. 
\end{aside}

We also give a (more terse) proof sketch that $[m(K)] = [K]\inv$:

\begin{claim}
    For any knot $K \subset S^3$, $K \# m(K)$ is slice.
\begin{proof}
    The identity cobordism $K \to K$ is a concordance. 
    Pick a point $p \in K$ and delete $p \times [0,1]$ from the identity cobordism. The resulting surface can be properly embedded in $B^4$ and viewed as a slice disk for $K \# m(K)$. 

    \begin{center}
        \includegraphics[width=3in]{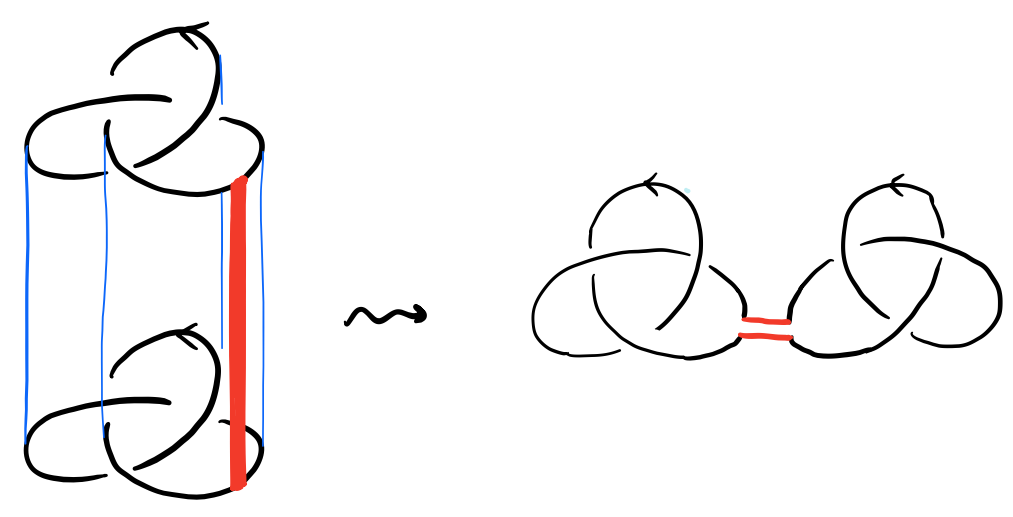}
    \end{center}
    
\end{proof}
\end{claim}

\begin{remark}
One can also define a notion of concordance between links. However, the notion of a `link concordance group' isn't obvious, and is an area of active research.
\end{remark}

\subsection{Obstruction to ribbon concordance}

In Example \ref{eg:square-knot-slice}, we envisioned a slice disk by seeing that its projection to $S^3$ had only \emph{ribbon singularities}. The following definition gives a more general definition of such phenomena.

\begin{definition}
Let $F: L_0 \to L_1$ be a cobordism embedded in $S^3 \times I$ such that the second coordinate gives a Morse function. 
We say $F$ is \emph{ribbon} if, with respect to the Morse handle decomposition, $F$ contains only 0- and 1-handles.

A ribbon cobordism that $D: \emptyset \to K$ that is a disk is called a \emph{ribbon disk}, and is a special case of a slice disk. 

A knot $K$ that bounds a ribbon disk is called a \emph{ribbon knot}. 
\end{definition}

\begin{example}
Here is a movie for the ribbon disk for the square knot that we visualized in Example \ref{eg:square-knot-slice}:
\begin{center}
    \includegraphics[width=2in]{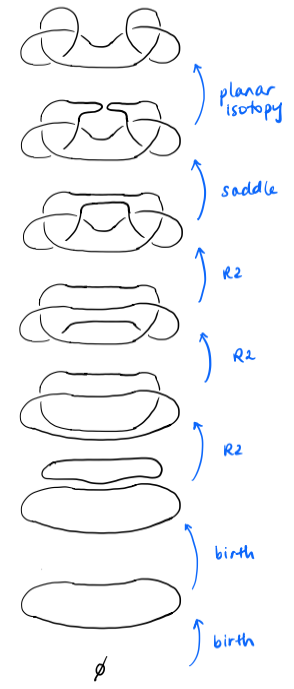}
\end{center}
\end{example}

From the definition, the `ribbon concordant' relation is not symmetric. If we flip a ribbon concordance containing a $0$-handle upside-down (by reversing the Morse function),  the upside-down concordance $\bar C$ clearly has a 2-handle.

\begin{exercise}
    Draw a movie for the ribbon disk for the Stevedore knot $6_1$ implied in the diagram below:
    \begin{center}
        \includegraphics[width=1.5in]{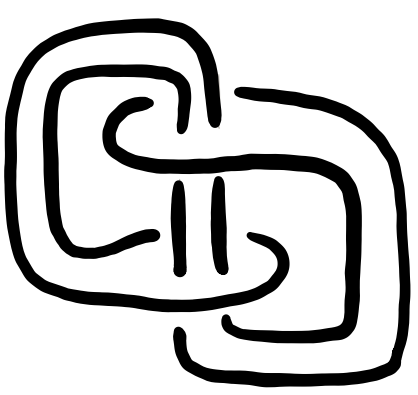}
    \end{center}
\end{exercise}

\begin{remark}
Unfortunately, we naturally would want to say `$6_1$ is concordant to the unknot', which is true, but not if we add the word `ribbon'.
Mathematically it makes more sense to say $U$ is ribbon concordant to $6_1$, and this is the language used in \cite{Levine-Zemke-ribbon}. Historically, some people defined ribbon surfaces by taking the descending Morse function and requiring only 1- and 2-handles, so be careful.
For my sanity, I will always specify the direction of the cobordism as a morphism.
\end{remark} 

\begin{remark}
    In terms of movies, a ribbon disk is a `happy movie', where the only critical moments are births of circles and merges of circles. 
    There are no scenes where circles split or die. 
\end{remark}

\begin{conjecture}(Open: Slice-Ribbon Conjecture)
All slice knots are ribbon.
\end{conjecture}

It is clear that ribbon disks are slice, so any ribbon knot is a slice knot. 
However, not all slice \emph{disks} are (isotopic to) ribbon disks (see \href{https://people.mpim-bonn.mpg.de/aruray/documents/slicenotes.pdf}{Aru Ray's notes}, Proposition 1.8). 
So, the conjecture posits that if $K$ bounds a slice disk, it bounds a (potentially non-isotopic) ribbon disk.

\begin{remark}
We will see soon that there are knots that bound non-isotopic ribbon disks \cite{hay-sun}.
\end{remark}

Levine--Zemke showed that we can use Khovanov homology to obstruct the existence of ribbon cobordisms between two knots:

\begin{theorem}[\cite{Levine-Zemke-ribbon}]
\label{thm:Kh-ribbon-obstruction}
    If $C: K_0 \to K_1$ is a ribbon concordance, then the morphism
    \[
        \Kh(C): \Kh(K_0) \to \Kh(L_1)
    \]
    is injective, with left inverse $\Kh(\bar C)$.
\end{theorem}

Here are some immediate corollaries that follow from basic algebra:
\begin{corollary}[\cite{Levine-Zemke-ribbon}]
Suppose $C: K_0 \to K_1$ is a ribbon concordance.
    \begin{enumerate}
        \item At any bigrading, $\Kh^{i,j}(K_0) \into \Kh^{i,j}(K_1)$ as a direct summand.
        \item If additionally there is a ribbon concordance $C': K_1 \to K_0$, then $\Kh(K_0) \cong \Kh(K_1)$.
    \end{enumerate}
\end{corollary}
\note{Seriously, prove these for yourself.}
There are four more corollaries in the paper, if you're interested. 

The proof uses a topological lemma first appearing in Zemke's \cite{Zemke-HFK-ribbon}. This paper sparked a whole series of papers proving similar or related results for other functorial link homology theories.

The proof of this lemma is embedded in the proof of Zemke's main theorem:

\begin{lemma}[\cite{Zemke-HFK-ribbon}]
\label{lem:Zemke-tubing}
Let $C: K_0 \to K_1$ be a ribbon concordance, and let $\bar C$ be the upside-down (and opposite orientation) concordance to $C$, a morphism $K_1 \to  K_0$. 

If a movie presentation for $C$ has $n$ births and $n$ saddles (note that the Euler characteristic of a concordance is 0), then there is a movie presentation for $\bar C \circ C$ with $n$ births, $n$ merge saddles, $n$ split saddles that are dual to the merge saddles, and $n$ deaths. 
\end{lemma}

\begin{center}
    \includegraphics[width=2in]{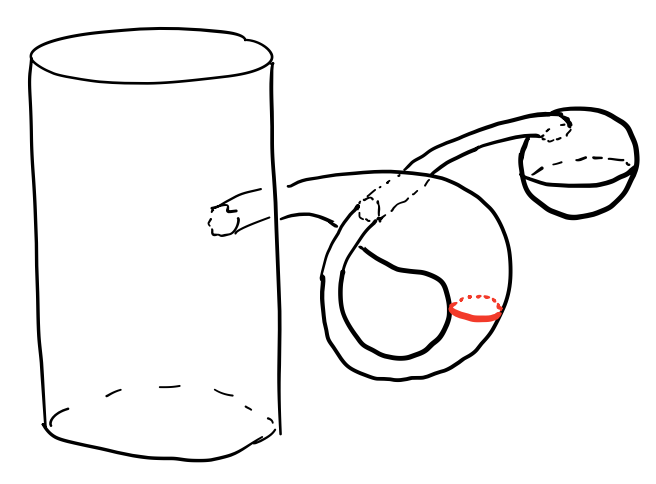}
\end{center}

In particular, the lemma tells us that $\bar C \circ C$ is isotopic to a cobordism $K_0 \to K_0$ that looks like the identity cobordism for $K_0$ with $n$ spheres tubed on. Because of nontrivial knot theory for surfaces in dimension 4, this cobordism might not be isotopic to the identity cobordism, but nevertheless, Khovanov homology can't tell because of the neckcutting relation:

\begin{proof}[Proof of Theorem \ref{thm:Kh-ribbon-obstruction}]
Take $C$ and $\bar C$ as in Lemma \ref{lem:Zemke-tubing}.
In $\TLcat_0$, by neckcutting, we see that the morphism $\bar C \circ C$ is equal to the morphism with $2^n$ summands that are all of the form $\id_{K_0}$ disjoint union with $n$ spheres, and all summands have just one dot somewhere. After deleting the summands containing an undotted sphere, the only remaining morphism is the one consisting of $\id_{K_0}$ and $n$ dotted spheres, which evaluate to a coefficient of 1. 
Therefore $\bar C \circ C = \id_{K_0}$ as morphisms in $\TLcat_0$.

By (projective) functoriality of $\Kh$, 
we have $\Kh(\bar C)\circ \Kh(C) = \Kh(\bar C \circ C)$, and the remainder of the theorem follows.
\end{proof}

\subsection{Lee homology and the $s$ invariant}
\label{sec:rasmussen-s}

The most important numerical invariant(s) that you can distill from Khovanov homology techniques is Rasmussen's $s$ invariant, which was introduced in \cite{Rasmussen-s}.
It is a `concordance homomorphism,' i.e.\ a homomorphism of groups 
\[
    s: \concgroup \to \Z.
\]
Rasmussen showed that $s$ gives a lower bound on the \emph{slice genus} of a knot. The following definition defines two knot invariants:

\begin{definition}
A \emph{Seifert surface} for a knot $K \subset S^3$ is an oriented surface embedded in $S^3$ whose boundary is $K$.
The 3-ball genus or \emph{Seifert genus} of a knot $K$, $g_3(K)$, is the minimal genus of a Seifert surface for $K$: 
\[
    g_3(K) = \min \{ g(F) \st F \into S^3 \text{ with } \partial{F} = K\}
\]

Analogously, the 4-ball genus or \emph{slice genus} of a knot $K \in S^3$ is
    \[
        g_4(K) = \min \{ g(F) \st F \into B^4 \text{ with } \partial{F} = K\},
    \]
the minimal genus of a slice surface for $K$. 
\end{definition}

\begin{example}
    Here is a Seifert surface for the right-handed trefoil, otherwise known as the torus knot $T_{2,3}$:
    \begin{center}
        \includegraphics[height=1.5in]{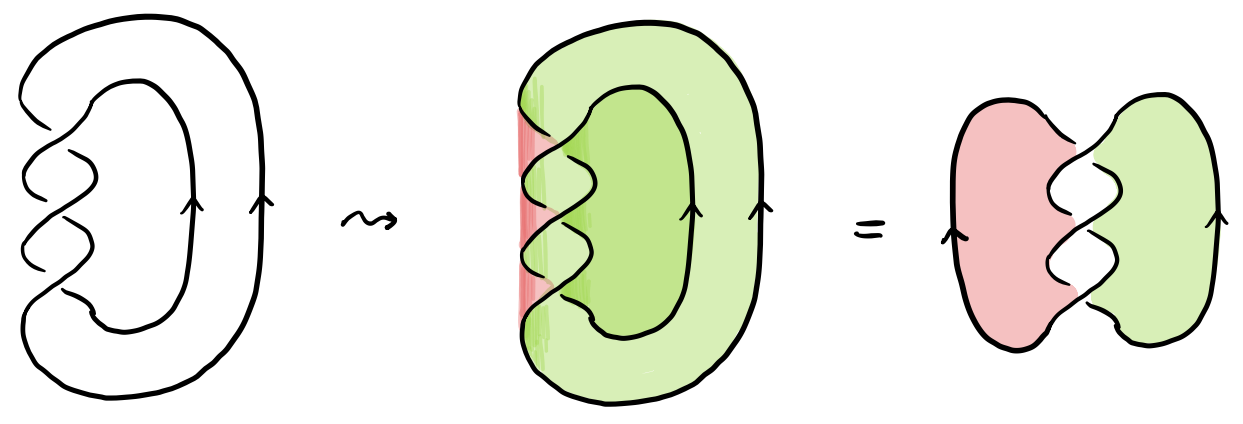}
    \end{center}
\end{example}

Since any Seifert surface can be `punched in' to $B^4$ to become a slice surface, for any knot $K$, we have $g_4(K) \leq g_3(K)$, and therefore $s$ also gives a lower bound on the Seifert genus of $K$.

Rasmussen used the $s$ invariant to give a combinatorial proof of the (topological) Milnor conjecture: 

\begin{theorem}[\cite{KM-gauge-embedded-I}, \cite{Rasmussen-s}]
The slice genus of a torus knot $T(p,q)$ is 
\[
    g_4(T(p,q)) = g_3(T(p,q)) = \frac{(p-1)(q-1)}{2}.
\]
\end{theorem}
\note{I bet the first Seifert surface you draw for $T(p,q)$ is the minimal genus (both $g_4$ and $g_3$!) surface!}
 
Another major application of Rasmussen's $s$ invariant in 4D topology is Piccirillo's proof that the Conway knot is not slice \cite{Piccirillo-conway-knot}; this was a long-standing conjecture until she proved it in less than 8 pages. This is an \textit{Annals of Mathematics} paper.
\note{For those of you interested in 4-manifolds and using Kirby calculus, this is a potential final project idea.}

The $s$ invariant relies on \emph{Lee homology}, a version of Khovanov homology introduced by E.~S.~Lee in her study of Khovanov homology of alternating knots \cite{Lee-endomorphism}, very soon after Khovanov homology was first introduced. We now understand Lee homology as the version corresponding to the Frobenius algebra $\CA' = \Q[X]/(X^2-1)$.

In this section, we will discuss Lee homology and Rasmussen's $s$ invariant, as well various related ideas, including extremely useful algebraic constructions such as filtration spectral sequences and their relationship to torsion order.
This will take a while, but these constructions were fundamental to the various applications of Khovanov homology in the following two decades.


\subsubsection{Lee homology and quantum filtration}

In this section we replace Khovanov's TQFT $\CA$ with Lee's, which we denote using primes (e.g.\ $\CA'$), following Rasmussen's paper \cite{Rasmussen-s}.
\note{In future sections of these notes we will likely use more specific notation.}

\emph{Lee's TQFT} comes from the Frobenius system 
\[
    (\Q, \CA' = \Q[X]/(X^2-1), \iota', m', \Delta', \varepsilon')
\]
where the structure morphisms are given by 

    \begin{align*}
        \iota' : \Q &\to \CA' \\
                1 &\mapsto 1
    \end{align*}
    \begin{align*}
        m' : \CA' \otimes \CA' &\to \CA' \\
                1 \otimes 1 &\mapsto 1 \\
                X \otimes 1, \ 1 \otimes X &\mapsto X \\
                X \otimes X & \mapsto \alert{1}
    \end{align*}
    \begin{align*}
        \Delta' : \CA' &\to \CA' \otimes \CA' \\
                1 &\mapsto X \otimes 1 + 1 \otimes X \\
                X &\mapsto X \otimes X  + \alert{1 \otimes 1}
    \end{align*}
    \begin{align*}
        \varepsilon' : \CA' &\to \Q \\
                1 &\mapsto 0\\
                X &\mapsto 1.
    \end{align*}

The most important thing to notice right now is that these maps are not grading preserving! After all, if you mod out by a non-quantum-homogeneous polynomial $X^2 - 1$, you will not get a graded theory. 
\note{But since $\deg X^2 = 4$, you \emph{do} get a $\Z/4\Z$-quantum-graded theory, and this \emph{will} be used in Rasmussen's proofs.}

Lee originally described her differentials as a perturbation of Khovanov's. Indeed, we may write 
\[
    d_{\Lee} = d_{\Kh} + \alert{ \Phi}
\]
where $\Phi$ consists of the red terms in the $m'$ and $\Delta'$ morphisms above. 

Check by inspection that while $d_{\Kh}$ preserves quantum grading, $\Phi$ increases quantum grading by 4. 
So even though $d_{\Lee}$ does not preserve grading, the Lee chain complex 
\[
    \CLee(K) = (\CKh(K), d_{\Kh} + \Phi)
\]
is \emph{filtered}:

\begin{definition}
\label{defn:filtered-complex}
    
    A \emph{filtration} of a chain complex $(\CC, d)$ is a sequence of subcomplexes $F_\bullet$ of $(\CC, d)$
    \[
        \cdots \supseteq  F_i
        \supseteq  F_{i+1}
        \supseteq  F_{i+2}
        \supseteq  F_{i+3}
        \supseteq \cdots
    \]
    such that 
    \[
        \bigcap_i F_i = \emptyset 
        \qquad
        \text{and}
        \qquad
        \bigcup_i F_i = \CC
    \]
    and where $d(F_i) \subseteq F_i$.

    A chain map $f: (\CC, d) \to (\CC', d')$ between filtered complexes $\CC = \bigcup F_i$ and $\CC' = \bigcup F'_i$
    is a chain map that respects the filtration: $f(F_i) \subseteq F'_i$. 
\end{definition}

\begin{example}
    Here is a cartoon of a filtered chain complex:
    \begin{center}
        \includegraphics[width=2in]{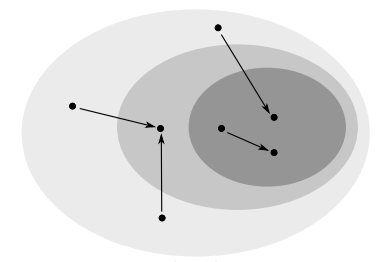}
    \end{center}
\end{example}

\begin{warning}
    This is only one type of filtration; see a homological algebra book for a more general definition.
    We will in particular be using \emph{finite length filtrations}, where only finitely many $F_i$ are not $\emptyset$ or $\CC$. 
\end{warning}

The filtration on the Lee complex is very special, because it is defined by a  \emph{filtration grading} $\gr_q$ on a free module generated by a distinguished, $\gr_q$-homogeneous basis.

Let $\Kg(L)$ denote the set of pure tensors that generate the Khovanov chain group $\CKh(L)$. 
Each chain is a finite sum $x = \sum_{i=1}^k c_i g_i$ where $g_i \in \Kg(L)$, and $c_i \in \Q^\times$.
The filtration grading is defined for non-homogeneous chains by the following formula:
\begin{equation}
\label{eq:quantum-filtration-grading}
    \gr_q(\sum_{i=1}^k c_i g_i) 
    = \min \{ \gr_q(g_i)\}_{i=1}^k.
\end{equation}
Observe that, by this definition, for all $j \leq \gr_q(x)$, we have $x \in F_j$; meanwhile, for $j \gneq \gr_q(x)$, $x \not\in F_j$.

\begin{remark}
    Filtered complexes give rise to \emph{filtration spectral sequences}, which we will likely discuss later on. For now, just know that because the Khovanov differential is the grading-preserving piece of the Lee differential, there is a spectral sequence relating Khovanov homology to Lee homology.
\end{remark}

\subsubsection{Lee's basis and canonical generators}

At first, the Lee complex might look \emph{more} complicated than the Khovanov one, but this is only because we made an inconvenient choice of distinguished basis, because in the Khovanov TQFT, we wanted a homogeneous basis. 
Note that $\CA'$ is still isomorphic, as a non-graded vector space, to what we will still call $V = \Q v_+ + \Q v_-$, \note{even though we've switched to $\Q$ coefficients, so we're technically working with $V \otimes_{\Z} \Q$}.
Lee introduced a much better distinguished basis for her TQFT, where 
\[
    \bolda = v_- + v_+ = 1 + X
    \qquad
    \boldb = v_- - v_+ = 1 - X.
\]
Notice that these are the factors of $X^2 - 1$; therefore, you already automatically know that 
$m'(\bolda \otimes \boldb) = m'(\boldb \otimes \bolda) = 0$, which is nice.
Here are the other morphisms, in this new basis:

\begin{notation}
    \label{nota:lee-better-basis}
    \begin{align*}
        \iota' : \Q &\to \CA' \\
                1 &\mapsto \frac{\bolda - \boldb}{2}
    \end{align*}
    \begin{align*}
        m' : \CA' \otimes \CA' &\to \CA' \\
                \bolda \otimes \bolda &\mapsto 2\bolda \\
                \bolda \otimes \boldb,\  \boldb \otimes \bolda &\mapsto 0\\
                \boldb \otimes \boldb & \mapsto -2\boldb
    \end{align*}
    \begin{align*}
        \Delta' : \CA' &\to \CA' \otimes \CA' \\
                \bolda &\mapsto \bolda \otimes \bolda \\
                \boldb &\mapsto \boldb \otimes \boldb
    \end{align*}
    \begin{align*}
        \varepsilon' : \CA' &\to \Q \\
                \bolda, \  \boldb  &\mapsto 1
    \end{align*}
\note{You may wish to check that this indeed defines a Frobenius algebra by checking the required relations.}
\end{notation}

With a Frobenius algebra structure so simple (and with 20 years of reflection and hindsight), one would expect that the homology would be very easy to compute. Indeed, Lee showed that Lee homology is very simple: 

\begin{theorem}[\cite{Lee-endomorphism}, Theorem 4.2]
\label{thm:Lee-structure-theorem}
For a link $L$ with $\ell$ components, there is an isomorphism of (ungraded) vector spaces $\Lee(L) \cong (\Q \oplus \Q)^\ell$. 
\end{theorem}

Moreover, she knows the homology classes that generate the $2^{\ell}$-dimensional homology; these correspond to the $2^{\ell}$ orientations of the link!
These special homology classes are now known as \emph{ Lee's canonical generators}. \note{And are so, so, so important.}

Here is an algorithm used for obtaining Lee's canonical generators. 

\begin{algorithm}
\label{algo:Lee-canonical-generators}
    Suppose we are given a link $L$, an orientation $o$ on $L$, and a diagram $D$ of $L$. 
    Fix a checkerboard coloring of the complement of $D$ in the plane.
    (A standard convention is to leave the infinite region unshaded.)

    Let $D_o$ denote the oriented resolution: that is, resolve each crossing so that the orientations of the arcs are all preserved. 
    The resulting circles have induced orientations. For each circle, draw a small dot slightly to the left of any point on the circle (based on the orientation of the circle). 
    \begin{itemize}
        \item If the dot lands in a shaded region, label the circle $\bolda$.
        \item If the dot lands in an unshaded region, label the circle $\boldb$. 
    \end{itemize}
    Let $\fraks_o$ denote the chain correspond to the above labeling of the resolution $D_o$
\end{algorithm}

\begin{exercise}
    Prove that $\fraks_o$ is closed (i.e.\ in the kernel of the differential).
\end{exercise}

Lee's canonical generators are the set of homology classes 
\[
    \{ [\fraks_o] \st o \text{ is an orientation for } L\}.
\]

In the remainder of this (subsub)section, we give a proof of Lee's Theorem \ref{thm:Lee-structure-theorem} that is due to Bar-Natan--Morrison \cite{BN-Morrison}.
\note{There are many versions of Lee's proof in the literature. They all involve similar ideas to those in the original proof.}

First of all, we need to set up the Bar-Natan category that works for Lee homology. 
In particular, we know that $X^2=1$, so now `two dots' must evaluate to 1, rather than zero. 
Below are the Bar-Natan categories we will use. Beware that the notation is not standard, and in fact does not appear anywhere in the literature that I know of (\cite{BN-Morrison} calls these categories `$\mathcal{C}\mathit{ob}_1(\partial T)$'). However, we use this primed notation to be consistent with the notation used in Rasmussen's $s$ invariant paper \cite{Rasmussen-s}.

\begin{definition}
    The preadditive categories $\TLcat'_n$ for $n \in \N \cup \{0\}$ have the same objects as $\TLcat_n$. The morphisms are also cobordisms, but modulo these slightly modified relations:
    \begin{center}
        \includegraphics[width=3in]{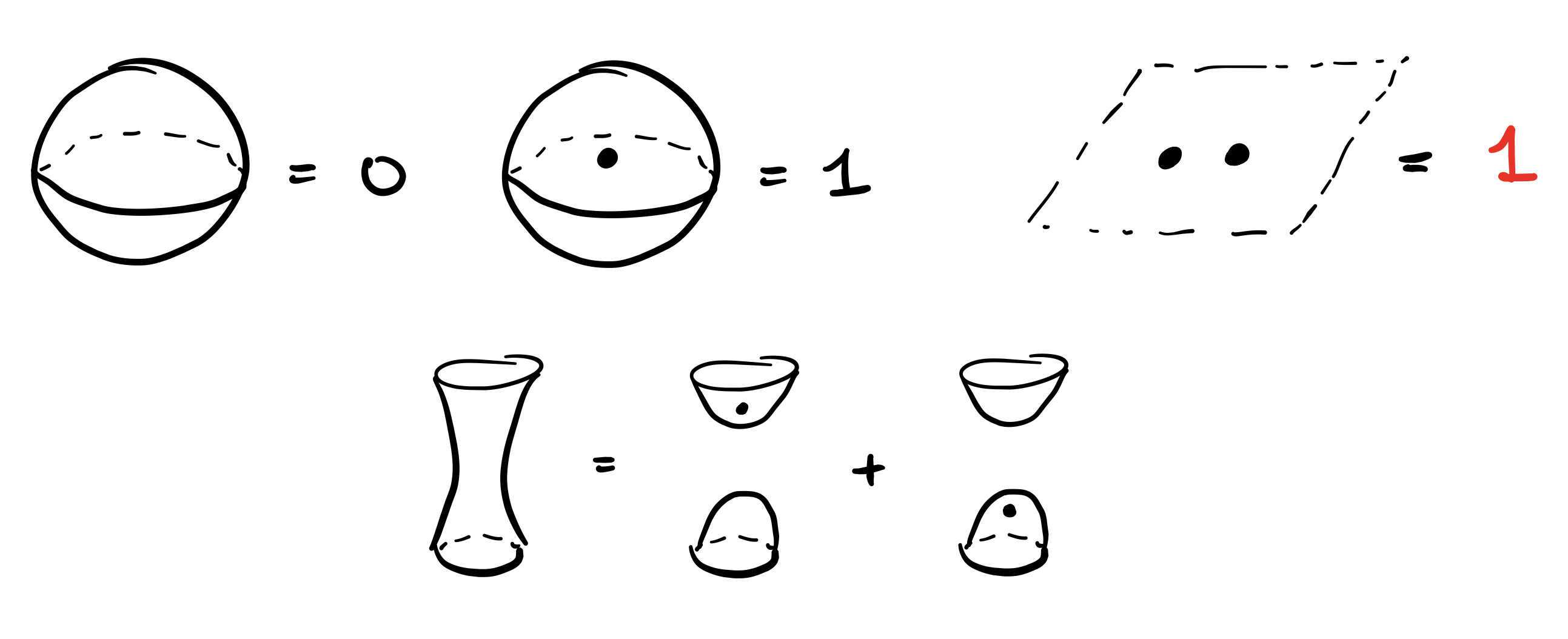}
    \end{center}
\end{definition}

\begin{remark}
    If you wanted to, you could verify that these relations still hold. The dotted identity still means `$\cdot X$'. What's more interesting is coming up with a new decorated cobordism category that innately uses the $\{\bolda, \boldb\}$ basis for $\CA'$. (So, define new kinds of dots.) What would your morphism relations look like, with these new decorations?
    \note{There is an answer in the literature, and major hints in the remainder of the section. Also, I could have phrased this as an Exercise but I would not want to grade this problem!}
\end{remark}

We will need an additional algebraic construction; we follow the exposition in \cite[Section 3]{BN-Morrison}.
Recall that in any category $\CC$, any endomorphism $p \in \Hom_{\CC}(X,X)$ satisfying $p^2 = p\circ p = p$ is called a \emph{projection}.
\note{We say that $p$ is \emph{idempotent}.}
The \emph{Karoubi envelope} of $\CC$, denoted by $\Kar(\CC)$, is an enlargement of the category $\CC$ that includes all images of idempotent endomorphisms:

\begin{definition}
The \emph{Karoubi envelope} (or \emph{idempotent completion}) of $\CC$, denoted by $\Kar(\CC)$, is the category where
\begin{itemize}
    \item objects are pairs $(X,p)$ where $X \in \Ob(\CC)$, $p \in \Hom_{\CC}(X,X)$, and $p^2 = p$; and 
    \item morphisms from $(X,p)$ to $(X',p')$ are the morphisms $f \in \Hom_{\CC}(X,X')$ that respect the projections: $f\circ p = p' \circ f$.    
\end{itemize}
\end{definition}

In other words, if you ever find a projection $p$ in (the morphisms of) $\CC$, then its image $\im(p)$ is an object in $\Kar(\CC)$. In particular, every object in $\CC$ also appears in $\Kar(\CC)$, in the form $(X, \id_X)$. 

\begin{example}
The Karoubi envelope occurs in the wild!
For example, let $\mathbf{Open}$ be the category consisting of 
\begin{itemize}
    \item objects: open subsets of Euclidean space (all $\R^n$ for $n \in \N$)
    \item morphisms: smooth maps between these open sets.
\end{itemize}
The Karoubi envelope of $\mathbf{Open}$ turns out to be the category of all smooth (closed, finite-dimensional) manifolds $\mathbf{Man}$. 
For example, while $S^n$ is not an open subset of any $\R^m$, it is the image of the projection
\begin{align*}
    p: U &\to U \\
        x &\mapsto \frac{x}{|x|}
\end{align*}
where $U$ is the open subset $\R^n - \{0\} \subset \R^n$. 
This example comes from the ncatlab page for `Karoubi envelope,' \cite{nlab:karoubi_envelope}.
\end{example}

Here are some more important algebraic facts that are not hard to prove.

\begin{proposition}
\label{prop:karoubi-facts}
Let $\CC$ be an additive category.
\bea
    \item If $p \in \Hom(X,X)$ is a projection, then so is $\id-p \in \Hom(X,X)$.  Thus in $\Kar(\CC)$, $X \cong \im(p) \oplus \im(\id - p)$. 
    \item Let $A$ and $B$ be chain complexes in $\Kom(\CC)$. If $A$ and $B$ are chain homotopy equivalent in $\Kom(\Kar(\CC))$, then they were also homotopy equivalent in $\Kom(\CC)$. 
\ee
\end{proposition}

\begin{exercise}
    Prove Proposition \ref{prop:karoubi-facts}. 
\end{exercise}

\begin{remark}
Recall from, say, Mat167 that involutions are related to projections:
if $P$ is a projection, then $I-2P$ is an involution. 
Geometrically, this involution reflects across the plane that $P$ projects onto, because you (a point in $\R^n$) are being translated by a vector to the plane, and then you overshoot by traveling by that vector again.

Conversely, given an involution $S$, the \emph{eigenprojections} of $S$ are given by 
$\frac{I \pm S}{2}$. 
If $P$ is an eigenprojection for $S$, then $SP = \lambda P$ where $\lambda = \pm 1$. 
These correspond to $\im(P)$ and $\im(P)^\perp$, because $P(I-P) = I$.

For example, the projection $P = \begin{pmatrix} 1 & 0 \\ 0 & 0 \end{pmatrix}$ projects points in $\R^2$ to the $x$-axis. The projection $I-P$ projects the $y$-axis. 
The involution (reflection) $I-2P$ reflects across the $x$-axis. 
This fixes the points on the $x$-axis ($\lambda= 1$) but reflects the points on the $y$-axis ($\lambda = -1$).
\end{remark}

Now consider the additive category $\Mat(\TLcat_1')$. The object with no closed components
\begin{center}
\begin{tikzpicture}[scale=.5]
    \draw[dotted] (0,0) circle (1cm);
    \draw[thick] (0,-1) -- (0,1);
\end{tikzpicture}
\end{center}
has an interesting involution: the `$\cdot X$' map (however, we can only use this terminology after passing through the TQFT, which we have not). Recall that this morphisms is just the dotted identity map. We will call this endomorphism $b$ for \texttt{\textbackslash bullet}, following Bar-Natan--Morrison. The `two dots = 1' relation in $\TLcat'$ means that $b^2 = \id$. 
The corresponding idempotents for this involution are
\[
    {\color{Red} r = \frac{\id+b}{2}} 
    \qquad \text{and}  \qquad 
    {\color{OliveGreen} g = \frac{\id-b}{2}}
\]
and they satisfy $r^2=r$, $g^2=g$, $r+g = \id$, $rg = 0$, $br=r$ and $bg=-g$. 
So, in $\Kar(\Mat(\TLcat_1'))$, we have a new decomposition of our  object 
\[
\begin{tikzpicture}[scale=.5, baseline=-.33em]
    \draw[dotted] (0,0) circle (1cm);
    \draw[thick] (0,-1) -- (0,1);
\end{tikzpicture}
=
\begin{tikzpicture}[scale=.5, baseline=-.33em]
    \draw[dotted] (0,0) circle (1cm);
    \draw[thick, Red ] (0,-1) -- (0,1);
    \node at (.3,0) {$r$};
\end{tikzpicture}
\oplus
\begin{tikzpicture}[scale=.5, baseline=-.33em]
    \draw[dotted] (0,0) circle (1cm);
    \draw[thick, OliveGreen] (0,-1) -- (0,1);
    \node at (.3,0) {$g$};
\end{tikzpicture}
\]

In $\Mat(\TLcat_n')$, for every object containing no closed components, we have a similar story; for each strand, the dotted identity (with the dot on the sheet corresponding to that strand) is an involution, and we get a projections $r$ and $g$ for each strand.
In short, an $n$-strand crossingless matching has $2^n$ colorings that assign either $r$ or $g$ to each strand. The $2^n$ colors correspond to the $2^n$ projections $\{r,g\}^{\otimes n}$.

Let's now study the chain complex associated to a crossing in the category $\Kom(\Kar(\Mat(\TLcat_2')))$:

\begin{center}
    \includegraphics[width=3in]{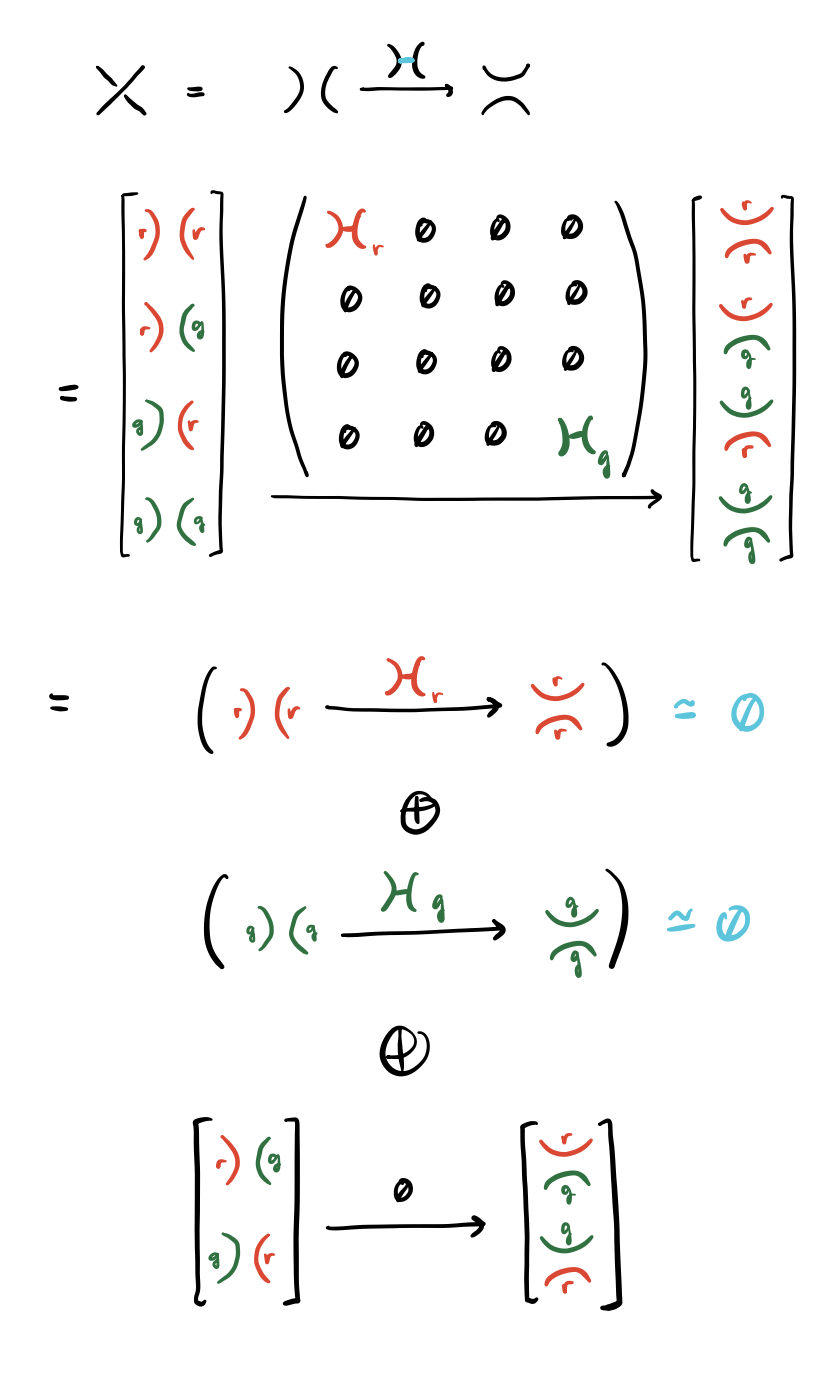}
\end{center}

Above, we claim that first two summands are contractible. To show this, it suffices to see that the red and green saddle maps are chain homotopy equivalences. \note{To see this, try to come up with inverses for these saddle maps.}
Therefore the chain complex for the crossing is equivalent to the last summand, which has no differential.
The generators of homology are therefore the four objects drawn. 

Now consider a link with $n$ crossings. A neighborhood of every crossing (a diagram representing a complex in $\Kom(\Kar(\Mat(\TLcat_2')))$ is equivalent to the complex above with 0 differential. The surviving homology classes are therefore obtained by gluing these complexes in the spaghetti-and-meatballs picture; more precisely, by tensoring together these complexes in a planar algebra. 
Observe that if, upon tensoring, a single strand is colored by red and green, that diagram dies, because $rg = 0$. 
\note{Make sure this really actually makes sense to you.}
The only remaining objects are complete resolutions where each planar circle is colored entirely either red or green.

\begin{aside}
    We never carefully defined what a planar algebra is, mostly because it is quite intuitive what `gluing' means (and our intuition is correct). 

    Here is a slightly more precise definition for our particular setting, though we will remain somewhat vague; we will often refer to this whole story as the `spaghetti-and-meatballs' picture. 

    Our $\TLcat_n$ categories are categorifications of the Temperley-Lieb algebras $\TL_n(\delta)$. You can look these up, but the main point is that these are generated by crossingless tangles with $2n$ endpoints, and any time you see a closed component, it evaluates to some $\delta \in R$, where $R$ is the base ring. 

    Consider a planar tangle $\mathcal{T}$ properly embedded in a disk with a finite number of smaller \emph{input disks} removed where each boundary component contains an even number of endpoints. 

    Suppose the boundary of the input disk $D_i$ contains $2n_i$ tangle endpoints. Then we can input (i.e.\ glue) any object of $\TL_{n_i}$ into this hole. 
    If the boundary of the big disk (the dinner plate) contains $2n_0$ endpoints of the tangle, then $\mathcal{T}$ gives us a map of $R$-algebras
    \[
        \TL_{n_1} \otimes \TL_{n_2} \otimes \cdots \otimes \TL_{n_k} \to \TL_{n_0}.
    \]
    
    Similar, for the categories $\TLcat_{n_i}$, the tangle $\mathcal{T}$ gives us a $\Z$-linear multifunctor
    \[
        \TLcat_{n_1} \times \TLcat_{n_2} \times \cdots \times \TLcat_{n_k} \to \TLcat_{n_0}.
    \]
\end{aside}

\begin{exercise}
In this exercise, you will fill in some details in the proof above of Lee's structure theorem.
\bea
\item Prove that 
\begin{center}
    \includegraphics[width=2in]{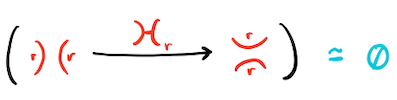}.
\end{center}
\item Explain carefully (i.e.\ at the level of homological algebra, planar algebras, and Karoubi envelopes) why $rg=0$ implies that gluing a red strand to a green strand results in an object that is equivalent to the zero object in $\Kar(\Mat(\TLcat_1'))$. 
\ee
    
\end{exercise}

This concludes the proof of Theorem \ref{thm:Lee-structure-theorem}, as we stated it. 
However, it remains to show that the surviving objects correspond to orientations on the link diagram, as described in Algorithm \ref{algo:Lee-canonical-generators}. We leave this as an exercise for the reader.

\begin{exercise}
Let $P$ be a surviving $\{r,g\}$-colored complete resolution of a link diagram $D$, in the proof of Lee's structure theorem. 
\bea
\item Prove that $P$ satisfies the following property: if two circles in $P$ both abut (a small neighborhood of) a crossing $c$ in $D$, then their colors are different (or more precisely, orthogonal!). 
\item Use this to argue that $P$ corresponds to some $\fraks_o$. \note{Perhaps use a counting argument.}
\item Determine a formula for the homological grading of the canonical generator $[\fraks_o]$. \note{The quantum filtration level is harder to compute, and is captured by a generalization of the $s$ invariant to links.}
\ee
\end{exercise}

\begin{warning}
\label{warn:char-2-Lee-vs-BN}

Nowadays, when discussing other TQFTs, we also use the term `Lee homology' to mean a very specific construction that results in a more degenerate homology theory.

\emph{Universal Khovanov homology} is the theory over the very general Frobenius algebra $\Bbbk[h,t,X]/(X^2 - hX - t)$. 
The corresponding `Lee homology' 
(\note{sorry about the confusing language}) is obtained by inverting the discriminant $\mathcal{D} = \sqrt{h^2+4t}$ (see \cite{Kh-frobext}).
In Lee's original theory, we set $h=0$ and $t=1$, so $\mathcal{D} = \sqrt{4} = 2$. 
So, we didn't need to work over $\Bbbk = \Q$; we just needed $\Bbbk$ to contain $\frac{1}{2}$. Otherwise, we'd be dividing by 0.
That is to say, you shouldn't do Lee homology over $\F_2$ if you're not ok with localizing at 0 (I'm not!); use the Frobenius algebra $\F_2[X]/(X^2-X)$ instead. 
This theory appears in earlier literature as `Bar-Natan homology'; nowadays, `Bar-Natan homology' typically refers to the theory for $\F_2[h,X]/(X^2-hX)$, which is more powerful.

Finally, also notice that for the original Frobenius algebra, the discriminant is 0 (which happens when the two roots of the quadratic are the same), so we can't actually invert the discriminant to get a corresponding `Lee homology'.

For my sanity, I prefer to call the `corresponding Lee homology' to a flavor of Khovanov homology as the `localized homology' because we localize at the discriminant.
\end{warning}

\subsubsection{Defining the $s$ invariant}

We now return back to \cite{Rasmussen-s}. 
Throughout, let $K$ be a knot. Then $\Lee(K) \cong \Q \oplus \Q$. 
Rasmussen's $s$ is defined using the filtration grading $\gr_q$ on the Lee complex. 

First note that our definition for the filtration grading of a chain $x \in \CLee(K)$ in \eqref{eq:quantum-filtration-grading} is equivalent to
\[
    \gr_q(x) = \max \{ j \in \Z \st x \in F^j(\CLee(K))
\]
where $F^\bullet$ is the filtration induced by $\gr_q$ on the homogeneous distinguished generators.

This induces a filtration grading on the homology class $[x] \in \Lee(K)$ by
\[
    \gr_q([x]) = \max \{ \gr_q(y) \st y \text{ is homologous to } x\}.
\]

Rasmussen defines two knot invariants, $s_{\max}$ and $s_{\min}$, using this filtration grading on Lee homology (\cite[Definition 3.1]{Rasmussen-s}):

\begin{equation}
    \label{eq:s-min-defn}
    s_{\min}(K) = {\min} \{\gr_q([x]) \st [x] \in \Lee(K), \ [x] \not= 0\}
\end{equation}
\begin{equation}
    \label{eq:s-max-defn}
    s_{\max}(K) = {\max} \{\gr_q([x]) \st [x] \in \Lee(K), \ [x] \not= 0\}
\end{equation}

\note{Prove to yourself that 
(1) these two values are always odd integers and 
(2) these are knot invariants.}

The $s$ invariant is defined to be the average of these two odd integers:
\begin{definition}
    For a knot in $K \subset \R^3$, 
    \[
        s(K) = \frac{s_{\max}(K) + s_{\min}(K)}{2}.
    \]
\end{definition}

In fact, $s(K)$ is always an even integer because of the following proposition, whose proof is the focus of the remainder of this section.

\begin{proposition}
\label{prop:smax-smin}
$s_{\max} = s_{\min} + 2$.
Therefore $s(K) = s_{\max}-1 = s_{\min}+1$. 
\end{proposition}

For a knot $K$, the quantum gradings of the homogeneous generators are all odd.
\note{(Why? Prove this to yourself! I do have a favorite proof of this.)}
Now recall that $d_{\Lee} = d_{\Kh} + \Phi$ where $\Phi$ raises $\gr_q$ by 4, and so we concluded that $\CLee(K)$ is $\Z/4\Z$-quantum graded. But in fact, we only have generators in degrees $1 \mod 4$ and $3 \mod 4$. 
We get a decomposition of the chain complex
\begin{equation}
\label{eq:lee-q-grading-direct-sum}
    \CLee(K) \cong \CLee_1(K) \oplus \CLee_{-1}(K),
\end{equation}
where $\CLee_i(K)$ is the summand at the $\Z/4\Z$-quantum degree $i \mod 4$. 

\begin{lemma}[cf.\ Lemma 3.5 of \cite{Rasmussen-s}]
\label{lem:ras-s-quantum-summands}
Let $o$ be an orientation of $K$, and fix a diagram $D$ for $K$. Then the two cycles $\fraks_o \pm \fraks_{\bar o}$ are contained in two different $\Z/4\Z$-quantum grading summands of $\CLee(D)$. 
\begin{proof}
The Lee generators $\fraks_o$ and $\fraks_{\bar o}$ both live in the oriented resolution $D_o$ of a diagram $D$ for $K$. 
The quantum grading of any homogenenous generator $g$ at $D_o$ is given by 
\[
    \gr_q(g) = \gr_h(D_o) + p(g) + n_+ - n_-
\]
where 
\begin{itemize}
    \item $\gr_h(D_o)$ is the homological degree at the resolution $D_o$ and
    \item $p(g) = \#(v_+ \text{ in }g) - \#(v_-\text{ in }g)$ (following notation in \cite{BN-Kh}).
\end{itemize}
\note{For the purposes of this proof, we only need to worry about $p(g)$, since the actual quantum gradings for all generators appearing in this proof differ from this local $p$ grading only by an overall shift (of $\gr_h(D_o) +  n_+ - 2n_-$).}

Define an involution $\tau: V \to V$ by $v_- \mapsto v_-$, $v_+ \mapsto -v_+$. Extend this to tensor products $V^{\otimes k}$, and in particular, $V^{\otimes |D_o|}$, \note{which is isomorphic to  $\CLee(D_o)$ up to the global shift we discussed.}

Observe that 
\begin{enumerate}
    \item $\tau$ acts as identity on generators (pure tensors) containing an even number of instances of $v_+$; call the span of these generators $\CLee_{even}(D_o)$.
    \item $\tau$ acts as $-1$ on generators with an odd number of instances of $v_+$; call the span of these generators $\CLee_{odd}(D_o)$.
\end{enumerate}
Since the number of tensor factors $(|D_o|)$ is fixed, these two summands of $\CLee(D_o)$ are at two different $\Z/4\Z$ quantum gradings. 

Since $\bolda = X-1$ and $\boldb = X+1$, we have $\tau(\bolda) = \boldb$ and $\tau(\boldb) = \bolda$; therefore 
\[  
    \tau(\fraks_o) = \fraks_{\bar o}.
\]
Therefore
\begin{itemize}
    \item $\tau$ acts as identity on the sum $\fraks_o +  \fraks_{\bar o}$, so we must have $\fraks_o +  \fraks_{\bar o} \in \CLee_{even}(D_o)$.
    \item $\tau$ acts as $-1$ on the difference $\fraks_o - \fraks_{\bar o}$, so we must have $\fraks_o -  \fraks_{\bar o} \in \CLee_{odd}(D_o)$.
\end{itemize}

Switching $p$ out for $\gr_q$ (by applying the quantum shift), we still have that the cycles $\fraks_o \pm \fraks_{\bar o}$ are supported in different $\Z/4\Z$-quantum grading summands.
\end{proof}
\end{lemma}

\begin{remark}
    In general, of $L$ has $n$ components, then the chains $\CKh(L)$ are all supported in quantum gradings $n \mod 2$. 
    Lemma 3.5 of \cite{Rasmussen-s} gives a more general statement than our discussion above.
    The definition for the $s$-invariant can be extended to oriented links, though the filtration gradings of the canonical classes behave more complexly; see \cite{Beliakova-Wehrli-ras-s-links}.
\end{remark}

Lemma \ref{lem:ras-s-quantum-summands} actually tells us more. 
Because $\Lee(K)$ is 2-dimensional, we now know that we can choose the basis $\{ \fraks_o \pm \fraks_{\bar o}\}$, and that these generate the summands of the direct sum decomposition on homology
\[
    \Lee(K) \cong \Lee_1(K) \oplus \Lee_{-1}(K).
\]
arising from the decomposition of the chain complex in \eqref{eq:lee-q-grading-direct-sum}.
We conclude that 
\begin{equation}
\label{eq:smax-gneq-smin}
    s_{\max}(K) \gneq s_{\min}(K).
\end{equation}

Furthermore, we must have 
\begin{equation}
\label{eq:smin-canonical-generators}
    s_{\min}(K) = \gr_q([\fraks_o]) = \gr_q([\fraks_{\bar o}])
\end{equation}
because both have components in the $\Z/4\Z$-quantum grading summand corresponding to $s_{\min}(K)$. 

To finish the proof of Proposition \ref{prop:smax-smin}, we will need another lemma, which will let us bound the difference between $s_{\max}$ and $s_{\min}$. 

\begin{definition}
    A map $f: C \to C'$  between filtered chain complexes (of the form in Definition \ref{defn:filtered-complex}
    is \emph{filtered of degree $k$} if $f(F_i) \subseteq F'_{i+k}$. 
\end{definition}

\begin{lemma}[\cite{Rasmussen-s}, Lemma 3.8]
\label{lem:lee-connected-sum-ses}
Let $K_1, K_2$ be knots. There is a short exact sequence 
\[
    0 \to \Lee(K_1 \# K_2) \map{p^*} 
    \Lee(K_1) \otimes \Lee(K_2) \map{m^*}
    \Lee(K_1 \# K_2)
    \to 0
\]
where $p^*$ and $m^*$ have filtered degree $-1$.
\begin{proof}
Let $D_1, D_2$ be diagrams for $K_1, K_2$, respectively.
We have the following mapping cone:
\begin{center}
    \includegraphics[width=3in]{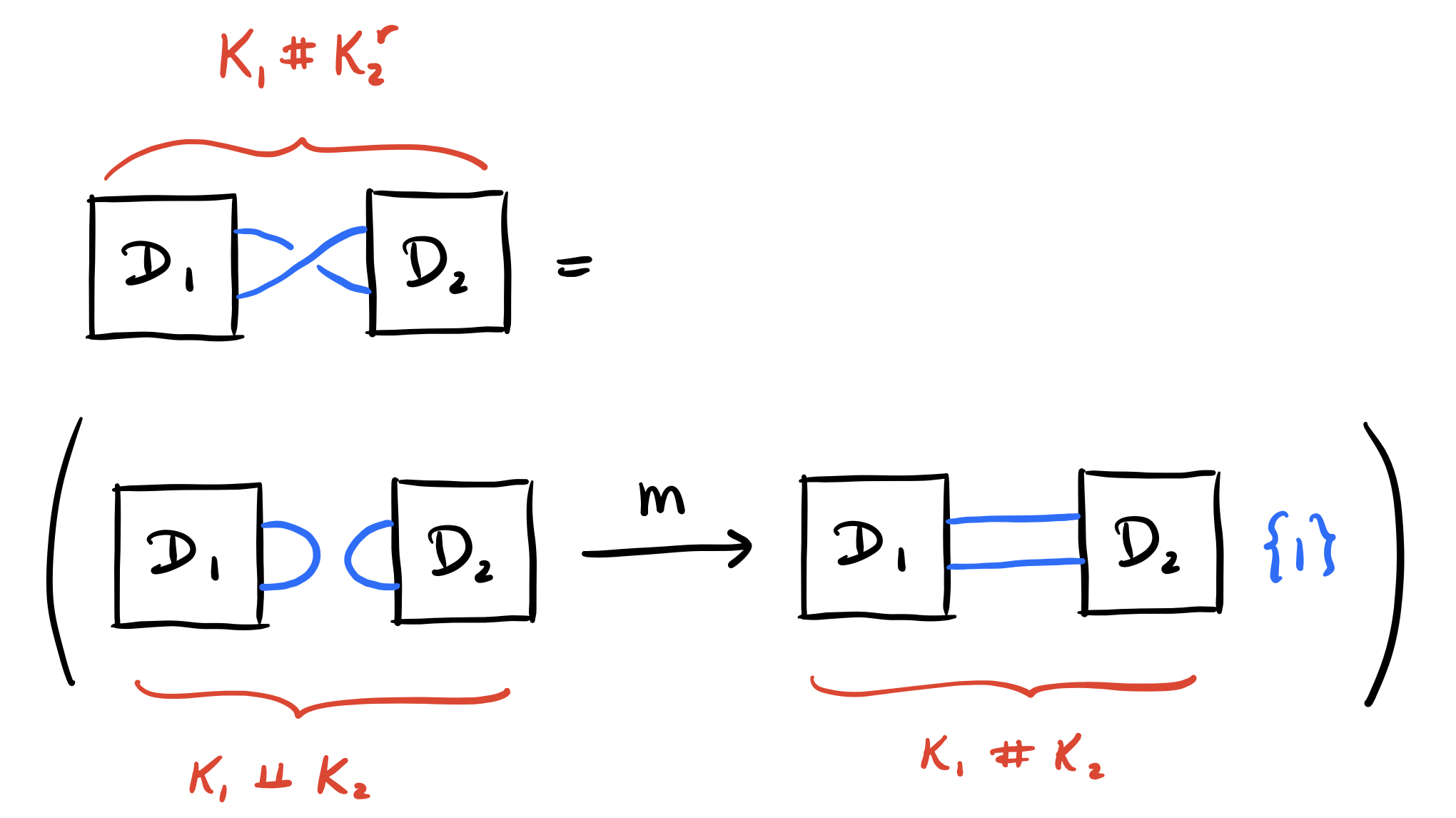}
\end{center}
\note{
Note that the first diagram isn't exactly $D_1 \# D_2^r$, but you can visualize an isotopy that spins the $D_2^r$ around the $x$-axis so that it looks like a disjoint union. So, the Lee complex is chain homotopy equivalent to $\CLee(D_1\# D_2^r)$, which is in turn obviously isomorphic to $\CLee(D_1 \# D_2)$. 
}
This yields a short exact sequence (see Section \ref{sec:mapping-cones}):
\[
    0 \to \CLee(K_1 \# K_2)\{1\} \map{i}
    \CLee(K_1 \# K_2^r) \map{p}
    \CLee(K_1 \sqcup K_2) \to 0
\]
\note{where $\CLee(K_1 \# K_2^r)$ is the chain homotopy representative given by the specific diagram as shown above, and the quantum shift comes from the internal quantum shift in the Khovanov chain complex construction}.
This gives you a long exact sequence on Lee homology. But because of Lee's structure theorem, we know the dimensions of the $\Q$ vector spaces in the sequence. Since $2+2 = 4$, the long exact sequence splits \note{(yay vector spaces!)}, i.e.\ $i^* = 0$. 

We conclude that we have a short exact sequence
\[
    0 \to \Lee(K_1 \# K_2) \map{p^*} 
    \Lee(K_1) \otimes \Lee(K_2) \map{m^*}
    \Lee(K_1 \# K_2)
    \to 0,
\]
but we've kind of discarded the internal quantum degree shifts when we write this down. 
If we're careful and  keep track of everything (i.e.\ keep track of orientations), we can figure out that both $p^*$ and $m^*$ are filtered of degree $-1$, as follows.

Without loss of generality, we can choose diagrams $D_1, D_2$ so that our crossing is positive. Then the actual mapping cone would be
\begin{center}
    \includegraphics[width=4in]{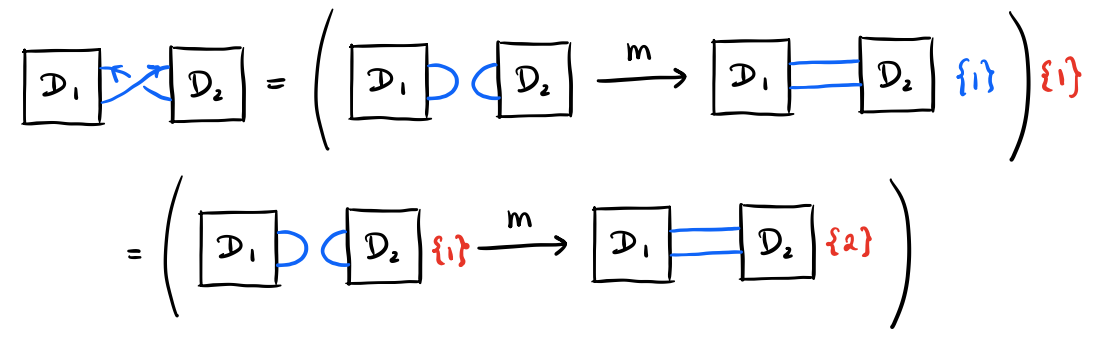}
\end{center}
By construction, $m$ (secretly $m'$ because we are doing Lee homology) 
\begin{center}
    \includegraphics[width=3in]{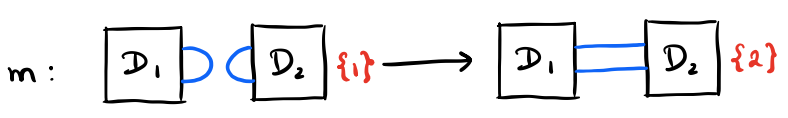}
\end{center}
is filtered of degree 0. 
Also, tautologically, $p$
\begin{center}
    \includegraphics[width=3in]{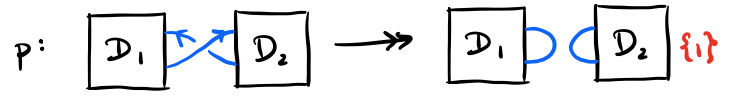}
\end{center}
is filtered of degree 0.
Therefore the \textit{actual} SES maps 
\begin{center}
    \includegraphics[width=3in]{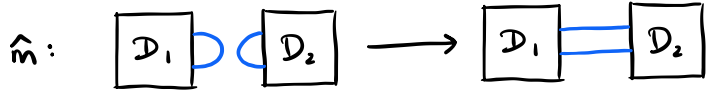}\\
    \includegraphics[width=3in]{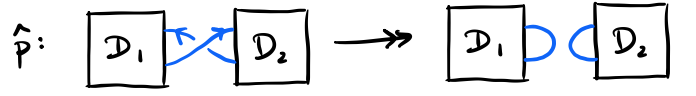}
\end{center}
(that we meant to use to obtain our LES)
are actually filtered of degree $-1$.

\note{Alternatively, Rasmussen suggests that we first note that the maps must be filtered of some degree, because, without shifts, $m^*$ is induced by a saddle (filtered degree $-1$) and $p^*$ is induced by a quotient (a filtered degree $0$ operation). 
So, he instead just uses the case $K_1, K_2 = U$ to figure out the filtered degree of these maps, and concludes that they are indeed both  $-1$.}

\end{proof}
\end{lemma}


\begin{proof}[Proof of Proposition \ref{prop:smax-smin}]

Now let's apply Lemma \ref{lem:lee-connected-sum-ses} to $K_1 = K$ and $K_2 = U$; we get a short exact sequence
\[
    0 \to \Lee(K) \map{p^*} \Lee(K) \otimes \Lee(U) \map{m^*} \Lee(K) \to 0
\]
where each nontrivial map has filtered degree $-1$. 
We know $\Lee(U) = \Q\bolda \oplus \Q\boldb$. 
We also know that one of $\{\fraks_o, \fraks_{\bar o}\}$ has label $\bolda$ on the component where the connected sum operation occurs; let's call this generator $\fraks_a$, and call the other $\fraks_b$. 

We know that $\gr_q([\fraks_a + \varepsilon \fraks_{\bar b}]) = s_{\max}$, where $\varepsilon$ is either $1$ or $-1$. 
In any case, we have
\[
    m^*([\fraks_a + \varepsilon \fraks_{\bar b}] \otimes \bolda) = [\fraks_a].
\]
Because $m^*$ is filtered of degree $-1$, it can only decrease the filtration degree by at most 1, i.e.
\[
    \gr_q([\fraks_a]) \geq 
    \gr_q([\fraks_a + \varepsilon \fraks_{\bar b}] \otimes \bolda) -1,
\]
i.e.
\[
    \gr_q([\fraks_a + \varepsilon \fraks_{\bar b}] \otimes \bolda)
    \leq 
    \gr_q([\fraks_a]) + 1.
\]
By \eqref{eq:smin-canonical-generators}, we have 
\[
    s_{\max}(K) - 1 \leq s_{\min}(K)+1,
\]
and by \eqref{eq:smax-gneq-smin}, we conclude that 
\[
    s_{\max}(K) = s_{\min}(K) + 2.
\]

\end{proof}

\subsubsection{$s$ is a concordance homomorphism}

We begin this section with the properties of $s$ that show it is a concordance homomorphism. 

These are proven in \cite{Rasmussen-s}, but that I leave for you to ponder on your own. If you aren't able to prove the statements on your own (but with the hints provided), feel free to look at Rasmussen's proof (and understand it) and then put it into your own words.

\begin{exercise}[\cite{Rasmussen-s}, Proposition 3.9]
\label{ex:s-invt-mirror}
\note{(Maybe hard?)}
Let $\bar K$ denote the mirror of $K$. Prove that $s(\bar K) = -s(K)$. 

\emph{Hint:}
Let $D$ be a diagram for $K$, and let $\bar D$ be the diagram where all the crossing information is reversed.
Explain briefly why, in the category of chain complexes of filtered vector spaces over $\Q$, we have 
\[
    \CLee(\bar D) \cong \CLee(D)^*
\]
where $(\cdot)^*$ is the dualizing functor that 
\begin{itemize}
    \item sends $v_{\pm} \mapsto v_{\mp}^*$; here $\gr_q(v_{\mp}^*) = \mp 1$; \note{I don't like this notation... The point is that under the degree-preserving isomorphism $V^* \cong V$, we have $v_-^* \mapsto v_-$. You can therefore think of the functor $(\cdot)^*$ as taking $v_\pm \mapsto v_{\mp}$. }
    \item reverses all quantum shifts (in fact, all quantum gradings);
    \item reverses all homological gradings and is contravariant on morphisms.
\end{itemize}
\end{exercise}

\begin{remark}
    In Khovanov's notation, the dualizing functor is denoted by $(\cdot )^!$. In other contexts, you may also see it written as $(\cdot )^\vee$, though the diagram mirroring rule might be different (e.g.\ reflect across $x$-axis). Nevertheless, however you mirror a diagram, there is a clear identification of chains for $D$ and for $\bar D$.
\end{remark}

\begin{exercise}[\cite{Rasmussen-s}, Proposition 3.11]
\note{(Maybe hard?)}
For knots $K_1, K_2$, prove that $s(K_1 \# K_2) = s(K_1) + s(K_2)$.

\emph{Hint:}
Use Lemma \ref{lem:lee-connected-sum-ses} for bounds on the filtration degree of morphisms, and then flip your morphisms upside down and use the result of Exercise \ref{ex:s-invt-mirror} too.

\end{exercise}

\subsubsection{Behavior of canonical generators under cobordisms}

In the proof of Proposition \ref{prop:smax-smin}, we saw how the filtration degree of maps 
helped us bound the behavior of $s$. 
Now, we will study the behavior of $s$ under cobordisms, by using the fact that every cobordism can be decomposed into elementary cobordisms, and we know the filtration degree of the Lee map associated to each elementary cobordism $F$: it's $\chi(F)$!

But this is not enough to know what happens to $s$: after all, the induced map could be 0. 

\begin{remark}
    Our discussion here will differ from Rasmussen's original paper, mainly because it was written before Bar-Natan's categories were introduced. (Rasmussen posted \cite{Rasmussen-s} in February 2004, whereas Bar-Natan posted \cite{BN-tangles} in October 2004.)
    Rasmussen carefully proves that the Reidemeister maps in Lee homology are filtered of degree 0. 
    For us, this fact is almost immediate because Bar-Natan's proof of Reidemeister invariance (in the most general $S,T,4Tu$ cobordism categories) is grading-preserving. 
    Instead, we can imagine passing first through a TQFT for the Frobenius algebra $\Z[T,X]/(X^2 = T)$, where $\deg_q(T) = -4$, and then setting $T = 1$ (and then localize at $\sqrt{\mathcal{D}} = 2$, or just $\otimes \Q$). This preserves the \emph{filtration} grading, because now all the arrows that used to have a $T$ coefficient (and were grading-preserving) are now increase filtration degree by 4.
\end{remark}

We know the Reidemeister moves are chain homotopy equivalences in Bar-Natan's category, so these maps are nonzero. We are more worried about compositions of Morse moves.
Luckily, we have the following lemma:

\begin{lemma}[see proof of Proposition 4.1 in \cite{Rasmussen-s}]
\label{lem:compatible-orientations}
    Let $S: L \to L'$ be a cobordism with no closed components. 
    Let $o$ be an orientation on $L$.
    Then 
        \[ 
            \phi_S([\fraks_o]) 
            = \sum_{I} a_I [\fraks_{o'_I}]
        \]
        where $\{o'_I\}$ is the set of all orientations on $L'$ that are compatible with $(L,o)$ on the bottom boundary of $S$, and where all coefficients $a_I$ are nonzero.
\end{lemma}

Let us state this another way, using Rasmussen's vocabulary.
There are two types of connected components $F$ of the cobordism $S$:
\begin{itemize}
    \item $F$ is of the \emph{first type} if it has a boundary component in $L$. In this case, the chosen orientation on the bottom boundary of $F$ determines the orientation on the rest of the boundary components on $L'$.
    \item $F$ is of the \emph{second type} if it only has boundary on $L'$. Then either orientation of $F$ is compatible with $(L,o)$. This yields two different possible sets of orientations on the components of $L'$ at the (top) boundary if $F$. 
\end{itemize}
An orientation on $F$ is \emph{permissible} (with respect to $(L,o)$) if the bottom boundary of $F$ agrees with $(L,o)$. The lemma therefore states that the set of orientations $\{o'_I\}$ is the set of top boundary orientations of the set of permissible orientations on $S$.

\begin{remark}
    Requiring that $S$ contain no closed components obstructs the situation where you have, say, a closed undotted sphere that makes the entire cobordism evaluate to 0.
\end{remark}

\begin{remark}
\alert{Clarify this on Friday}
\mz{add remark on cobordism boundary orientation (in and out). }
\end{remark}

Before we prove Lemma \ref{lem:compatible-orientations}, let's state the important corollaries:

\begin{corollary}[Proposition 4.1 of \cite{Rasmussen-s}]
If $S$ is an oriented cobordism from $(L,o)$ to $(L',o')$ that is \emph{weakly connected}, i.e.\ every component of $S$ has a boundary component in $L$ (i.e.\ $S$ has no components of the second type), then 
$\phi_S([\fraks_o])$ is a \emph{nonzero} multiple of $\fraks_{o'}$. 
\end{corollary}

\begin{corollary}[Corollary 4.2 of \cite{Rasmussen-s}]
    If $S$ is a connected cobordism between knots, then $\phi_S$ is an isomorphism.
\end{corollary}

This corollary is particularly important because we can then study the filtration degree of $\phi_S$ to understand the behavior of $s$.

\begin{proof}[Proof of Lemma \ref{lem:compatible-orientations}]
To set up, first isotope $S$ so that you may select regular values $0 < t_1 < t_2 < \ldots < t_k = 1$ such that, between any $t_i$ and $t_{i+1}$, there is at most one critical value of $S$, and there are no critical values between $0$ and $t_1$.
\note{This hopefully feels like a familiar procedure by now!}

Let $L_i = S \cap \R^3 \times \{t_i\}$, and let $S_i$ be the cobordism $S \cap [0,t_i]$, a cobordism $L \to L_i$. 

The proof basically proceeds by induction on the heights of elementary cobordisms.
The cobordism $S_1$ is a planar isotopy; this is the base case, and the proposition holds because $\phi_{S_1} = \id$. 

Assume the proposition holds for $S_i$. Then $S_e := S_{i+1} - S_i$ consists of a single elementary cobordism (union identity everywhere else, of course). 
The induction step holds if $S_e$ is a Reidemeister move, so we only need to consider the cases of the four possible Morse moves; see Notation \ref{nota:lee-better-basis}. 
\note{We will abuse notation a bit below and let $\phi_{S_e}$ also denote the chain map in $\Kom(\ggVect_\Q)$, rather than just the map on homology.} 
\mz{These are individually not hard to prove; you should actually think through these yourself.}
\begin{enumerate}
    \item[(0)] If $S_e$ is a 0-handle attachment (birth), then $\phi_{S_e}(\fraks(o_I)) = \fraks(o_I) \otimes \frac{1}{2} \bolda \otimes \boldb$. (This is a disjoint union, so the differentials and homology respect the tensor product decomposition.) \mz{Done!}
        \begin{center}
            add pic
        \end{center}
    \item[(1)] If $S_e$ is a 1-handle attachment (i.e.\ a merge or a split), once again you can check by inspection that the induction step holds. \mz{Verify it! Here you actually need to think about orientations in the plane.}
        \begin{center}
            add pic
        \end{center}
    \item[(2)] If $S_e$ is a 2-handle attachment (death), then there's only one possible compatible orientation for $L_{i+1}$; since $\varepsilon'(\bolda) = \varepsilon'(\boldb) = 1$, you basically just erase the component that died.
    But we need to check that only one permissible orientation on $S_i$ (and thus on $L_i$) extends to this unique orientation on $L_{i+1}$. \mz{Here, argue why if two different permissible orientations on $S_i$ extended to the orientation on $L_{i+1}$, then $S$ overall would have a closed component.}
\end{enumerate}
\end{proof}

\begin{exercise}
    Fill in the details in the proof of Lemma \ref{lem:compatible-orientations} for the four Morse moves, by supplying pictures working out the possible cases and explaining why the induction step holds in each case.
\end{exercise}

\begin{remark}
\note{This is a very subtle comment. Feel free to ignore.}
Notice that while we are using the Lee maps associated to cobordisms, nowhere did we actually use the fact that this assignment of cobordisms to filtered chain maps actually defines a functor. In other words, Rasmussen's invariant can be defined in settings where you have something weaker than an actual functor; we only needed the fact that isotopic cobordisms have the same Euler characteristic, and that Euler characteristics add when you glue cobordisms along collections of circles.
\end{remark}

\subsubsection{Profit!}

Now that we understand how Lee's canonical generators behave under cobordisms, we reap the benefits and prove a bunch of facts about the behavior of $s$ under cobordisms embedded in 4D. Clever mathmaticians (\mz{like you!}) can then use these facts to do interesting topology.

\begin{theorem}[\cite{Rasmussen-s}, Theorem 1]
\label{thm:s-invt-genus-bound}
\[ 
    |s(K)| \leq 2g_4(K). 
\]
\begin{proof}
Suppose $K \subset S^3$ bounds a genus-$g$ slice surface $F \subset B^4$.
Let $S$ be the cobordism $K \to U$ obtained by deleting a 0-handle from $F$ and flipping it upside down. Then $\chi(S) = -2g$.
Let $[x]$ be a homology class where $\gr_q([x]) = s_{max}(K)$, and further pick $x \in [x]$ so that $\gr_q(x) =\gr_q([x]) =  s_{\max}(K)$. \note{Note that none of these classes or cycles are 0.}

Since $s_{\max}(U) = 1$, we have
\[
    \gr_q(\phi_S(x)) \leq 1.
\]
Since $\phi_S$ is filtered of degree $-2g$, we have 
\[
    \gr_q(\phi_S(x)) \geq \gr_q(x) - 2g 
\]
and so
\[
    s(K)+1 = s_{\max}(K) = \gr_q(x) \leq \gr_q(\phi_S(x)) + 2g \leq 1 + 2g
\]
so we conclude that 
\[
    s(K) \leq 2g.
\]

To get that $s(K) \geq -2g$, we use the same argument but for $\bar K$, which is cobordant to the unknot via the surface $\bar S$, and we similarly conclude that $s(\bar K) \geq 2g$.
But $s(\bar K) = -s(K)$, so we multiply the whole inequality by $-1$ to get $s(K) \leq -2g$. The theorem follows.
\end{proof}
\end{theorem}

\begin{exercise}
    Suppose knots $K$ and $K'$ differ by just one crossing change. 
    \bea
    \item Prove that $|s(K) - s(K')| \leq 1$.
    \item From this, obtain a bound on the \textit{unknotting number} $u(K)$ of knot $K$, where $u(K)$ is the minimal number of self-intersections needed in a homotopy of $K \into S^3$ that changes $K$ into an unknot.
    \ee
\end{exercise}

\begin{theorem}
    If $D$ is a positive diagram for a positive knot $K$, then 
    \[
        s(K)  = s_{\min}(K) + 1 = \gr_q(\fraks_o) +1.
    \]
\begin{proof}
    The oriented resolution $D_o$ is the unique resolution at the left-most vertex of the cube of resolutions. There are no differentials into this homological grading, so $\fraks_{o}$ and $\fraks_{\bar o}$ are alone in their respective homology classes.
\end{proof}
\end{theorem}

\begin{exercise}[\cite{Rasmussen-s}, Corollary 1](The Milnor Conjecture)
Let $\gcd(p,q) = 1$.
For the braid-closure diagram of the torus knot $T_{p,q} = \hat \beta$ for 
$\beta = (\sigma_1\sigma_2\cdots\sigma_{p-1}^q$, compute $\gr_q(\fraks_{o})$. 
Deduce the Milnor Conjecture:
\[
    g_4(T_{p,q}) = g_3(T_{p,q}) = \frac{(p-1)(q-1)}{2}.
\]
\end{exercise}

\begin{remark}
    Rasmussen also shows that the $s$ invariant is not so useful for alternating knots, as it gives the same bound as the classical knot signature. Manolescu--Ozsv\'ath later proved that, in fact, Khovanov homology is not very interesting for a larger class of recursively defined knots called \emph{quasialternating knots} \cite{MO-quasialternating}. \mz{We probably won't cover this; good idea for a final project.}
\end{remark}

\subsection{Non-isotopic (smooth) surfaces in the 4-ball; exotic disks}

Here we give some more examples where Khovanov homology was used to study surfaces embedded in 4D.

Recall that an $(n+1)$-dimensional TQFT will give an invariant of $(n+1)$-dimensional \emph{closed} manifolds: a cobordism $F: \emptyset \to \emptyset$ yields map $Z(F): R \to R$, which is determined by $Z(F)(1) \in R$. 
Rasmussen \cite{Rasmussen-KJ} and Tanaka \cite{Tanaka-KJ} showed that this invariant for (vanilla) Khovanov homology is always just determined by the genus of the surface $F$. 

However, if we instead consider a relative version of such an invariant, we do get a rich invariant; this is possible because the target of the map $\Kh(K)$, is big.

\begin{definition}
Let $F$ be a slice surface for a knot $K$. 
The \emph{Khovanov-Jacobsson class} of $F$ is 
\[
    \Kh(F)(1) \in \Kh(K).
\]
\end{definition}

\begin{example}
    Sundberg--Swann \cite{Sundberg-Swann} showed that the slice knot $K = 9_{46}$ has a pair of non-isotopic slice disks $D, D'$ by showing that their \emph{Khovanov-Jacobsson classes} are unequal:
    \[
        \Kh(D)(1) \not= \Kh(D')(1).
    \]
    Here are the disks: 
    \begin{center}
        \includegraphics[height=1.5in]{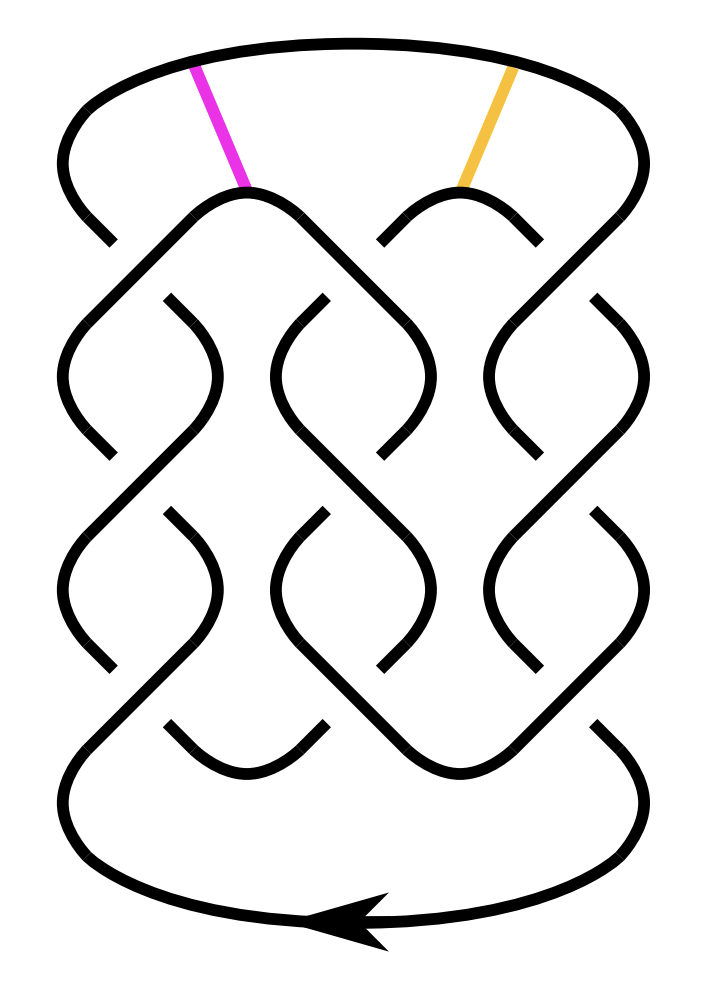}
    \end{center}
    \note{(Image taken from their paper.)}
    For the pink disk, surgery the knot at the pink arc; you will find the result unravels to a two-component unlink. Fill these in and run your movie backwards. 
    Do the same for the orange disk.
\end{example}

\begin{remark}
    Sundberg--Swann's paper is much more general; they show that for any pretzel knot $K = P(p, -p, p)$, you can use the same trick to get two different slice disks. Furthermore, they show that you get the same result if you take arbitrary connected sums of such $K$. 
\end{remark}

When a topologist sees this kind of result, they will naturally wonder, ``Is this because the disks were \emph{topologically} non-isotopic in the first place, or are these disks actually an \emph{exotic pair}?''
And a certain type of topologist will get very excited if the latter is true. 

Let's casually explain the difference. (All embeddings are proper.)

Let $F$ be a surface homeomorphically embedded in $B^4$, without any further regularity conditions. The embedding is \emph{locally flat} if, in any neighborhood of any point on $F \subset B^4$, we can choose coordinates so that $F$ looks flat, in a smooth sense.

For example, using the Jordan curve theorem, you can convince yourself that any embedding $S^1 \into \R^2$ is locally flat. However, as we saw on the first day of class, this is no longer true for $S^1 \into \R^3$:
\begin{center}
    \includegraphics[width=2.5in]{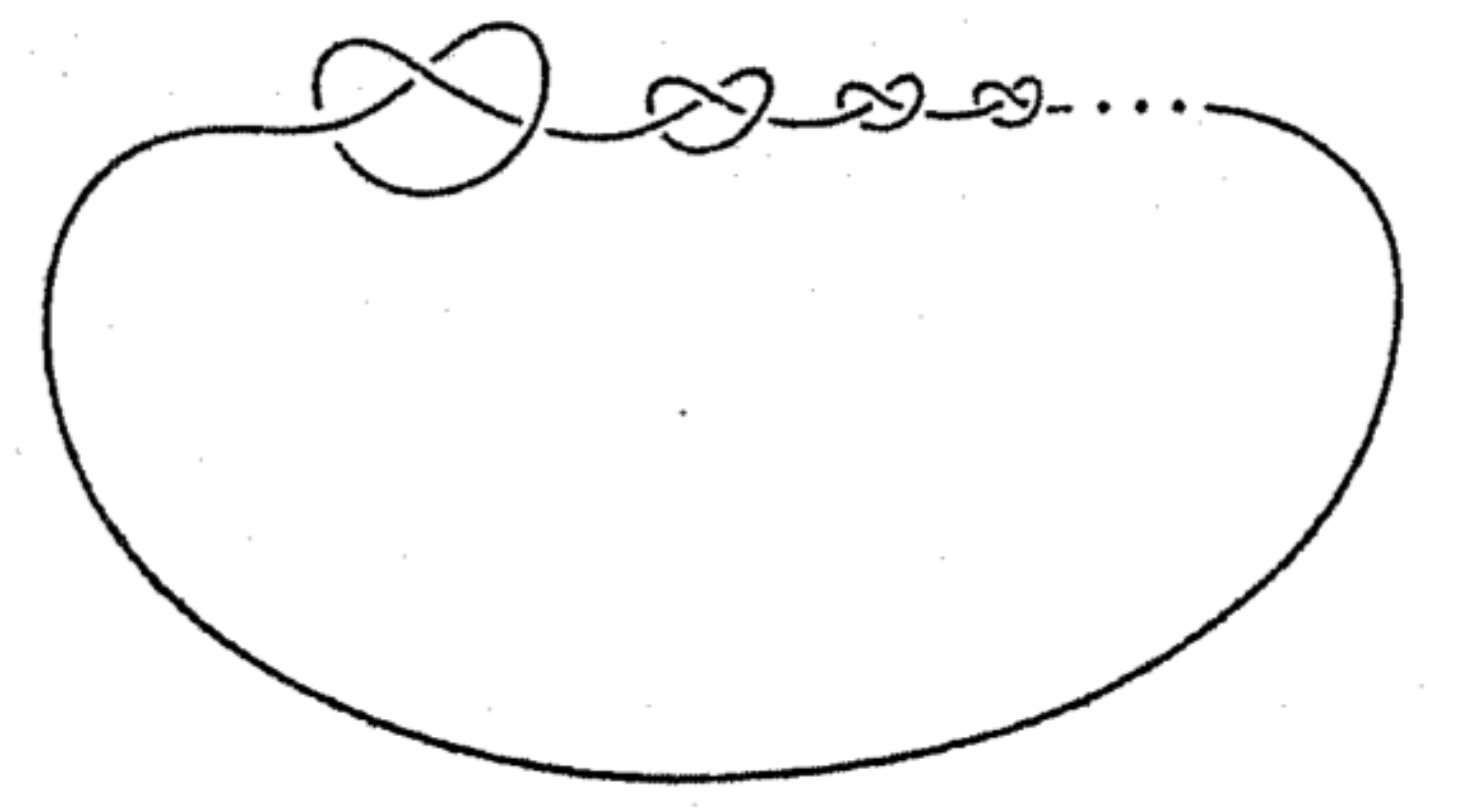}
\end{center}
or $S^2 \into \R^3$:
\begin{center}
    \includegraphics[width=2in]{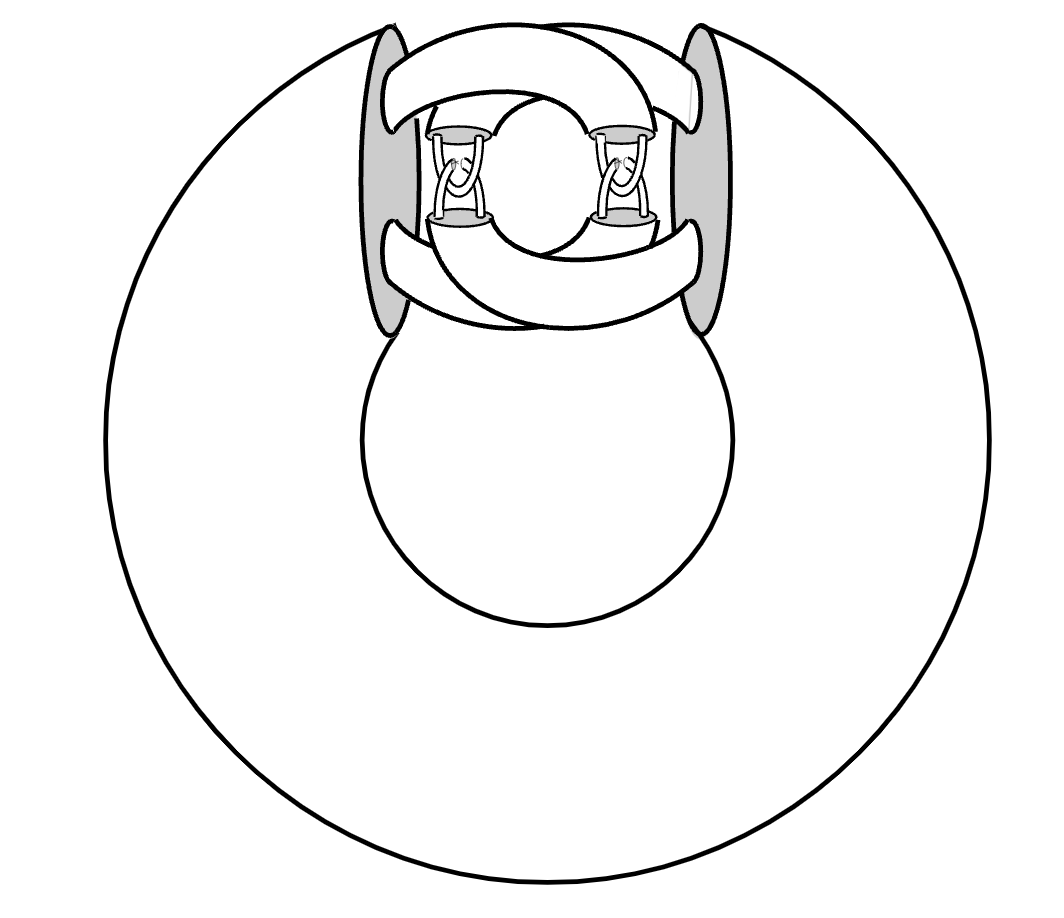}
\end{center}
\note{(Image taken from wolfram.com's article.)}

\begin{definition}
Let $F$ and $F'$ be two \textit{smoothly} embedded surfaces in $B^4$ that share the same boundary in $S^3 = \partial B^4$. 
\begin{itemize}
    \item We say $F$ and $F'$ are \emph{topologically (locally flatly)  isotopic} if there is an ambient isotopy taking $(B^4, F)$ to $(B^4, F')$ through locally flat(ly embedded) surfaces, rel boundary. 
    \item If $F$ and $F'$ are topologically isotopic but \emph{not smoothly isotopic}, then we say that they are an \emph{exotic pair} (or just \emph{exotic}, for short). 
\end{itemize}
\end{definition}

\begin{remark}
    Sometimes I cook food from my ancestral culture and it's referred to as `exotic'. I like to think that the people who say this really mean to say that my dish and their dish are an `exotic pair'. After all, being `exotic' is a symmetric relation.
\end{remark}

One particularly famous question about exotic phenomena in dimension 4 is THE: 
\begin{conjecture}[Smooth Poincar\'e Conjecture in dimension 4 (SPC4)] 
    (Open) If $X$ is homeomorphic to $S^4$, then $X$ is diffeomorphic to $S^4$. 
\end{conjecture}

As a smooth invariant, Khovanov homology can help with questions of this type. We will see some results about disks here, and then later about 4-manifolds when we talk about skein lasagna modules.

Here is an example, again only a specific one taken from a larger class described in the paper:
\begin{example}
    Hayden--Sundberg showed that the following knot $J$ bounds slice disks $D$ and $D'$ that are topologically isotopic (by work of Akbulut) but not smoothly isotopic, by using Khovanov homology. 
    \begin{center}
        \includegraphics[width=3.5in]{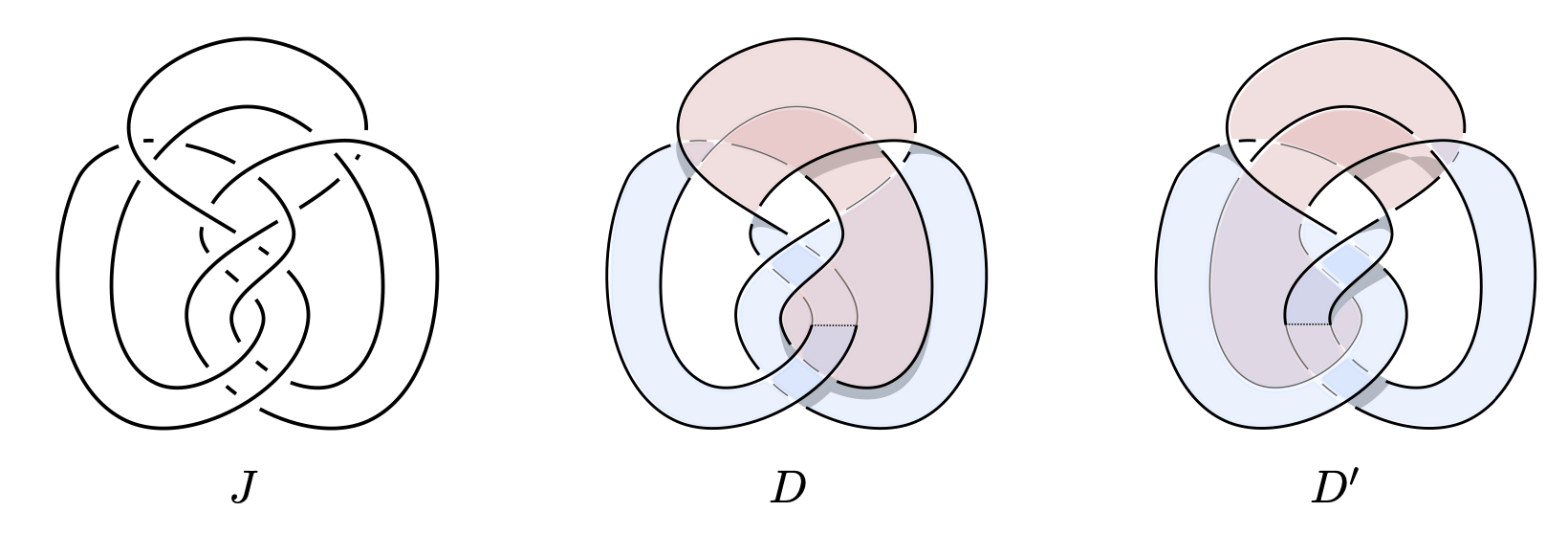}
    \end{center}    
    \note{(Image taken from their paper.)}
    They view the two disks upside-down, and study the maps $\Kh(D)$ and $\Kh(D')$ by identifying what happens to a class called \emph{Plamenevskaya's class} $\phi \in \Kh(J)$ under these maps. For one of these disks, $\Kh(D)(\phi) = 0$ but for the other, $\Kh(D')(\phi)\neq 0$.
\end{example}

\note{Plamenevskaya's class will show up again very very soon!}

That excited topologist from before will probably then also ask, ``Do these surfaces remain exotic after \emph{stabilization}? How many times do I need to \emph{stabilize} before the surfaces become smoothly isotopic?''

In this case, \emph{stabilization} refers to connect summing with a standard (small, boundary-parallel) torus embedded in $B^4$. 
In 4D topology, one expects that exoticness between two objects is eventually lost after stabilizations, because, roughly speaking, there is more room for maneuvering. \note{We'll talk a little more about this later on, but not about the details.}

It turns out that if we work with a TQFT over a polynomial ring, such as $\F_2[H]$, then Khovanov homology `sees' stabilization (`adding a handle'), e.g.\ as multiplication by $2X-H$ in the specific case of $\CA = \F_2[H,X]/(X^2 = HX)$ \note{(this theory is usually called \emph{Bar-Natan homology})}. 

\begin{example}
Hayden \cite{Hayden-atomic} gave an example (again, one of many) where (reduced) Bar-Natan homology can show us two disks remain smoothly non-isotopic even after one stabilization. 
He first finds two disks slice disks for a knot $K$ that are not topologically isotopic, and therefore also not smoothly isotopic. 
Then, he computes that 
\[
    H \cdot \left ( \BNred(D)(1) - \BNred(D')(1) \right ) \not= 0
\]
and therefore
\[
    H\cdot \BNred(D)(1) \not= H \cdot \BNred(D')(1)
\]
as classes in $\BN(K)$. 
\note{Where did the $2X$ part of the handle $2X-H$ go? We're over characteristic 2!}
\end{example}

\begin{remark}
    We did not talk about reduced homology, but will when we start talking about relationships between Khovanov homology and Floer homologies. For now, just know that it's a slimmer version of full Khovanov homology, so if it can tell you something, so can the full theory.
\end{remark}


\section{Legendrian, transverse, and annular links}

\subsection{Knots in the standard contact $\R^3$ and $S^3$}

In this section, we give just enough exposition on the standard contact structure on $S^3$ as well as Legendrian and transverse knots in this setting to proceed with our applications of Khovanov homology. See \href{https://etnyre.math.gatech.edu/preprints/papers/legsur.pdf}{John Etnyre's notes} \cite{Etnyre-leg-notes} for a fairly standard introductory reference on this material, and much more context. 

\subsubsection{Legendrian knots in $(\R^3, \xistd)$}

Roughly speaking, a \emph{contact structure} on a 3-manifold is a 2-plane field that is \emph{maximally non-integrable}, i.e. not tangent to any surface on any neighborhood of any point on the surface. 

\begin{definition}
    The \emph{standard contact structure on $\R^3$} is 
    \[
        \xistd = \ker (dz - ydx).
    \]
    Here $\alpha = dz-ydx$ is the \emph{standard contact 1-form} on $\R^3$.
\end{definition}

The \emph{contact planes} are invariant under $x$ and $z$ translation (i.e.\ parallel transport) and look like this at (the tangent spaces of the points on) the plane $z=0$:
\begin{center}
    \includegraphics[width=2in]{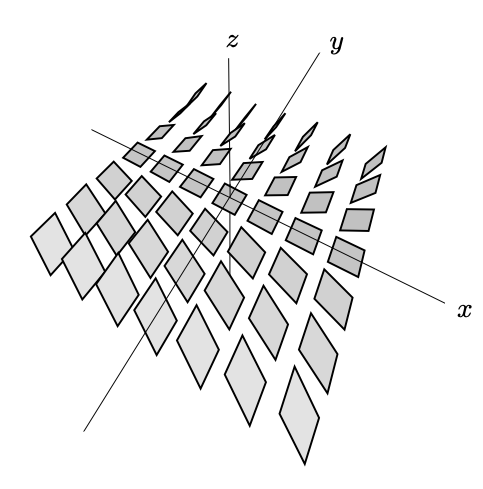}
\end{center}
(Image from \cite{Etnyre-leg-notes}, where he says the images are courtesy of S. Sch\"onenberger.)

Along each line parallel to the $y$-axis, you can think of these planes as the tangent spaces to a belt that is overall a half-twisted band, but whose ends are at infinity. 
\note{Let's call this an `infinite half-twisted band'.}
In particular, the planes are \textit{nearly} vertical (i.e.\ $yz$-planes) when $|y|$ is large. 

While embedded 2-manifolds can't be tangent to $\xistd$ at more than isolated points, embedded 1-manifolds can. For example, notice that the $x$-axis is everywhere tangent to the contact planes.

\begin{definition}
    A \emph{Legendrian knot (or link)} $\Lambda$ in $(\R^3, \xistd)$ is a smoothly embedded knot (or link) that is everywhere tangent to $\xistd$. 

    \note{In other words, $T_*L \subset \xistd|_L$, i.e.\ the tangent bundle of $L$ is contained in the contact plane bundle over $L \subset \R^3$.}

    Again, we will consider Legendrian knots only up to \emph{Legendrian isotopy}, i.e.\ a smooth isotopy only through Legendrian knots.
\end{definition}

The $x$- and $z$-translation invariance of $\xistd$ allows us to draw Legendrian knots very easily, by projecting away the $y$-coordinate. 
The only thing we need to remember is that, if you see a crossing, the more negatively sloped strand is closer to you. 

A \emph{front projection} of a Legendrian knot (or link) $\Lambda \subset (\R^3, \xistd)$ is the (image of the) projection of $\Lambda$ onto the $xz$-plane.
\note{The viewer is at either $\pm \infty$ on the $y$-axis. It doesn't matter which one, as long as we're all at the same place.}

If you stare at the contact form again, you will notice that the front projection will satisfy these properties:
\begin{itemize}
    \item There cannot be any vertical tangencies. 
    \item Consequently, there will be left and right cusps:
        \begin{center}
            \includegraphics[height=1.2in]{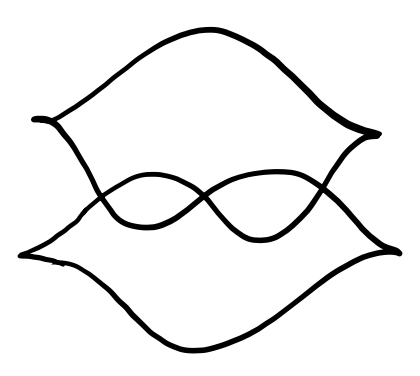}
        \end{center}
    \item At any crossing, the more negatively sloped strand is in front (i.e.\ closer to the viewer):
        \begin{center}
            \includegraphics[width=3in]{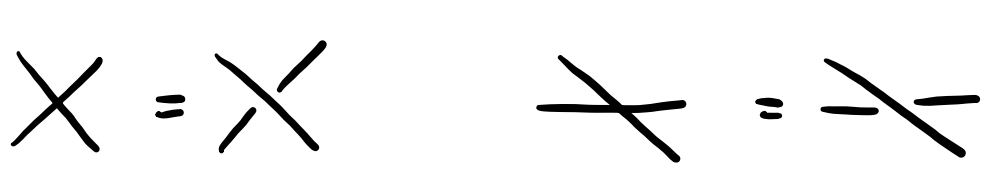}
        \end{center}
\end{itemize}

There are of course \emph{Legendrian Reidemeister moves} that relate front diagrams of the same Legendrian knot. See Figure 8 of \cite{Etnyre-leg-notes}.

\begin{remark}
Each Legendrian knot of course also represents a smooth knot. But because of the stricter geometric conditions on the equivalence between Legendrian knots, the set of Legendrian knots is a refinement of the set of smooth knots. 
For example, the Chekanov pair is a pair of Legendrian \emph{representatives} of the knot $m(5_2)$ that are famously hard to tell apart, but were proven to be in fact non-Legendrian-isotopic \cite{Chekanov-dga, Eliashberg-dga}:

    \begin{center}
        \includegraphics[width=3in]{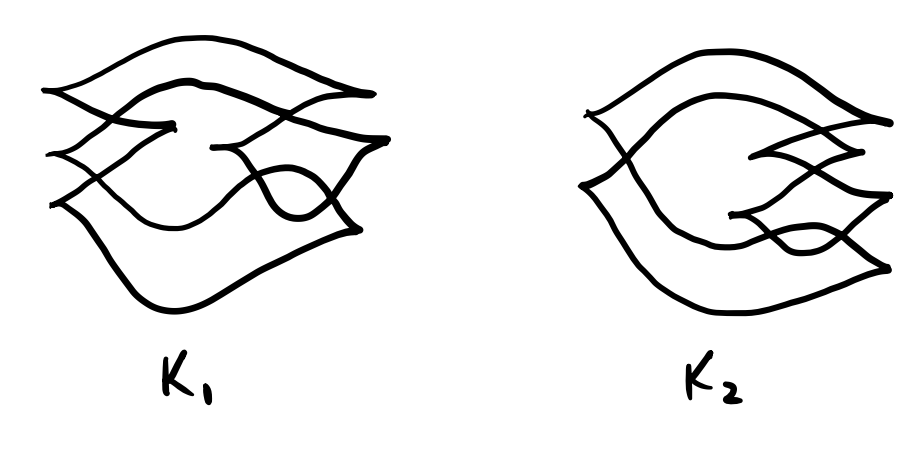}
    \end{center}

\note{($K_1$ is a sleeping cat with a long torso but short neck. $K_2$ is a sleeping cat with a short torso but a long neck.)}
\end{remark}

\subsubsection{Transverse knots in $(\R^3, \xirot)$}

We can also look at a less restrictive class of knots that are everywhere \emph{transverse} to the contact planes. 

\begin{definition}
    A \emph{transverse knot (or link)} in $(\R^3, \xistd)$ is a smoothly embedded knot (or link) $\transknot$ where at any point $p \in \transknot$, the tangent space is transverse to the contact planes.
    In other words, at all $p \in \transknot$,
    \[
        T_p\transknot \pitchfork (\xistd)_p.
    \]
\end{definition}

Since our contact structure is orientable (it's defined globally by a single 1-form, and therefore we can talk about the positive and negative sides of the contact planes), we will also require that our transverse knots intersect the contact planes \textit{positively}. This allows us to fix an orientation on transverse knots (and links) without having to draw any orientation arrows.
\note(Technically, these are \emph{positive transverse knots}, and we could also study the entirely analogous set of \emph{negative transverse knots}.)

\begin{remark}
    You can also describe front diagrams of transverse knots. The geometry will give you the following constraints:
    \begin{itemize}
        \item no downward pointing vertical tangencies (but upward is ok):
        \item no downward pointing positive crossings:
    \end{itemize}
There are also transverse Reidemeister moves for front diagrams; see Figure 14 of \cite{Etnyre-leg-notes}.
\end{remark}

\begin{remark}
    Transversality is a generic condition, so transverse knots are a less fine refinement of the set of smooth knots than Legendrian knots. In fact, by perturbing a Legendrian knot, you will land in a well-defined transverse knot class.
    
    The front projection gives us a diagrammatic way to make sense of the statement  
    \[
        \{\text{transverse knots}\} \leftrightarrow 
        \{\text{Legendrian knots}\}/\sim
    \]
    (which we will not discuss but that you can easily look up).
\end{remark}

Two contact structures $\xi_1, \xi_2$ on the same 3-manifold $M$ are \emph{contactomorphic} if there is a diffeomorphism that sends $(M, \xi_1) \to (M,\xi_2)$. 
In order to study a different type of knot living in contact 3-space, we sometimes prefer to instead work with a more rotationally symmetric version of the standard contact structure, by using cylindrical coordinates $(r,\theta, z)$ for $\R^3$ instead of the Cartesian $(x,y,z)$.

\begin{definition}
The \emph{rotationally symmetric contact structure on $\R^3$} is 
\[
    \xirot = \ker (dz + r^2 d\theta)\quad (= \ker(dz + xdy - ydx)). 
\]
\end{definition}
The contact plane for $\xirot$ are now $z$- and 
$\theta$-translation invariant. 
The contact planes look like this:
\begin{center}
    \includegraphics[width=2in]{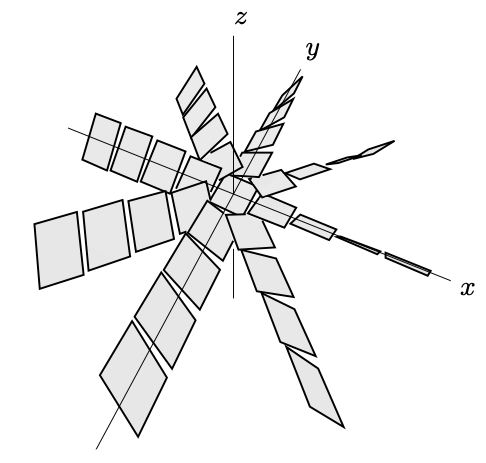}
\end{center}
(Image from \cite{Etnyre-leg-notes} again.)
Take the infinite half-twisted band from before, and now lay them out radially so that their flat planes are identified at the $z$-axis. Copy this to each $z$-coordinate. Take the tangent planes.

In particular, the contact planes are \textit{nearly} vertical when $|r|$ is large.
This is particularly exciting because this means that if you take a braid, wrap it around the $z$ axis and close it up, and then push it radially  far enough away from the origin, the result is a transverse link! 

The more surprising thing is that this characterizes \emph{all} transverse links.
In order to appreciate this fact, we need to take a detour to talk about braid closures.

\begin{remark}
    There are moves that relate braid closure representatives of $\transknot$. This will be discussed in the next section (the Transverse Markov Theorem).
\end{remark}

\subsubsection{Braid closures and annular links}

Let $\beta \in B_n$ be a braid on $n$ strands. 
The \emph{braid closure} of $\beta$, denoted by $\hat \beta$, is the link obtained by wrapping $\beta$ around the $z$-axis and gluing the strands:

\begin{center}
    \includegraphics[height=1.2in]{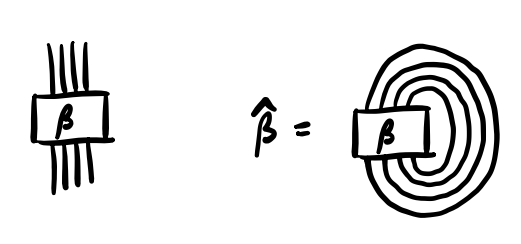}
\end{center}

\note{This is also referred to as taking the \emph{trace}. Prove to yourself that indeed, this operation satisfies the trace condition $\trace(\beta_1 \beta_2) = \trace(\beta_2 \beta_1)$.)
}

A braid closure naturally lives in a \emph{thickened annulus} $A \times I$, 
which is diffeomorphic to the complement of a neighborhood of the $z$-axis $\cup \{\infty\}$ in $S^3$. 

\begin{remark}
    The thickened annulus is also diffeomorphic to the solid torus. Also, if you're not too careful about details about open and closed sets, you can also think of this ambient space as 
    \[ 
        \R^3-\{z\text{-axis}\} = (\R^2-\{0\}) \times \R.
    \]
    In any case, the braid closure is properly embedded and compact, so any of these descriptions is valid. They all appear in the literature.
    The important thing is that you can draw your diagrams on an annulus by projecting away the `thickening' direction.
\end{remark}

\begin{theorem}
    Here are some classical theorems about knots in $S^3$.

\emph{Alexander's theorem} (proved by Alexander) showed that every knot or link in $S^3$ can be braided with respect to the $z$-axis. The \emph{Yamada-Vogel} algorithm (discovered by Vogel building on work of Yamada) gives a hands-on way to turn a link diagram into a braided diagram (through valid isotopies, of course).

The \emph{Markov theorem} (originally proven by Markov, but now there are many, many proofs) states that, for any two braids $\beta_1 \in B_{n_1}$ and $\beta_2 \in B_{n_2}$ whose closures represent a link $L$, $\beta_1$ and $\beta_2$ are related by a finite sequence of the following moves:
\begin{itemize}
    \item (Markov I) conjugation, i.e.\ replacing $\beta \in B_n$ with $\omega \beta \omega\inv$ where $\omega$ is any braid in $B_n$:

    \begin{center}
        \includegraphics[width=1.5in]{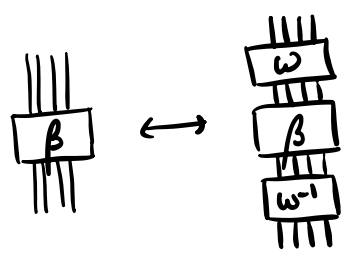}
    \end{center}
    
    \item (Markov II) positive or negative stabilization / destabilization, i.e.\ Reidemeister 1 of either type across the braid axis):
    
    \begin{center}
        \includegraphics[width=3in]{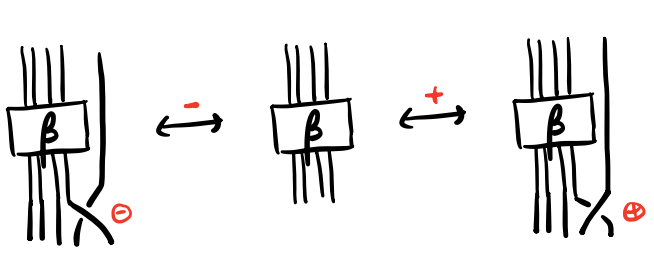}
    \end{center}
    
\end{itemize}
\note{These are meant to be construed in addition to braid isotopy in the annulus (i.e.\ braid-like Reidemeister 2 and 3, i.e.\ braid relations in the braid group).}
\end{theorem}

\begin{remark}
Together, these two theorems allow us to study knots in terms of their braid representatives. 
This can be very useful. For example, at the moment, \emph{triply-graded Khovanov-Rozansky homology} (a.k.a.\ HOMFLYPT homology) can only be computed from a braided diagram of a link, even though it's an invariant for links in $S^3$. 
\end{remark}

We have similar theorems for braid representatives of transverse links.
Bennequin proved the analogue to Alexander's theorem:

\begin{theorem}[\cite{Bennequin-alexander-thm}]
(Bennequin's theorem)
   For any transverse knot in $(\R^3, \xirot)$, there exists a representative that is braided with respect to the $z$-axis.  
\end{theorem}

Braids are also naturally oriented; a (positive) transverse link will be oriented counterclockwise. 

Nancy Wrinkle proved the following in her PhD thesis (!):

\begin{theorem}[\cite{Wrinkle-transverse-markov}, \cite{Orev-Shev-markov}]
(Transverse Markov theorem)
Any two braid representatives of a transverse link are related by conjugation and \emph{positive} Markov stabilization / destabilization.
\end{theorem}

These theorems allow us to study transverse knots via their braid representatives.

Finally, note that you can also have other non-braided links living in the solid annulus; braid (conjugacy classes) are just a special type of 
\emph{annular link}.

An \emph{annular link $L$} is a link $L \subset A \times I$. 
These links can be projected onto the annulus; these are called \emph{annular diagrams}. 
We usually draw them on the plane and put a marking where the axis is. 
In other words, an annular link is really just the data of a link $L$ and an unknotted axis $U$ inside $S^3$. 
We will talk more about annular links later. 

\subsubsection{Summary}

Let's summarize the relationships among the different types of links in $\R^3$ that we've introduced.

We have the following surjections of sets.
\begin{enumerate}
    \item Thinking about the front projection and $(\R^3, \xistd)$, we have 
    \[
        \text{Legendrian links} \onto
        \text{transverse links} \onto
        \text{smooth links}
    \]
    \item Thinking about braids and $(\R^3, \xirot)$, we have
    \[
        \text{braids} \onto
        \text{braid conjugacy classes} 
        \onto
        \text{transverse links}
        \onto \text{smooth links}.
    \]
    \item Finally, we obviously have 
    \[
        \text{annular links} \onto 
        \text{smooth links}
    \]
    by just forgetting the unknotted axis.
\end{enumerate}

\subsection{Plamenevskaya's transverse knot invariant}

The reference for this section is \cite{Plam-invt}

\begin{definition}
Let $\beta \in B_n$ be a braid representative for $\transknot = \hat \beta$. 
\begin{itemize}
    \item \emph{Plamenevskaya's cycle}, which we will denote by $\vec v_-$,  is the chain in $\Kh(\hat \beta)$ at the oriented resolution where all circles are labeled $v_-$. 
    \item \emph{Plamenevskaya's invariant}, denoted by $\psi$, is the homology class in $\Kh(\transknot)$ of $\vec v_-$. 
\end{itemize}
\end{definition}

Once again, this is a theorem-definition. 
We need to convince ourselves of the following facts:
\begin{enumerate}
    \item $\vec v_-$ is indeed a cycle. \note{In a previous exercise, you proved that there are no split maps out of an oriented resolution. Observe that any merge map out of $D_o$ sends $\vec v_- \mapsto 0$.}
    \item $\psi$ is invariant under the transverse Markov moves, so that it is actually an invariant of 
\end{enumerate}

Moreover, we also natrually want to know the answers to the following questions:
\begin{enumerate}
    \item Is $\psi$ actually a transverse knot invariant, or is it just a smooth knot invariant? 
    In other words, are there smoothly isotopic transverse knots $\transknot \neq \transknot'$ such that $\psi(\transknot) \neq \psi(\transknot')$? 
    \note{Note that `1' is a knot invariant, but it's not very useful.}
    \item How does $\psi$ behave under cobordisms? Is it \emph{functorial}, i.e.\ if $F: \transknot \to \transknot'$ is a cobordism, is $\Kh(F)(\psi(\transknot)) = \psi(\transknot')$? 
    \note{You can restrict to symplectic cobordisms $F$, but since Khovanov homology is a smooth invariant, $\Kh(F)$ only takes into account the underlying smooth cobordism.}
    \item Is $\psi$ \emph{effective}? In other words, can $\psi$ distinguish pairs of transverse knots that existing invariants can't?
\end{enumerate}

The first question is easy to answer, after we have the following proposition. 
\begin{proposition}
\label{prop:psi-neg-stab}
    If $\transknot$ is a negative stabilization of some $\transknot'$, then $\psi(\transknot) = 0$. 
\end{proposition}
Thuss it suffices to find $\hat\beta$ such that $\psi(\hat \beta) \neq 0$, since the negative stabilization of $\hat \beta$ represents the same \emph{smooth} knot type.

\begin{example}
Let $\One_n$ denote the identity braid in the braid group $B_n$. 
The transverse (fillable, max tb) unknot satisfies $\psi(\hat \One_1) \neq 0$ but its stabilization $\psi(\hat \sigma_1\inv) = 0$. 
\end{example}

To answer the last question, we need to know what the \emph{classical} transverse invariants are.

\begin{definition}
Let $\transknot$ be a transverse knot. 

\begin{itemize}
    \item Let $\beta \in B_b$ be a braid representative for $\transknot$.
    The \emph{self-linking number} of $\transknot$ is 
    \[
        \selflinking(\transknot) = \selflinking(\hat \beta) = -b + \writhe(\hat \beta).
    \]
    \note{In terms of a front diagram $D$ for $\transknot$, the self-linking number is just $\writhe(D)$.
    There is a more geometric definition of self-linking, by thinking of $\transknot$ as a framed knot.}

    \mz{You can quickly check for yourself that self-linking is preserved by the transverse Markov moves.}
    
    \item The \emph{smooth knot type} of $\transknot$ is simply the equivalence class of the underlying smooth knot.
\end{itemize}

A transverse invariant is called \emph{effective} if it can distinguish transverse knots better than these classical knots can.

\note{In other words, in practice we just need the invariant to tell us more than just the number $\selflinking(\transknot)$. 
}
\end{definition}

\begin{proposition}
\label{prop:psi-q-gr}
    The quantum grading of the class $\psi(\transknot)$ is the self-linking number of $\transknot$. 
\end{proposition}

\begin{exercise}
Prove Propositions \ref{prop:psi-neg-stab} and \ref{prop:psi-q-gr}.
\end{exercise}

\begin{question}
(Open) Is Plamenevskaya's invariant effective? 
\end{question}

In fact, there have been many refinements of Plamenevskaya's invariant defined throughout the years
(e.g.\ \cite{LNS-transverse-khovanov}). It is not known if these are effective either!

\subsection{Annular Khovanov homology}

Annular Khovanov homology is the a version of Khovanov homology for links in the thickened annulus. 
We will first give a definition for $\AKh$ as a tri-graded homology theory, analogously to how we defined Khovanov homology. Then, in \S \ref{sec:AKh-summary}, we give a more nuanced interpretation of $\AKh$ as it relates to $\Kh$.

\subsubsection{Annular Khovanov homology}

Let $L \subset A \times I$ be a link in the thickened annulus, and let $D$ be a diagram for $L$ drawn on the annulus $A$. 
Instead of drawing the annulus, we typically mark the position of the deleted $z$-axis with an asterisk or other marking.

The annular Khovanov chain complex $\CAKh(D)$ is generated by the same distinguished generators as for Khovanov homology, but we now distinguish between homologically trivial and nontrivial circles.
We assign them \textit{triply} graded modules, graded by $\gr_h$, $\gr_q$, and a new \emph{winding number grading}, which we interpret below in Remark \ref{rmk:winding-number-grading}.

\begin{center}
    \includegraphics[height=1.5in]{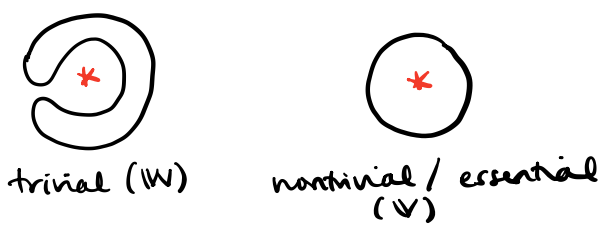}
\end{center}

\begin{itemize}
\item A circle $Z$ is \emph{trivial} if the marked point $*$ and $\infty$ are in the same region of $\R^2 \backslash Z$. To these circles, we associate the tri-graded module
    \[
        \WW := \Z \langle w_+, w_- \rangle
    \]
    where 
    \begin{align*}
        (\gr_h, \gr_q, \gr_k)(w_\pm) = (0, \pm 1, 0).
    \end{align*}
\item A circle $Z$ is \emph{non-trivial} or \emph{essential} if the marked point $*$ and $\infty$ are in different regions of $\R^2 \backslash Z$. To these circles, we associate the tri-graded module
    \[
        \VV := \Z \langle v_+, v_- \rangle
    \]
    where 
    \begin{align*}
        (\gr_h, \gr_q, \gr_k)(v_\pm) = (0, \pm 1, \pm 1).
    \end{align*}
\end{itemize}

The tri-grading for the Khovanov generators is then defined by extension to tensor products, and allowing for homological and quantum  shifts.

The annular Khovanov differential $d_{\AKh}$ is comprised of all the components of the Khovanov differential that preserve $\gr_k$ grading. 
One can work out all six possible interactions between trivial and essential circles to determine the tri-grading-preserving differential.

\mz{Instead of re-drawing these, here's stuff grad-school me TeXed up \cite{Zhang-akh}. The `type' names are not used in the literature; I just needed them for a proof later in the paper.}

\begin{center}
    \includegraphics[width=\textwidth]{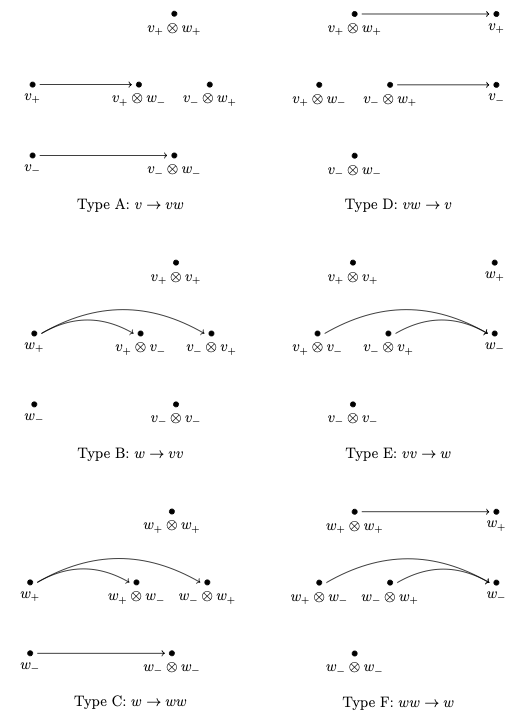}
\end{center}

The homology of $(\CAKh(D), d_{\AKh})$ is an invariant of the smooth isotopy type of the annular link.

\begin{remark}
    One can also define the annular Bar-Natan category, typically denoted by  $\BN(\mathbb{A})$ in the literature. We will not dig into this in this course.
\end{remark}

\begin{remark}
    \label{rmk:winding-number-grading}
\begin{enumerate}
    \item Pick a generic arc from $*$ to $\infty$. 
    For a given \emph{oriented} circle $\zeta$ in the plane missing the marked point $*$, $\gamma$ intersects $\zeta$ transversely (otherwise, perturb). The signed intersection (algebraic intersection) $I(\gamma, \zeta)$ is 0 if $\zeta$ is homologically trivial in the annulus, and is $\pm 1$ if $\gamma$ is essential. If we associate counter-clockwise circles with the `plus' labeling (and clockwise with `minus'), then $I(\gamma, \zeta) = \gr_k(\zeta)$, (interpreted appropriately).

    \item Note that this association of counter-clockwise/clockwise with plus/minus is a matter of convention and bookkeeping. This `orientation' on $\zeta$ is \emph{not} meant to be interpreted topologically in the context of cobordisms and such!

    \item Grigsby--Licata--Wehrli show that $\AKh(L)$ has a $\mathfrak{sl}_2(\C)$-representation, where $\gr_k$ is the weight-space grading. 
    From this point of view $\WW$ is two copies of the trivial representation, and $\VV$ is one copy of the defining representation.
    
    So, if you're working with $\C$ coefficients, it's also fair to call $\gr_k$ the weight-space grading. \note{I have in the past, quite unfortunately, also used this terminology even while working over $\F_2$ before really absorbing the `winding number grading' terminology, which is more intuitive.}
\end{enumerate}

\end{remark}

\subsubsection{Summary for $\AKh$}
\label{sec:AKh-summary}

Given the data of a link diagram $D$ in $\R^2$ and a basepoint $*$ in the complement of $D$ in $\R^2$, the winding number grading $\gr_k$ defines a filtration on the Khovanov chain complex $(\CKh, d_{\Kh})$.
Thus $(\CKh$ is $\gr_h$ and $\gr_q$ \emph{graded}, and $\gr_k$ \emph{filtered}.

The annular Khovanov homology of the annular link represented by $(D,*)$ is the homology of the \emph{associated graded object} to the $\gr_k$-induced filtration on $\CKh$. 
The Khovanov differential decomposes into two homogeneous-degree components:
\[
    d_{\Kh} = d_{(1,0,0)} + d_{(1,0,-2)}
\]

\begin{definition}
    Let $F_\bullet$ be a $\Z$-filtration on a chain complex $(\CC,d)$, with $F_i \supseteq F_{i+1}$. 
    \note{The situation with inclusion the other way is basically identical.} 
    The \emph{associated graded object} is the chain complex
    \[
        (\bigoplus F_i / F_{i+1}, \bar d)
    \]
    where $\bar d$ is the induced differential on the quotients. 
    \note{If $\CC$ is generated by distinguished generators that are homogeneous with respect to a filtration grading, then $\bar d$ is comprised precisely of the components of the differential between generators with the same filtration grading.}
\end{definition}

So, for an annular link $L \subset A \times I$, the annular Khovanov homology $\AKh(L)$ is a triply-graded invariant of annular links.

\begin{exercise}
\bea
\item Compute the annular Khovanov differentials for all six merge/split interactions between pairs of circles. Confirm that $d_{\Kh}$ only either preserves $\gr_k$ or decreases it by $2$. 
\item Since $\gr_k$ is defined on the level of Khovanov generators, this also gives a filtration on the \emph{Lee differential}. Identify all the $(\gr_h, \gr_q, \gr_k)$ homogeneous-degree components of the Lee differential.
\ee
\end{exercise}

\begin{remark}
Grigsby--Licata--Wehrli defined an annular version of Rasmussen's invariant by studying the annular Khovanov-Lee complex, which is $\gr_h$ graded and $(\gr_q, \gr_k)$-bifiltered. 
By taking a linear combinations of $\gr_q$ (denoted by $\gr_j$ in their paper) and $\gr_k$, they obtain a 1-parameter family $\{d_t\}$ of \textit{annular} concordance invariants \cite{GLW-akhlee}. 
This type of construction (mixing two filtration gradings) was popularized by Ozsv\'ath--Stipsicz--Szab\'o's $\Upsilon$ invariant from knot Floer homology. 

Here is a cartoon:
\begin{center}
    \includegraphics[width=2in]{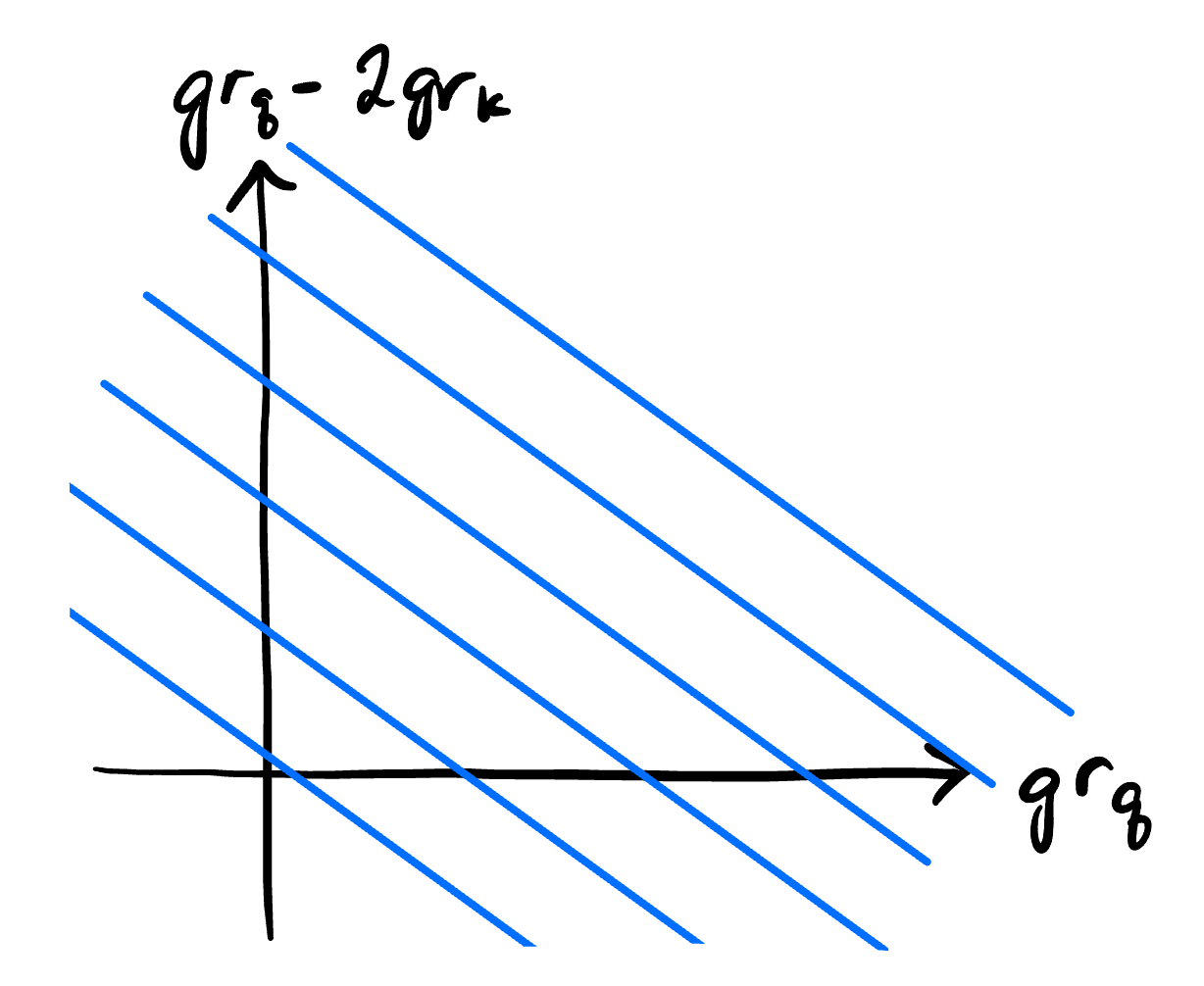}
\end{center}
The blue lines are level sets in the $\gr_t = \gr_q - t \gr_k$ filtration grading. For any $t \in [0,2]$, $\gr_t$ is a valid filtration grading on the annular-Khovanov-Lee complex.

Grigsby--Licata--Wehrli then apply these $d_t$ invariants to measure some contact topological quantities, such as \emph{right-veeringness} of the mapping class group element for the $b$-pointed disk given by a braid $\beta$.
\end{remark}

\begin{remark}
    We probably won't say much more about the mapping class group interpretation of braids in this course, but for those of you interested, here's an intuitive definition of the property of `right-veeringness'. 

    The braid group $B_n$ can be thought of as the mapping class group of a disk with $n$ (indistinguishable) marked points. 
    A braid $\beta \in B_n$ is then a dance of these $n$ marked points around each other, ending with the points returning to the original marked positions, but possibly permuted. This gives a sort of stirring action on the particles around them.

    Draw an oriented arc $\gamma$ from basepoint on the boundary of the disk to one of the marked points, then apply the stirring action by $\beta$. If the resulting arc $\gamma'$ veers right upon leaving the boundary, then $\beta$ is right-veering.

    We can use mapping class groups to describe the construction of contact 3-manifolds by Giroux's \emph{open-book decompositions}. Honda--Kazez--Mati\'c \cite{HKM-RV} showed that a contact structure $(M, \xi)$ is \emph{tight} if and only if \emph{all} its open books have right-veering monodromies.
\end{remark}

\subsubsection{Hubbard-Saltz's braid conjugacy class invariant $\kappa$}

As I mentioned previously
\note{but possibly only verbally in class}, Plamenevskaya's cycle $\tilde \psi$ / class $\psi$ gives us a homology class to measure, and thereby obtain invariants from in different settings. This is analogous to how Lee's canonical generators are measured to obtain Rasmussen's invariant.

As originally defined, $\psi$ is inherently annular. Hubbard and Saltz use the $\gr_k$ from $\AKh$ to define a braid conjugacy class invariant $\kappa$ by measuring the filtration level of $\psi$ \cite{Hubbard-Saltz-kappa}:

\begin{definition}[\cite{Hubbard-Saltz-kappa}, Definition 1]
Let $\beta \in B_n$. Let $\tilde \psi(\hat\beta) \in \CAKh(\hat \beta)$ denote Plamenevskaya's cycle in the annular Khovanov complex for the closure of $\beta$. 
Let $\{F_i\}$ denote the $\gr_k$-induced filtration on $\CAKh(\hat \beta)$. 
\begin{itemize}
    \item If $\psi(\hat \beta) \in \im(d_{\Kh})$, let 
        \[
            \kappa(\beta) = n + \min \{ i \st [\psi(\hat \beta)] = 0 \in H(F_i)\}
        \]
        \note{In other words, if $\psi$ dies in homology, $\kappa$ measures the filtration level at which it dies.}
    \item If $\psi(\hat \beta) \not= 0 \in \Kh(\hat \beta)$, then let $\kappa(\beta) = \infty$. 
        \note{If $\psi$ never dies, then $\kappa$ is infinite.}
\end{itemize}
\end{definition}

\begin{remark}
    In \cite[Theorem 2]{Hubbard-Saltz-kappa}, Hubbard--Saltz show that $\kappa$ sometimes increases by 2 under positive (Markov) stabilization, and is therefore not a transverse invariant.
\end{remark}

However, braid conjugacy class invariant can also be useful. For example, they use $\kappa$ to show the following:
\begin{theorem}[\cite{Hubbard-Saltz-kappa}, Corollary 17]
Let $\beta \in B_n$.
If $\kappa(\beta) \neq 2$ and $\kappa(m(\beta) \neq 2$, then $\beta = 1 \in B_n$. 
\end{theorem}

In short, satisfying both parts of the hypothesis tell us that the mapping class group element $\beta$ is both right-veering and left-veering, and by work of Baldwin-Grigsby \cite{BG-braids}, the only such braid is the identity braid.

\begin{remark}
    All of these `filtration level at which ...' definitions can equivalently be made in terms of spectral sequences. 
    For example, $\kappa$ gives a lower bound on the length of the spectral sequence $\AKh(\hat \beta) \abuts \Kh(\hat \beta)$. 
    
    Since many of you might not have worked with spectral sequences before, I will give an introduction to this algebraic tool next Wednesday before we talk about these in class.
\end{remark}

\begin{remark}
    Lipshitz, Ng, and Sarkar also give various filtered refinements of Plamenevskaya's invariant in \cite{LNS-transverse-khovanov}, which are more likely to be effective than $\psi$ (because they are refinements). 
    \note{If you are interested, this could be a good final project.}
\end{remark}

\subsection{Ng's Thurston-Bennequin bound for Legendrian knots}

Recall that a knot $\Lambda \subset (\R^3, \xistd)$ is \emph{Legendrian} if it is always tangent to the contact planes $\xistd = \ker(dz - y dx)$. 
We draw front diagrams to describe Legendrian knots, where the more negatively sloped strand is in front, and the diagram has no vertical tangents, only cusps.

We have talked about classical invariant for transverse knots: the topological (smooth) knot type, and the self-linking number. We now discuss classical invariants for Legendrian knots, and how smooth topology plus framing information can tell us about Legendrian representatives of knots $K$.

\begin{remark}
    Everything we discuss in this section is for oriented links. For knots, the orientation turns out not to matter. I will however write `knot' everywhere for simplicity.
\end{remark}

\begin{definition}
    Let $\Lambda$ be a Legendrian knot. We have the following classical invariants:
    \begin{enumerate}
        \item[(0)] Topological knot type, i.e.\ $K = [\Lambda]$, the equivalence class of $\Lambda$ under smooth isotopy
        \item[(1)] \emph{Thurston-Bennequin number}, which measures how much the contact planes twist as you follow the knot. Given a front diagram $F$ of $\Lambda$,
        \[
            \tb(\Lambda) = \writhe(F) - c(F)
        \]
        where 
        \[ 
            c(F) = \frac{1}{2} (\# \text{cusps}).
        \]
        \note{A zigzag in the front projection represents the knot traveling like a spiral staircase. If you think of this as the path of a car traveling up to down levels in a multi-story parking garage, $c(F)$ basically measures how many times you drive around in circles (ignoring the $z$-coordinate).} 
        \item \emph{Rotation number}, which measures (roughly speaking) how many spiral staircase floors you travel up and down (counted with sign). 
        Given a front diagram $F$ for $\Lambda$, 
        \[
            \rot(\Lambda) = \# ( \text{down cusps} - \text{up cusps}).
        \]
        \begin{center}
            \includegraphics[width=2in]{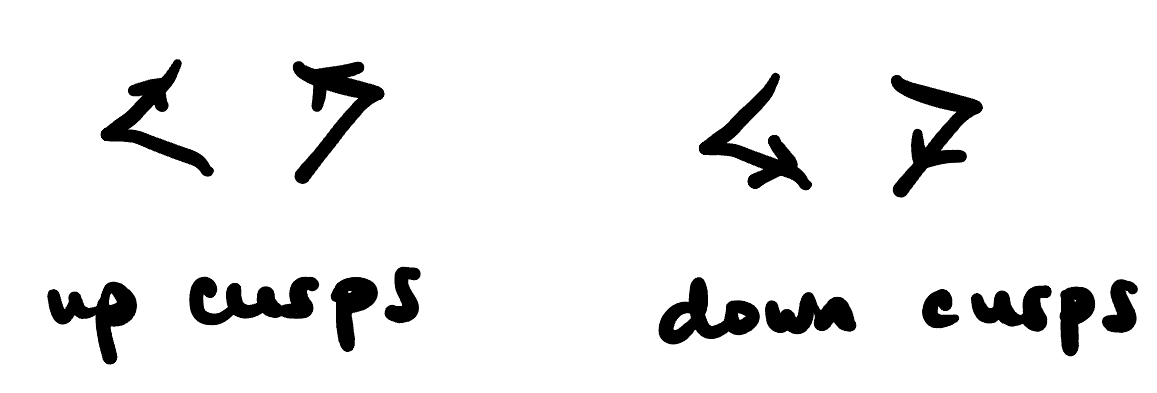}
        \end{center}
        \note{In my head, I think of this in terms of a kind of energy. Suppose you are running around a building with a lot of spiral staircases. You gain energy by descending staircases but lose energy by climbing staircases. In this Escher-esque world, you can travel the path of a knot and end up with more or less energy than you started with! (Sounds freaky, but this is just... monodromy.)}  
    \end{enumerate}
\end{definition}

\begin{remark}
    Recall that for a transverse knot $\transknot$, the self-linking number is just the writhe of a front projection. So, if $\transknot$ is a transverse push-off of $\Lambda$, then 
    \[
        \selflinking(\transknot) = \tb(\Lambda) + c(F).
    \]
\end{remark}

Our motivating question today is the following: 
\begin{question} 
Given a smooth knot $K$, what is the maximal $\tb(\Lambda)$ for Legendrian representatives $\Lambda$ of $K$?
\end{question}

The \emph{maximal Thurston-Bennequin number} achievable by a Legendrian representative of $K$ is called the ``max tb'' of $K$, and is denoted by $\maxtb(K)$. 

\begin{remark}
\begin{enumerate}
    \item Given $K$, we can study the Legendrian \emph{geography} and \emph{botany} of $K$. If you plot all Legendrian representatives of $K$ on a lattice with coordinates $(\rot, \tb)$, the support will be a ``mountain range''
    (see \cite{Etnyre-Honda-mountain}). The $\maxtb$ is the height of the tallest mountain.
    \textit{Geography} is the study of the support, and \textit{botany} is the study of the flora (i.e.\ Legendrian representatives) at each particular lattice point.
    \item The tallest mountain might not be unique. For example, the mountain range for positive torus knots looks like this:
    \begin{center}
        \includegraphics[width=1.5in]{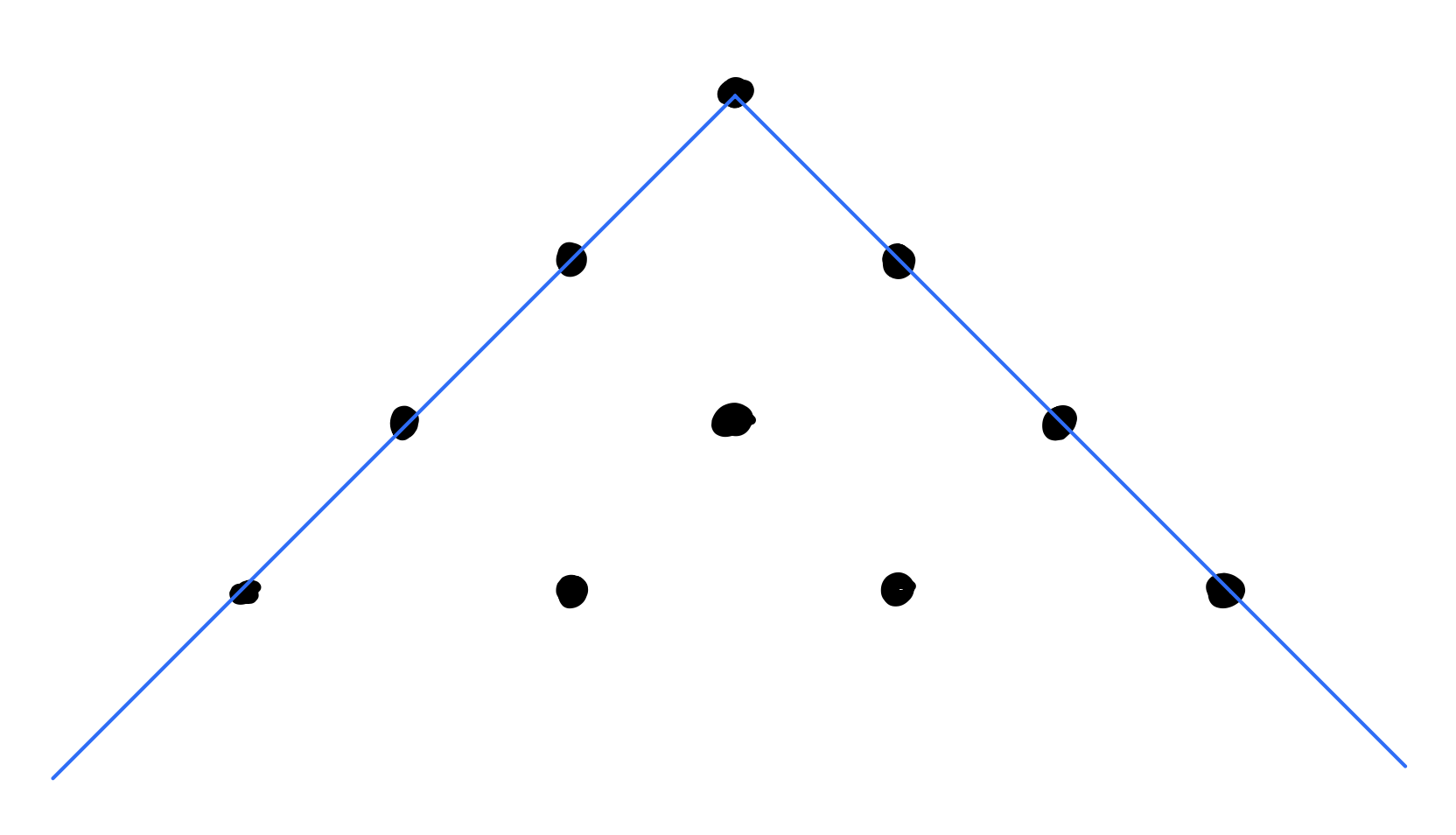}
    \end{center}
    while the mountain range for negative torus knots looks like this:
    \begin{center}
        \includegraphics[width=1.5in]{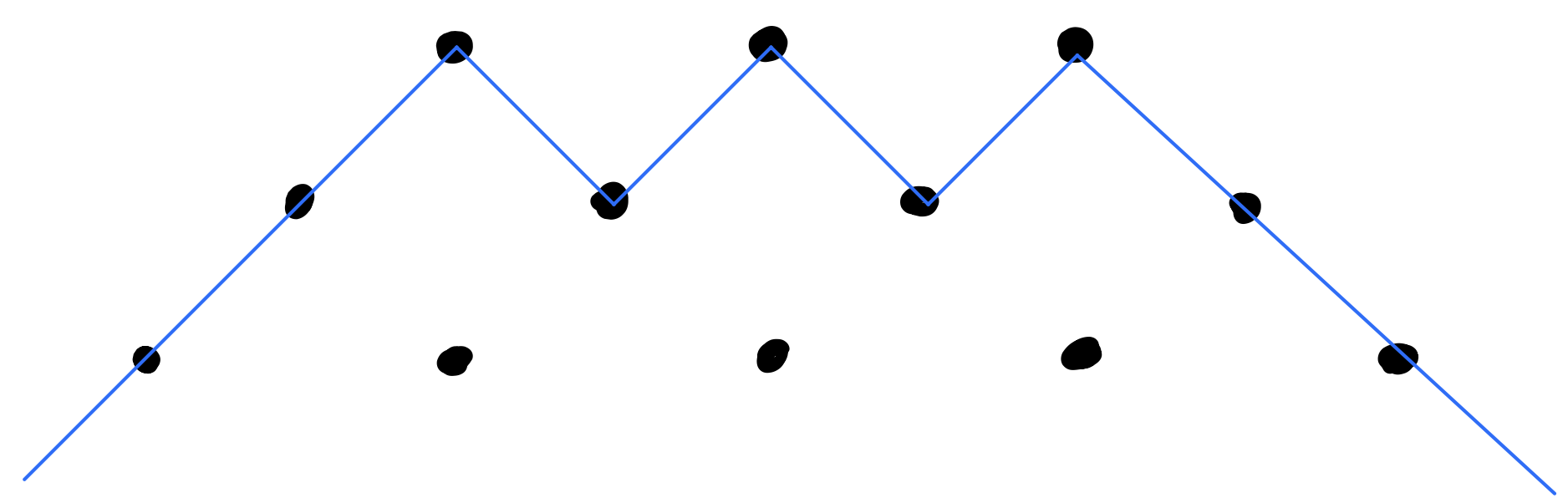}
    \end{center}
    \item It is easy to descend the mountain, i.e.\ decrease $\tb$, simply by negative and positive stabilization:
    \begin{center}
        \includegraphics[width=4in]{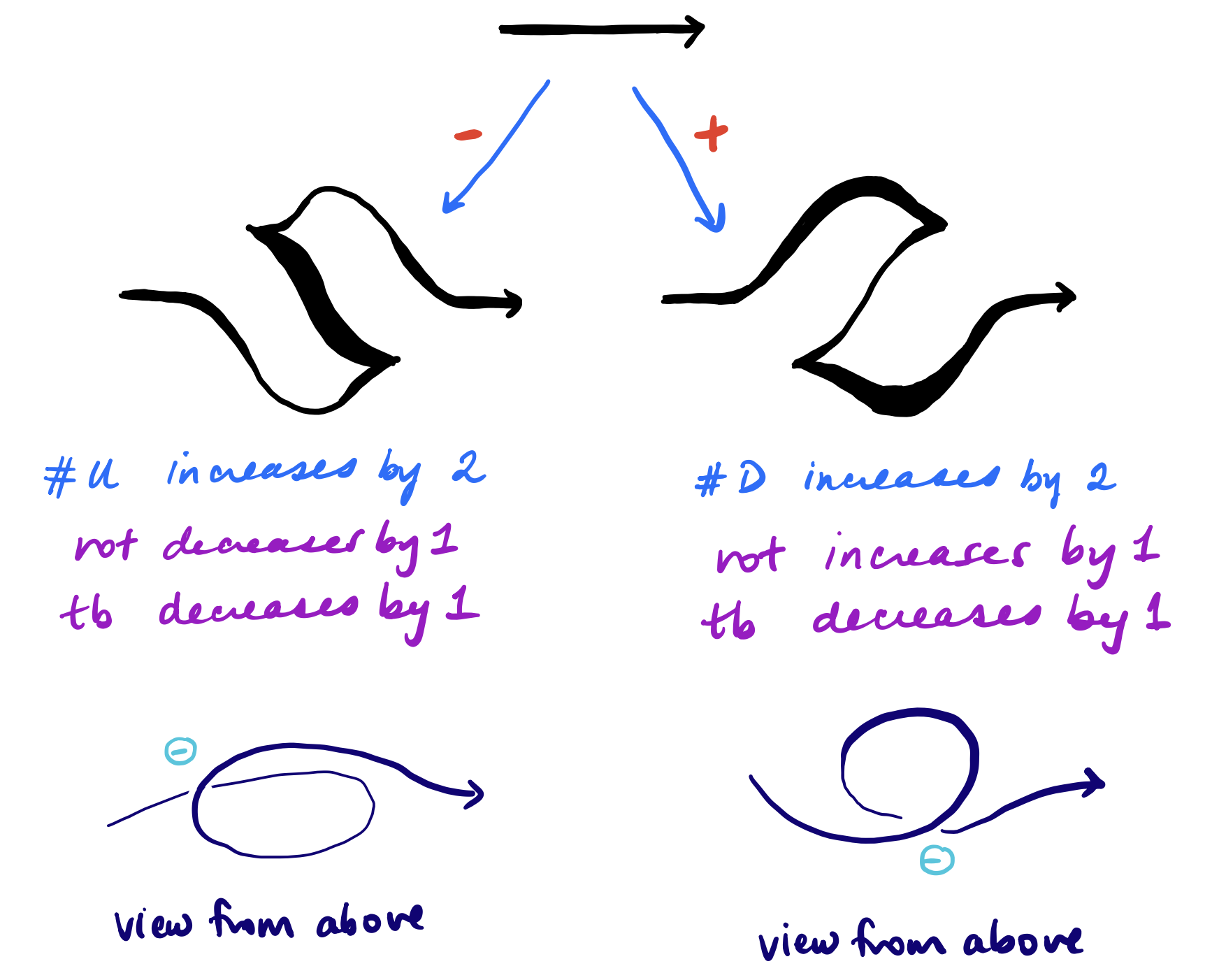}
    \end{center}
    However, it is not always possible to climb the mountain. 
\end{enumerate}

\begin{remark}
\begin{enumerate}
    \item In the drawings above, the thicker strands in the `S' and `Z' are closer to you. The `view from above' is a cartoon only. \note{This would normally be called a Lagrangian projection, but technically then one needs to be careful about the areas enclosed by curves in the projection. We will not comment on this further in this course.}
    \item Notice that the labels `positive' and `negative' for these Legendrian stabilizations have more to do with whether the rotation number increases or decreases, respectively. 
    From the point of view of smooth topology, these are both negative R1 moves. 
    \item On the other hand, the legal Legendrian R1 moves 
        \begin{center}
            \includegraphics[width=3in]{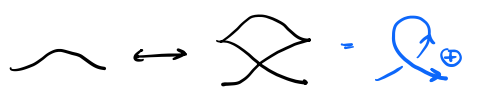}
        \end{center}
    (and the $180^\circ$ rotation of the above) are smoothly positive R1 moves. 
    This jives with the idea that, for the transverse pushoff of this Legendrian, positive stabilizations preserve transverse link type, but negative stabilization change the transverse link type (and therefore definitely also the Legendrian link type). 
        
\end{enumerate}
\end{remark}

\mz{I really ought to give citations for all this; I will need to spend some time finding the right references. Nevertheless, you do not need these references for the remainder of this course.}
\end{remark}

Bennequin \cite{Bennequin-tb-bound} discovered the first bound on $\maxtb$; we have since found stronger and stronger bounds. Here's a sampling:

\begin{proposition}
    Let $\Lambda$ be a Legendrian representative for $K$. 
    Then $\tb(\Lambda) \leq \maxtb(K)$, which is in turn bound by the following values:
\begin{enumerate}
    \item (Bennequin) \cite{Bennequin-tb-bound} $\maxtb(K) \leq 2g_3(K) - 1$ \mz{Impactful result -- was used to prove that there are contact structures on $\R^3$ that are homotopic but not contactomorphic to $(\R^3, \xistd)$; do not ask me for details. Maybe ask Orsola instead. :) }
    \item (Slice-Bennequin bound) \cite{Lisca-Matic-slice-bennequin}
    \[ \maxtb(K) \leq 2g_4(K) - 1\]
    So the original Bennequin bound was probably only detecting the 3D shadow of the real bound, which comes from 4D behavior! This makes sense because $(\R^3, \xistd)$ is the boundary of the standard symplectic 4-ball, which comes from the complex structure of the unit ball in $\C^2$.
    \item ($s$-Bennequin bound) \cite{Plam-invt, Shumakovitch-rasmussen-invt}
    \[ \maxtb(K) \leq s(K) - 1\]
    Here $s(K)$ is Rasmussen's invariant.
    \item (Ng's Khovanov bound) \cite{Ng-tb}
    \[ \maxtb(K) \leq \min \{ \delta \st \Kh^\delta(K) \neq 0 \} \]
    where $\gr_\delta = \gr_h - \gr_q$. 
    \mz{We will discuss the $\delta$ grading and prove this bound today. }
\end{enumerate}
\end{proposition}

\begin{remark}
    Observe that $(4) \implies (3) \implies (2) \implies (1)$. Such is the progress of humanity.
\end{remark}

\begin{exercise}
    Prove the $s$-Bennequin inequality. 
    
    \emph{Hint:} First relate $\maxtb$ with $\selflinking$. Then compare the quantum \emph{grading} of Plamenevskaya's cycle $\vec v_-$ in the Khovanov complex with the quantum \emph{filtration} of Lee's canonical classes in the Lee homology. 
\end{exercise}

Let's state Ng's `strong Khovanov bound' again the way he phrased it:

\begin{theorem}[\cite{Ng-tb}, Theorem 1]
\label{thm:Ng-tb-bound}
    For any link $K$, 
    \[
        \maxtb(K) \leq \min \{ k \st  \bigoplus_{i-j = k} \Kh^{i,j}(K) \neq 0 \}.
    \]
\end{theorem}

\begin{remark}
    For alternating knots, the Khovanov homology (over $\Q$) is rather simple; it is supported on two diagonals on the $(\gr_h, \gr_q)$ lattice. This was first shown by Lee in \cite{Lee-endomorphism}, and was extended to \emph{quasi-alternating links} by Manolescu and Ozsv\'ath \cite{MO-quasialternating}. 
    Consequently, the Khovanov homology of a quasi-alternating knot is entirely determined by two `classical' invariants, the knot signature and the Jones polynomial. \note{I was at a conference in Berkeley a few summers back where Vaughan Jones objected to his polynomial being called `classical'.}

    Ng proves that his bound is \emph{sharp} for alternating knots.
\end{remark}

\begin{remark}
    The $\delta$ grading here is significant because it is also the grading used by Seidel--Smith in their \emph{symplectic Khovanov homology} \cite{SS-khsymp}, $\Khsymp$, which is a Lagrangian intersection Floer homology interpretation of Khovanov homology that is isomorphic to (combinatorial, $\delta$-graded) Khovanov homology over fields of characteristic 0 \cite{Abouzaid-Smith-kh}.

    This means that Ng's bound can equivalently be viewed as a $\maxtb$ bound coming from symplectic Khovanov homology!
\end{remark}

In addition to working with $\delta$-graded Khovanov homology, Ng also makes use of a `shifted' version of Khovanov homology, which depends on the diagram; you may think of this as a version of Khovanov homology for \emph{framed links}, i.e.\ links made out of ribbon (or pappardelle) rather than string (or spaghetti).

Observe that, for a distinguished generator $g$ for $\CKh(D)$, the $\delta$-grading is given by 
    \[
        \gr_\delta(g) = p(g) + \writhe(D).
    \]
Ng's \emph{shifted Khovanov homology} $\Kh_{sh}$ is related to regular ($\delta$-graded) Khovanov homology by 
\begin{equation}
\label{eq:kh-shifted}
    \Kh^*(D) = \Kh_{sh}^{*-\writhe(D)}(D)
\end{equation}
so that the grading for $g$ in $\CKh_sh$ is simply $p(v)$. 

\begin{remark}
    This terminology (`shifted $\Kh$') is entirely local to this current section, and is not used in the literature, as far as I'm aware.
    However, the homology theory absolutely does appear elsewhere in the literature in various forms (and slightly different conventions...).
    Here's how I think about $\Kh$ vs $\Kh_{sh}$.

    If we were to define a bracket for $\Kh$ so that we didn't have to include global shifts at the end, we would use the bracket relations
    \begin{center}
        \includegraphics[height=2in]{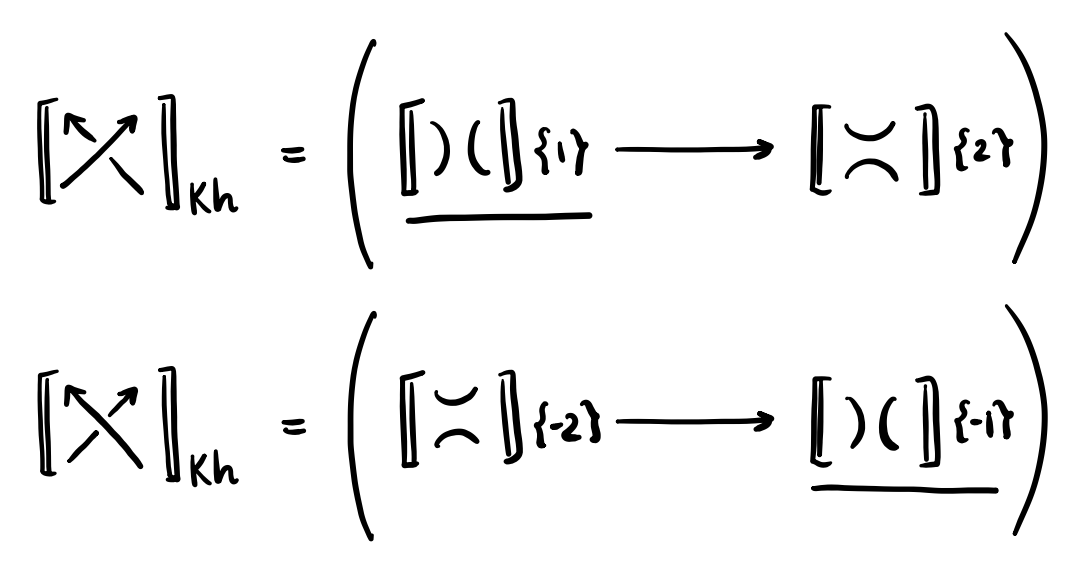}
    \end{center}

    For the framed link invariant $\Kh_{sh}$, the bracket relation would instead be 
    \begin{center}
        \includegraphics[height=2in]{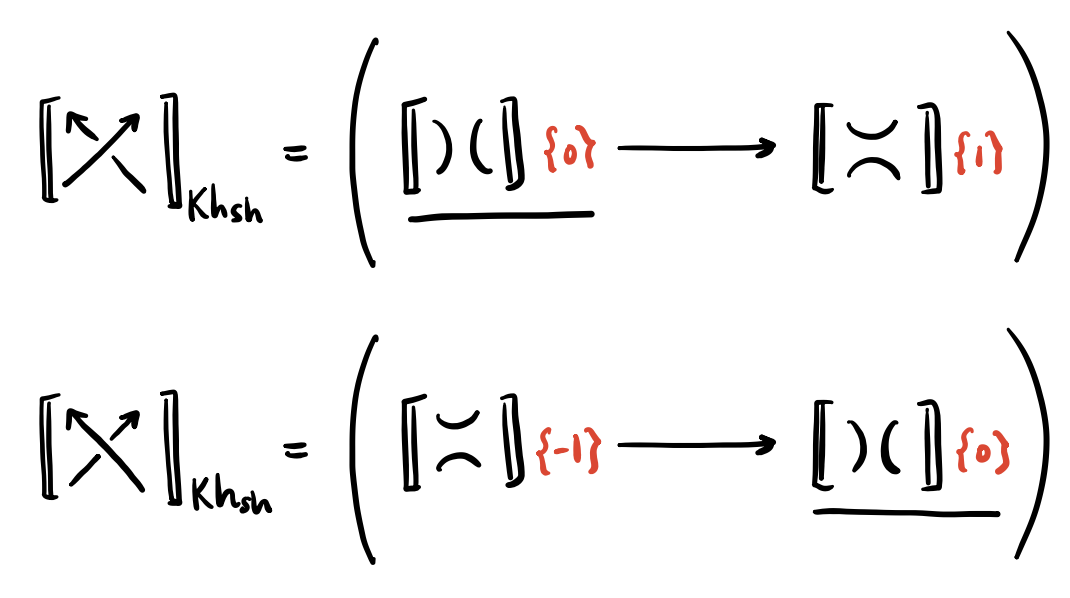}
    \end{center}

\end{remark}

The remainder of this section will be dedicated to proving the following proposition, which will immediately imply Theorem \ref{thm:Ng-tb-bound}.

\begin{proposition}[\cite{Ng-tb}, Proposition 5]
\label{prop:Ng-sh-tb-bound}
Let $D$ be a front diagram for a Legendrian knot $\Lambda$ with smooth knot type $K$. Then
\[
    \Kh_{sh}^*(D) = 0
\]
for all $* < -c(D)$, 
and therefore $\Kh^*(D) =0$ for all $* < \tb(D)$. 
\end{proposition}

The `therefore' part is clear after you remember that $\tb(D) = \writhe(D) - c(D)$. 

We will need the following lemma, which should hopefully feel somewhat similar to Lemma \ref{lem:lee-connected-sum-ses}.
\mz{If you want practice with long exact sequences coming from mapping cones, try proving this lemma!}

\begin{lemma}[\cite{Ng-tb}, Lemma 6]
\label{lem:Ng-leg-LES}
Let $D$ be a front diagram, and let $D_0, D_1$ denote the (Legendrian) $0$- and $1$- resolutions at a specific crossing:
    \begin{center}
        \includegraphics[width=2in]{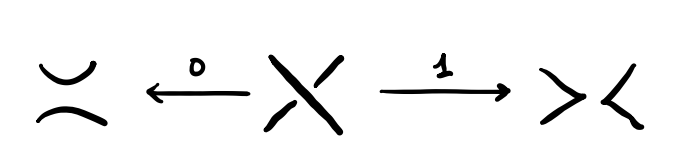}
    \end{center}

There is a long exact sequence
\begin{center}
\begin{tikzcd}
    \Kh_{sh}(D_0) \arrow{rr}{(-1)} & & \Kh_{sh}(D_1) \arrow{dl}{} \\
    & \Kh_{sh}(D) \arrow{ul}{} & 
\end{tikzcd}
\end{center}
where the $(-1)$ indicates a shift in $\delta$-grading.    
\end{lemma}

Because this triangle is exact, we immediately have the following corollary.

\begin{corollary}[\cite{Ng-tb}, Lemma 7]
If $\Kh_{sh}(D_0)$ and $\Kh_{sh}(D_1)$ are both supported only in $\delta$ gradings $\geq n$, then so is $\Kh_{sh}(D)$. 
\end{corollary}

As you can imagine, the proof of Proposition \ref{prop:Ng-sh-tb-bound} will proceed by induction on the number of crossings (since $D_0$ and $D_1$ both have one fewer crossing than $D$). 

\textbf{(Base case.)} If $D$ has no crossings, then $D$ is an unlink (call it $U^n$) of $n$ components, for some $n$. Since $\gr_\delta(g) = p(g)$ for all generators in $\CKh_{sh}(U^n)$, the complex (and homology) are supported on $\delta$ gradings $\{-n, \ldots, n\}$. 
The Legendrian representative consisting only of max-tb unknots 
    \begin{center}
        \includegraphics[height=1in]{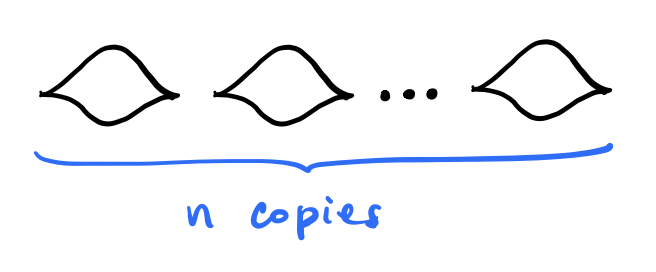}
    \end{center}
has $c(D) = n$ for the diagram above (and is in fact the max-tb representative of $U^n$).
So, indeed, the $\Kh_{sh}(U^n)$ is supported only in gradings $\geq c(D)$. 

\textbf{(Induction step.)} Now assume that for all diagrams with fewer crossings than $D$, the claim holds. 

\note{For the sake of understanding Ng's argument quickly, we are going to draw a concrete example to follow. But note that the proof is completely general.}

    \begin{center}
        \includegraphics[width=1.5in]{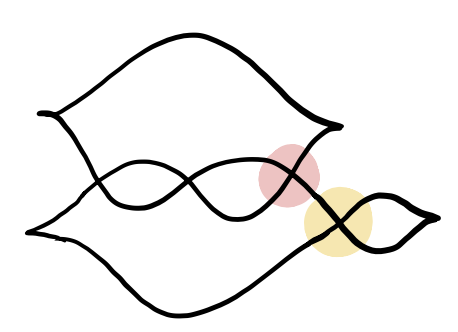}
    \end{center}

In $D$, pick a right-most crossing. There are two possible scenarios:
\begin{center}
    \includegraphics[height=1in]{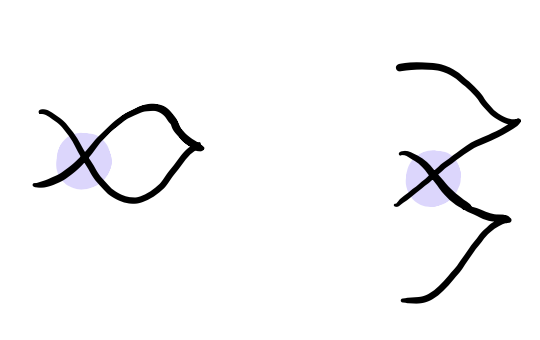}
\end{center}
\begin{enumerate}
    \item The two strands emanating rightward out of the crossing meet at a cusp, or 
    \item the two strands emanating rightward out of the crossing do not meet at a cusp, and in fact go to two separate cusps.
\end{enumerate}

We study the two cases separately. 
\note{In our example, Case 1 appears first, and then after reducing to $D_0$, we need to use Case 2, when the diagram looks like this:}
    \begin{center}
        \includegraphics[width=1.3in]{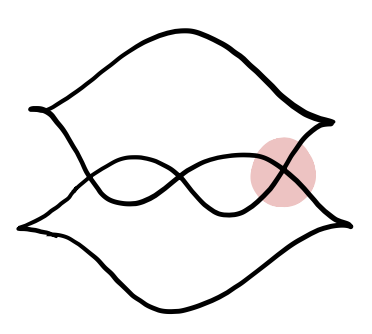}
    \end{center}

\textbf{Case 1.} 
Resolving the crossing in Case 1, we have 
\begin{center}
    \includegraphics[height=1in]{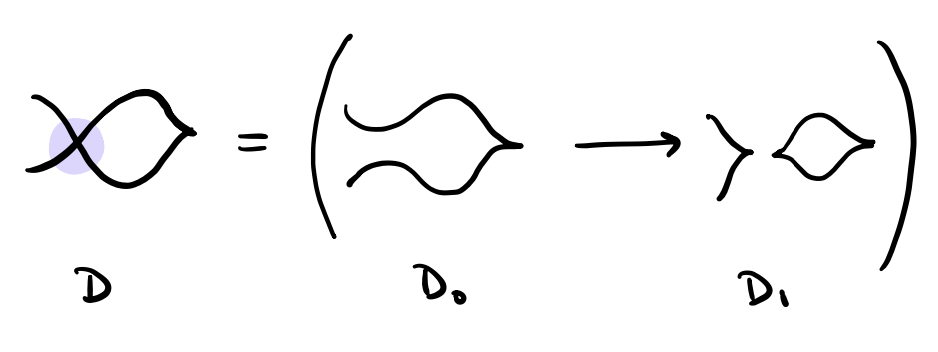}
\end{center}
By (smooth) R1 invariance, $\Kh(D) \cong \Kh(D_0)$. 
Note that $\writhe(D_0) = \writhe(D) -1$.
Combined, this tells us that $\Kh_{sh}^*(D) \cong \Kh_{sh}^{*-1}(D_0)$. 

On the other hand, $c(D_0) = c(D)$. 
By the induction hypothesis, $\Kh_{sh}(D_0)$ is supported on $\delta$-gradings $\geq c(D_0) = c(D)$.
Therefore $\Kh_{sh}^*(D) \cong \Kh_{sh}^{*-1}(D_0)$ must also be supported in these gradings as well. 

\textbf{Case 2.}
Resolving the crossing in Case 2, we have 
\begin{center}
    \includegraphics[height=1.2in]{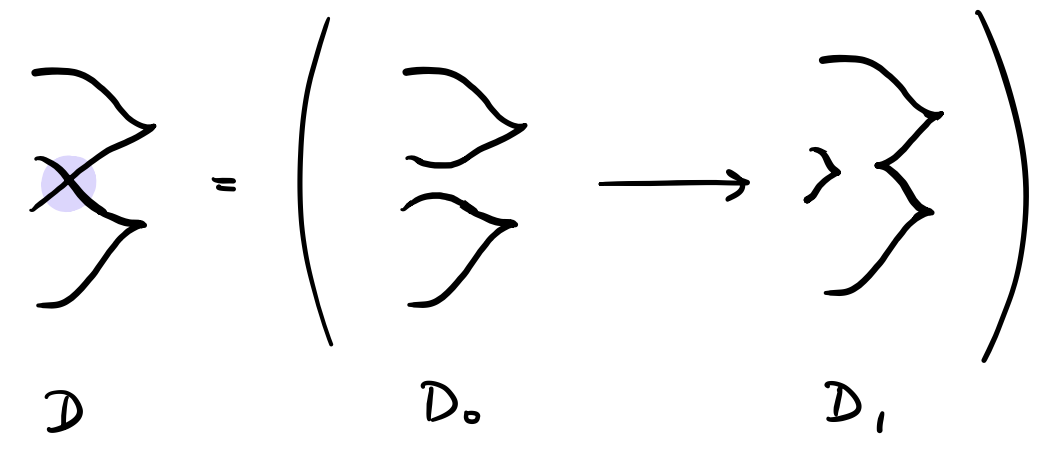}
\end{center}
By Lemma \ref{lem:Ng-leg-LES}, it suffices to show that both $D_0$ and $D_1$ are supported in degrees $\geq c(D)$. 
\begin{itemize}
    \item Note that $\writhe(D_0) = \writhe(D) - 1$, but $c(D) = c(D_0)$; by a similar argument as in Case 1, $D_0$ satisfies the induction hypothesis.
    \item While $c(D_1)$ is different from $c(D)$, $D_1$ is clearly isotopic to a diagram where you remove the stabilization on the far right strand. We can then apply the same argument as before. 
\end{itemize}

This concludes the proof of Proposition \ref{prop:Ng-sh-tb-bound}!

\section{Spectral sequences, $\Kh$'s relation to gauge / Floer theories}

Spectral sequences are an algebraic tool in homological algebra that computes homology by taking successively more accurate approximations.

In our text, the spectral sequence itself can be a useful object to study, e.g.\ as a knot invariant. 

We will start out with the most general version of a spectral sequence that you will see in this course; these come from filtered complexes. 
We'll then discuss more specific settings, like spectral sequences coming from bicomplexes.

\subsection{Spectral sequences from filtered complexes}

Let's begin by recalling some basics of filtered complexes in a different light, and in the process, set some notation (slightly different from before). 

A \emph{chain complex} (of, say, $\Z$-modules), in its most general form, is just a pair $(\CC, d)$ where $\CC$ is a $\Z$-module and $d$ is an endomorphism of $\CC$ satisfying $d^2 = 0$. 

In the situations we're usually most familiar with, there is a homological grading, and $\CC = \bigoplus C_i$ where $C_i$ are the homologically graded pieces. 

A ($\Z$-)\emph{filtration} on the complex $(\CC, d)$ is a sequence of subcomplexes $\{F_j \CC\}_{j \in \Z}$ such that 
\[
    F_j\CC \supseteq F_{j+1}\CC
\]
and $d(F_j\CC) \subseteq F_j\CC$. (In other words, $d$ preserves the filtration level.)

In our settings, we also require that there are levels $m$ and $M$ where $F_m\CC = \CC$ and $F_M \CC = \emptyset$. 
This makes our filtration \emph{finite-length}. 

\note{
    The filtration could also be indexed such that $F_j \CC \subseteq F_{j+1}\CC$. 
    We will see examples of both.
}

In our setting (i.e.\ Khovanov homology), $\CC$ is generated by a set of \emph{distinguished generators} that are homogeneous with respect to some \emph{filtration grading} $\gr_f$. 

In this case, we can decompose the differential as 
\[
    d = d_0 + d_1 + d_2 + \cdots
\]
where $d_i$ is the component of the differential that is $\gr_f$ graded degree $i$. Because our filtrations are finite-length, we know that eventually, for $k$ large enough, $d_k = 0$, so this is a finite sum.

\begin{remark}
    Notice that the associated graded chain complex can be written as just $(\CC, d_0)$ in this case.
\end{remark}

\begin{remark}
    Observe that since $d^2 =0$, we know that the graded pieces of the endomorphism $d^2$ must be 0, so we know
    \begin{align*}
        d_0^2 &= 0 \\
        d_1d_0 + d_0 d_1 &= 0\\
        d_2d_0 + d_0d_2 + d_1^2 &= 0\\
        \ldots & \\
    \end{align*}
    and so on. 
\end{remark}

\begin{definition}
The \emph{filtration spectral sequence} for the filtered complex $(\CC, d)$ with filtration $\{F_j\CC\}$ is a sequence of chain complexes, successively defined as follows:
    \begin{itemize}
        \item $(E^0, d^{(0)}) = (\CC, d_0)$; note that this is just the associated graded complex
        \item $(E^1, d^{(1)})$:
            \begin{itemize}
                \item $E^1$ is the homology of $E^0$
                \item $d^{(1)}$ is the graded degree 1 piece of the induced differential $\bar d$, viewed as a map on homology; this is the lowest graded piece of the induced map on homology
            \end{itemize}
        \item $(E^2, d^{(2)})$ is obtained the same way: $E^2$ is the homology of $E^1$, and $d^{(2)}$ is the lowest degree graded piece of the (now doubly induced map) $\bar \bar d$, which is in fact graded degree 2. 
    \end{itemize}
\end{definition}

\begin{remark}
First of all, if this definition is hard to parse, don't worry -- we will see an example and it will all be much clearer. 
    \begin{itemize}
        \item The notation in the definition above isn't really standard. For example, I added the parentheses around the superscript to remind you that, in general, $d^{(i)}$ is \emph{not} $d_i$. (It wouldn't even make sense since they're maps on completely different algebraic objects.) 
        \item The $E^i$ are called the \emph{pages} of the spectral sequence. Every time you turn the page, you are taking a subquotient (i.e.\ homology). 
        \item Subject to some boundedness conditions, these spectral sequences will \emph{collapse} (or \emph{terminate}) on a finite page. This means that, eventually, for $k \geq N$ for some large $N$, $E^k = E^N$, and all the $d^{(k)} = 0$. We say that the spectral sequence \emph{abuts} to the homology on page $E^N$, and we call this page the ``infinity page'' $E^\infty$. 
    \end{itemize}
\end{remark}

For instructional reasons, I will work over $\F_2$ today, so that we don't have to complicate the discussion with signs.
Our main tool is Gaussian elimination (see Corollary \ref{cor:cancellation-lemma}), which works over $\F_2$ as follows:

\begin{center}
    \includegraphics[width=2in]{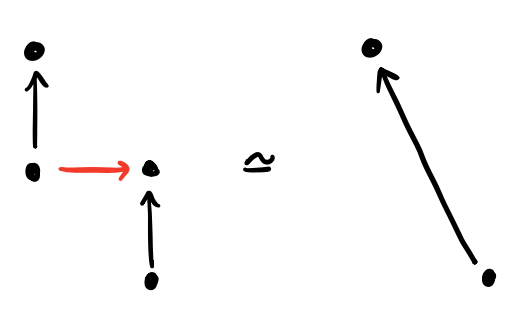}
\end{center}

\note{If you wish to work over characteristic 0, then you have to add a sign to your zigzag differentials.}

\begin{example}
    Here is an illustrative example of an abstract filtered complex and the pages of the associated spectral sequence. 
    \begin{itemize}
        \item On the left are the pages of the spectral sequence; induced maps that are not to be considered on the present page are grayed out.
        \item On the right is the actual computation one might do, where one keeps track of the longer induced differentials while working through the calculations of the pages.
    \end{itemize}

    \begin{center}
        \includegraphics[width=\linewidth]{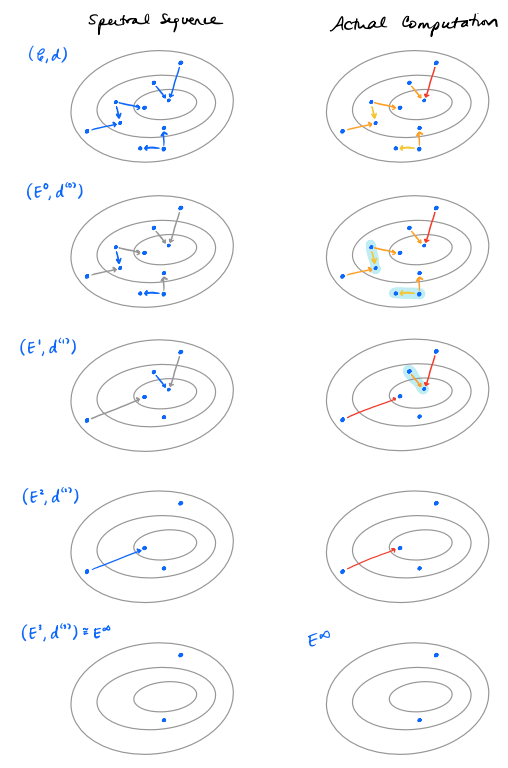}
    \end{center}
\end{example}

\begin{exercise}
    The Lee complex is $\gr_q$-filtered, and the homology of the associated graded complex is Khovanov homology. This gives a spectral sequence, known as the Rasmussen-Lee spectral sequence, or the Khovanov-to-Lee spectral sequence. 
    Compute the Khovanov--to--Bar-Natan spectral sequence over $\F_2$ for the trefoil appearing in \cite{BN-Kh}.

    Recall from Warning \ref{warn:char-2-Lee-vs-BN} that Bar-Natan homology corresponds to the Frobenius algebra $\F_2[X]/(X^2-X)$.
    The merge and split maps are shown below:
    \begin{center}
        \includegraphics[width=3in]{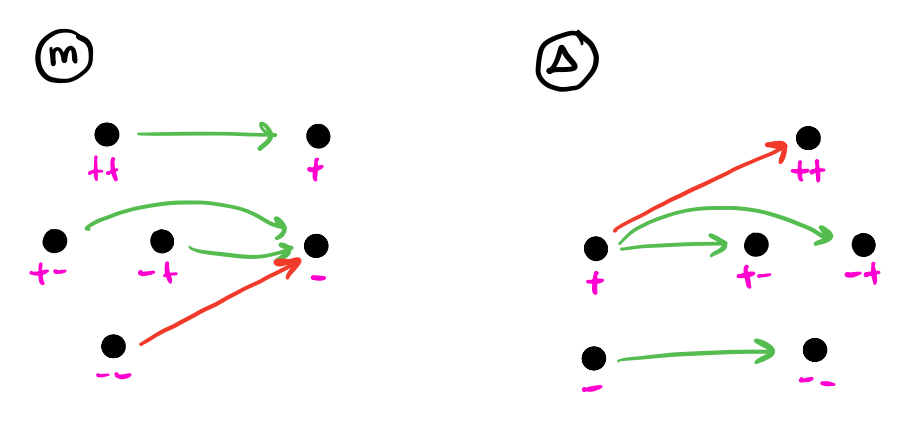}
    \end{center}

    \mz{The solution to this exercise is available on the class website.}
    
\end{exercise}

\note{
    To see how the Kh-to-Lee spectral sequence can be used to bound some topological quantities, see \cite{Alishahi-Dowlin-unknotting} and \cite{Caprau-etal-wiscon}.}

\begin{example}
Another filtration spectral sequence comes from the fact that $\AKh$ is the homology of the associated graded complex for $\Kh$ with respect to the $\gr_k$ grading. This filtration grading satisfies $F_k \subseteq F_{k+1}$. 
Since $d_{\Kh} = d_0 + d_{-2}$ where $d_0 = d_{\AKh}$, our spectral sequence first computes differentials that preserve $\gr_k$, then those that \emph{decrease} $\gr_k$ by $2$, then those induced differentials that \emph{decrease} $\gr_k$ by $4$, and so on. 

In the following annular Hopf link example, we don't have longer differentials, but you can see the difference between $\AKh(D)$ and $\Kh(D)$ on the different pages of the spectral sequence nevertheless. 

    \begin{center}
        \includegraphics[width=.9\linewidth]{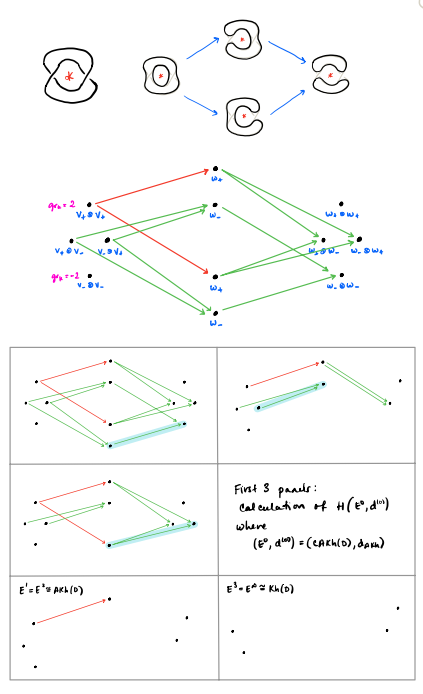}
    \end{center}
\end{example}

\subsection{Spectral sequences from bicomplexes}

Roughly speaking, a \emph{bicomplex} is a 2D complex with two differentials, one given by vertical arrows, and one given by horizontal arrows, as shown in the cartoon below:
\begin{center}
    \includegraphics[width=2in]{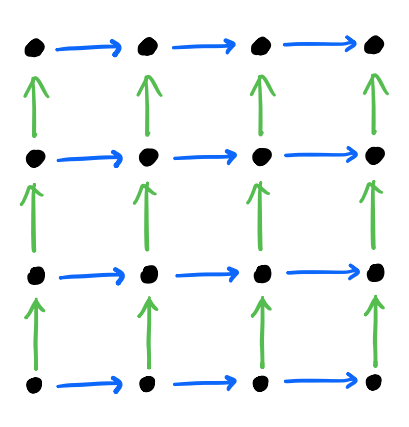}
\end{center}
It is automatically bigraded by $\Z \oplus \Z$ by the row number and the column number.

Here is a more formal definition with more details, now that you have some context.

\begin{definition}
A \emph{bigraded complex} or \emph{bicomplex} is the data of 
\begin{itemize}
    \item chain groups $C_{i,j}$ for $i,j \in \Z$
    \item two gradings $\gr_r$ and $\gr_c$, where $(\gr_r, \gr_c)(C_{i,j}) = (i,j)$
    \item two commuting differentials $d_h$ and $d_v$ such that 
        \begin{itemize}
            \item the $(\gr_r, \gr_c)$ bidegree of $d_h$ is $(0,1)$ and 
            \item the $(\gr_r, \gr_c)$ bidegree of $d_v$ is $(1,0)$.
        \end{itemize}
\end{itemize}

\end{definition}

\begin{remark}
    The \emph{total complex} of the bicomplex is a flattening of the above data into a single (one-dimensional) complex, with homological grading given by $\gr_r + \gr_c$ (whose level sets are lines of slope $-1$ in the cartoon). 

    However, since $d_h$ and $d_v$ commute, their sum is not a differential. 
    So, we need to introduce a minus sign to every other row of $d_h$ differentials in order to make each square \emph{anticommute}:

    \begin{center}
        \includegraphics[width=2in]{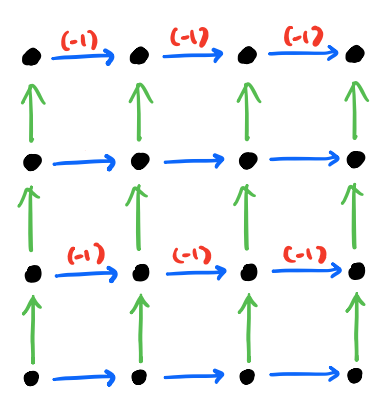}
    \end{center}
    
    Then, we get the \emph{total complex} of the bicomplex, which is a flattening of the bicomplex into a 1-dimensional, linear complex. 
    If $\mathcal{B}$ denotes the bicomplex, we denote the total complex (or \emph{totalization}) by $\Tot(\mathcal{B})$. 
\end{remark}

\mz{We again work over $\F_2$ in this section, in order to keep things as simple as possible while we are learning a new algebraic gadget. 
In this case, there is no difference between commuting and anticommuting differentials.}

Suppose we have a bicomplex $(\CC = \bigoplus C_{i,j}, d_h, d_v)$ of $\F_2$ vector spaces. The total complex clearly has two filtration gradings: $\gr_r$ and $\gr_c$. We call these the \emph{row-wise} and \emph{column-wise} filtrations:

\begin{center}
    \includegraphics[width=\linewidth]{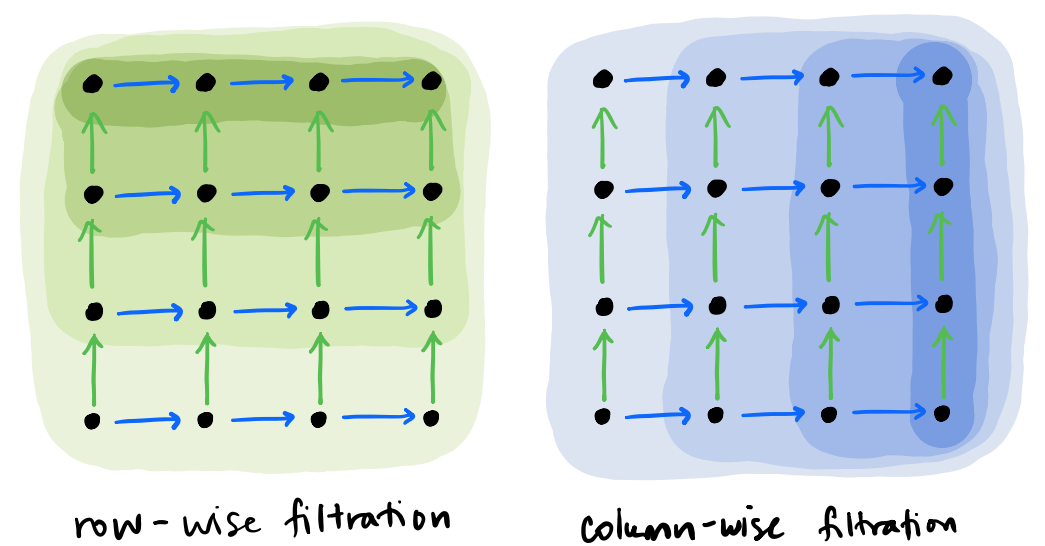}
\end{center}

There are therefore two associated filtration spectral sequences:
\begin{itemize}
    \item The filtration spectral sequence for the row-wise filtration (for filtration grading $\gr_r$) first computes homology with respect to $d_h$ on page 0. 
    On page 1, it proceeds to compute homology with respect to the chain maps induced by $d_v$.
    So, we denote this spectral sequence by $^{hv}E_\bullet$. 
    \item The filtration spectral sequence for the columnwise-wise filtration (for filtration grading $\gr_c$) first computes homology with respect to $d_v$ on page 0. 
    On page 1, it proceeds to compute homology with respect to the chain maps induced by $d_h$.
    So, we denote this spectral sequence by $^{vh}E_\bullet$. 
\end{itemize}
Both spectral sequences ultimately compute the homology of the total complex, so they abut to isomorphic vector spaces: $^{hv}E_\infty \cong ^{vh}E_\infty$.

\begin{remark}
\alert{Beware}, though, that because the total differential is filtered with respect to two different filtration gradings in these spectral sequence computations, the gradings will be messed up.
That is, the two infinity pages might have support on different lattice points! 
So all you can really say is that the ungraded homologies are isomorphic.

However, if you have some third unrelated grading $\gr_u$ ($u$ for \emph{unrelated} to the row and column gradings) hanging around, then obviously the spectral sequences will not affect that grading. 
In that case, you really have many spectral sequences, one for each value of $\gr_u$, and they don't talk to each other at all.
\end{remark}

The longer differentials (i.e.\ differentials on pages $\geq 2$ are induced by cancellation of differentials in the previous pages. 
For example, the `length-2' differentials on page 2 of the two spectral sequences travel like this:

\begin{center}
    \includegraphics[width=\linewidth]{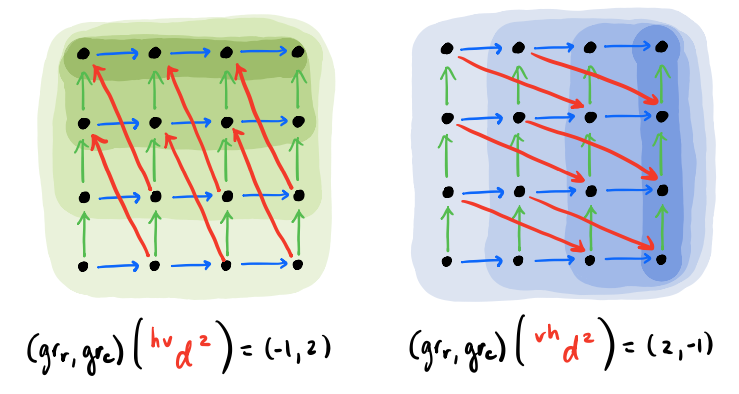}
\end{center}

The length-$n$ differentials would similarly travel the path determined by zigzags caused by Gauss elimination.

\begin{remark}
I'm using the terms `Gauss elimination' and `taking homology' sort of interchangeably, because we are working over a field. Allow me to explain:

Recall that `taking homology' is essentially the process of finding the smallest chain homotopy representative in a chain homotopy equivalence class; this is the one with zero differential, and we call it the homology.

Gaussian elimination (a.k.a.\ cancellation) is an algorithm for finding this smallest chain homotopy representative. 
Given a chain complex $(\CC, d)$, we can Gaussian eliminate away components of the differential $d$, \emph{in some order}, until we end up at a complex with zero differential. 
Because we are \alert{working over a field}, we are able to cancel along any component of the differential, because every nonzero coefficient is a unit.

This process is equivalent to changing the basis for our chain groups so that our original chain complex becomes a direct sum of the homology and a bunch of acyclic complexes, which we then discard.
\end{remark}

As a demonstration of the power of these two spectral sequences $^{hv}E_\bullet$ and $^{vh}E_\bullet$ computing the total homology, we will now discuss examples coming from equivariant homology. 
The goal is to obtain a spectral sequence relating two interesting objects $X$ and $Y$ via the following strategy.
\begin{enumerate}
    \item Set up a bicomplex using the topological data relating $X$ and $Y$.
    \item Arrange so that $^{vh}E_1$ is easy to compute, and describes $X$.
    \item Arrange so that $^{hv}E_\bullet$ is easy to compute, so that you know $^{hv}E_\infty$, which describes $Y$.
    \item Assuming the bicomplex is sufficiently bounded, we know that $^{vh}E^\infty$ must then also describe $Y$, and we are done. (The $^{vh}E_\bullet$ spectral sequence therefore relates $X$ and $Y$.)
\end{enumerate}

\subsubsection{Classical Smith inequality}

Let $p$ be a prime.
Let $X$ be a topological space with a $\Z/p\Z$ action generated by $\sigma: X \to X$ (that is, $\sigma^p = \id_X$). 
Let $X^\fix$ denote the fixed-point set, which itself is also a topological space.

The \emph{classical Smith\footnote{P.A.Smith, back in the 1920s, extensively studied $\Z/p\Z$ actions on topological spaces.
} inequality} gives an inequality on the total dimensions  of the singular homologies of $X$ and $X^\fix$, over $\F_p$:
\[ 
    \dim H_*(X; \F_p) \geq \dim H_*(X^\fix; \F_p).
\]

\begin{remark}
Borel's proof requires us to work over $\F_p$; see \cite{Borel-seminar-on-transformation-groups} for more details.
\end{remark}

We give Borel's proof \cite{Borel-seminar-on-transformation-groups} of the classical Smith inequality for $p=2$, working over $\F_2$ (as required). 
The proof for $p \geq 3$ is similar. 
Also, Borel's proof uses singular homology; we will instead demonstrate the proof using CW-homology, which means it's really only a proof for CW-complexes.
\note{Also, if you look at the pictures, this is kind of a `proof' by example; but you can piece together the real proof from the written text if you wanted to.}
We basically follow the proof found on pages 5-6 of \cite{Lipshitz-Treumann-smith}.

Our example to keep in mind is the following $S^2$ with the reflection involution $\tau$:
    \begin{center}
        \includegraphics[width=3in]{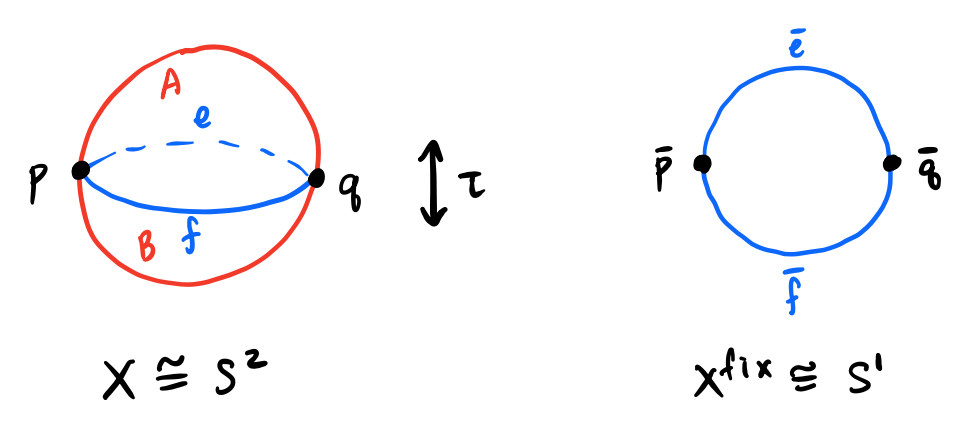}
    \end{center}
We've chosen a cell decomposition that plays well with this involution. 
The fixed point set is a copy of $S^1$ with the cell structure shown on the right of the figure.

The \emph{Tate bicomplex} is the bi-infinite bicomplex obtained by basically tensoring the complex $(C_*(X), \partial)$ with the localized polynomial ring $\F_p[\theta, \theta\inv]$, and then adding horizontal differentials $1+\tau$:

\begin{center}
    \includegraphics[width=4in]{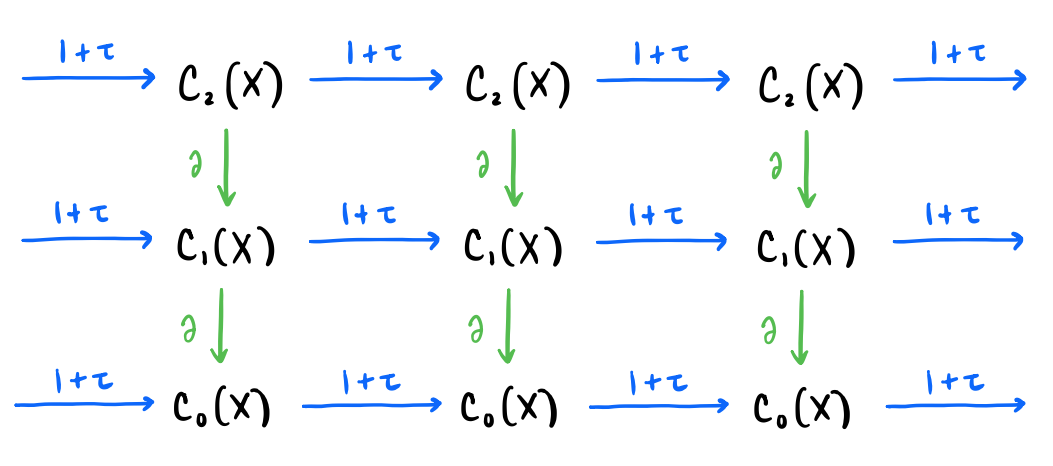}
\end{center}
The action of $\theta$ moves you to the next column over (to the right). 

\note{I'm abusing notation and letting $\tau$ denote both the topological involution as well as the involution on chains. This is because I want to use $\tau_*$ to denote the induced map on homology at some point.}

The $^{vh}E^\bullet$ spectral sequence is very hard to compute, but it's easy to compute up to page 1:
\begin{itemize}
\item $(^{vh}E^0, d^0) =$
    \begin{center}
        \includegraphics[width=4in]{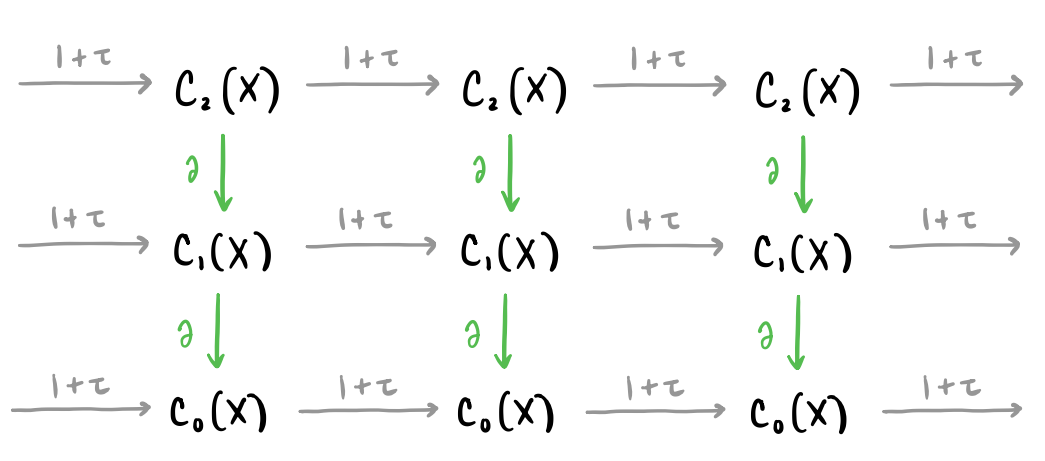}
    \end{center}
    The homology is just $H_*(X) \otimes \F_2[\theta, \theta\inv]$, which leads us to ...
\item $(^{vh}E^1, d^1) =$ 
    \begin{center}
        \includegraphics[width=4in]{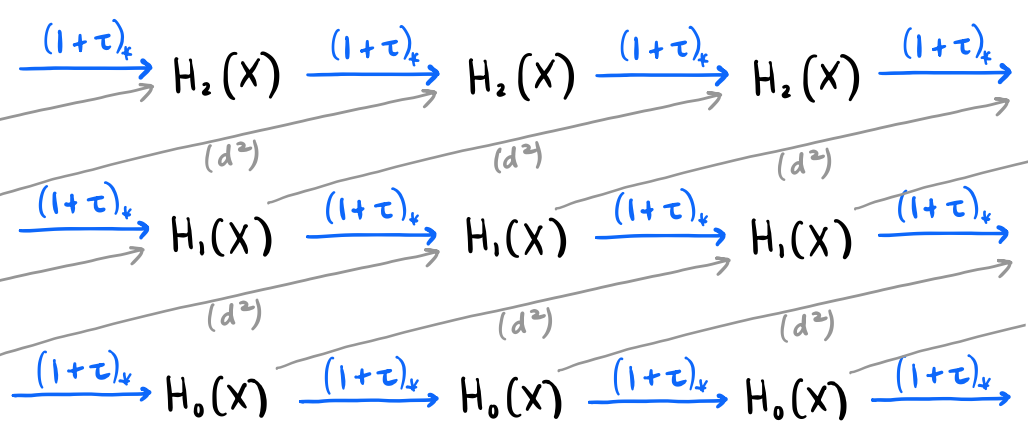}
    \end{center}
    I could think hard about what $(1+\tau)_*$ does, but this page is already interesting -- it describes $X$.
    In any case, $^{vh}d^2$ looks pretty bad, so I'm going to stop computing this spectral sequence here.
\end{itemize}

The $^{hv}E^\bullet$ spectral sequence turns out to be easy to compute:
\begin{itemize}
    \item $(^{vh}E^0, d^0) =$
    \begin{center}
        \includegraphics[width=4in]{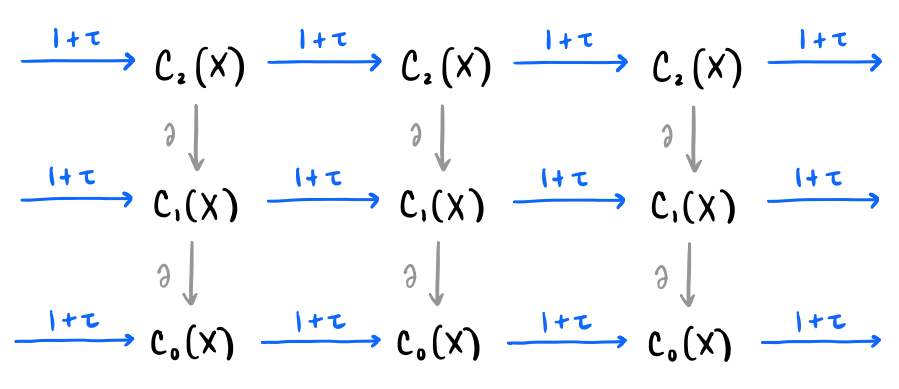}
    \end{center}
    The kernel of $d^0 = 1+\tau$ consists of precisely two types of chains:
    \begin{itemize}
        \item cells that are literally fixed by $\tau$ ($p,q,e,f$),  and 
        \item orbits of cells that are not fixed by $\tau$ ($A,B$).
    \end{itemize}
    The second class are the image of $1+\tau$. 
    Therefore the homology of this page can be canonically identified with $C_*(X^\fix)$. 
    \item Now what is the are the induced differentials $\partial_*$? These are just the restriction of $\partial$ to the fixed cells, of course -- the gluing maps definitely haven't changed! Let $\partial^\fix$ denote the CW differentials for $C_*(X^\fix)$.
    
    So, $(^{hv}E^1, d^1) =$
    \begin{center}
        \includegraphics[width=4in]{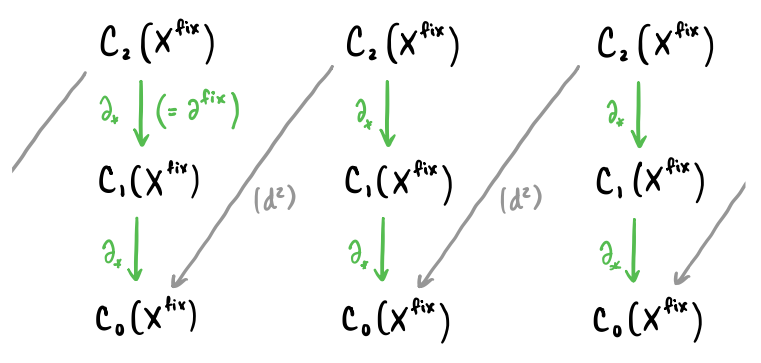}
    \end{center}
    which is just $(C_*(X^\fix), \partial^\fix) \otimes \F_2[\theta, \theta\inv]$.

    \item Hence $^{hv}E^2 \cong H_*(X^\fix) \otimes \F_2[\theta, \theta\inv]$. 
    We now need to reckon with the $d^2$ differentials shown in gray in the previous figure. These would have been induced by zigzags of the form
    \begin{center}
        \includegraphics[width=1in]{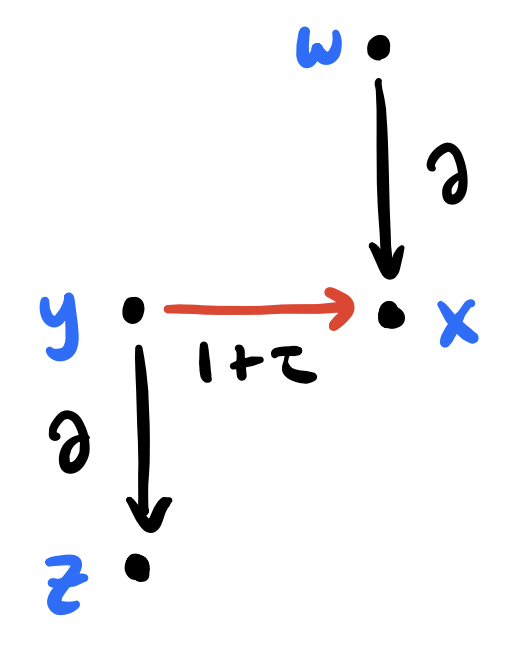}
    \end{center}
    For such a length-2 arrow to exist (i.e.\ have nonzero coefficient) on page 2, then both $w$ and $z$ needed to have survived to page 2. 

    \note{A subtlety here is that, technically speaking, we are thinking of $w$ as a homology class in $H_*(X^\fix)$. However, because we have been Gaussian eliminating this whole time, it's also ok to think of $w$ as a generator (CW cell) that survived to page 2; to be more precise, a chain representative $\tilde w$ of $w$ survived to page 1. But we will abuse notation and just call this chain $w$, since Gaussian elimination picks out a (noncanonical) representative / generator in a new basis.}
    
    Since $w$ survived to page 2, we know that it's (1) a fixed cell under $\tau$ and (2) in the kernel of $\partial^\fix$. But then $\partial w = 0$ as well, so the first step of this zigzag is 0. 

    \emph{There are no more longer differentials!}

    \item 
    So the spectral sequence collapses, and we know that 
    \[
        ^{hv}E^2 \cong H_*(X^\fix; \F_2) \otimes \F_2[\theta, \theta\inv].
    \]
    Furthermore, this is also isomorphic to both $^{hv}E^\infty$ and the homology of the total complex, $H_*(\Tot(\CC, \partial + (1+\tau)))$. 
\end{itemize}

Since $^{vh}E^\bullet$ also computes the homology of the total complex, we now know that this is a spectral sequence
\[
    H_*(X; \F_2) \otimes \F[\theta,\theta\inv] 
    \abuts
    H_*(X^\fix; \F_2) \otimes \F[\theta,\theta\inv].
\]
Now since every page turn of spectral sequence is a homology computation, the rank of the $\F_2[\theta, \theta\inv]$-module on each page must be monotone decreasing. 

After convincing ourselves that the rank of 
$H_*(Y; \F_2) \otimes \F[\theta,\theta\inv]$
as an $\F[\theta,\theta\inv]$-module is the same as the dimension of $H_*(Y; \F_2)$ over $\F_2$, we recover the classical Smith inequality.

\subsubsection{Khovanov homology and periodic links}
\label{sec:periodic-links}

We now use the same technique to prove a similar Smith inequality for the Khovanov homology of knots with a type of $\Z/2\Z$ symmetry. The reference for this section is \cite{Zhang-akh}; a more general reference is \cite{Stoffregen-Zhang-periodic}, but this uses techniques we have yet to discuss in this course.

Let $p$ be a prime.
Consider a link $\tilde L \subset S^3$ together with an orientation-preserving diffeomorphism $\sigma: S^3 \to S^3$ of order $p$ that preserves $\tilde L$ set-wise. 

The \emph{Smith conjecture} (named after P.A.Smith, and now a theorem \cite{Waldhausen-smith}, \cite{Morgan-Bass-book-smith}, tells us that 
the fixed point set $\sigma$ must either be (1) empty  or (2) an unknotted fixed-point axis. 

If $\sigma$ has an unknotted axis and $\tilde L$ is disjoint from this axis, then we say that $\tilde L$ is a \emph{$p$-periodic link}.

In 2010, Seidel and Smith\footnote{a different, much younger Smith} proved the following Smith-like inequality for symplectic Khovanov homology:

\begin{theorem}[\cite{Seidel-Smith-localization}]
Let $\tilde L$ be a 2-periodic link, and let $L$ denote the quotient link. Then 
\[  
    \dim \Khsymp(\tilde L; \F_2) \geq \Khsymp(L;\F_2).
\]
\end{theorem}

While Abouzaid--Smith proved that $\Khsymp$ agrees with $\delta$-graded $\Kh$ over characteristic 0, we still don't know if there is an isomorphism over $\F_2$. 
So this result does not tell us if there's a similar dimension inequality for regular $\Kh$. 

Let's see what happens when we use Borel's technique. Our goal is to relate $\Kh(\tilde L)$ and $\Kh(L)$ (all over $\F_2$), by hoping to see $\Kh(L)$ on some page of a spectral sequence built from $\CKh(\tilde L)$. 

Here is the Khovanov-Tate complex for a two-periodic diagram $\tilde D$:
\begin{center}
    \includegraphics[width=0.5\linewidth]{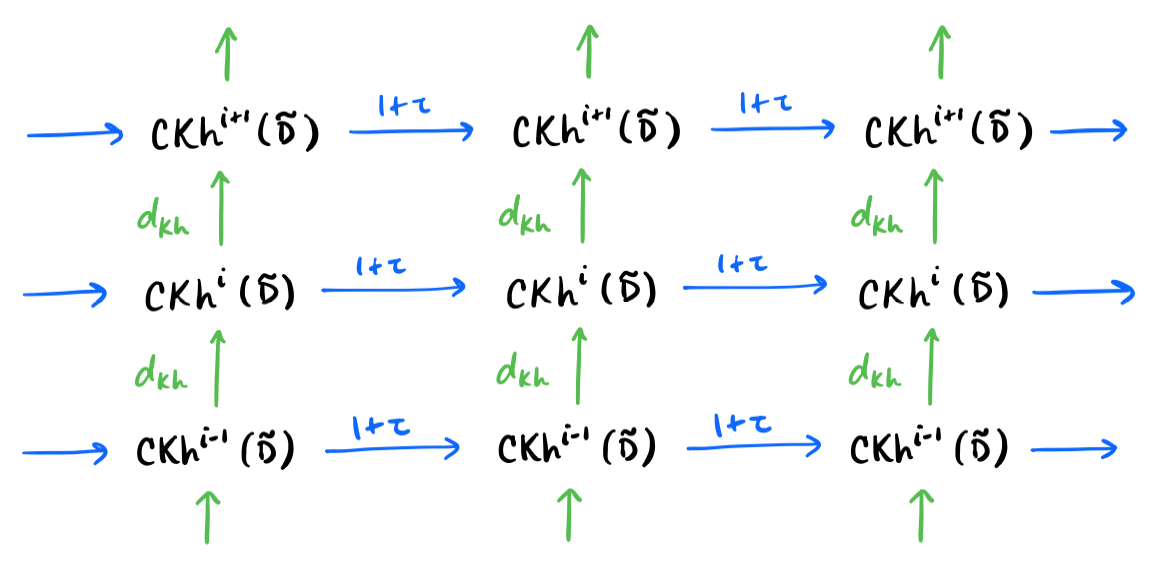}
\end{center}

For concreteness, let's consider a specific example. 
Let $\tilde D$ denote the following 2-periodic diagram for the Hopf link, along with its cube of resolutions:
\begin{center}
    \includegraphics[width=\linewidth]{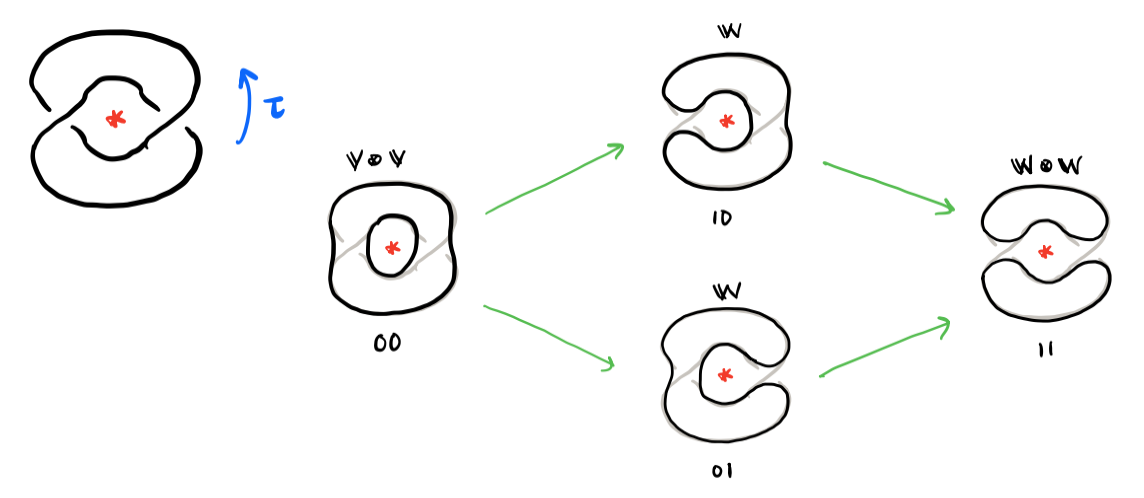}
\end{center}
The red asterisk {\color{red}*} marks the location of the unknotted axis of symmetry. 

Also, we have used annular Khovanov notation ($\VV$ and $\WW$) to keep track of the relationship between each planar circle and the axis of rotation. 
For $\VV$ circles, the quotient by $\tau$ is still a $\VV$ circle. 
For $\WW$ circles, each pair (orbit under $\tau$) maps to a single $\WW$ circle in the quotient.
\note{Draw this for yourself to convince yourself.}

The quotient diagram $D$ is a diagram for the unknot, with one Reidmeister move across the axis of symmetry.

Here is how $\tau$ acts on the distinguished generators:
\begin{center}
    \includegraphics[width=4in]{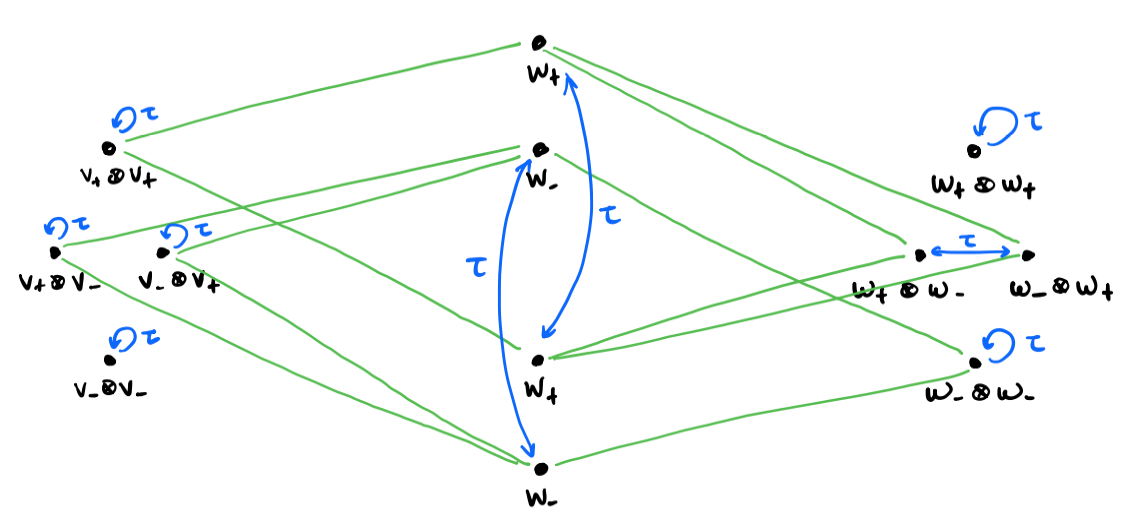}
\end{center}

Looking back at the Khovanov-Tate complex, consider the two spectral sequences, $^{vh}E^\bullet$ and $^{hv}E^\bullet$.

First of all, note that $^{vh}E^\bullet \cong \Kh(\tilde L) \otimes \F_2[\theta, \theta\inv]$. 

On the other hand, observe that 
$^{vh}E^1 \cong \CKh(D) \otimes \F_2[\theta, \theta\inv]$. 
There are no fixed resolutions in odd cube gradings, so $d^1 = 0$.
Therefore $^{vh}E^2 \cong \CKh(D) \otimes \F_2[\theta, \theta\inv]$ as well.

The length-2 differentials are computed as zigzags of the form 
(1) map by $d_\Kh$, (2) lift by $1+\tau$, (3) map by $d_\Kh$. These are shown below:
\begin{center}
    \includegraphics[width=4in]{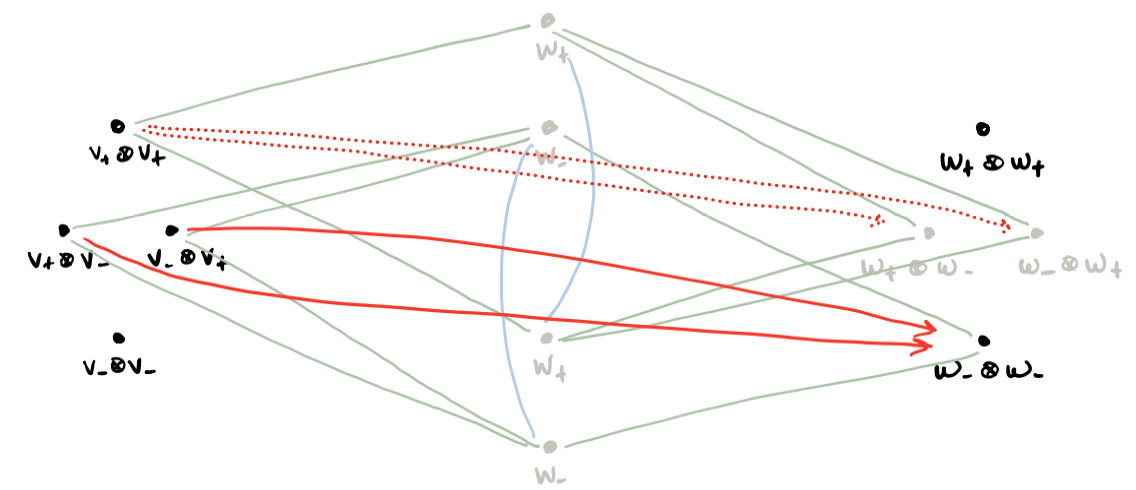}
\end{center}
The dotted maps are 0 because the targets did not survive to page 2. 

We now compare this with the $\CKh(D)$ complex:
\begin{center}
    \includegraphics[width=3in]{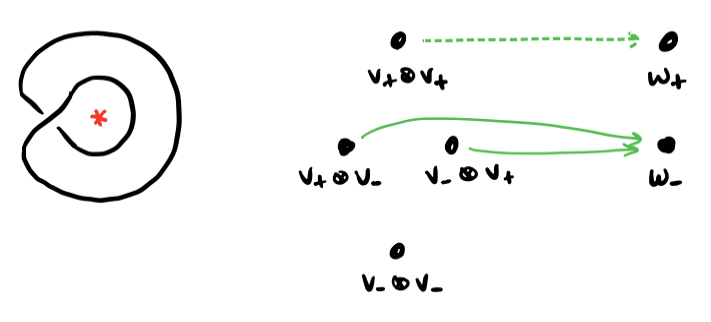}
\end{center}
Observe that there is a discrepancy! The dashed component does not correspond to an $^{hv}d^2$ differential, whereas the other two do. 

So what we \emph{actually} have is the \emph{annular} Khovanov differential, and you can check this in the following exercise.

\begin{exercise}
\bea
\item For each of the six possible Khovanov differentials, in the presence of the axis of symmetry ($\VV \otimes \VV \to \WW$, $\VV \otimes \WW \to \VV$, etc.), compute the induced $^{hv}d^2$ differentials, and verify that they correspond to the annular Khovanov differential, under the identification of generators $^{hv}E^2 \cong \CKh(D) \otimes \F_2[\theta, \theta\inv]$ discussed above.
\item Verify the same is true if you build the \emph{annular Khovanov-Tate} complex, i.e.\ use $d_{\AKh}$ as the vertical differentials in the bicomplex. 
\ee
\end{exercise}

Using grading arguments, it's possible to show that for the annular Khovanov-Tate complex, the $^{hv}E^\bullet$ spectral sequence collapses at page 3, and obtain the following theorem:

\begin{theorem}[\cite{Zhang-akh}, Theorem 1]
There is a spectral sequence 
\[
    \AKh(\tilde L;\F_2) \otimes \F_2[\theta, \theta\inv] 
    \abuts
    \AKh(L;\F_2) \otimes \F_2[\theta, \theta\inv]
\]
which implies the Smith-type inequality
\[
    \dim \AKh(\tilde L;\F_2) \geq \dim \AKh(L;\F_2).
\]
\end{theorem}
\note{The actual theorem is more refined, and includes some grading information.}

However, the same techique does not work for the Khovanov-Tate complex. In order to prove the corresponding theorem for $\Kh(\tilde L)$, and also to generalize this to all prime periodicities, we needed to use the Lipshitz--Sarkar Khovanov stable homotopy type, which we will discuss later in the course.

\subsection{Khovanov-to-Floer spectral sequences}

While this is not a course about Floer homologies, the relationship between categorified quantum invariants and Floer-type invariants should not be understated, for at least two reasons. 
First of all, the group of people studying these two types of invariants often overlap; these tools are often used in conjunction in low-dimensional topology, for example. Second, both ultimately have a homological flavor, and algebraic tricks can often be emulated on either side.

\begin{example}
    Recall that earlier in the course we mentioned the following wide-open problem:
    \begin{quote}
        \textit{Does the Jones polynomial detect the unknot?}
    \end{quote}
    Kronheimer-Mrowka proved that \emph{Khovanov homology} in fact \emph{does} detect the unknot \cite{KM-unknot-detector}. 
    The proof has two steps:
    \begin{enumerate}
        \item Find a spectral sequence from Khovanov homology to instanton Floer homology
        \item Prove that instanton Floer homology detects the unknot.
    \end{enumerate}
    \note{Prove to yourself that this completes the proof!}

    \mz{The paper is of course very involved and very long. The two steps are very hard to prove!}
\end{example}

Ozsv\'ath--Szab\'o's \emph{Heegaard Floer homology} \cite{OS-HF-3mfld}, and invariant of 3-manifolds, is most closely intertwined with Khovanov homology; both were discovered around the same time and were studied and expanded upon by the same cohort of grad students! (These grad students later became PhD advisors to my generation!)

In this section, we will focus on the first concrete connection between the two theories, a spectral sequence discovered by Ozsv\'ath--Szab\'o
\begin{equation}
\label{eq:OS-spectral-sequence}
    E^2 \cong \Khred(\bar L) \abuts E^\infty \cong \HFhat(\Sigma_2(S^3,L))
\end{equation}
relating the \emph{reduced Khovanov homology} of (the mirror of) a link $L$ to the Heegaard Floer homology of the branched double cover of $S^3$ along $L$ \cite{OS-dbc}.

We will explain all the components of this statement at various levels of detail.

We first give a very, very, almost criminally elementary introduction to Heegaard Floer homology. See \cite{OS-HF-intro}, \cite{Lipshitz-HF-intro}, and \cite{Hom-HF-intro} for good introductory readings on Heegaard Floer homology.

\subsubsection{Heegaard diagrams for closed, connected, orientable 3-manifolds}

We will work with the most basic flavor of Heegaard Floer homology and only concern ourselves with the simplest case of 3-manifolds.

\begin{notation}
    This is how I draw a 2-sphere (left) vs a 3-ball (right):
    \begin{center}
        \includegraphics[height=1in]{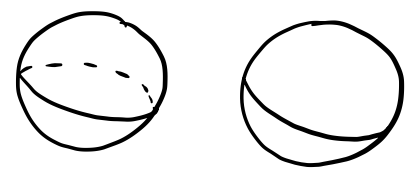}
    \end{center}

\end{notation}

Let $Y$ be a closed, connected, orientable 3-manifold. 
Let $f: Y \to \R$ be a self-indexing Morse function on $Y$;
that is, the index $i$ critical points are all in the critical level set $f\inv(i)$.

Observe the following:
\begin{itemize}
\item We can arrange so that there are unique index-0 and index-3 critical points. 
\item By Poincar\'e duality, 
    \[ \# \text{ index-1 critical points} =\# \text{ index-1 critical points}.\]
\item If there are $g$ index-1 critical points, (and $g$ index-2 critical points), then the level set $f\inv(3/2)$ is a closed, connected, orientable surface with genus $g$.
\end{itemize}

In this context, the surface $f\inv(3/2)$ splits $Y$ into two handlebodies. This decomposition is called a \emph{Heegaard splitting} for $Y$, and the surface $f\inv(3/2)$ is called the \emph{Heegaard surface} for this splitting.

\begin{center}
    \includegraphics[width=4in]{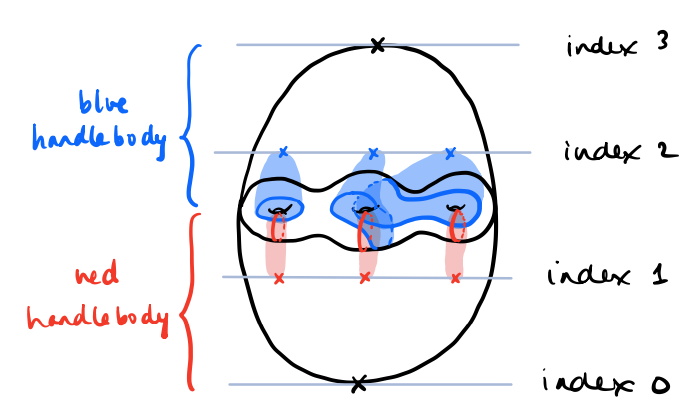}
\end{center}

A \emph{Heegaard diagram $\HD$} for a $Y$ is a combinatorial model for  a Heegaard splitting, and consists of the following data:
\begin{itemize}
    \item a genus $g$ surface $\Sigma_g$;
    \item a set of $g$ {\color{red} red} \emph{$\alpha$-curves} that each bound disks in the $f\inv([0,3/2])$ handlebody, and such that after cutting along these disks, the result is a 3-ball;
    \item a set of $g$ {\color{blue} blue} \emph{$\beta$-curves} that each bound disks in the $f\inv([3/2,3])$ handlebody, and such that after cutting along these disks, the result is a 3-ball.
\end{itemize}

Our running example will be this genus-2 Heegaard diagram for $S^3$:
\begin{center}
    \includegraphics[width=2.5in]{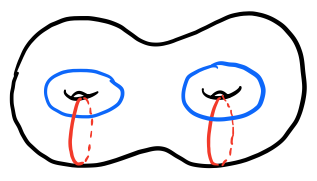}
\end{center}

Given a Heegaard diagram $\HD$ for $Y$, we know how to build $Y$:
\begin{enumerate}
    \item Start with a thickening of $\Sigma_g$.
    \item On inside, glue on thickened red disks, along the $\alpha$ curves.
    \item On the outside, glue on thickened blue disks, along the $\beta$ curves.
    \item Fill in the remaining two ball-shaped holes on each side with two 3-balls. 
\end{enumerate}

Starting with a thickened version of $\Sigma_2$, we can build the red handlebody by gluing in the disks on the inside with the red curves as boundary:
\begin{center}
    \includegraphics[width=4in]{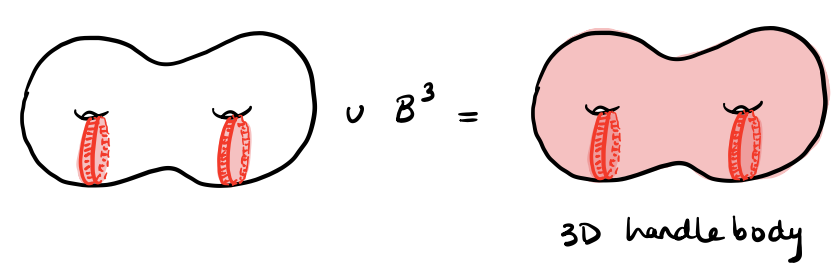}
\end{center}
Then, we can build the remainder of $S^3$ by attaching the blue disks on the outside, whose boundaries are the blue curves to get a ball, and then gluing to this $B^3$ another $B^3$:
\begin{center}
    \includegraphics[width=4in]{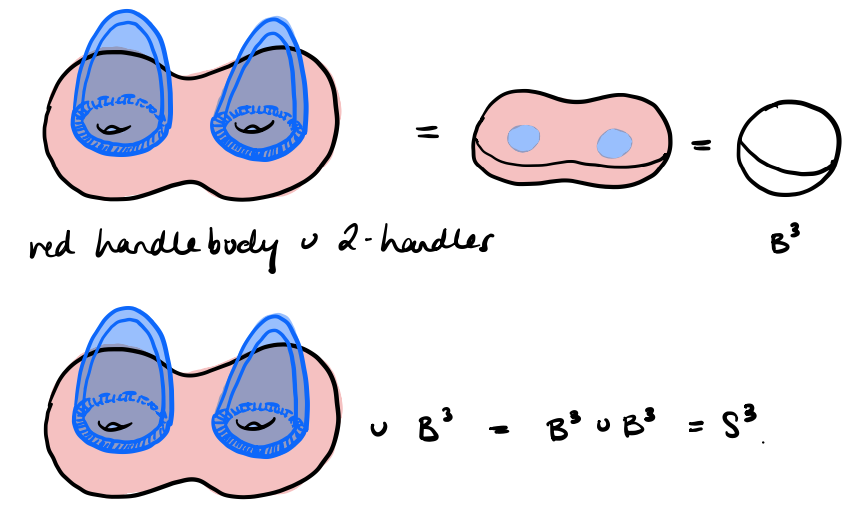}
\end{center}
We can relate this construction (which started with the thickened $\Sigma_2$ with the handle decomposition given by our Morse function as follows:
\begin{center}
    \includegraphics[width=4in]{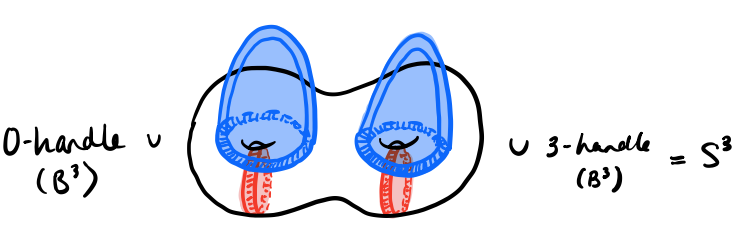}.
\end{center}

\begin{remark}
    This corresponds to a handle decomposition of $Y$ with 
    \begin{itemize}
        \item one 0-handle: the ball on the inside
        \item $g$ 1-handles: thickened version of the part of the ascending manifold for each index-1 critical point, in $f\inv([1,\frac{3}{2}])$
        \item $g$ 2-handles: thickened version of the part of the descending manifold for each index-2 critical point, in $f\inv([\frac{3}{2},2])$
        \item one 3-handle: the ball on the outside.
    \end{itemize}
\end{remark}

\begin{example}
\label{eg:s1xs2-HD}
    Let's now consider $S^1 \times S^2$. 
    This manifold is basically obtained by gluing two solid tori $S^1 \times D^2$ together so that, for each $\theta \in S^3$,  the disks $\{\theta\} \times D^2$ on the two solid tori are glued together along their boundary to form a $S^2$.

    Here is a genus-1 Heegaard splitting for $S^1 \times S^2$: 

    \begin{center}
        \includegraphics[width=3in]{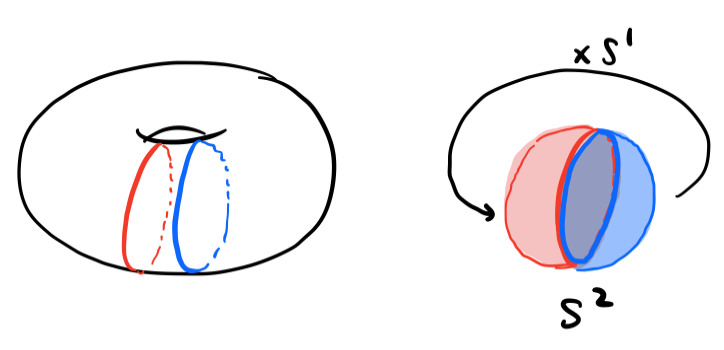}
    \end{center}

    The red and blue curves bound disks on the inside and outside. If you slide the blue disk so that its boundary is the same as the red curve, then you'll clearly see an $S^2$ with a red side and a blue side.
\end{example}

\subsubsection{Heegaard Floer homology basics}

Heegaard Floer homology was originally defined by Ozsv\'ath and Szab\'o, in \cite{OS-HF-3mfld}. 
There are many different version now, but today we will discuss only the most basic version, the `hat' flavor $\HFhat$, for closed, connected, orientable 3-manifolds $Y$. 

Here is a summary of the main steps:
\begin{enumerate}
    \item Start with a Heegaard diagram for $Y$. Pick a basepoint $z$ and perturb the $\alpha$ and $\beta$ curves so that the diagram is \emph{admissible} (see Remark \ref{rmk:admissibility-HF}).
    \item Build a symplectic manifold from $\Sigma_g$, as well as two Lagrangian submanifolds $\TT_\alpha$ and $\TT_\beta$ from the data of the $\alpha$- and $\beta$-curves, respectively.
    \item The chain complex is generated as an $\F_2$ vector space by the points $\TT_\alpha \cap \TT_\beta$.
    \item The differential counts moduli spaces of holomorphic disks with boundary on $\TT_\alpha \cup \TT_\beta$, that do not intersect the divisor defined by $z \in \Sigma_g$.
\end{enumerate}

For a genus $g$ surface $\Sigma_g$, $\Sym^g(\Sigma_g)$ is defined as the quotient of the  $g$-fold Cartesian product of copies of $\Sigma_g$ by the obvious action of  the symmetric group $S_g$:
\[
    \Sym^g(\Sigma_g) = \Sigma_g \times \cdots \times \Sigma_g / S_g.
\]
Observe that the points in $\Sym^g(\Sigma_g)$  are basically \textit{unordered} $g$-tuples of points in $\Sigma_g$. 

Now $\Sigma_g$ has a complex structure, so $\Sigma_g^{\times g}$ does as well.
What's surprising is that, while $\Sym^g(\Sigma_g)$ is a priori only an orbifold (the action of $S_g$ is obviously not free), it's in fact smooth, and moreover has an induced complex structure!
This makes $\Sym^g(\Sigma_g)$ into a symplectic manifold as well, with symplectic form induced by the complex structure.

Inside $\Sym^g(\Sigma_g)$ are two totally real tori 
\begin{align*}
    \TT_\alpha &= \alpha_1 \times \cdots \times \alpha_g / \sim \\
    \TT_\beta &= \beta_1 \times \cdots \times \beta_g / \sim
\end{align*}
which are Lagrangian with respect to the symplectic structure.

Heegaard Floer homology is then the Lagrangian intersection Floer homology of $(\Sym^g(\Sigma_g), \TT_\alpha, \TT_\beta)$...

... More concretely, here's how you think about the generators and the differential.

The generators of the chain complex are the points $\TT_\alpha \cap \TT_\beta$. In other words, these are unordered $g$-tuples of intersection points between $\alpha$ and $\beta$ curves.

The coefficient of the differential from one intersection point $x$ to another $y$ is determined by counting the size of the moduli space of holomorphic disks from $x$ to $y$ of the following form:

\begin{center}
    \includegraphics[width=4in]{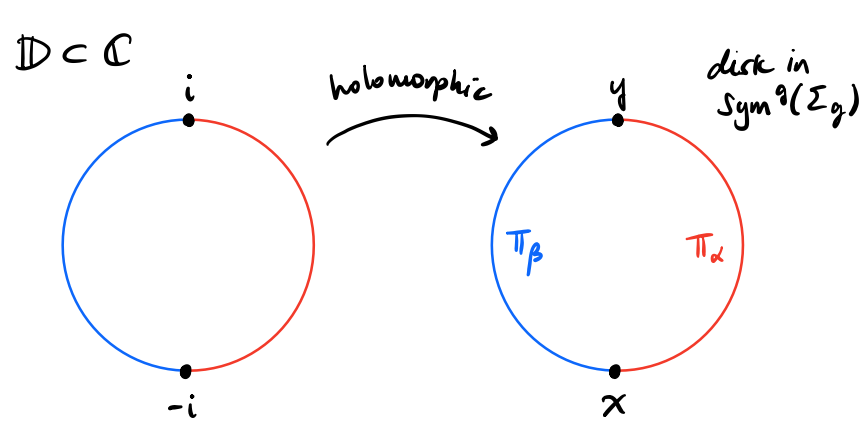}
\end{center}

Officially, here is the definition of the differential:
\[  
    \partial x = \sum_{y \in \TT_\alpha \cap \TT_\beta} 
    \sum_{\phi \in \pi_2(x,y),  \mu(\phi) = 1, n_z(\phi) = 0} \# \overline{\mathcal{M}(\phi)} y.
\]
\begin{itemize}
    \item $\pi_2(x,y)$ is the set of homotopy classes of Whitney disks from $x$ to $y$ as shown in the above drawing.
    \item $\mathcal{M}(\phi)$ is the moduli space of holomorphic representatives of $\phi$; these disks have an $\R$ action by parametrization (complex automorphisms of the unit disk in $\CC$ that preserve $\pm i$)
    \item $\mu(\phi)$ is the Maslov index of $\phi$, the `expected dimension' of the smooth, real manifold $\mathcal{M}(\phi)$; \note{Lipshitz gave more explicit formulas for $\mu(\phi)$ in his cylindrical reformulation for Heegaard Floer homology.}
    \item $\overline{\mathcal{M}(\phi)} = \mathcal{M}(\phi)/\R$ is a finite set of points, assuming that $\mu(\phi) = 1$.
\end{itemize}

\mz{For our purposes, it is necessary for you to really understand only the following examples.}

\begin{example}
\label{eg:s1xs2-CF}
    Our first example is the genus-1 Heegaard diagram for $S^1 \times S^2$ from Example \ref{eg:s1xs2-HD}, but now perturbed so that we actually have intersection points between the $\alpha$ and $\beta$ curves, and with a basepoint slapped on the complement of the $\alpha$ and $\beta$ curves:

    \begin{center}
        \includegraphics[width=2.5in]{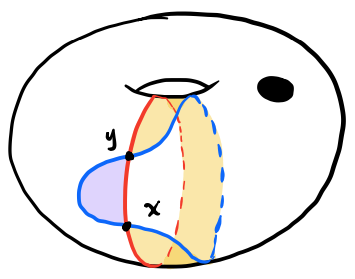}
    \end{center}

    In this $g=1$ example, we are lucky that $\Sym^g(\Sigma_g) = \Sigma_g$, and the drawn $\alpha$ and $\beta$ curves are literally Lagrangian.
    What you see is what you get, so we can explicitly determine the chain complex. 
    
    There are two intersection points, $x$ and $y$, as shown in the Heegaard diagram $\HD$ drawn above.
    Therefore 
    \[
        \CFhat(\HD) = \F_2 x \oplus \F_2 y.
    \]

    We now compute the differential; recall that we are working over $\F_2$:
    \begin{itemize}
        \item $\partial x = 0$: There are two holomorphic disks from $x$ to $y$; they are shaded in purple and yellow. 
        \item $\partial y = 0$: There are no holomorphic disks from $y$ to $x$ (the disks from computing $\partial x$ have the wrong orientation!). 
    \end{itemize}

    Therefore $\HFhat(S^1 \times S^2) \cong \F_2^2$. 

    It turns out that $\HFhat(Y)$ enjoys an action by $H_1(Y; \Z)/tor$. 
    For $S^1 \times S^2$, either generator $\theta$ of $H_1(Y; \Z) \cong \Z$ sends $x \mapsto y$. 
\end{example}

\begin{remark}
\label{rmk:admissibility-HF}
Admissibility of a Heegaard diagram ensures that we can actually build a chain complex, and that the invariant is well-defined. The actual definition of admissibility is somewhat involved, as the requirements really come from studying the spaces of holomorphic disks between generators. 

Our original diagram for $S^1 \times S^2$ is \emph{not} admissible:
\begin{center}
    \includegraphics[width=2.5in]{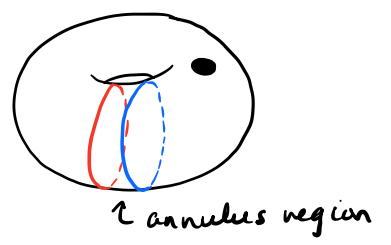}
\end{center}
We needed to perturb the $\beta$ curve so that, if we cut $\Sigma_1$ along both $\alpha$ and $\beta$, we get only disks, aside from the piece containing the basepoint.
\end{remark}

\begin{example}
\label{eg:s1xs2-CF-2copies}
    Let's do another example: $Y = (S^1 \times S^2) \# (S^1 \times S^2)$. 
    Here is a Heegaard diagram:
    \begin{center}
        \includegraphics[width=3in]{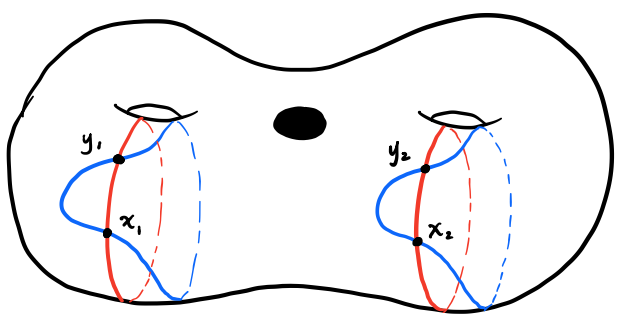}
    \end{center}

    There are four generators: $\{x_1,y_1\} \times \{x_2, y_2\}$. 

    Consider first the space of disks from $(x_1,x_2)$ to $(y_1, x_2)$. 
    There are two, given precisely by the product of the disks from Example \ref{eg:s1xs2-CF} from $x_1$ to $y_1$ and the constant disk at $x_2$. 

    After working through all four differentials (and remembering that we're working over $\F_2$, you should find that $\partial =0$, and each of the four generators represent a distinct homology class. 

    The $H_1$ action is likewise induced by the $H_1$ action on the two `sides' of the Heegaard diagram. 
    Fix generators $\theta_1$ and $\theta_2$ of the $H_1$ of the left and right copies of $S^1 \times S^2$, respectively. Then the $H_1(Y)$ action is given by 
    \begin{center}
    \begin{tikzcd}
        & (x_1, x_2) \arrow{dl}[swap]{\theta_1} \arrow{dr}{\theta_2} & \\
        (y_1, x_2) \arrow{dr}[swap]{\theta_2} & & (x_1, y_2)  \arrow{dl}{\theta_1}\\
        & (y_1, y_2) & 
    \end{tikzcd}
    \end{center}
    This should start to remind you of Khovanov homology...
\end{example}

Heegaard Floer homology behaves very well under connected sum:
    \[
        \HFhat(Y \# Y') \cong \HFhat(Y) \otimes \HFhat(Y').
    \]
If $\HD$ is a Heegaard diagram for $Y$, and $\HD'$ is a Heegaard diagram for $Y'$, then the connected sum of $\HD$ and $\HD'$ in the regions containing the basepoints is a Heegaard diagram for $Y \# Y'$. 
You can now check that, at least on the level of chain groups, it is perhaps believeable that 
\[
    \CFhat(\HD \# \HD') \cong \CFhat(\HD) \otimes \CFhat(\HD').
\]
And finally, because the basepoint in $\HD \# \HD'$ is in the connect-sum region, no holomorphic disks pass `between' the two sides.
More precisely, if $\mathbf{x}$ and $\mathbf{x'}$ are tuples of intersection points for $\HD$ and $\HD'$, respectively, then $(\mathbf{x}, \mathbf{x'})$ (obtained by concatenating the tuples) is a tuple of intersection points for $\HD \# \HD'$. Define another tuple $(\mathbf{y}, \mathbf{y'})$ similarly. 
Then holomorphic disks from $(\mathbf{x}, \mathbf{x'})$  to $(\mathbf{y}, \mathbf{y'})$ are in bijection with 
\[
    \{ \text{hol. disks } \phi: \mathbf{x} \to \mathbf{y} \}
    \times 
    \{ \text{hol. disks } \phi': \mathbf{x'} \to \mathbf{y'} \}
\]
by taking the diagonal disk in $\phi \times \phi'$.

Moreover, the action of $H_1(Y \# Y'; \Z)/tor \cong H_1(Y; \Z)/tor \otimes H_1(Y';\Z)/tor$ is induced by the actions of 
$H_1(Y; \Z)/tor$ and $H_1(Y';\Z)/tor$ on the tensor factors of $\HFhat(Y) \otimes \HFhat(Y')$, respectively. 

\begin{example}
We now know how to compute $\HFhat(\#^k S^1 \times S_2)$. And we also understand how $H_1$ acts.
\end{example}

\subsubsection{Branched double covers}

In Example \ref{eg:s1xs2-CF-2copies}, we saw an inkling of how Khovanov homology (for unlinks, at least) might be related to $\HFhat$ (for connected sums of $S^1 \times S^2$, at least). 

We now explain the topological relationship between planar circles and connected sums of $S^1 \times S^2$, by discussing branched double covers. 

Let $M$ be a manifold, and let $L \subset M$ be a submanifold. The $k$-fold branched cover of $M$ along branching locus $L$ is a manifold $\Sigma_k(M)$ with a (continuous) map $\pi: \Sigma_k(M) \to M$ such that
\begin{itemize}
    \item for each point $p \in M \backslash L$, $\pi\inv(p)$ consists of $k$ copies of $p$ (a covering map), and 
    \item for each point $p \in L$, $\pi\inv(p)$ is just one copy of $p$.
\end{itemize}

A prototypical example is the branched covering map $\pi: \C \to \C$ given by $z \mapsto z^2$.
In the following example, we describe a homeomorphic version of this branched covering, and how to construct the 2-fold branched cover, also known as the \emph{branched double cover}.

The typical notation for the branched double cover of $M$ along $L$ is $\Sigma_2(M, L)$. However, to avoid notation collision with a genus-2 Heegaard surface, we will instead just write $\Sigma$ to denote `branched double cover'.

\begin{example}
    Let $D$ be a disk, and let $p$ be a point in the interior of the disk:
    \begin{center}
        \includegraphics[width=2in]{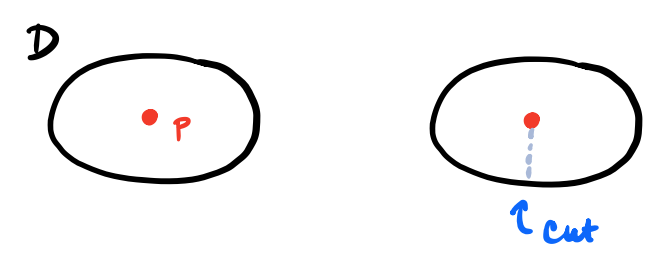}
    \end{center}
    We wish to build $\Sigma(D, p)$.
    
    We first make a branch cut (right side of above figure), and then take two copies of our snipped disk:
    \begin{center}
        \includegraphics[width=3in]{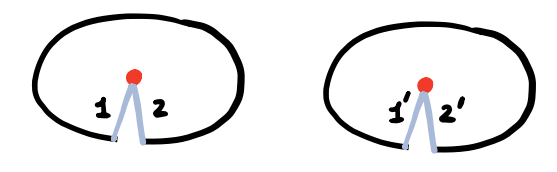}
    \end{center}
    We now  glue $1 \to 2'$ and $2 \to 1'$ so that the resulting surface $F$ has two lifts of each point other than $p$ (including the points along the branch cut -- notice that these were doubled already when we made the cut, so it's ok that we then glued them to their twins on the second copy). There is, however, only one lift of the branch point $p$. 

    We can visualize gluing the two shown disks above to get one very floppy disk.

    We can also think of $F$ as the ramp in a parking garage where you can climb from the 1st to the 2nd floor, but then when you keep climbing, you end up back at the 1st floor.

\end{example}

\begin{example}
    Now consider a disk with two branch points, red and blue:
    \begin{center}
        \includegraphics[width=3in]{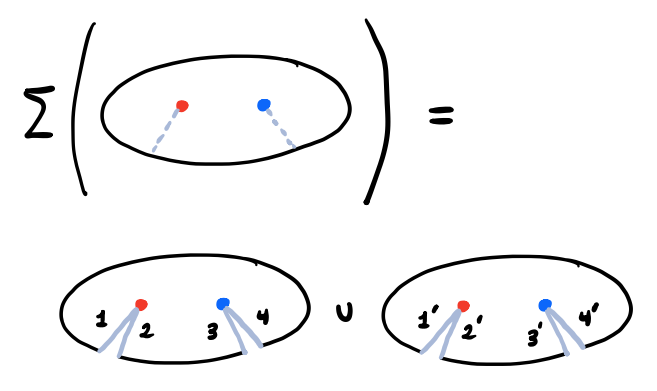}
    \end{center}
    We again make some branch cuts, and make two copies of our snipped disk.  
    (These cuts are just my choice, you could also decide to make a cut that connects the red and blue. However, I find the drawn visualization easier to explain.)
    
    When we glue the two pieces together by $1 \to 2'$, $2 \to 1'$; $3 \to 4'$, $4 \to 3'$, we get an annulus:
     \begin{center}
        \includegraphics[width=4in]{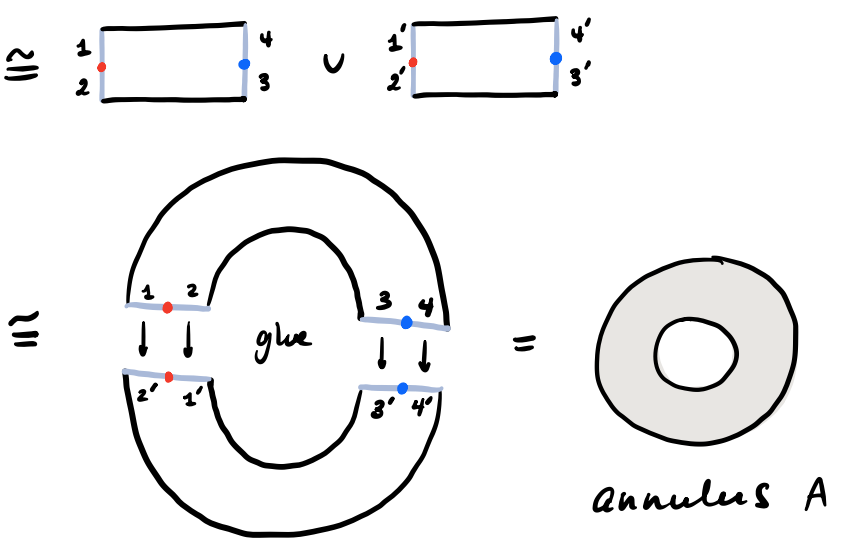}
    \end{center}

    So $\Sigma(D, \text{2 interior points}) \cong A,$ an annulus. 
\end{example}

\begin{example}
    Now consider the branched double cover of two arcs inside a 3-ball. This is homeomorphic to the thickening of the previous example by an interval $I$:
     \begin{center}
        \includegraphics[width=2in]{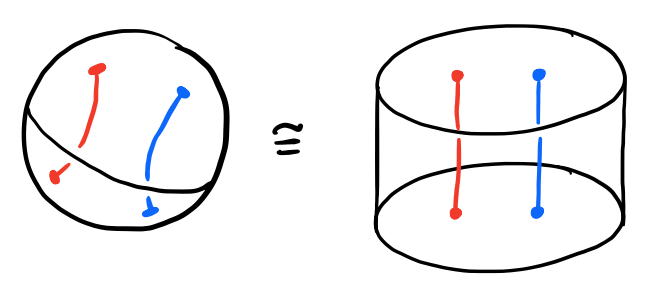}
    \end{center}
    We make the same branch cuts, but now thickened by $I$ as well:
    \begin{center}
        \includegraphics[width=2in]{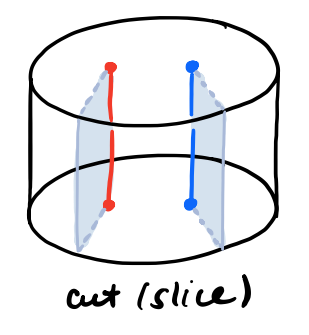}
    \end{center}
    Following the previous example, we make two copies and glue them together. This shows that the branched double over of a sphere with branch locus given by two properly embedded arcs is a solid torus:
    \begin{center}
        \includegraphics[width=4in]{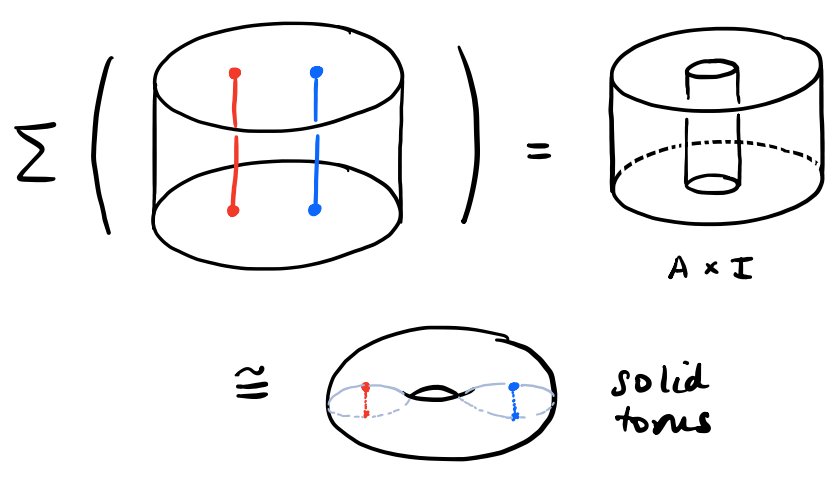}
    \end{center}
\end{example}

\begin{example}
    Now we are ready to find the branched double cover of a 2-component unlink inside $S^3$:
    \begin{center}
        \includegraphics[width=2in]{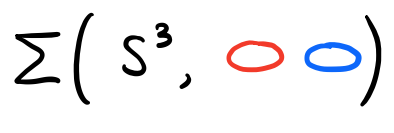}
    \end{center}
    The trick is to first cut $S^3$ into two balls as shown:
    \begin{center}
        \includegraphics[width=2in]{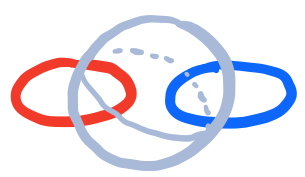}
    \end{center}
    Using the previous example, we have $\Sigma(S^3, \fullmoon \sqcup \fullmoon) \cong$
    \begin{center}
        \includegraphics[width=4in]{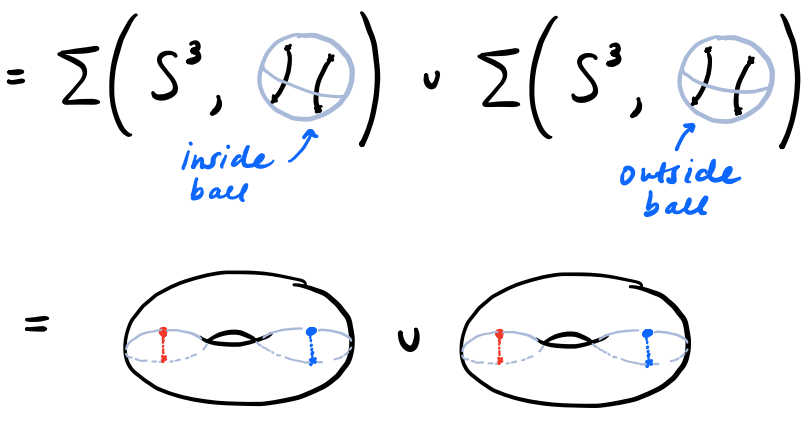}
    \end{center}
    Noting that the red endpoints need to glue to the red endpoints, and similarly for the blue endpoints, we reason that the result is 
    \[
        \Sigma(S^3, \fullmoon \sqcup \fullmoon)
        \cong 
        S^1 \times S^2.
    \]
\end{example}

\subsubsection{Reduced Khovanov homology}

In Example \ref{eg:s1xs2-CF-2copies}, I hinted at the fact that 
$\HFhat((S^1 \times S^2) \# (S^1 \times S^2))$ 
looks similar to $\Kh(\fullmoon \sqcup \fullmoon) = V \otimes V$. 
In the previous section, we hinted that the relationship between $\#^k S^1 \times S^2$ and unlinks.

But notice that there is a discrepancy: 
\[
    \Kh(\fullmoon \sqcup \fullmoon) \cong \HFhat((S^1 \times S^2) \# (S^1 \times S^2))
\]
but
\[
    \Sigma(\fullmoon \sqcup \fullmoon) \cong S^1 \times S^2.
\]
So, $\Kh(\fullmoon \sqcup \fullmoon) \not\cong \HFhat(\Sigma(S^3, \fullmoon \sqcup \fullmoon))$! 

This is one motivation for defining and studying \emph{reduced Khovanov homology}. 
This is a version of Khovanov homology defined for \emph{based} links, i.e.\ links with a chosen basepoint:
\begin{center}
    \includegraphics[width=1.5in]{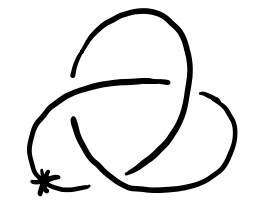}
\end{center}

\begin{definition}
    Let $(D,*)$ be a based diagram for based link $(L,p)$. 

    The \emph{reduced Khovanov complex} for $(D,*)$ is 
    \[
        \CKhred(D,*) := \left ( \CKh(D) / \im (\xi) \right )\{-1\}
    \]
    where $\xi$ is the chain map that merges a small circle labeled $X$ into the arc containing the basepoint:
    \begin{center}
        \includegraphics[width=1.5in]{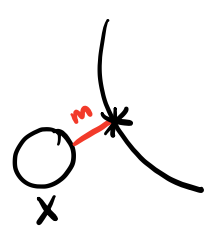}
    \end{center}
\end{definition}

\note{
In other words, the $\CKhred(D,*)$ complex is the quotient complex you get from $\CKh(D)$ by erasing all the generators that label the based component by $X$; 
the label on the component in a complete resolution $D_u$ containing the based point \emph{must} be $1$. 
}

\begin{example}
    Here is part of the chain complex for the based trefoil:
    \begin{center}
        \includegraphics[width=4in]{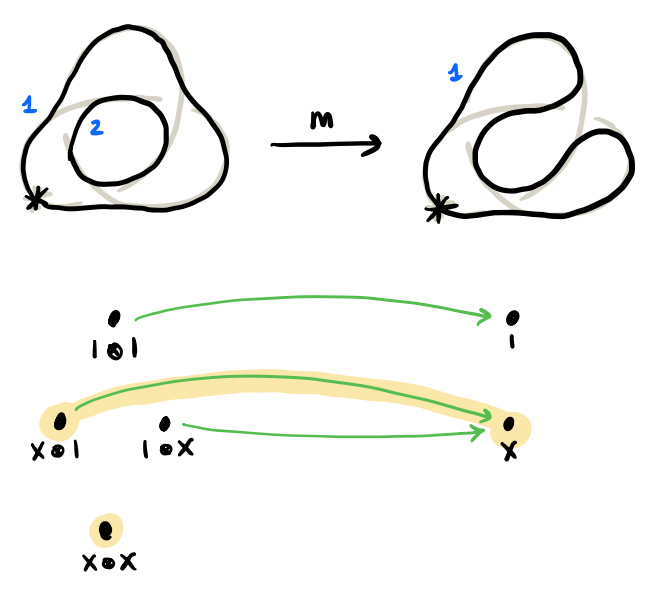}
    \end{center}
    The highlighted region is part of the subcomplex $\im(\xi)$. 
\end{example}

\begin{exercise}

\bea
\item 
Determine the reduced Khovanov homology of the unlink with $k$ components.
\item 
    Compute the reduced Khovanov homology of the trefoil. 
    \note{ Feel free to work over $\F_2$. Remember to incorporate the quantum shift in the definition!}

\ee
\end{exercise}

\begin{fact}
    For knots, the isomorphism class of $\Khred$ does not depend on the location of the basepoint. 
\end{fact}

\subsubsection{Dehn surgery}

\emph{Dehn surgery} is a method for building 3-manifolds. 
Here's the general idea:
\begin{enumerate}
    \item Pick a knot $K \subset S^3$ and a framing of the knot, i.e.\ a homologically nontrivial circle $\gamma$ on the boundary of a neighborhood of $K$ (a longitude for a tubular neighborhood of $K$).
    \item Drill out the solid donut that is a neighborhood of $K$. Now, glue it back in but so that the $\gamma$ bounds a disk for the inside solid torus.
\end{enumerate}

Lickorish \cite{Lickorish-surgery} and Wallace \cite{Wallace-surgery} proved that any closed, connected, orientable 3-manifold can be obtained in this way. 

There are varying conventions in the literature, and we will stay agnostic about these conventions. Everyone agrees on the following two examples, though:

\begin{center}
\includegraphics[width=4in]{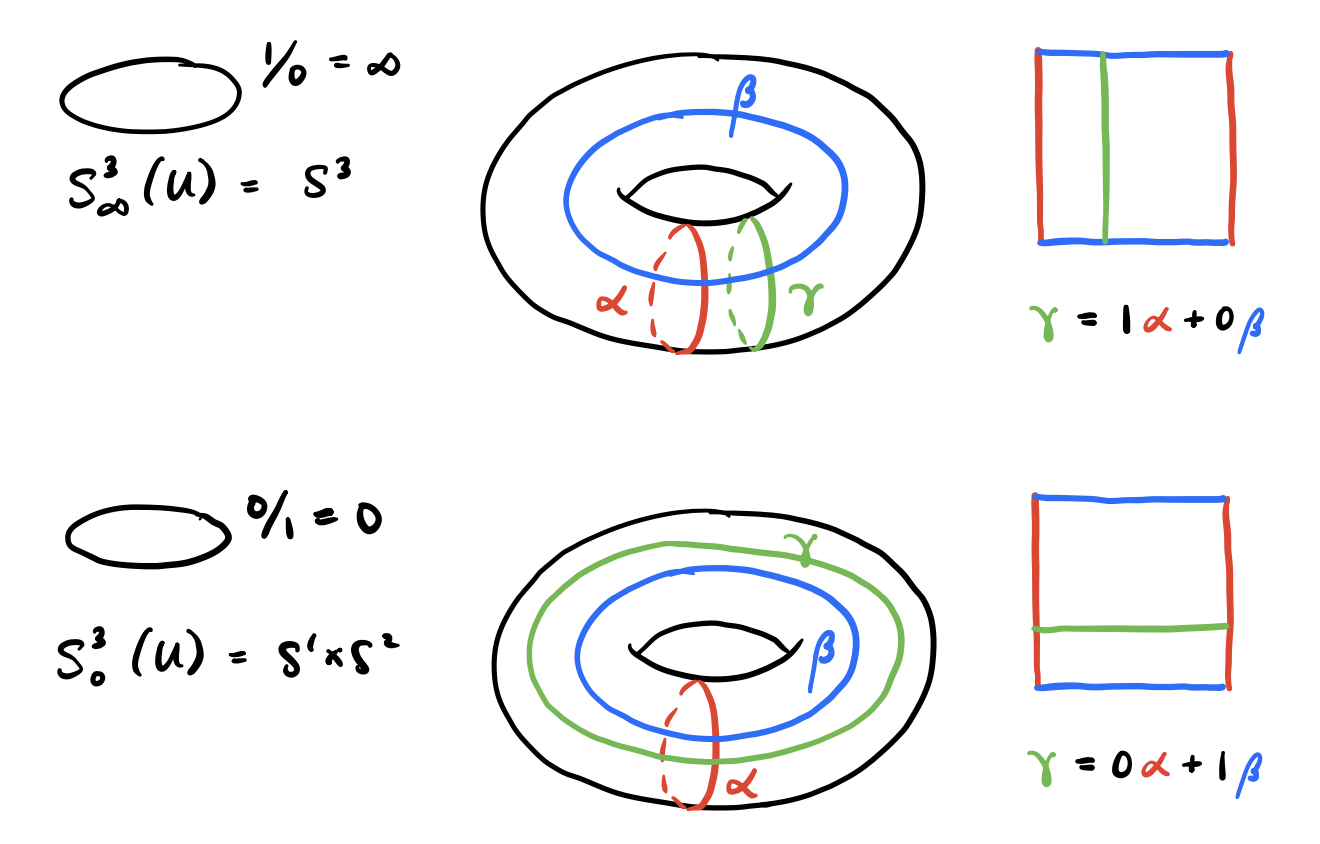}
\end{center}

In the Heegaard diagrams above, $(\Sigma_1, \alpha, \beta)$ is a Heegaard diagram for $S^3$. 
If we replace $\alpha$ with $\gamma$, then 
$(\Sigma_1, \gamma, \beta)$ is the Heegaard diagram for the manifold we get after surgery. 

The $\infty$-surgery of $S^3$ along the unknot $U$ is just $S^3$, since $\gamma$ is parallel to $\alpha$. 

The $0$-surgery instead yields $S^1 \times S^2$, because now, on the inside of the torus, $\gamma$ is supposed to bound a disk.

\subsubsection{The Ozsv\'ath--Szab\'o spectral sequence}

We are now ready to describe, in very broad strokes, the Ozsv\'ath--Szab\'o spectral sequence \eqref{eq:OS-spectral-sequence}.

Recall that we have already arranged so that for the $k$-component unlink $U^k$, 
\[
    \Khred(U^k) \cong \HFhat(\Sigma(S^3, U^k)).
\]
over $\F_2$ coefficients.

In order to boost this relationship to links other than the unlink, we use the same cube-of-resolutions idea from Khovanov homology: we essentially use an iterated mapping cone construction. 
The result will be a spectral sequence rather than an immediate identification.

If we see a neighborhood of a crossing in $L$, we can associate the $0$ and $1$ resolutions as shown below, with conventions reversed from Khovanov homology (hence the mirroring in  \eqref{eq:OS-spectral-sequence}).

In the drawing below, imagine that $L$ has the red crossing. 
Let $L_0$ be the link with the red crossing replaced with the blue arcs, and let $L_1$ be the link with the red crossing replaced with the green arcs.

\begin{center}
\includegraphics[width=2in]{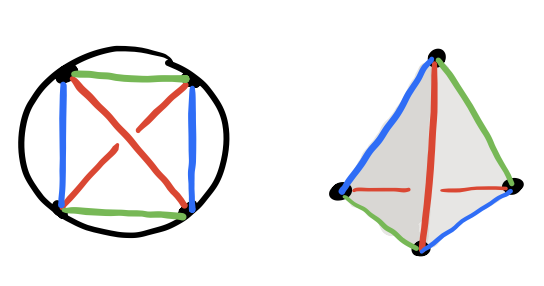}
\end{center}

These have a cyclic relationship, because they can be viewed as the three pairs of opposite edges of a tetrahedron.

We want to take the branched double cover of the ball along the red arcs first. To help with visualization, we will first rotate the inscribed tetrahedron (by rotating the ball) so that we see the diagram on the left:

\begin{center}
\includegraphics[width=3in]{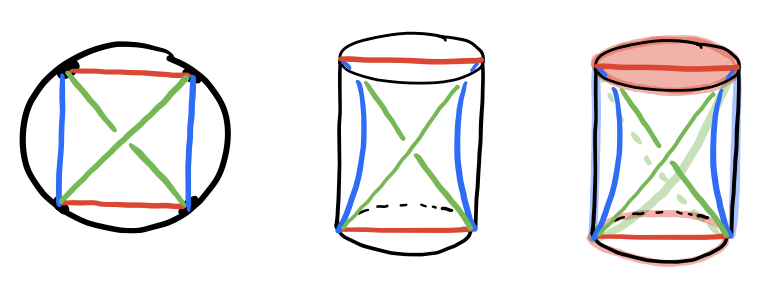}
\end{center}

We can then make a branch cut and isotop to get the solid cylinder in the middle, and note that once we build the branched double cover, the red arc will be contained in the red disk shown on the right. 

In other words, if we branch along the red arcs, the parallel arcs on the boundary of the ball will, together with their double, bound a disk in the resulting solid torus. 

A similar situation would occur if we took the branched double cover along blue or green, of course. 
On the right, we've faintly highlighted the boundary-parallel copies of the blue and green curves. 

After taking the branched double cover of the ball along the red arcs, we have the following solid torus, where we now only draw the boundary-parallel copies of the original red, blue, and green arcs (actually, just one copy):
\begin{center}
\includegraphics[width=3in]{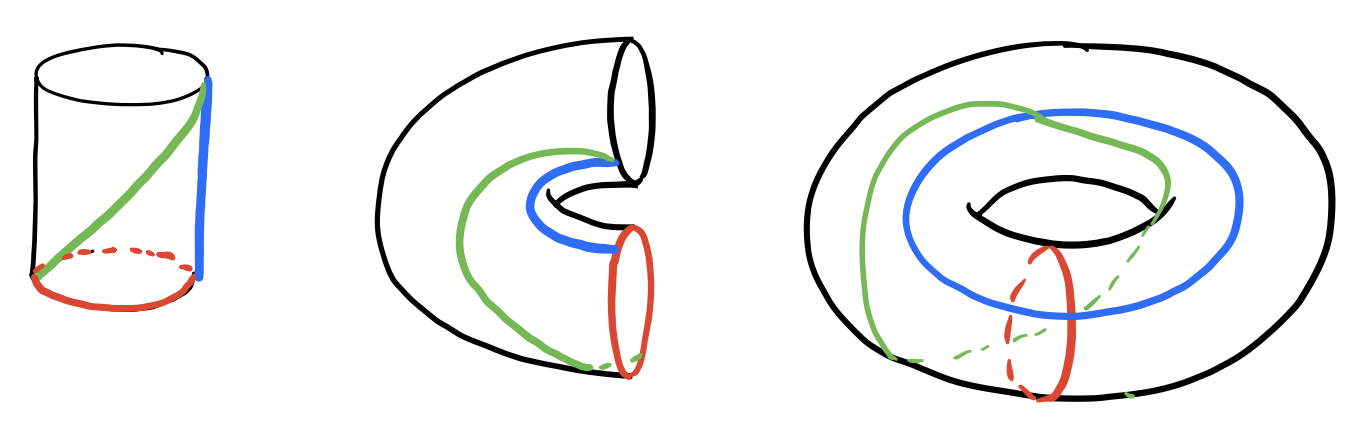}
\end{center}
\note{We only need to keep track of one copy of, say, the red circle, because the other copy will be parallel.}

Let's study the Heegaard diagram on the right more carefully. In fact, let's first isotop the green curve so that its intersections with red and blue are clearer:

\begin{center}
\includegraphics[width=2in]{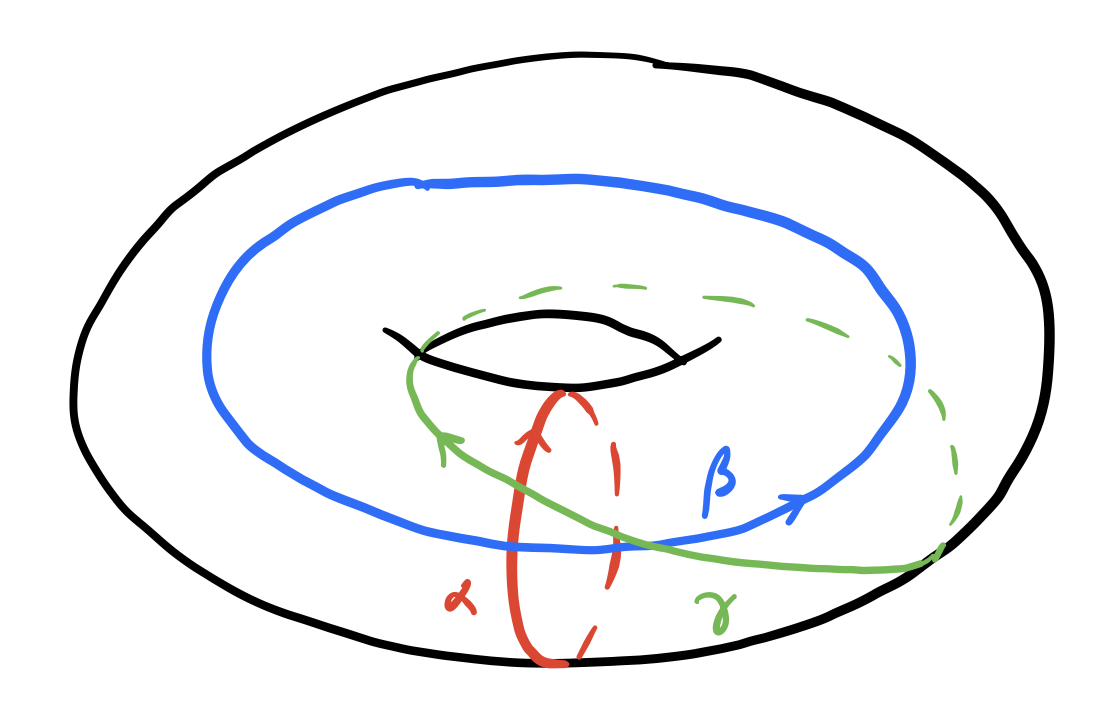}
\end{center}

In the diagram above, we've labeled the red, blue, and green curve as $\alpha$, $\beta$, and $\gamma$, and given them orientations so that their pairwise algebraic intersection numbers satisfy the following equation:
\begin{equation}
    \label{eq:triad-intersections}
    I(\alpha, \beta) 
    = I(\beta, \gamma)
    = I(\gamma, \alpha)
    = -1.
\end{equation}

A consequence of this cyclic relationship is that the three 3-manifolds $(\Sigma(S^3, L) = \Sigma(S^3, L_\infty), \Sigma(S^3, L_0), \Sigma(S^3, L_1))$
form a \emph{triad}; 
Ozsv\'ath--Szab\'o show that there is then a \emph{surgery exact triangle} (i.e.\ long exact sequence) relating their Heegaard Floer homologies:
\begin{center}
\begin{tikzcd}
    \HFhat(\Sigma(S^3, L_0)) \arrow{rr}{} 
      &  & \HFhat(\Sigma(S^3, L_1)) \arrow{dl}{} \\
    & \HFhat(\Sigma(S^3, L_\infty)) \arrow{ul}{}
\end{tikzcd}
\end{center}

This reveals how we can use the mapping cone construction: to compute the bottom object, we can `replace' it with the cone of the top row. 
\note{This is mostly analogous to how the Khovanov bracket works; see Remark \ref{rmk:HF-cone}.}

\begin{remark}
\label{rmk:HF-cone}
    Ozsv\'ath--Szab\'o show that for a triad $(Y, Y_0, Y_1)$, 
    \[
        \CFhat(Y) \simeq \cone \left ( \CFhat(Y_0) \map{\hat f} \CFhat(Y_1) \right)
    \]
    where $\hat f$ is the appropriate surgery map, and $\simeq$ means `quasi-isomorphic', which is the same as `chain homotopic' over a field. 
    Unlike in Khovanov homology, there will likely be a nontrivial homotopy that realizes this homotopy equivalence. 
\end{remark}

Ozsv\'ath--Szab\'o's key proposition tells us that these surgery maps actually correspond to the (reduced) Khovanov merge and split maps!

So, on $E^0$ of spectral sequence, we see a cube with $\CFhat$ of various connected sums of $S^1 \times S^2$ at the vertices. The differential $d^0$ computes the Heegaard Floer homology  for these very simple 3-manifolds; we can identify $E^1$ with the $\Khred$ cube, with unlinks at each vertex. 
One the next page, $d^1$ corresponds to the $\Khred$ differential, by the key proposition alluded to in the previous paragraph. 
Therefore $E^2 \cong \Khred(\bar L)$. 

We can view the remainder of the spectral sequence as a consequence of Remark \ref{rmk:HF-cone}. Due to the nontrivial homotopy equivalences, we have to add homotopies of all cube-degree lengths to the components of $d^1$ in order to truly recover the Heegaard Floer chain complex:
\begin{center}
    \includegraphics[width=4in]{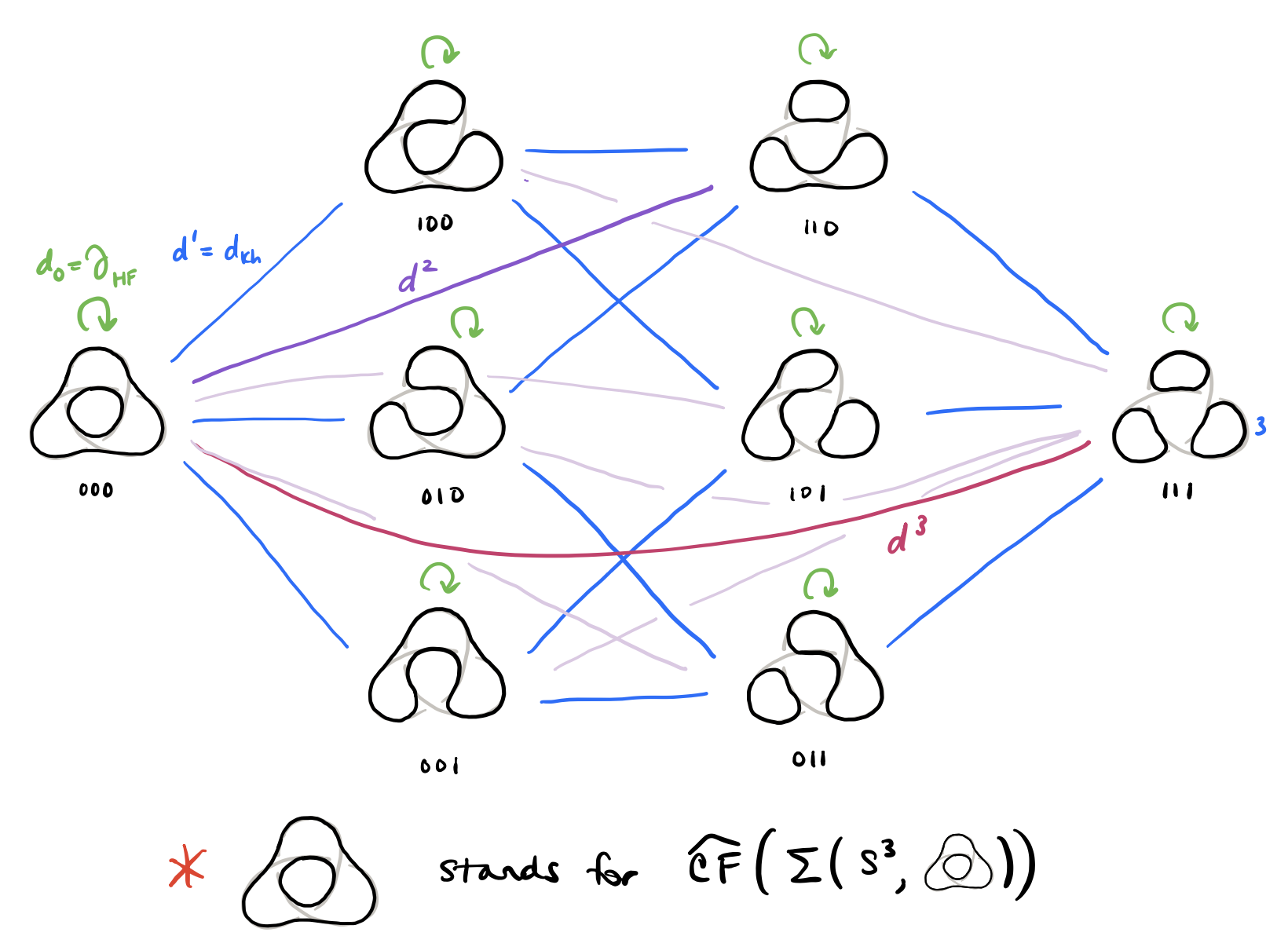}
\end{center}

So, the spectral sequence is really a filtration spectral sequence with respect to this cube degree.

\section{Khovanov stable homotopy type}

\note{In this course, we first carefully went through a definition of Khovanov homology and how one works with it. Then, we went on a tour of various applications of Khovanov homology. We now slow down (sort of) once again, to study a topic relatively carefully. 
But don't be fooled: our four-lecture journey through the construction of a Khovanov stable homotopy type is meant only as a primer, to help you read the papers in case you choose to dig deeper.}

We began this course by defining the Jones polynomial, and introduced Khovanov homology as a higher-level version of the invariant: while the Jones polynomial can only describe the knot embedded in $S^3$, the Khovanov invariant, as a TQFT, can also be used to describe cobordisms between knots. 

In this section, we will see how Lipshitz--Sarkar further boost the Khovanov homology invariant to a `space-level' invariant $\XKh$. Roughly speaking, for a link $L \subset S^3$, $\XKh(L)$ is a sufficiently high-dimensional space whose singular cohomology agrees with $\Kh(L)$, after an appropriate shift in homological grading. 

Just as the Khovanov invariant can be thought of as a chain homotopy equivalence class of chain complexes, the Lipshitz--Sarkar invariant is really a \emph{stable homotopy type}, rather than an actual particular space. (The `stable' part refers to the fact that it is an invariant once you reach a sufficiently high dimension.)

We first begin with some motivation, starting with a brief introduction to / recollection of Morse theory.
A standard reference for Morse theory is \cite{Milnor-morse-book}; if you want a quick overview, I am willing to bet that there are many, many video resources online covering Morse theory.

\subsection{The idea of Morse homology}

\begin{center}
    \includegraphics[width=1in]{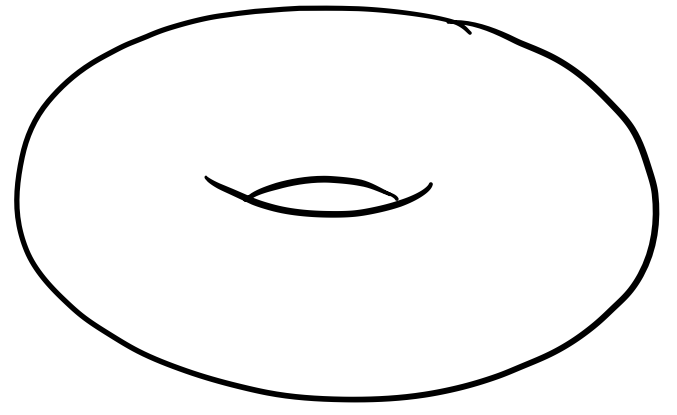}
\end{center}

Recall that we can build the torus $S^1 \times S^1$ using one 0-handle, two 1-handles, and one 2-handle:

\begin{center}
    \includegraphics[width=3.5in]{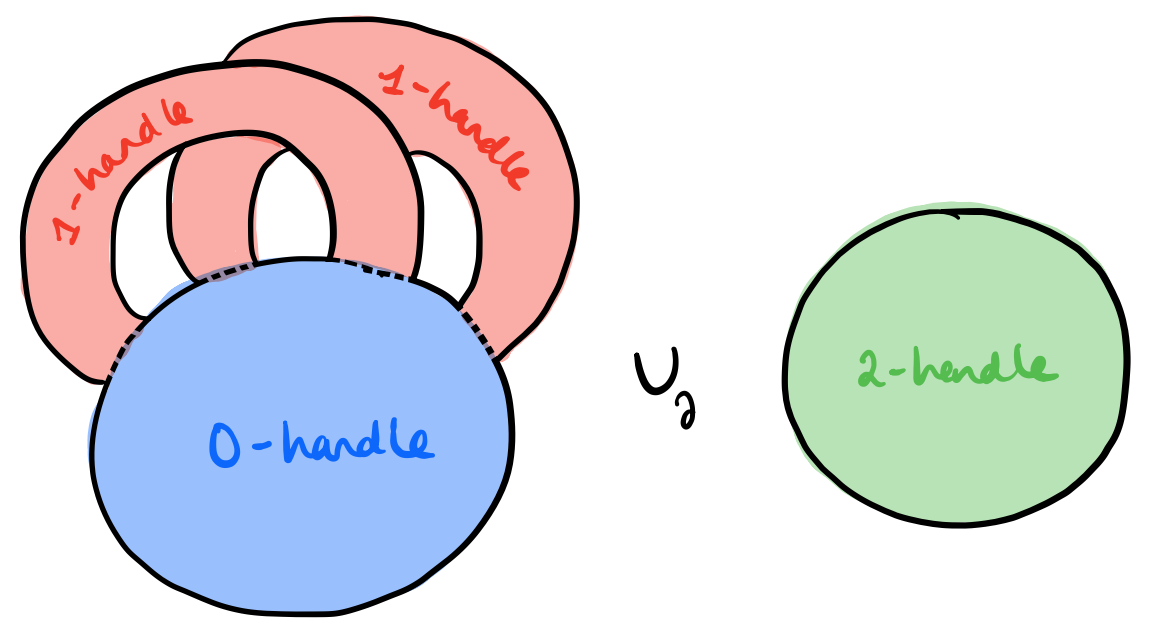}
\end{center}

As a CW-complex, we can also build it using one 0-cell, two 1-cells, and one 2-cell:

\begin{center}
    \includegraphics[width=3.5in]{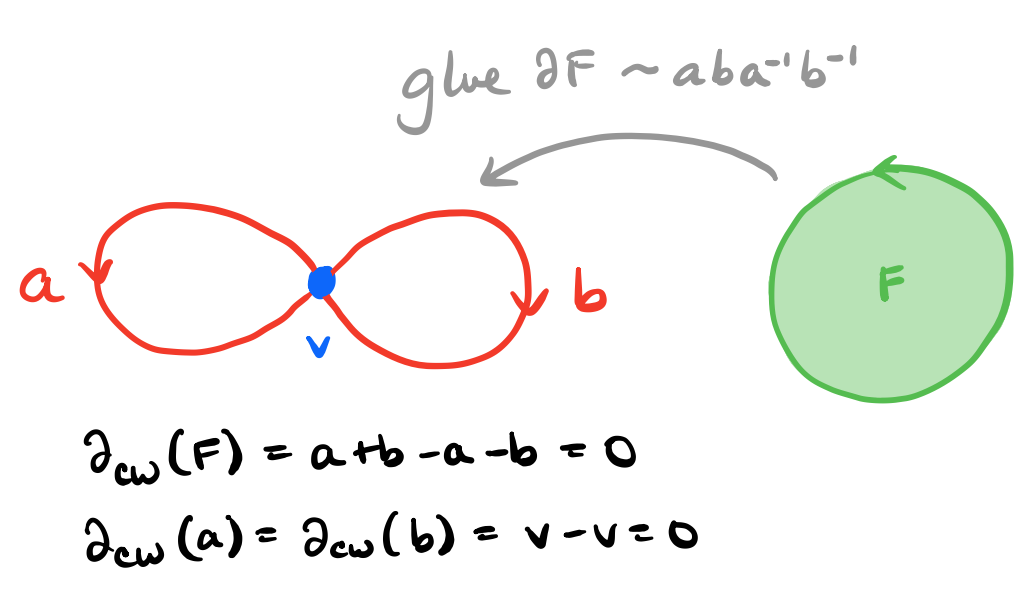}
\end{center}

The CW chain complex $\Z \to \Z \oplus \Z \to \underline{\Z}$  is generated by the cells:
\[
    \langle F \rangle \to \langle a,b \rangle \to \underline{\langle v \rangle}.
\]
Based on the attaching maps shown in the drawing, we have
\begin{align*}
    \partial F &= a + b -a -b = 0 \\
    \partial a &= \partial b = v-v =0
\end{align*}
so $\partial = 0$.

We will now describe a third way to compute the same homology, by looking at flow lines dictated by a sufficiently generic height function, called a \emph{Morse function}. 

Lean $T$ against a wall, and imagine dripping chocolate from the point at the top:

\begin{center}
    \includegraphics[width=2.5in]{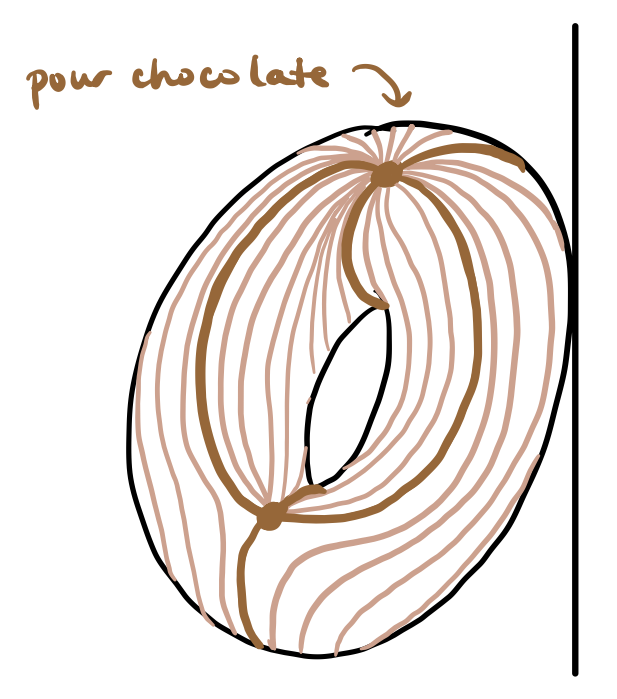}
\end{center}

At most points on the torus, chocolate flows. However, at the four critical points of the height function, the chocolate `pools'. Two are shown above; the other two are in the back of the torus, and you can't see them because chocolate is opaque. 

Here is a better picture; the two critical points on the back side of the torus are shown as \texttt{x}'s:

\begin{center}
    \includegraphics[width=2.5in]{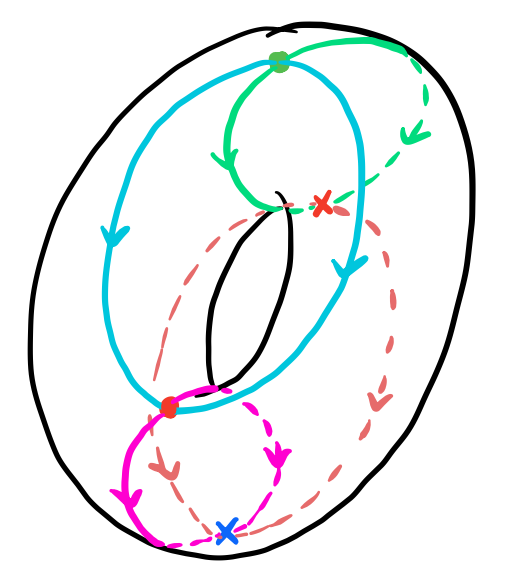}
\end{center}

There are three types of critical points shown. 
\begin{itemize}
    \item At the green critical point, if you shift slightly away in any direction, the chocolate flows away. More formally, we say that the \emph{descending manifold}, a.k.a. \emph{unstable manifold}, of the green critical point is 2-dimensional. 
    \item At each of the red critical points, shifting in two different directions might yield two different results: in one direction, the chocolate will flow back to the critical point; in another, the chocolate will flow away from the critical points. We say these critical points have a 1-dimensional descending / unstable manifold; they also have a 1-dimensional ascending / stable manifold.
    \item At the blue critical point at the bottom, any two directions you choose to shift in will yield chocolate flowing back into the critical point. So this point has a 0-dimensional descending / unstable manifold, and a 2-dimensional ascending / stable manifold.
\end{itemize}

The generators for the Morse homology complex are these critical points, and the homological grading of a critical point is given by the dimension of its descending / unstable manifold. 
The homological grading of a critical point is usually called the \emph{index} of the critical point. 

\begin{remark}
    \alert{Beware:} We will unfortunately also use the term `index' later to describe a difference between homological gradings, i.e.\ a homological degree. 
\end{remark}

A height function is said to be \emph{Morse-Smale} if 
\begin{itemize}
    \item the critical points are isolated and 
    \item the ascending and descending manifolds of different critical points intersect transversely.
\end{itemize} 
In this case, the Morse chain complex is well-defined and yields the Morse homology. 
Here are some examples of height functions on the torus that are \emph{not} Morse-Smale:

\begin{center}
    \includegraphics[width=4in]{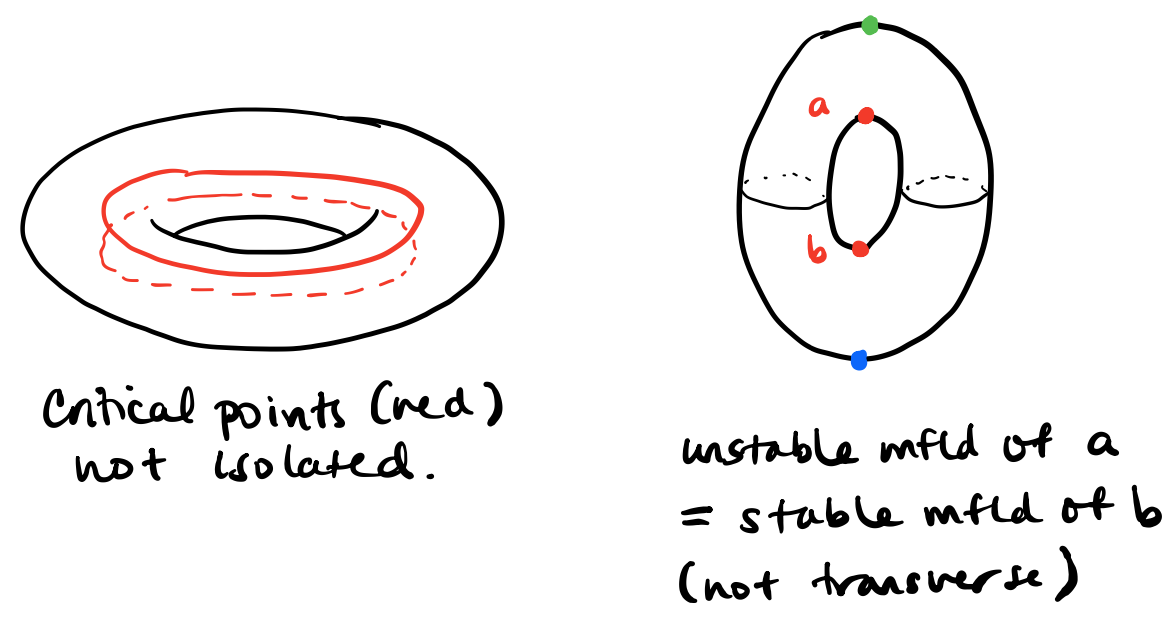}
\end{center}

Our leaning-torus height function is Morse-Smale, so we can describe the differential. 
The Morse differential counts flows from index $i$ critical points to index $i-1$ critical points, with sign. 
\note{
This means that the critical points, which are indeed points, actually have a framing. You could imagine them as tiny ball creatures facing a particular direction, which is either $+$ or $-$.} 

Let's revisit the picture of the flow lines we saw before, but now with the generators labeled:

\begin{center}
    \includegraphics[width=2.5in]{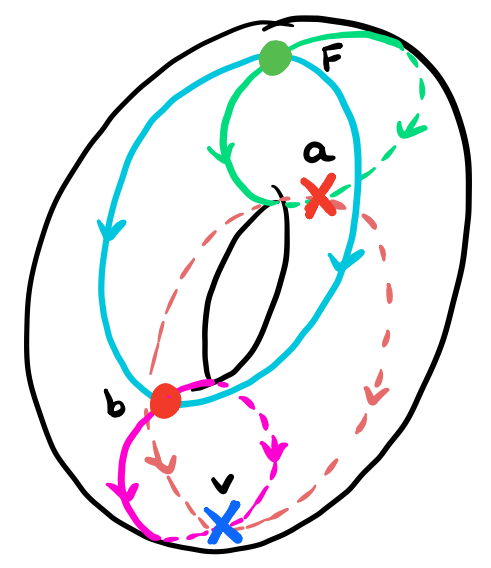}
\end{center}

We see that out of $F$ there are green and teal flow lines to index-1 critical points $a$ and $b$, and they come in pairs, with opposite signs. For example, the two teal flow lines from $F$ to $b$ clearly oppose each other while entering $b$. 
\note{There are ways to make all of this precise, of course! We unfortunately don't have enough time to stray off-course.}
Similarly, both $a$ and $b$ have pairs of oppositely-oriented flow lines into $v$. 
We conclude, just as we did in our CW-homology calculation, that the differential is 0.

Observe that, after choosing our Morse-Smale height function, we only actually needed the data of the 1-skeleton to complete our computation. 
This should give you a sense of how much information is being lost when we compute homology from a topological space. 

For example, we could have also thought about all the flow lines from $F$ to $v$. There are four 1-parameter families of these flow lines, one of which is shaded in purple below:

\begin{center}
    \includegraphics[width=3.5in]{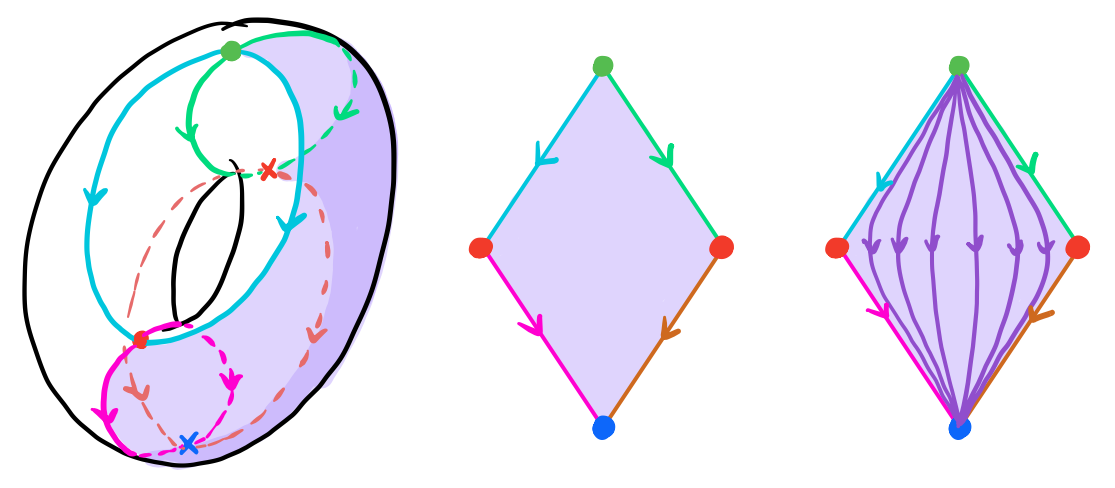}
\end{center}

The \emph{moduli space} of flow lines from a critical point $e^j$ of index $j$ and a critical point $e^i$ of index $i < j$ is a space $\moduli(e^j, e^i)$ where each point represents a single flow line;  the identification is continuous, so that flow lines close to each other correspond to points in the moduli space that are close to each other.
\note{Here is where we use the word `index' again: these are \emph{flows} of index $j-i$.}

In our situation, $\moduli(F,v)$ is the disjoint union of four intervals. The boundaries of each interval represent \emph{broken flow lines}, e.g.\ the concatenation of a flow line from $F$ to $a$ with a flow line from $a$ to $v$. (See the right-most drawing in the future above.)

\begin{remark}
    The actual proof that $\partial^2 = 0$ in Morse homology essentially boils down to the fact that the moduli spaces of index 2 (i.e.\ flows of index 2) are one manifolds, and we really understand 1-manifolds.
\end{remark}

All this is to say that, when we compute Morse homology, we are really only concerning ourselves with the data of moduli spaces of flows of index 1 (and you could maybe say 2, since we do want to check that $\partial^2 = 0$). 

\subsection{The spatial refinement problem}

Now suppose I \emph{start} with a chain complex of free $\Z$-modules (a free resolution!), and I wish to build a space whose Morse homology (or CW, singular, etc.\ homology) agrees with that of my chain complex.

The obvious way is to compute the homology, and then wedge together a bunch of spheres of the appropriate homological dimension. 
This method guarantees that your space will be no more interesting than your chain complex.

For example, had we start with the CW chain complex for the torus, we would have ended up with the space 
$S^0 \vee S^1 \vee S^1 \vee S^2$ (eww) which is definitely not our nice, smooth, beloved (not to mention \textit{connected}) torus $S^1 \times S^1$. 

Suppose we want to really build the space up with cells corresponding to the generators of the chain complex. The generators would tell us how many points to start with, and the differentials would tell us how to glue together the 1-skeleton. But at this point, we might still have 63 more dimensions to go --- how do we glue on the faces (corresponding to the moduli spaces of index 2 flow lines), or the 3-dimensional pieces? What if we glued together our 2-skeleton so weirdly that it's impossible to glue in a 3-ball where we want it?

This gluing information is what is lost when we take homology.

\begin{remark}
One should expect that this lost information is nontrivial. 
For example, in \S \ref{sec:periodic-links} we saw (asserted) that if a CW complex has a $\Z/2\Z$ action, then we can build the Tate bicomplex, and the $^{hv}E^\bullet$ spectral sequence is guaranteed to collapse by page 3. 
Perhaps you could imagine that you could build an abstract chain complex to place in the vertical columns of the Tate bicomplex, in such a way that the spectral sequence \emph{does not} collapse by page 3; then you'd know that the particular chain complex you chose could not have come from a $\Z/2\Z$ equivariant CW decomposition of a space! 
\end{remark}

While homology is just a graded $\Z$-module, a space has a \emph{cohomology ring}, whose multiplication is given by the cup product. 
This means that two spaces can have the same homology, but different cohomology rings.

Lipshitz--Sarkar's Khovanov stable homotopy type is a great example of this phenomenon, though in the category of suspension spectra rather than topological spaces.
The analogous operation to cup product operations $a \smile a$ for suspension spectra are \emph{Steenrod squares} $\Sq^i$. 
Seed showed by computing Steenrod squares that $\XKh$ can distinguish between knots and Khovanov homology can't \cite{Seed-sq2}! 

So indeed, the space-level refinement of Khovanov homology is a stronger invariant than Khovanov homology.

There are currently three methods for constructing $\XKh(L)$, which we outline briefly below. All three were shown  to be equivalent (in the appropriate sense) by Lawson--Lipshitz--Sarkar \cite{LLS-burnside-products}. 

\begin{enumerate}
    \item \textbf{Morse flow category approach} by Lipshitz--Sarkar. 
    
    This is the original method Lipshitz--Sarkar used to construct the spectrum. The main idea is to use Cohen--Jones--Segal's proposed idea \cite{CJS} for capturing the full data needed to rebuild a space, by building a \emph{framed flow category} that describes all the moduli spaces of flow lines, with framings, i.e.\ full gluing instructions.
    
    \note{This is the most hands-on method, and we will study this one in detail.}
    
    \item \textbf{Burnside functor approach} by Lawson--Lipshitz--Sarkar.
    
    This method is still quite hands-on, in the sense that some powerful theorems allow you to work with some very simple categories. The idea is to capture the information that a framed flow category would, but in a more abstract way. 
    We can view the Khovanov complex as a functor from category $\cubecat^n \to \Zmod$. We boost this to a functor to a version of the Burnside category $\Burn$, by making some careful choices. 
    Using the Pontrjagin--Thom construction, this gives us a functor $\cubecat^n \to \Topcat_*$ to based topological spaces. We then take the hocolim of this diagram of spaces, which gives us a (very high dimensional) space. Finally, suspend and desuspend to get a suspension spectrum. 

    \note{We might touch upon this construction if we have time.}

    \item \textbf{$K$-theory approach} by Hu--Kriz--Kriz. 

    This method feels the least hands-on to me, and I am actively trying to learn enough homotopy theory to be able to work with it. The basic idea is to build a permutative category which, again, contains all the framing information that we need to be able to build a space, and then pass this category through the Elmendorf--Mandell $K$-theory machine, which magically spits out a spectrum. 
    
\end{enumerate}

Finally, to bring us back down to earth, here is a warm-up question from class:

\begin{warmup}
    In our torus example, we saw that $\moduli(F,a)$ is two dots (there were two flow lines).     
    \bea
    \item What does $\moduli(F,v)$ look like?
    \item There are no moduli spaces for higher index flows in our example, but what would you guess a moduli space of index-3 flows would look like?
    \ee
\end{warmup}

\subsection{Cohen--Jones--Segal construction}

\subsubsection{Examples to keep in mind}

Let $f_1(x) = 3x^2 - 2x^3$. 
This is a Morse function on $\R$ with one index 0 critical point and a one index 1 critical point:

\begin{center}
    \includegraphics[width=4in]{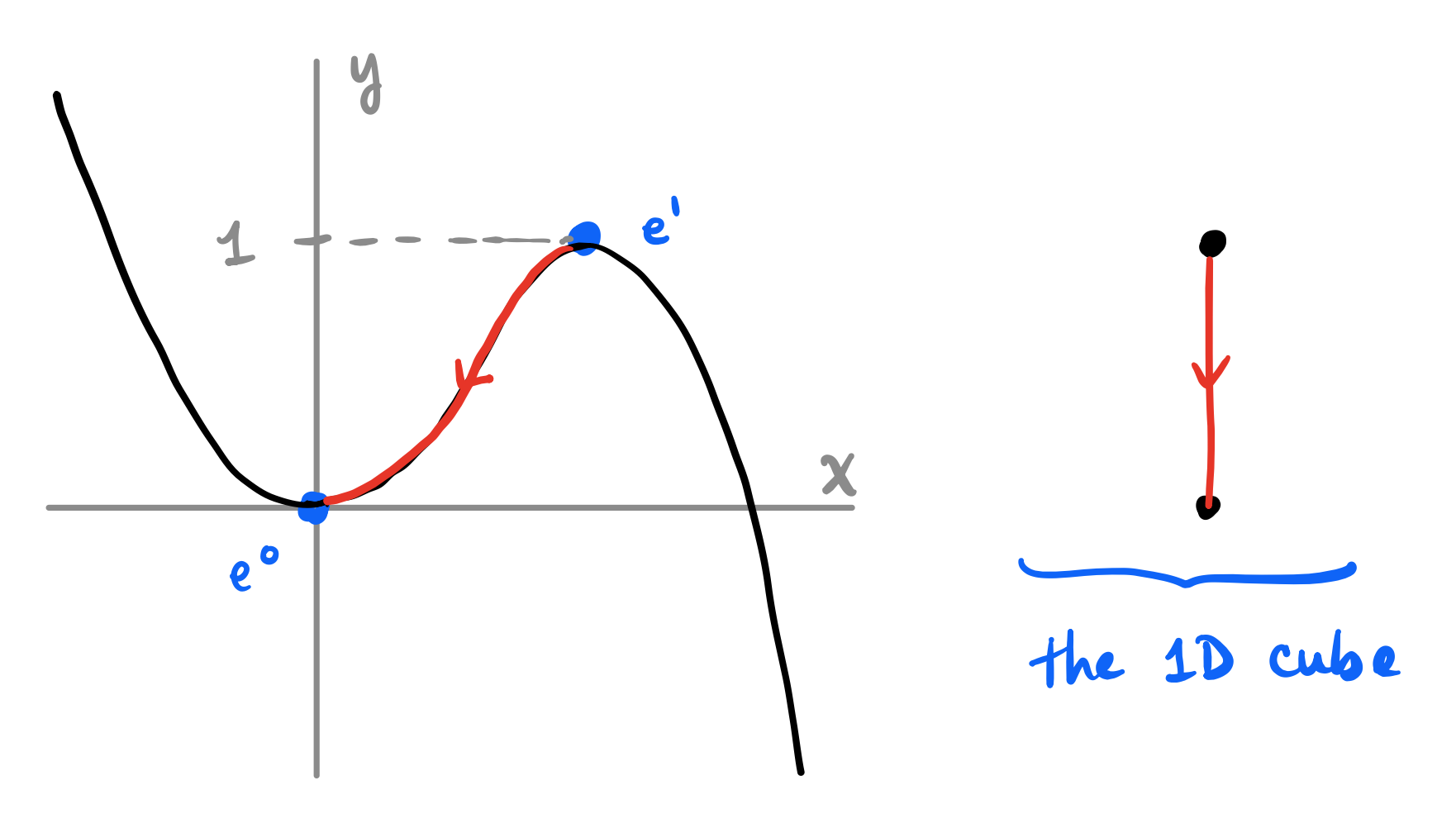}
\end{center}

The flows together form the 1-dimensional cube.
The Morse homology complex is captured by the flows of index 1 between the critical points.

\begin{remark}
    The Morse function $f_1: \R \to \R$ for $\R$ is indeed Smale. The ascending manifold of $e^0$ and descending manifold of $e^1$ are shown below:

    \begin{center}
        \includegraphics[width=4in]{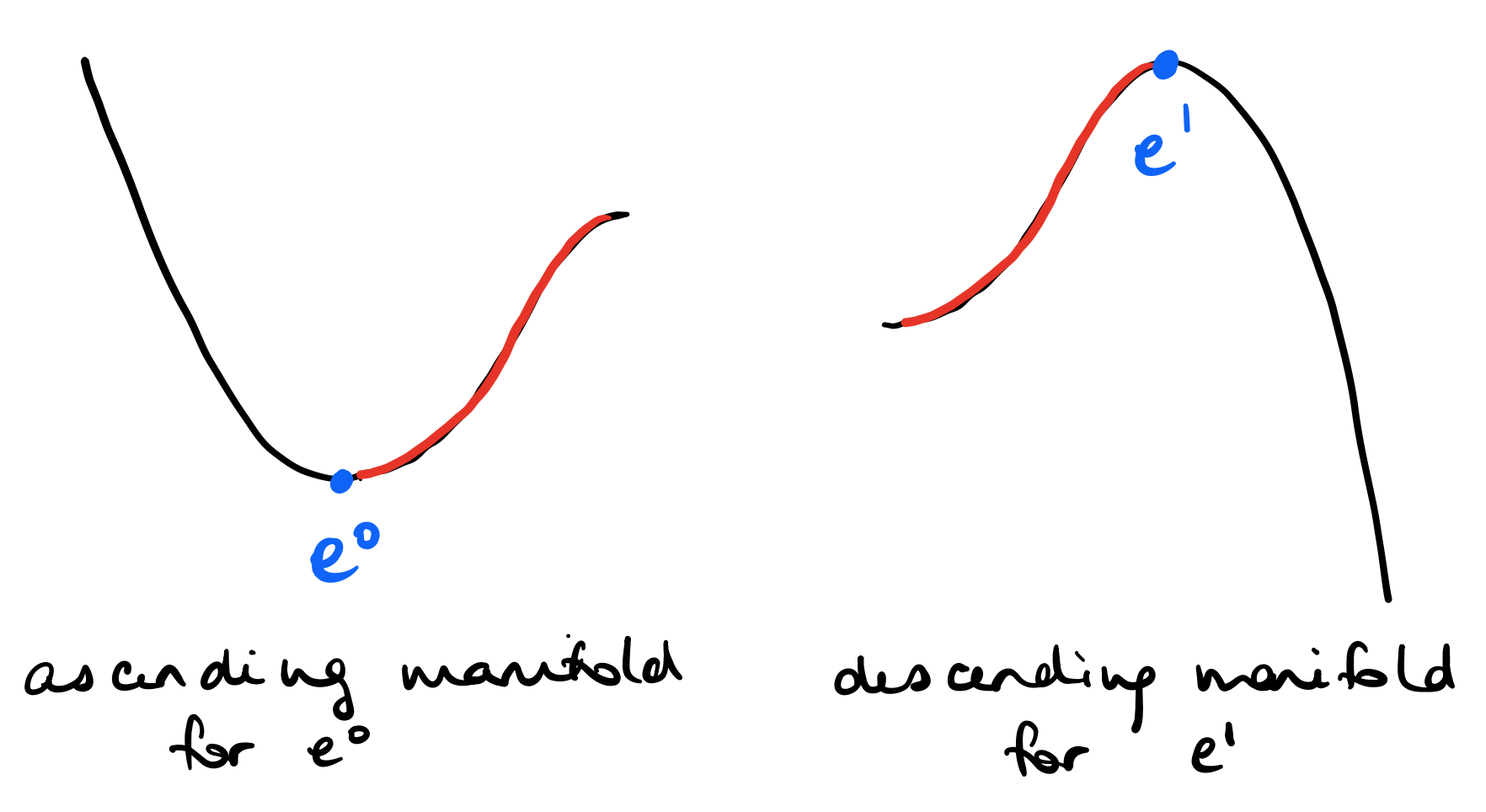}
    \end{center}    

    Their codimensions in $\R$ are both 0. Their intersection, the red curve, also has codimension 0. Since $0+0=0$, the intersection is transverse.

    (For more on transversality, see \cite{Guillemin-Pollack-difftop-book}, a standard textbook on differential topology.)
\end{remark}

Now let $f_2(x_1, x_2) = f_1(x_1) + f_1(x_2)$; this is a Morse function for $\R^2$. 
Here is the contour (i.e.\ topographic) map of the graph of this function:

    \begin{center}
        \includegraphics[width=2.5in]{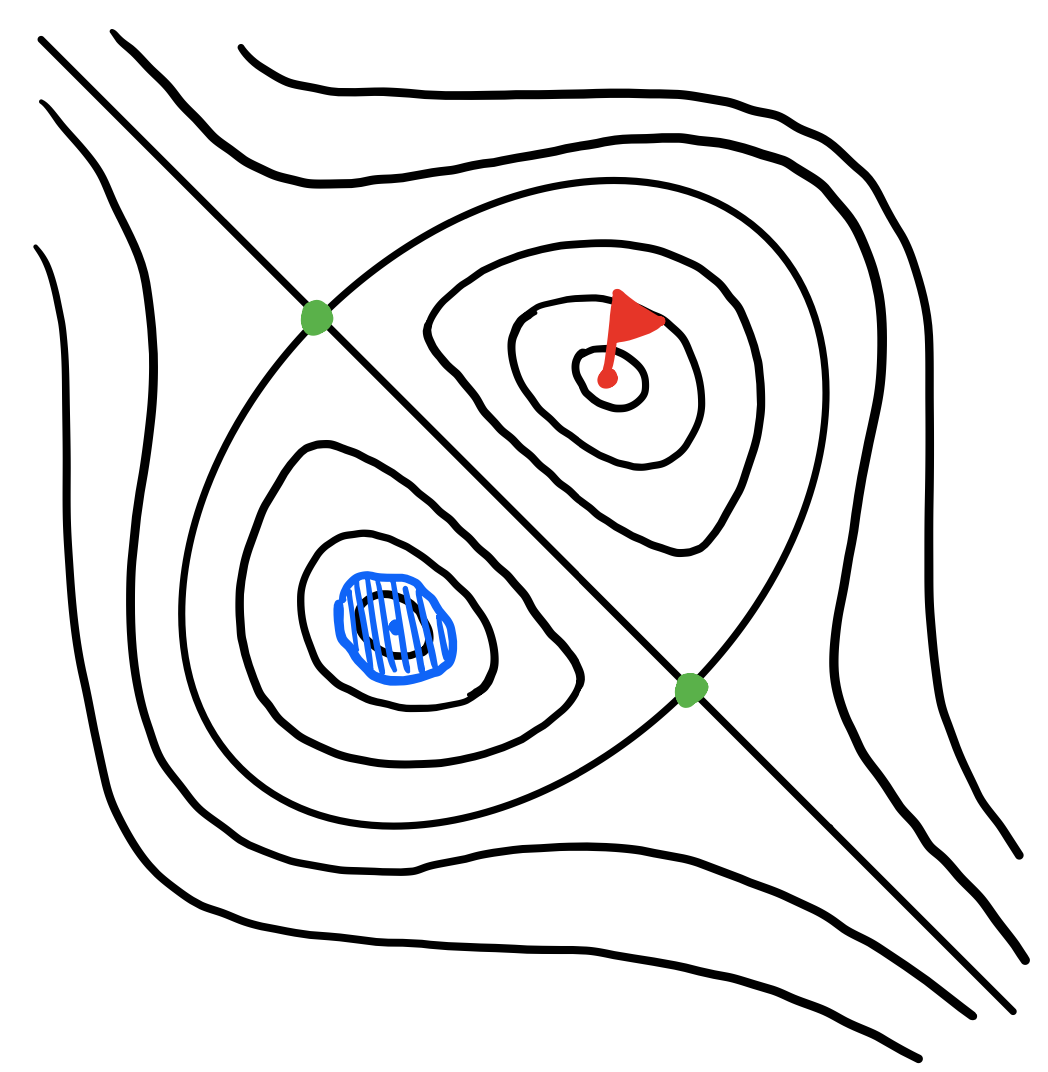}
    \end{center}

The red flag is at the top of a hill, and the blue dot is at the bottom of a lake. These points are at $(1,1)$ and $(0,0)$.  The green dots, located at $(0,1)$ and $(0,1)$, are saddle points. A skier starting slightly away from the red dot following the path of gradient descent to the lake has a 1-parameter family of possible paths, depending on where they started:

    \begin{center}
        \includegraphics[width=3in]{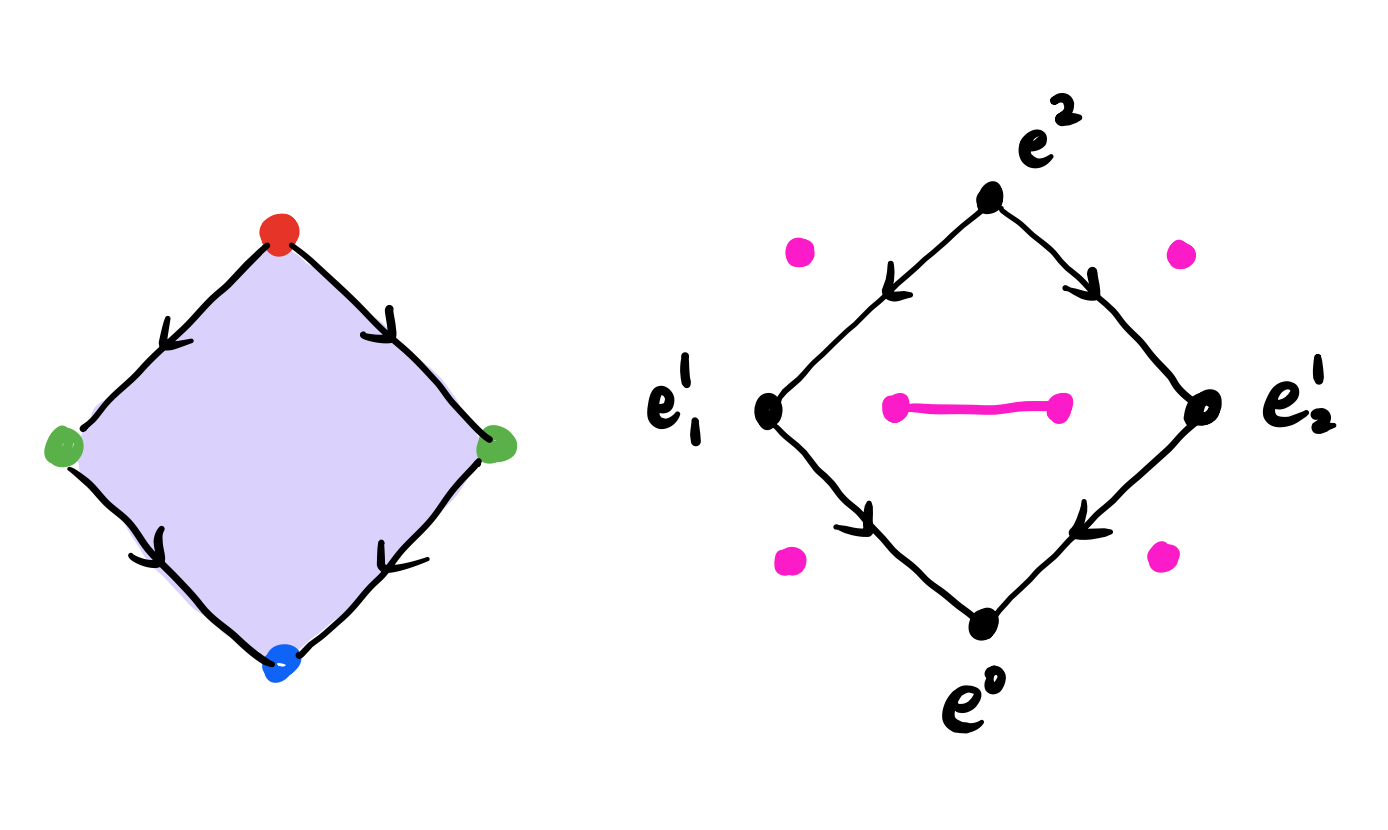}
    \end{center}  

The compactification of this moduli space of flows includes the flows where the skier skis from the red point to a green point, pauses (for eternity) and then skis to the blue point.

We write $\moduli(e^2, e^0)$ to represent the (compactified) moduli space of flows from $e^2$ to $e^0$, the pink interval.
Notice that the flows all together form a 2D cube.

Now let $f_3(x_1, x_2, x_3) = f_1(x_1) + f_1(x_2) + f_1(x_3)$. 
To visualize the graph of this function, imagine swimming in a pool, where the water temperature varies based on location. 
There is a local maximum in temperature at $(1,1,1)$ and a local minimum in temperature at $(0,0,0)$. At $(1,0,0)$, traveling away along the $x_1$ dimension means you'll feel colder; along the $x_2$ and $x_3$ dimensions, you'll feel warmer. 

The Morse flows form a 3D cube:

\begin{center}
    \includegraphics[width=2in]{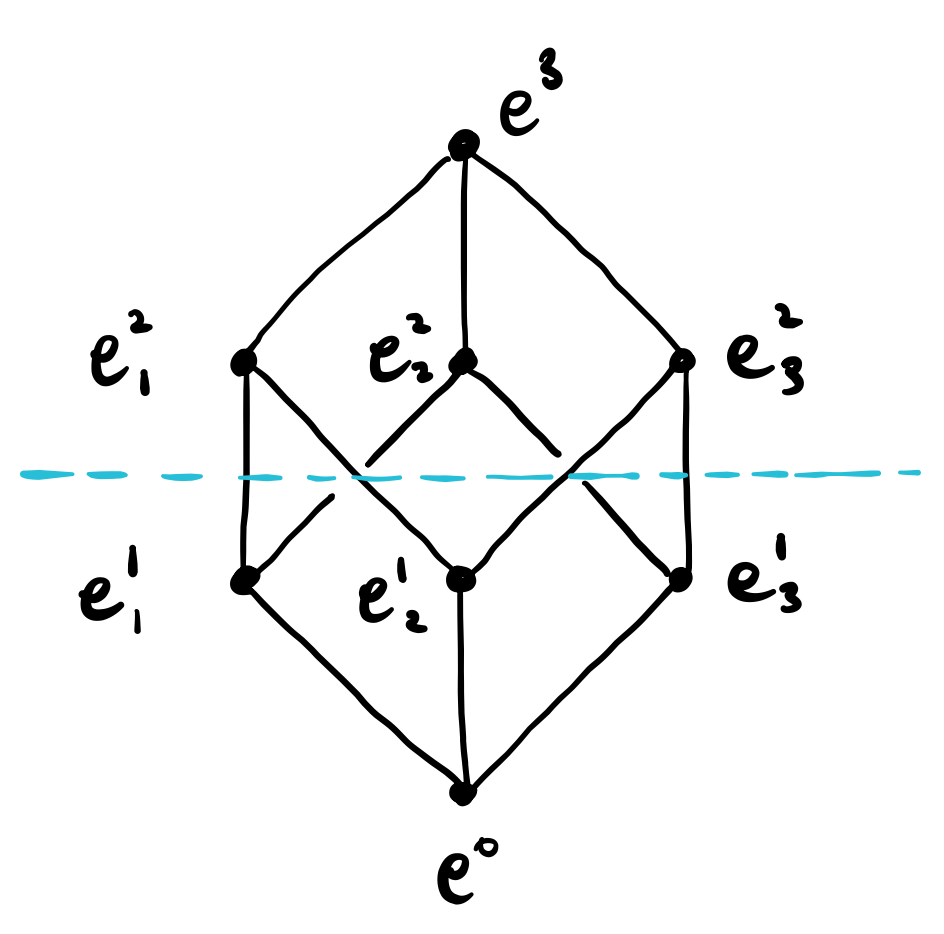}
\end{center} 

The moduli spaces of index 1 and 2 flows are dots and intervals, respectively. The moduli space $\moduli(e^3, e^0)$ of flows from the unique maximum to the unique minimum is a hexagon, which you can see by taking a cross-section of all the flows along the dotted line shown. 

The boundary of this hexagon is \emph{stratified}, i.e.\ decomposed into submanifolds of various dimensions. Here is the most natural stratefication:

\begin{center}
    \includegraphics[width=4in]{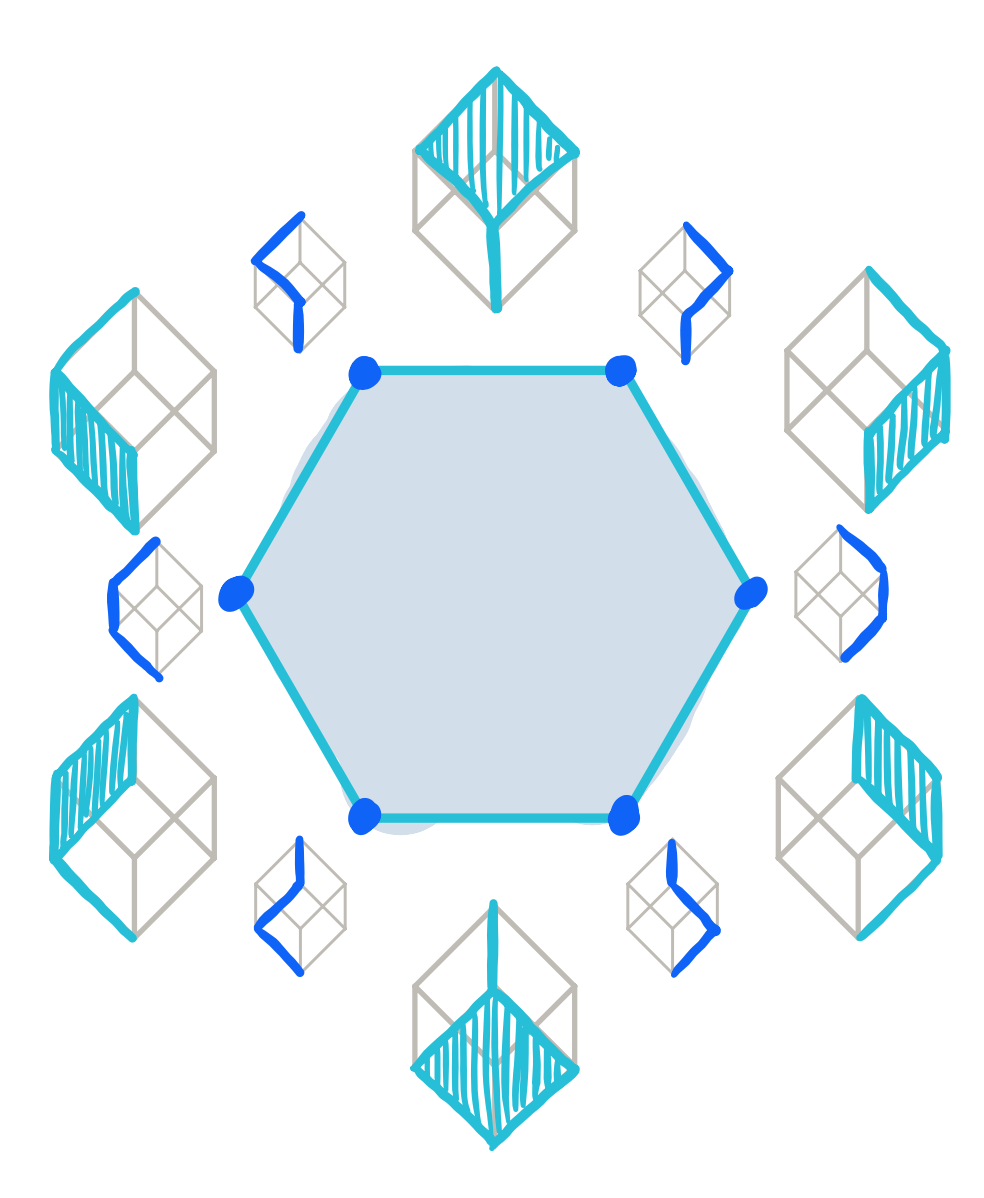}
\end{center}  
\note{This is most natural because, if you wanted to build this hexagon out of gumdrops, dried spaghetti, and cloth, you would build it this way.}

\begin{itemize}
    \item The 6 vertices are \emph{codimension 2 strata}. These correspond to \emph{twice broken} flows.
    \item The 6 edges are \emph{codimension 1 strata}. These correspond to \emph{once broken} flows. 
    These come in two flavors: 
    \begin{itemize}
        \item index 2 flow followed by index 1 flow, i.e.\ a composition
        \[
            \moduli(e^1_i, e^0) \times \moduli(e^3, e^1_i) \map{\circ} \partial \moduli(e^3, e^0)
        \]
        
        \item index 1 flow followed by index 2 flow, i.e.\ a composition
        \[
            \moduli(e^2_i, e^0) \times \moduli(e^3, e^2_i) \map{\circ} \partial \moduli(e^3, e^0)
        \] 
    \end{itemize}
\end{itemize}

The codimension 1 strata of the boundary are going to play a very important role in determining how the rest of the boundary looks.
We will think of the flavors of codimension 1 strata as colors of paint that we will use to paint the boundary of a $\langle n \rangle$-manifold,  which we define in the following section.

In general, the function 
\begin{equation}
\label{eq:fn-Morse-function-for-cube}
    f_n(x_1, \ldots, x_n) = \sum_{i=1}^n f_1(x_i)
\end{equation}
is a Morse function for $\R^n$. 
\begin{itemize}
    \item The flows between critical points glue together to form a cube of dimension $n$. 
    \item The moduli spaces that appear are all \emph{permutohedra}, which we will study more carefully soon. 
    \item The codimension $k$ strata correspond to $k$-times broken flows. 
\end{itemize}

The reason why we are so focused on this particular family of Morse functions is because the Khovanov chain complex lies over the $n$-dimensional cube, for a diagram with $n$ crossings. 
Lipshitz--Sarkar's \emph{Khovanov flow category}, which we will define in a few lectures, will be built by heavily relying on the \emph{cube flow category}, which is what we have been studying carefully. 

\subsubsection{$\langle n \rangle$-manifolds}

The moduli spaces of Morse flows will be very special kinds of manifolds with boundaries of all possible codimension, and that have a certain kind of combinatorial rigidity. 

\begin{definition}
A $k$-dimensional \emph{manifold with corners} locally looks like a neighborhood of $(\R_+)^k$, where $\R_+ = [0,\infty)$. 
\end{definition}

\begin{remark}
\begin{itemize}
    \item In other words, it can have boundary points belonging to strata of any codimension $c$ where $1 \leq c \leq k$. 
    \item This loose definition is meant to be analogous to how you would explain what a manifold is to a friend, without having to talk about charts and atlases.
    \item This teardrop-shaped disk is a 2-dimensional manifold with corners:
        \begin{center}
            \includegraphics[width=2in]{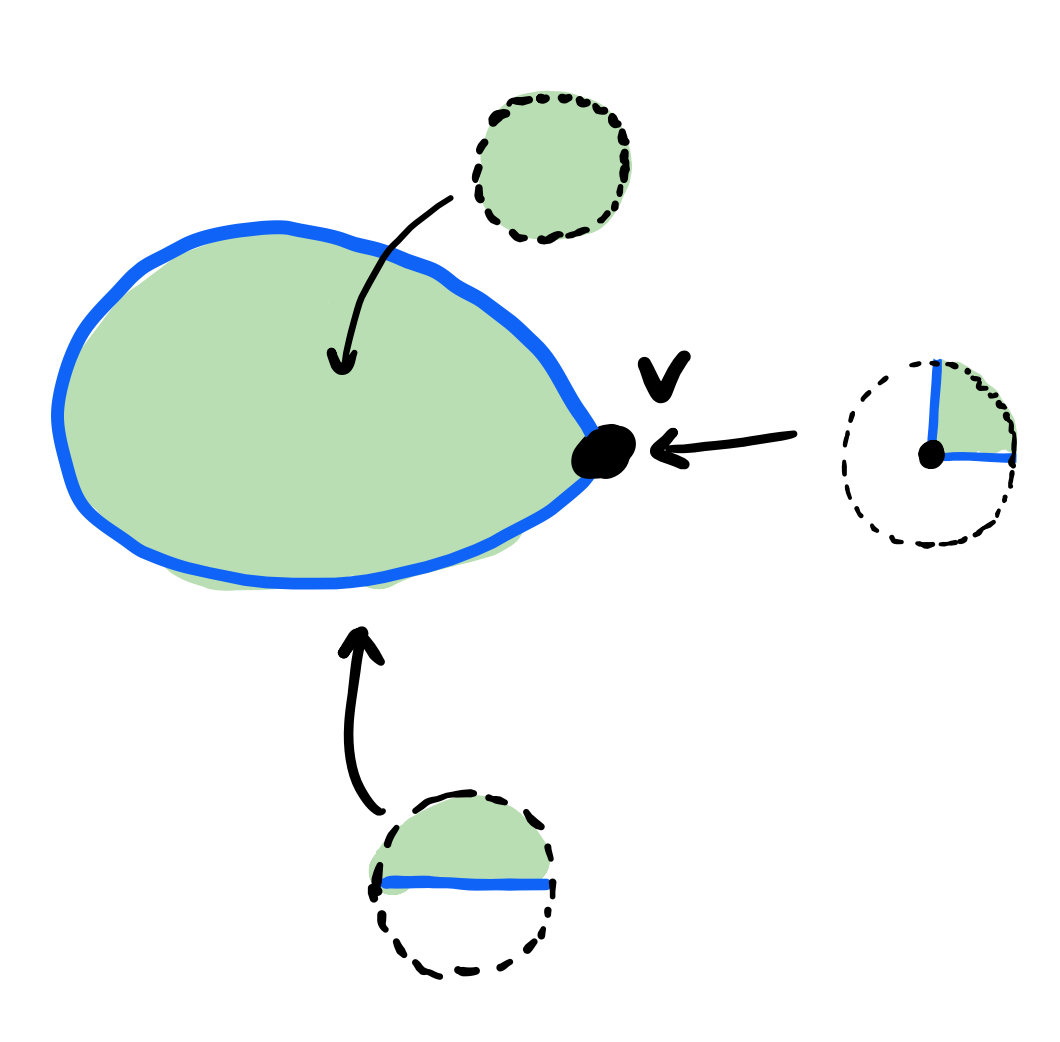} 
        \end{center}
\end{itemize}
\end{remark}

We will need a very precise definition of the term `face'.

\begin{definition}
Let $X$ be a manifold with corners. 
\begin{itemize}
    \item A \emph{connected face} is the closure of a connected component of the codimension-1 boundary of $X$.
    
    \note{For example, the 6 teal arcs in the stratified hexagon \eqref{eq:stratefied-hexagon} are, individually, connected faces of the hexagon.}
    \item A \emph{face} of $X$ is a possibly empty union of \textbf{disjoint} connected faces of $X$. 
    
    \note{For example, the pink subspace below is a face of the hexagon, but purple subspace below is \emph{not} a face of the hexagon:}
        \begin{center}
            \includegraphics[width=3in]{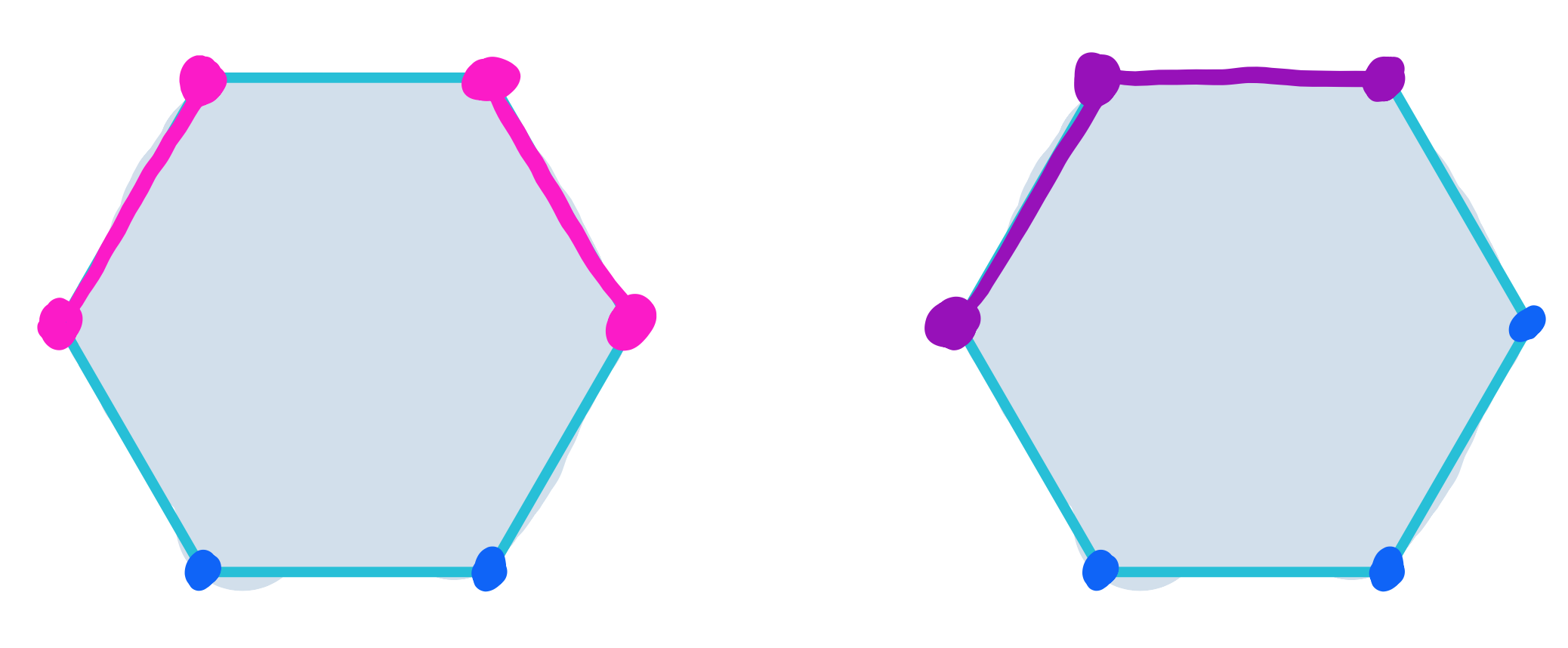}
        \end{center}
\end{itemize}
\end{definition}

I have been asking you to vibe out the `obvious' stratefication of the boundary of a manifold with corners.
To be more precise, we can describe the strata as the connected components of codimension $i$ boundary, with $i$ ranging from $1$ to $k$. 

For each $x$ in a manifold with corners $X$, let $c(x)$ denote the codimension of the stratum that $x$ lives in. (If $x$ is in the interior of $X$, let $c(x) = 0$.)

\begin{definition}
A $k$-dimensional \emph{manifold with faces}
is a $k$-dimensional manifold with corners such that, 
for each $x \in X$, 
$x$ belongs to exactly $c(x)$ different connected faces of $X$. 
\end{definition}

\begin{remark}
\begin{enumerate}
    \item The teardrop example from before is \emph{not} a manifold with corners because the vertex $v$ is adjacent to only one distinct edge, even though $c(v) = 2$. 
    \item This triangle is a manifold with faces:
        \begin{center}
            \includegraphics[width=2in]{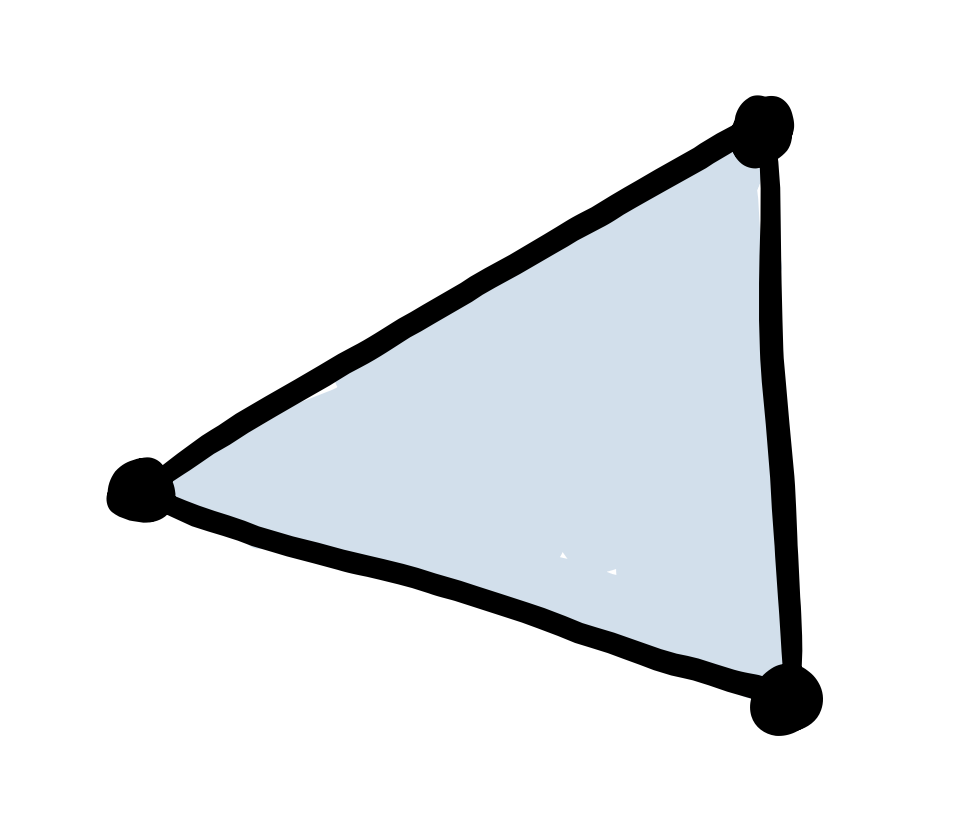}
        \end{center}
    \item Any face of a manifold with faces must necessarily also be a manifold with faces.
    \mz{explain why}
\end{enumerate}
\end{remark}

\begin{definition}
    A \emph{$k$-dimensional $\langle n \rangle$-manifold}
    is a $k$-dimensional manifold with faces 
    along with an ordered $n$-tuple of faces of $X$ where
    \begin{itemize}
        \item $\partial X = \bigcup_i \partial_i X$
        \item for $i\neq j$, $\partial_i X \cap \partial_j X$ is a face of both $\partial_i X$ and $\partial_j X$. 
    \end{itemize}
\end{definition}

\begin{remark}
    \begin{enumerate}
        \item We can think of each $i \in \{1,\ldots, n\}$ as a color of paint. 
        We paint a face of $X$ with each color, so that by the time we're done painting with color $n$, we have covered the entire boundary of $X$ with paint.
        \item A point $x \in \partial X$ will be covered with $c(x)$ different layers of paint, all of different colors, corresponding to the codimension-1 faces that it abuts. 
        \item For this reason, I verbally refer to $\langle n \rangle$-manifolds as `$n$-painted manifolds'.
        \item Notice that $k$ and $n$ are different and basically unrelated. 
        \item The triangle previously shown is cannot possibly be a 2-dimensional $\langle 2 \rangle$-manifold, but it can be a $\langle 3 \rangle$-manifold:
        \begin{center}
            \includegraphics[width=2in]{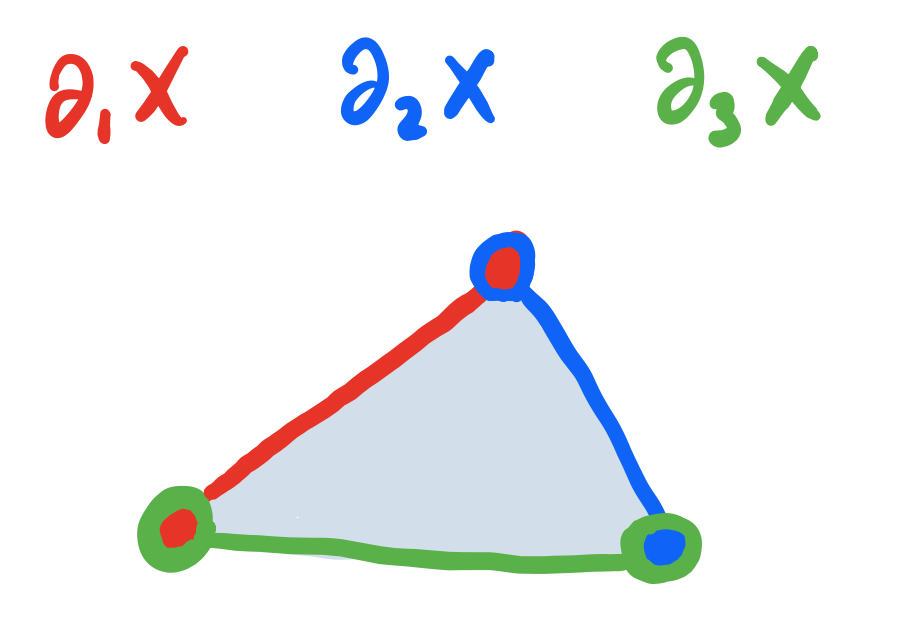}
        \end{center}
        Because the colors of the corners (boundary of codimension $\geq 2$) are determined by the colors painted on the codimension-1 boundary, it suffices to draw the figure below:
        \begin{center}
            \includegraphics[width=2in]{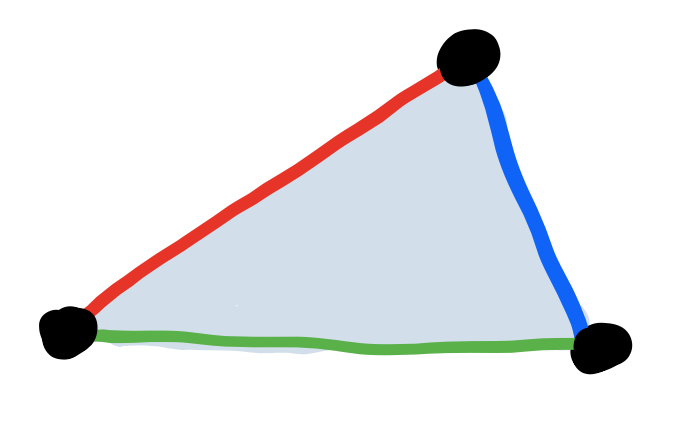}
        \end{center}
        
    \end{enumerate}
\end{remark}

\begin{example}(Permutohedra)

The \emph{permutohedron} $P_{n+1}$ is the most important example of an $n$-dimensional $\langle n \rangle$-manifold, for our purposes. 

\begin{itemize}
    \item As a space, $P_{n+1}$ is the convex hull in $\R^{n+1}$ of the $(n+1)!$ points 
    \[
        \{ ( \sigma(1), \ldots, \sigma(n+1)) \st \sigma \in S_{n+1} \}
    \]
    where $S_{n+1}$ is the symmetric group on the letters $[n+1] := \{1, 2, \ldots, n+1\}$. 
    \item The connected faces of $P_{n+1}$ are in bijection with proper subsets $\emptyset \neq S \subsetneq [n+1]$; 
    for a proper subset $S$ of cardinality $k$,
    the corresponding face $F_S \subset \partial P_{n+1}$ is the intersection of $P_{n+1}$ with the plane 
    \[
        \sum_{i \in S} x_i = \frac{k(k+1)}{2}.
    \]
    \item The $\langle n \rangle$-manifold structure is given by 
    \[
        \partial_i P_{n+1} = \bigcup_{|S|=i} F_S.
    \]
    \item Here is the hexagon $P_3$ again \footnote{Thanks to Josh Turner for spotting and fixing a previous error in this example!}:
        \begin{center}
            \includegraphics[width=4in]{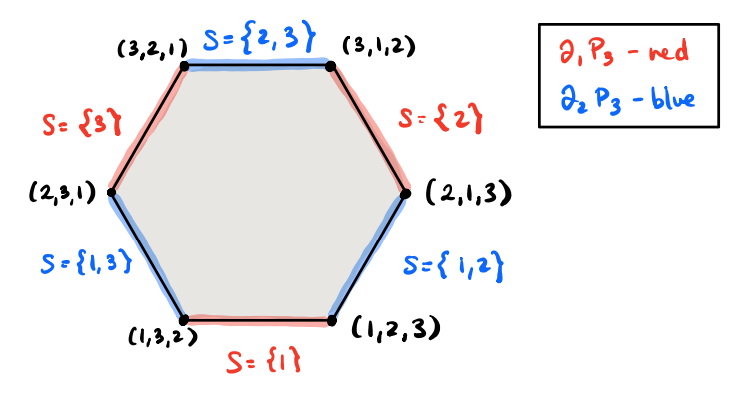}
        \end{center}
        Notice that the two `flavors' of once-broken flows on the boundary of $\moduli(e^3,e^0)$ exactly correspond to the faces $\partial_1 P_3$ and $\partial_2 P_3$. 
\end{itemize}

\end{example}

\subsubsection{Flow categories}

We are now ready to look at the definition of a \emph{flow category}. If you keep in mind the Morse flows picture, all of these conditions should feel quite natural.

\mz{Note that I am using different variables names from the source. Sorry.}

\begin{definition}[\cite{LS-stable-homotopy-type}, Definition 3.13 ] 
A \emph{flow category} is a category $\SC$ consisting of the following data:
\begin{itemize}
    \item finitely many objects $\Ob(\SC)$
    \item a $\Z$-grading on objects $\gr : \Ob(\SC) \to \Z$ 
    \item morphisms between objects $\Hom(x,y)$
\end{itemize}
subject to the following conditions:

\begin{enumerate}

\item $\Hom(x,x) = \{ \id_x \}$ for all $x \in \Ob(\SC)$. \note{This is so that we have a category. Note that $\moduli(x,x) = \emptyset$.}

\item For distinct objects $x \neq y$, $\Hom(x,y)$ is a compact $(\gr(x) - \gr(y) -1)$-dimensional $\langle \gr(x) - \gr(y) -1 \rangle$-manifold. 
This space is denoted by $\moduli(x,y)$.

\item Let $x,y,z \in \Ob(\SC)$ be distinct, with $\gr(x) \geq \gr(y) \geq \gr(z)$, with 
\begin{align*}
    \gr(x) - \gr(y) = \ell  \\
    \gr(y) - \gr(z) = m \\
    \gr(x) - \gr(z) = n = \ell + m.
\end{align*}
\begin{center}
    \begin{tikzcd}
        x \arrow{d}{\text{index } \ell} \arrow[bend right]{dd}[swap]{\text{index }{n}}\\
        y \arrow {d}{\text{index } m}\\
        z
    \end{tikzcd}
\end{center}

\begin{itemize}
    \item The composition map 
        \[
            \circ: \Hom(y,z) \times \Hom(x,y) \to \Hom(x,z) 
        \]
        is an embedding into $\partial_m \Hom(x,z)$.
        \note{Recall this is the face of $\Hom(x,z)$ painted with the color $m$.}
    \item For $i< m$, the preimage of $\partial_i\Hom(x,z)$ 
    \note{under the particular composition map we are looking at here} is
    \[
        \circ\inv (\partial_i\Hom(x,z)) = \partial_i \Hom(y,z) \times \Hom(x,y).
    \]
    \item For $j > m$, the preimage of $\partial_j \Hom(x,z)$ \note{under the particular composition map we are looking at here} is
    \[
        \circ\inv (\partial_i\Hom(x,z)) = \Hom(y,z) \times \partial_{j-m}\Hom(x,y).
    \]

\end{itemize}

\item 
If you run over all the $y$ with grading in between that of $x$ and $z$, you will be able to piece together the whole $\partial \Hom(x,z) = \partial \moduli(x,z)$:
\[
    \partial_m \Hom(x,z) \cong \coprod_{y \st \gr(y) = \gr(z) + m} \Hom(y,z) \times \Hom(x,y).
\]
\note{Here $\cong$ means `diffeomorphism'. }

\end{enumerate}

\end{definition}

\begin{example}
    For a Morse function $f: X \to \R$, we get the \emph{Morse flow category} for $f$, exactly the way we have been motivating the definition of `flow category'. 
\end{example}

\subsubsection{Framing the flow categories}

In this section we give a sense of what it means to upgrade a flow category to a \emph{framed flow category}. 

A `framing' typically refers to a continuous choice of orthogonal basis (a `frame') for the fibers of a bundle. 
When we talk about a framing for a knot $K$, for instance, we pick a choice of longitude $\lambda$ for the solid torus  $\nu (K)$; this gives a continuous choice of orthogonal bases for the normal bundle of the knot in $S^3$:

\begin{center}
    \includegraphics[width=2.5in]{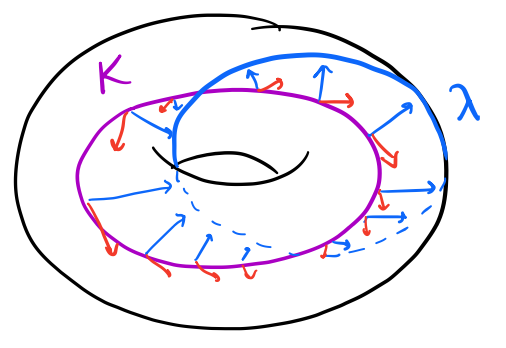}
\end{center}

Our goal is to provide framings for the moduli spaces in a flow category, in such a way that they can be glued together nicely. 
To do this, we embed the moduli spaces into some high dimensional Euclidean spaces and pick framings for their normal bundles there.

First, recall that for finite dimensional smooth manifolds $X$, we have the \emph{Whitney embedding theorem}, which guarantees that we can find some large $N$ such that there exists an embedding $X \into \R^N$.

In a way, this theorem tells us that Euclidean space $\R^d$ are the 'model smooth manifolds'. We already knew that, of course, but this raises the question,
\begin{center}
    ``what is a model $\langle n \rangle$-manifold?''
\end{center}

We already saw that manifolds with corners are modeled after $\R_+^d$, so we will certainly need at least that much information. 
The important difference between $\langle n \rangle$ manifolds and run-of-the-mill manifolds with corners is the painting of the faces; so, we need model spaces that are also $\langle n \rangle$-manifolds, where it is very clear how to paint the faces with $n$ colors.

\begin{definition}
    Let $\vec d = (d_0, d_1, \ldots, d_n) \in \N^{n+1}$ \note{(here $\N = \N \cup \{0\}$)}, and let
    \[
        \BE_n^{\vec d} := 
        \R^{d_0} \times 
        \left ( \R_+ \times \R^{d_1} \right ) \times 
        \left ( \R_+ \times \R^{d_2} \right ) \times 
        \cdots \times 
        \left ( \R_+ \times \R^{d_n} \right ). 
    \]
    \note{Note that there is a copy of $\R_+$ for each $i \in \{1, \ldots, n\}$. The boundaries of these give us $n$ different faces.}
    The $\langle n \rangle$-manifold structure is given by 
    \[
        \partial_i(\BE_n^{\vec d}) = 
        \R^{d_0} \times 
        \left ( \R_+ \times \R^{d_1} \right ) \times 
        \cdots \times
        \left ( {\color{red} \{0\}} \times \R^{d_i} \right ) \times 
        \cdots \times 
        \left ( \R_+ \times \R^{d_n} \right ).
    \]
\end{definition}

The appropriate, structure-preserving notion for an embedding of $\langle n \rangle$-manifolds, for our framing purposes, is the following:

\begin{definition}
\label{def:neat-embedding}
    A \emph{neat embedding} $\iota$ of an $\langle n \rangle$-manifold is a smooth embedding 
    \[
        \iota: X \into \BE^{\vec d}_n
    \]
    for some $\vec d \in \N^{n+1}$, where
    \begin{itemize}
        \item $\iota\inv(\partial_i \BE^{\vec d}_n) = \partial_iX$ for all $i$ \note{(respecting the $\langle n \rangle$-manifold structure)} and 
        \item whenever a stratum of $X$ abuts a lower-dimensional (higher-codimensional) stratum of $\BE^{\vec d}_n$, it does so perpendicularly.
        \note{This is the \emph{neat} part; see Remark \ref{rmk:neat-perpendicular} for a more precise definition.}
    \end{itemize}
\end{definition}

\begin{example}
    The triangle, as a 2-dimensional $\langle 3 \rangle$-manifold, neatly embeds in $\BE_3^{(0,0,0,0)} = (\R_+)^3$ as an octant of the unit 2-sphere.
\end{example}

\begin{remark}
\label{rmk:neat-perpendicular}
Lawson--Lipshitz--Sarkar note that a $\langle n \rangle$-manifold can be viewed as a functor $\cubecat^n \to \Topcat$ by
\[
    X(u) = 
    \begin{cases}
        X & \text{if }u = \overline 1 \\
        \bigcap_{i \in \{ i \st a_i = 0\}} \partial_i X & \text{otherwise}.
    \end{cases}
\]
For example, here is the hexagon as a cube-shaped diagram of topological spaces:
\begin{center}
    \includegraphics[width=3in]{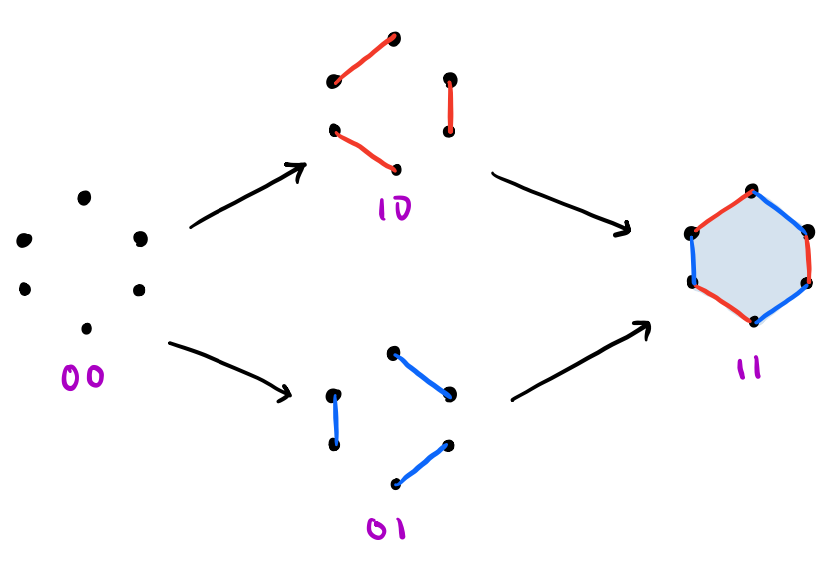}
\end{center}

The second requirement in Definition \ref{def:neat-embedding} becomes:
\begin{quote}
    For all $u < v$ in $\{0,1\}^n$, 
    $X(v) \perp \BE_n^{\vec d}(u)$. 
\end{quote}

\end{remark}

There is an analogue to the Whitney embedding theorem for $\langle n \rangle$-manifolds, which we will paraphrase loosely:

\begin{fact}[see Lemma 3.11 of \cite{LS-stable-homotopy-type} for more detail]
Let $X$ be a $\langle n \rangle$-manifold. 
Given a neat embedding $\iota$ of $\partial X \into \BE_n^{\vec d}$, there exists a neat embedding of $X \into \BE_n^{\vec d'}$ for some $\vec d' \geq \vec d$, induced by $\iota$.
\note{Here, $\vec d' \geq \vec d$ means $d_i' \geq d_i$ for all $i$.}
\end{fact}

Note that, since we are quite capable of neatly embedding finite sets of points, by induction we know that every $\langle n \rangle$-manifold can be neatly embedded.

We now return back to framing flow categories, and give very rough definitions.

A \emph{neat embedding} of a flow category $\SC$  is a collection of neat embeddings for all the moduli spaces in $\SC$, in a coherent way so that composition of moduli spaces is respected. (See Definition 3.16 of \cite{LS-stable-homotopy-type} for more detail.)
\begin{itemize}
    \item To make sure all the embeddings are in the same spatial universe (in the non-technical sense), we have to first pick a bi-infinite sequence $\vec d: \Z \to \N$. \note{Don't worry, we'll only use a finite subsequence of this, because $\SC$ only has finitely many objects.}
    \item Then, let $\moduli(i,j)$ be the union of all the moduli spaces $\moduli(x,y)$ where $\gr(x) = i$ and $\gr(y) = j$; neatly embed $\moduli(i,j)$ into pieces edgy spaces of the form $\BE_{i-j-1}^{(d_j, d_{j+1}, \ldots, d_{i-1})}$. Call these embeddings $\iota_{i,j}$.
    \item If $p \in \moduli(x,y)$ and $q \in \moduli(y,z)$, then the composition condition requires that 
    \[
        \iota_{x,z}(q \circ p) = (\iota_{y,z}(q), 0, \iota_{x,y}(p)).
    \]
    \note{The most important thing to note here is that $\iota_{x,y}$ and $\iota_{y,z}$ determine $\iota_{x,z}(q \circ p)$.}
\end{itemize}

Observe that at this point, we haven't really actually picked a \emph{framing} for anything. However, we do have moduli spaces embedded in Euclidean spaces, with normal bundles ready to be framed. 

We can now give a rough definition of framed flow categories:

\begin{definition}
    A \emph{framed flow category} is a neatly embedded flow category $\SC$ \textbf{plus} a \emph{coherent} framing of the normal bundles. 
\end{definition}

The word `coherent' is pulling a lot of weight here. This is the point at which we make sure we'll actually be able to glue our moduli spaces together to make a nice space. 
The main requirement is that the product framing of the normal bundles of two embeddings of moduli spaces
$\nu(\iota_{y,z}) \times \nu(\iota_{x,y})$
equals the pullback framing of $\circ \inv \nu(\iota_{x,z})$ induced by the composition map.

Finally, Lipshitz--Sarkar give explicit instructions for how to take the data of a framed flow category and transform them into a CW-complex (of some dimension). 
See Figures 3.3 and 3.4 of \cite{LS-stable-homotopy-type} to get a sense of how to do this.
\note{We talked about these pictures in class, but I will not be including my commentary in these notes.}

\subsection{A framed flow category for Khovanov homology}

Let $L$ be an oriented link, and let $D$ denote a diagram for $L$ with $n$ crossings.
Lipshitz--Sarkar build the framed flow category for Khovanov homology as follows:
\begin{enumerate}
    \item Frame the \emph{$n$-dimensional cube flow category} $\SC_C(n)$.
    \item Build a Khovanov flow category $\SC_K(D)$ with moduli spaces corresponding to \emph{decorated resolution configurations} for $D$.
    \item Make sure that the moduli spaces of $\SC_K(D)$ are all \emph{trivial} covers of the corresponding moduli spaces in $\SC_C(n)$.
    \item Conclude that framings can be lifted from $\SC_C(n)$ to $\SC_K(D)$, thereby boosting $\SC_K(D)$ to a \emph{framed} flow category.
\end{enumerate}

There is a lot to say and prove about all of these steps, but we will give an overview of the key philosophical ideas in each step (as opposed to the key technical ideas). 

\subsubsection{Framing the cube flow category: obstruction theory}

Let's first talk a bit about the idea behind \emph{obstruction theory}.

Consider the unknot $U \subset S^3$. 
A framing $\lambda$ of $U$ is a continuous choice of frame (orthonormal basis) for the fibers of the normal bundle. 

\begin{center}
    \includegraphics[width=1.5in]{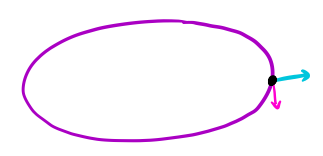}
\end{center}

This is equivalent to an integral choice of longitude on the torus $\partial \nu(U)$,  the boundary of a tubular neighborhood of $U$:

\begin{center}
    \includegraphics[width=2.5in]{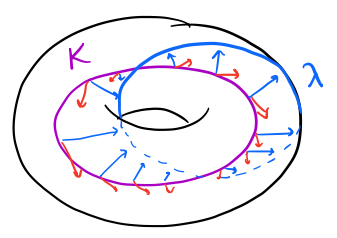}
\end{center}

Here are two different framings for $U$. 

\begin{center}
    \includegraphics[width=4in]{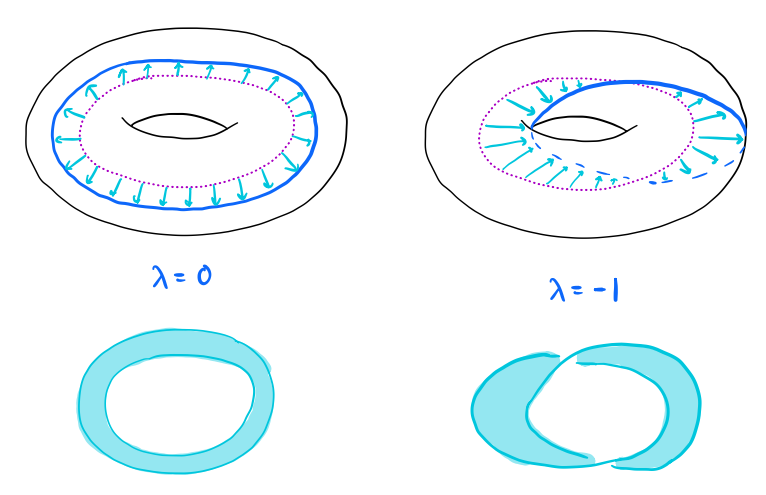}
\end{center}

Notice that the $0$-framed annulus can be extended to an embedded disk in $S^3$ bound by $U$, where as the $(-1)$-framed annulus cannot. 

In this sense, the value $\lambda \in \Z$, when nonzero, provides an \emph{obstruction} to extending the framed annulus to a disk.

Pictorially, we can think of $\lambda$ as a longitude. 
But another way to think about $\lambda$ is as an element of $\pi_1(O(2))$, as follows.
The columns of a matrix in $O(2)$ give an orthogonal basis for the fiber of the normal bundle over a point $p \in U$. 
Moving $p$ all the way around $U$ is a loop in $O(2)$, which corresponds to an element of $\pi_1(O(2))$. 
\note{You may convince yourself why we only care about these loops up to homotopy.}

In order to frame the cube flow category $\SC_C(n)$, Lipshitz--Sarkar first prove that there exists a coherent set of orientations for the moduli spaces 
\note{(we didn't really talk much about this)}. 
Then, suppose we are given a coherent orientation and a neat embedding of the $\SC_C(n)$ (both of which we know exist), we wish to frame the moduli spaces, starting from the 0-dimensional ones, whose framings are given by the orientations chosen, and then building up to framings for the moduli spaces of higher dimension --- if there is no obstruction. 

Suppose you already have a coherent framing of all the moduli spaces of dimension $< k$.
This data produces an obstruction class 
\[
    \mathfrak{o} \in C^{k+1}(\CC(n), \pi_{k-1}(\mathbf{O})).
\]
\begin{itemize}
    \item Here, $\CC(n)$ is `the cube' of dimension $n$, i.e.\ the CW-complex you get from gluing all the flow lines between critical points of $f_n(x_1, \ldots, x_n)$ (from \eqref{eq:fn-Morse-function-for-cube}). 
    \item $\mathbf{O} = \colim_{k \to \infty} O(k)$, the colimit of all the orthogonal groups of finite dimension, related by inclusion maps.    
\end{itemize}
While we will not discuss how the obstruction class $\mathfrak{o}$ is obtained, but we can still get an idea of why $\mathfrak{o}$ makes sense as the obstruction to extending the framing to the $k$-dimensional moduli spaces. 
\note{Recall that all our moduli spaces are permutahedra, which are in particular homeomorphic to balls of various dimensions.}
\begin{itemize}
    \item A (component of a) $k$-dimensional moduli space corresponds to a $k+1$-dimensional cell in the CW decomposition of $\CC(n)$ given by the spaces of flow lines.
    \item A framing of a $k$-dimensional moduli space $\moduli(x,y)$ is a choice of an orthogonal basis for the fiber of the normal bundle above every point.
    \note{So, depending on how high-dimensional the space you're embedded in is, your normal bundle is some rank, and you'd be working with that many vectors in your orthogonal bases.}
    \item For the framing along the boundary $\partial \moduli(x,y) \cong S^{k-1}$ to extend to a framing of $\moduli(x,y)$,  we need the boundary framing to `not be twisted,' or more precisely, for the map $S^k \to \mathbf{O}$ to be nullhomotopic.
\end{itemize}

Lipshitz--Sarkar show that indeed, the cochain $\mathfrak{o}$ is in fact a coboundary, and modifications can be made so that the cochain itself vanishes, so that there is no obstruction.

\subsubsection{Resolution configurations, ladybug configuration}
\label{sec:ladybug-configuration}

In the Khovanov flow category, the objects are in correspondence with the distinguished generators of the Khovanov chain complex.

\begin{remark}
    This time, I really do mean the \emph{chain} complex, rather than the technically cohomologically graded version we've been discussing all quarter. 
    But, since we are just getting a sense of the Khovanov stable homotopy type construction, I will not spend time worrying about things like whether we should be using $\cubecat$ or $\cubecat^{op}$. 
    Perhaps take this as a \alert{warning} that I am not being careful about these conventions here.
\end{remark}

The 0-dimensional moduli spaces are in correspondence with the components of the differential between distinguished generators (the `arrows'). 

If we think of the `dots and arrows' picture of the Khovanov chain complex as a direct graph, then the 1-dimensional moduli spaces come from paths of length 2
(these are the analogues to once broken flow lines in the Morse picture). 
We capture these moduli spaces by drawing \emph{resolution configurations}.

We will now introduce some vocabulary words in context, via the following very important example.

Consider the cube of resolutions for this two-component unlink diagram:
\begin{center}
    \includegraphics[width=4in]{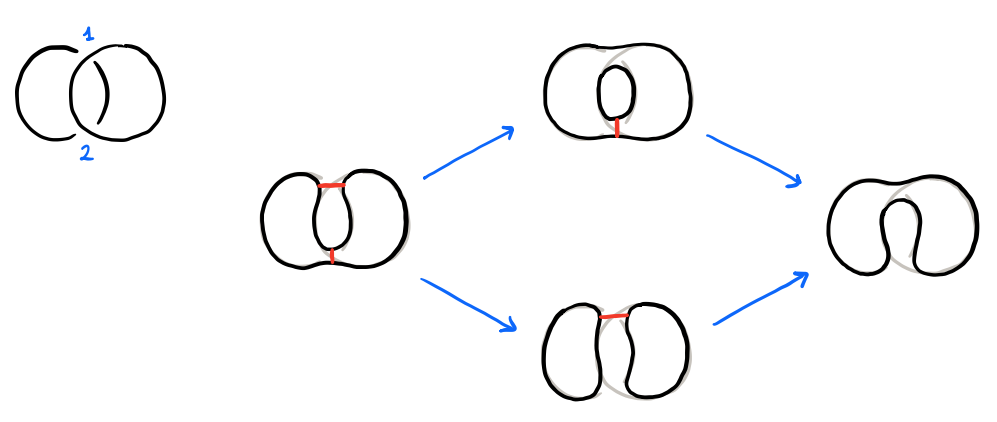}
\end{center}

\begin{itemize}
    \item At vertex $00$ is a complete resolution, along with two red arcs indicating that there are two crossings at which we have yet to switch the resolution to the 1 resolution. 
    \note{In other words, we have yet to \emph{surger} along the red arcs. 
    Yes, we are doing surgery on 1-manifolds.}
    \item At the $10$ and $01$ resolutions, we have already surgered along one red arc, but have yet to surger along the other. 
    \item At the $11$ resolution, all red arcs have been surgered, so there is nothing left to do. 
\end{itemize}

These complete-resolutions-with-red-arcs are called \emph{resolution configurations}. 
These diagrams are useful because they point out 'subcubes', in the sense that the $00$ resolution configuration in the above example can represent the whole 2-dimensional cube that you get from surgering the red arcs in the two different orders.

In general, a resolution configuration need not only describe a cube whose extremal vertex is the all-ones resolution. By omitting some red arcs, one can describe a subcube somewhere in the middle of the cube of resolutions. 
\note{In other words, the 2D cube of resolutions above could appear somewhere within a much bigger cube of resolutions for a much more complicated link diagram.}

A \emph{labeled resolution configuration} additionally specifies a particular distinguished generator. 
For example, for the unlink diagram above, at quantum grading $\gr_q = 0$ we see the following subcomplex:
\begin{center}
    \includegraphics[width=2.5in]{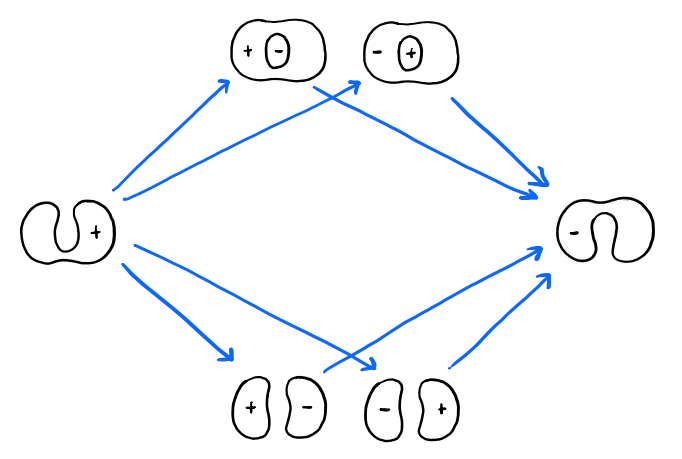}
\end{center}
I can specify this entire subcomplex by drawing
\begin{center}
    \includegraphics[width=1.5in]{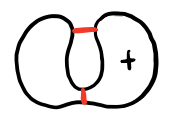}
\end{center}

In other words, a labeled resolution configuration is a pair $(D,x)$ where $D$ is a resolution configuration and $x$ is a labeled of all circles in $D$ with $v_+$ or $v_-$.

Finally, a \emph{decorated resolution configuration} specifies both the beginning and the end of the path.
\note{In the example above, there was a unique distinguished generator at the $11$ resolution that length-2 paths from the shown generator map to. In general, this might not be true. So a decorated resolution configuration also specifies which distinguished generator your path should end at.}

In other words, a decorated resolution configuration is a triple $(D,x,y)$ where $(D,x)$ and $(s(D),y)$ are both labeled resolution configuration; here, $s(D)$ means `result of surgery along the red arcs of $D$'. 

So, if $D$ is a resolution configuration with $k$ red arcs, we say that $(D,x,y)$ is an \emph{index $k$} decorated resolution configuration. 
These should correspond to $(k-1)$-dimensional moduli spaces in the Khovanov flow category. 

Notice that in our example, there were two possible ways that we could build our 1-dimensional moduli spaces:

\begin{center}
    \includegraphics[width=4in]{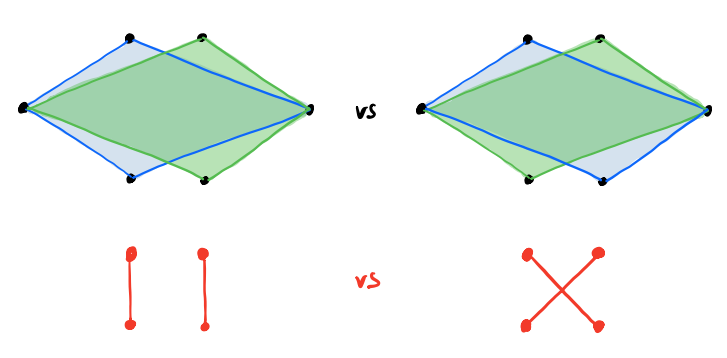}
\end{center}

A `correct choice' (i.e.\ a choice that allows us to obtain a framed flow category) will be determined by the constraints in the following section. 

\subsubsection{Khovanov flow category as a trivial cover}

\begin{definition}
    A grading-preserving functor $\mathscr{F}: \SChat \to \SC$ between flow categories is a \emph{cover} if, 
    for all $x,y \in \SChat$, 
    whenever $\moduli_{\SChat}(x,y) \neq \emptyset$, 
    the map of $\langle \gr(x) - \gr(y)-1 \rangle$ manifolds
    \begin{equation}
        \label{eq:flow-cat-cover-moduli-space-map}
        \mathscr{F}: \moduli_{\SChat}(x,y) \to \moduli_{\SC}(\mathscr{F}(x), \mathscr{F}(y))
    \end{equation}
    \begin{itemize}
        \item respects the $\langle \gr(x) - \gr(y)-1 \rangle$ manifold structures,
        \item is a local diffeomorphism, and 
        \item is a covering map. 
    \end{itemize}

    The cover is \emph{trivial} if the map \eqref{eq:flow-cat-cover-moduli-space-map} is a trivial covering map. 
\end{definition}

\note{Note that the term `trivial' makes it sound like the whole covering situation is somehow trivial; this is not true, because moduli spaces can glue together along their boundaries in many interesting ways. }

Trivial covers are particularly important to the construction at hand, because Lipshitz--Sarkar essentially prove the following:

\begin{proposition}[\cite{LS-stable-homotopy-type}, Proposition 5.2, (E-3), casually stated]
If the Khovanov flow category trivially covers the cube flow category $\SC_C(n)$, then we can use the framing of $\SC_C(n)$ to frame the Khovanov flow category. 
\end{proposition}

We return to the ladybug configuration to see the difference between trivial and nontrivial covers:
\begin{center}
    \includegraphics[width=3.5in]{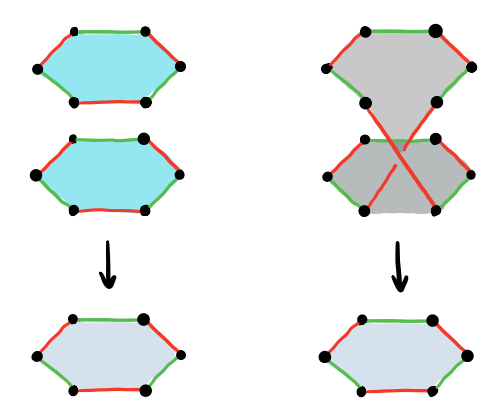}
\end{center}
We assume that all the 0- and 1-dimensional moduli spaces are trivial covers, as shown. 

In order to make sure the covering is trivial for the hexagonal 2-dimensional moduli space, we have to make sure that, on its boundary, we get two 6-cycles (left), rather than a single 12-cycle (right).

Lipshitz--Sarkar cleverly use the orientation of the diagram plane to make a global choice for any ladybug configurations that arise, ensuring that all covering maps for hexagons look like the picture on the left.

\subsubsection{Finishing up}

It turns out that this is the last important choice to make! 
Moduli spaces of dimension $k \geq 3$ are copies of the permutahedron $P_{k+1}$, which is homeomorphic to  $B^k$ (recall that it's the convex hull of a bunch of points). 
This means the boundary of the moduli space is homeomorphic to $S^{k-1}$, which is simply connected, and therefore only have trivial covers!
\note{Recall that connected covering spaces of $X$ correspond to subgroups of $\pi_1(X, *)$.}

To finish up:
\begin{itemize}
    \item We can build our moduli spaces inductively, and frame them using our framing of the cube flow category. 
    This gives us a framed flow category refining the Khovanov chain complex we started with.
    \item Next, use the Cohen--Jones--Segal construction to build a high-dimensional CW-complex from the framed flow category, called the \emph{realization}. 
    \item Finally, take the suspension spectrum of the realization (and formally desuspend appropriately to get the correct homological gradings) to get the link invariant $\XKh$. 
\end{itemize}

\begin{remark}
Actually, we get a spectrum for each quantum grading $\gr_q = j$:
\[
    \XKh(L) = \bigvee_{j} \XKh^j(L)
\]
\end{remark}

\subsection{The Burnside functor approach}

The reference for this section is \cite{LLS-burnside-products}.

Lawson--Lipshitz--Sarkar give a different way to build equivalent Khovanov spectra, via the following process:
\begin{itemize}
    \item Start with the data of the Khovanov complex (but not totalized), drawn as a (commuting) diagram of abelian groups and maps between them. 
    If a link diagram has $n$ crossings, then we can think of this diagram as a functor $\cubecat^n \to \Zmod$. 
    \item Make some choices (i.e.\ add information) to lift this to a functor $\cubecat^n \to \Burn$, where $\Burn$ is the \emph{Burnside 2-category}. 
    \item Use the \emph{Pontrjagin-Thom} construction to turn this data into a functor $(\cubecat^{op})^n \to \Topcat_*$.
    \item Check that this diagram of topological spaces (actually, wedges of spheres of some large dimension) is \emph{homotopy coherent}, and then take the homotopy colimit to obtain a space. (Then, as usual, take the suspension spectrum and desuspend to make up for the large dimension of the spheres you used.)
\end{itemize}

While we sketch out `Burnside-functor' approach to building $\XKh$, look out for the following important points we encountered while discussing the flow categories construction method:
\begin{itemize}
    \item What elements of this construction correspond to the higher-dimensional moduli spaces?
    \item Where do we make the ladybug matching choice?
    \item At what index (for the resolution configurations) can be stop making choices, and why?
\end{itemize}

\subsubsection{The Burnside category}

The \emph{Burnside category} $\Burn$ is a 2-category, meaning that there are objects, morphisms, and also morphisms between morphisms (level 2 homomorphisms).
In $\Burn$:
\begin{itemize}
    \item Objects are finite sets $X$
    
    \item 1-morphisms are \emph{finite correspondences}: a morphism in $\Hom(X,Y)$ is a diagram of set maps
    \begin{center}
    \begin{tikzcd}
        & A \arrow{dl}[swap]{s_A} \arrow{dr}{t_A} & \\
        X & & Y
    \end{tikzcd}
    \end{center}
    where $s_A$ and $t_A$ are called the \emph{source} and \emph{target} set maps, respectively. 

    \note{You can think of the elements of $A$ as `arrows' form points in $X$ to points in $Y$; an arrow $a \in A$ has a source point and a target point, and $s_A(a)$ and $t_A(a)$ are just telling us which points these are.}

    \item Composition of 1-morphisms is given by \emph{pullback} (i.e.\ \emph{fiber product}):
    the composition of $(A,s_A, t_A) \in \Hom(X,Y)$ and $(B, s_B, t_B) \in \Hom(Y,Z)$ is the finite correspondence $(C, s_A \circ s_C, t_B \circ t_C) \in \Hom(X,Z)$ 
    \begin{center}
    \begin{tikzcd}
        & & C \arrow{dl}[swap]{s_C} \arrow{dr}{t_C} & & \\
        & A \arrow{dl}[swap]{s_A} \arrow{dr}{t_A} & & B \arrow{dl}[swap]{s_B} \arrow{dr}{t_B} \\
        X & & Y & & Z 
    \end{tikzcd}
    \end{center}
    where 
    \[
        C = A \times_Y B 
        = \{ (a,b) \in A \times B \st t_A(a) = s_B(b).
    \]
    \note{The maps $s_C$ and $t_C$ are induced by projection from $A \times B$, i.e.\ $s_C((a,b)) = a$ and $t_C((a,b)) = b$.}

    \item 2-morphisms are \emph{isomorphisms of correspondences}: 
    given 
    $(A,s_A, t_A)$ and $(B, s_B, t_B)$ in $\Hom(X,Y)$, 
    an element of $2\Hom(A,B)$ is an isomorphism of sets $f: A \to B$ such that the following diagram commutes:
    \begin{center}
    \begin{tikzcd}
        & A \arrow{r}{f (\cong)} \arrow{dl}[swap]{s_A} \arrow{drr}[swap]{t_A} & 
            B \arrow{dll}{s_B} \arrow{dr}{t_B}
            & \\
        X &  & & Y
    \end{tikzcd}
    \end{center}

\end{itemize}

\subsubsection{Cube-to-Burnside functor}

Given a Khovanov cube-shaped complex $\cubecat^n \to \Zmod$, we want to build a lift of this functor to $\cubecat^n \to \Burn$. 
Here is how we define the new functor.

\textbf{(Objects)} At each vertex $u$ of the cube, the distinguished generators (pure tensors in $V^{\otimes |\pi_0(D_u)|}$) become the elements of the finite set at $u$.

\textbf{(1-Morphisms)} The arrows in the Khovanov differential between distinguished generators determine the finite correspondences. 
    
For example, consider the merge map $m: V \otimes V \to V$.
\begin{center}
    \includegraphics[width=4in]{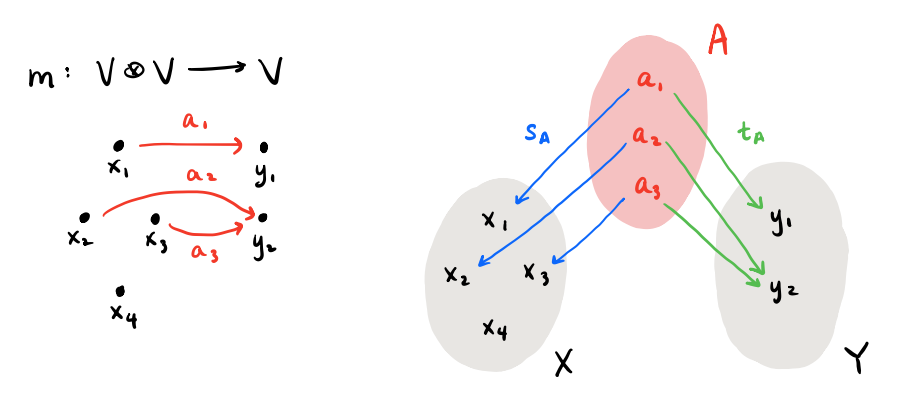}
\end{center}
On the left is our familiar dots-and-arrows picture, but now each arrow has a name. 
On the right is the \note{... corresponding...} correspondence in the Burnside category. 

\textbf{(2-Morphisms)} 
\note{To define a 2-functor from $\cubecat^n$, we technically need to think of $\cubecat^n$ as a 2-category as well. We do this by asserting that $\cubecat^n$ has no non-identity 2-morphisms.}

Defining our 2-morphisms is where new choices are made. 
For an example, we will once again study the ladybug configuration (see \S \ref{sec:ladybug-configuration}).

Here is the usual dots-and-arrows picture of the chain complex (at the interesting quantum grading), but now the arrows have names:

\begin{center}
    \includegraphics[width=4in]{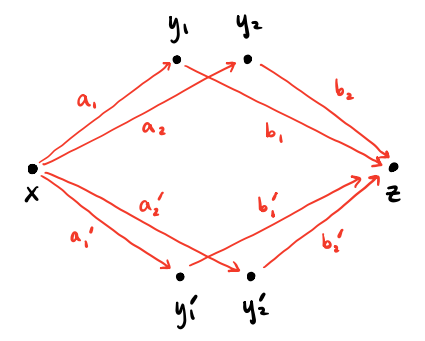}
\end{center}

Translating to a diagram in the Burnside category, we have the following picture, where a blue lines are parts of source maps, and green lines are parts of target maps.

\begin{center}
    \includegraphics[width=4in]{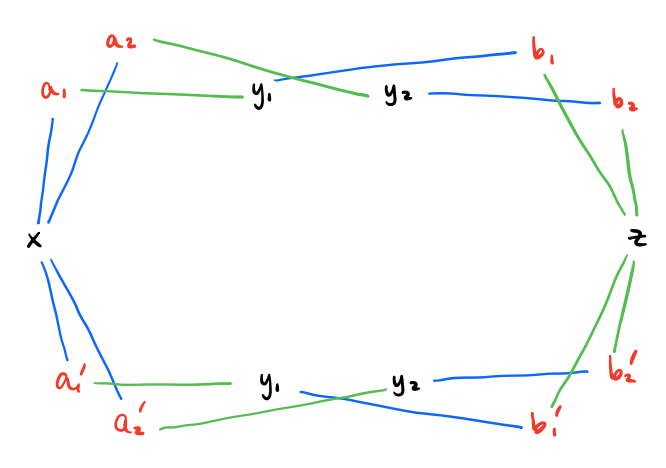}
\end{center}

Now let's consider the 1-morphisms we get from length-2 paths through the cube; these are compositions of the arrows in the dots-and-arrows picture.

\begin{center}
    \includegraphics[width=4in]{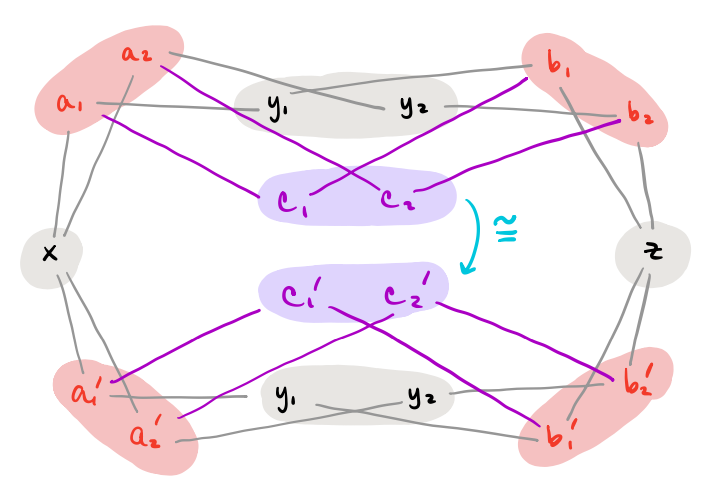}
\end{center}

\begin{itemize}
    \item Along  the top half of the cube, the composition of 1-morphisms $A = \{a_1, a_2\}$ and $B = \{b_1, b_2\}$ yields a 1-morphism $C \in \Hom(X = \{x\}, Z = \{z\})$ containing two elements,
        \[
            c_1 = (a_1, b_1) 
            \quad \text{and} \quad 
            c_2 = (a_2, b_2).
        \]
    \item Along the bottom half of the cube, we get a different composite 1-morphism $C' \in \Hom(X,Z)$ also containing two elements,
        \[
            c_1' = (a_1', b_1') 
            \quad \text{and} \quad 
            c_2' = (a_2', b_2')
        \]
\end{itemize}
An element of $2\Hom(C, C')$ is a choice of isomorphism between these two 2-element sets. There are exactly two possible isomorphisms (either $c_1 \mapsto c_1'$ or $c_1 \mapsto c_2'$). 

Lawson--Lipshitz--Sarkar check through all the combinatorics in the Khovanov cube, and this resolution configuration is, once again, the most interesting. 

\subsubsection{Burnside functor to homotopy coherent diagram of wedges of spheres}

We now use the Burnside functor $\cubecat^n \to \Burn$ to obtain a diagram of topological spaces, using the Pontrjagin-Thom construction. 
We will stick with our running ladybug example. 

First, replace all the objects in the Burnside functor diagram with $k$-dimensional disks, with the size of the disks increasing with the cube grading $|u| = \sum_i u_i$. 
In the drawing below, $k=2$:

\begin{center}
    \includegraphics[width=4in]{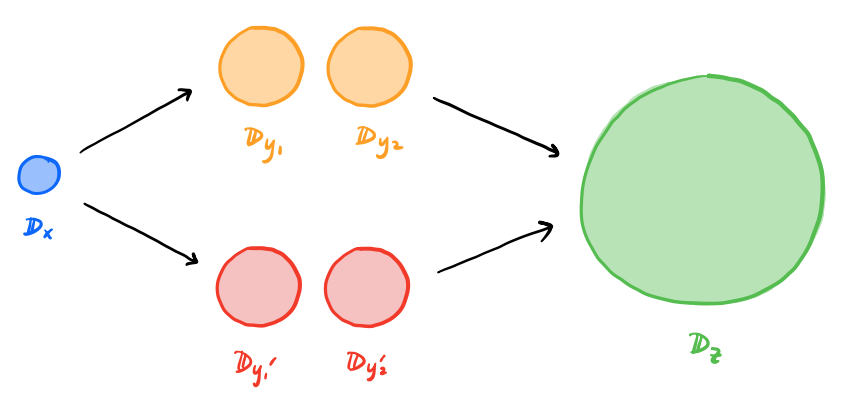}
\end{center}

The correspondences tell us where to embed smaller disks into larger disks, so that images are disjoint.

\alert{Beware:} \note{ The following picture is just for showing you how we're embedding the smaller disks into bigger ones. 
For example, I've drawn the image of the embeddings $\mathbb{D}_x \into \mathbb{D}_{y_i}$ ($i = 1,2$), 
in order to indicate a \emph{multi-embedding} (see Remark \ref{rmk:multi-embedding}) of 
$\mathbb{D}_x$ in $\mathbb{D}_{y_1} \sqcup \mathbb{D}_{y_2}$. 
}

\begin{remark}
\label{rmk:multi-embedding}
    \alert{Beware} that what I've drawn is not a map of topological spaces. This is not even a function, but rather a \emph{multifunction}, where each point maps to multiple points. 
    For this reason, I will be referring to these disk embeddings as `multi-embeddings'. 
    \note{I don't believe this is standard language. 
    For more technical language, you can look up the \emph{little boxes operad}. }
\end{remark}

\begin{equation}
\label{eq:pont-thom-disks-embedded}
    \includegraphics[width=4in]{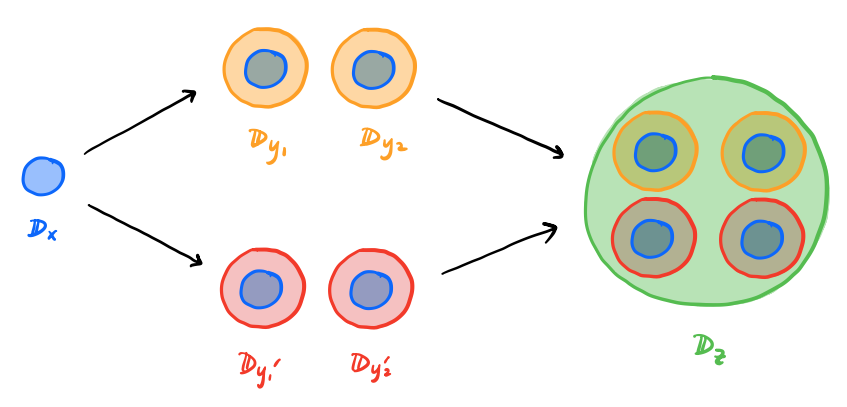}
\end{equation}

We are now set up to build a functor $(\cubecat^{op})^n \to \Top_*$. 
First, at each vertex of the cube, identify all of the boundaries to a single basepoint $*$:

\begin{center}
    \includegraphics[width=4in]{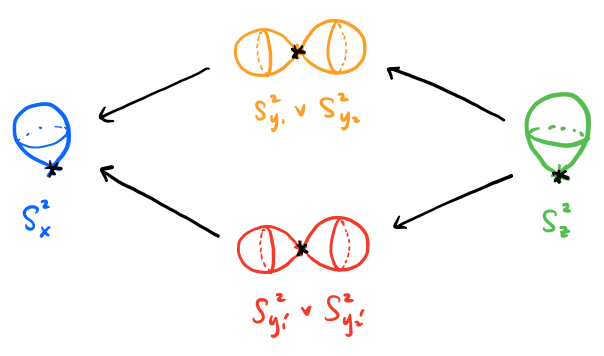}
\end{center}

We can get a map from $S_{y_1}^2 \to S_x^2$ by collapsing everything outside of the image of $\mathbb{D}_x \into \mathbb{D}_{y_1}$ to the basepoint in $S_x^2$. 
The same process gives us a map $S_{y_2}^2 \to S_x^2$. 
Combined, we obtain a map of based topological spaces $S^2_{y_1} \vee S^2_{y_1} \to S^2_{x}$.
The map from $S^2_{z} \to S^2_{y_1} \vee S^2_{y_1}$ is slightly more interesting, but defined the same way.
\note{Try to visualize this yourself, by imagining that you are, perhaps, making a balloon animal, grabbing a ball of dough from a large mass of dough, or balling up some mozzarella from a vat of curds.}

Even more interesting is the composition of embeddings.

\begin{itemize}
\item 
Along the top of the cube in \eqref{eq:pont-thom-disks-embedded}, we get two homeomorphic images of $\mathbb{D}_x$ inside $\mathbb{D}_z$. These are the blue disks inside the orange circles. 

The composite multi-embedding of $\mathbb{D}_x$ to $\mathbb{D}_z$ agrees with the collapse map $S^2_z \to S^2_x$
: we crush everything outside of the two images of  $\mathbb{D}_x$ to the basepoint, and get a degree-2 map   $S^2_z \to S^2_x$. 

\item Along the bottom of the cube, we get a different set of homeomorphic images of $\mathbb{D}_x$ inside $\mathbb{D}_z$. These are the blue disks inside the red circles. 

This path gives us a different degree-2 map $S^2_z \to S^2_x$. 
\end{itemize}

In order to make sure that the two maps of topological spaces $S^2_z \to S^2_x$ are homotopic, we need the dimension $k$ to be large enough (we will discuss more about this soon).
But because this square is potentially only a small part of a much larger cube, we actually need to \emph{pick specific homotopies} between the two maps. 
To do this, we use the 2-morphisms we chose when we built the Burnside functor. 

From the Burnside functor, we get a particular matching of the orange-blue disks (associated with $C = \{c_1, c_2\}$ in the Burnside functor)
with the red-blue disks (associated with $C' = \{c_1', c_2'\}$). 
Suppose we chose the matching $c_i \mapsto c_i'$. 
To choose a homotopy, we just need to pick two paths: one that takes the orange-blue disk associated with $c_1$ to the red-blue disk associated with $c_1'$, 
and another that takes the orange-blue disk associated with $c_2$ to the red-blue disk associated with $c_2'$

After choosing these homotopies, we also need to choose homotopies among homotopies when we zoom out and look at an index-3 subcube (or index-3 resolution configuration), and then homotopies of homotopies of homotopies ... etc. 
If we can successfully find homotopies at all levels all the way to index-$n$ resolution configurations (recall that our link diagram has $n$ crossings), then we have a \emph{homotopy coherent} diagram of topological spaces. 

We want to make as few choices as possible.
To do this, we just make sure $k$ is sufficiently large, because this will ensure that all possible choices are effectively equivalent (i.e.\ homotopic). 
Observe that the data of a multi-embedding $e: \mathbb{D}_x \to \mathbb{D}_y$ can be captured by just recording the centers of the image disks. 
Suppose there are $m$ image disks.
Then, $e$ is just an element of $\Conf(\mathbb{D}_y, m\text{ points})$, the configuration space of $m$ points inside the ($k$-dimensional) disk $\mathbb{D}_y$.

\begin{fact}
The configuration space of finitely many points in a $k$-dimensional disk (ball) is $(k-2)$-connected, i.e., 
for any $j \leq k-2$,
\[
    \pi_{j}(\Conf(\mathbb{D}^k, m \text{ points}) = 0. 
\]
\end{fact}

Using this fact, we see that as long as $k \geq n+2$, no matter how we choose to multi-embed our disks, our diagram will be homotopy coherent.

At this point, we can proceed to the last step, where we use this diagram to build a space, by throwing our homotopy coherent diagram into a machine that spits out the \emph{homotopy colimit}. 

\subsubsection{Homotopy colimit}

If you want to see a concrete model of a homotopy colimit of based topological spaces, you can check out the reference for this whole subsection, \cite{LLS-burnside-products}. 

Instead, let's talk about the significance of homotopy colimits in general, given  that we are currently in possession of a homotopy coherent diagram of topological spaces $(\cubecat^{op})^n \to \Topcat_*$. 

When we work with topological spaces, we really consider them up to homeomorphism. 
When I take a colimit of a diagram of topological spaces, the resulting space should be well-defined up to homeomorphism.

On the other hand, in homotopy theory, we want to consider spaces up to homotopy equivalence.
When I take a colimit of a diagram of topological spaces, I should expect that if I replace one of the spaces in the diagram with something homotopy equivalent to it, that I should get the same colimit. 
But this is not always true, as evidenced by the following classic example.

Consider the two diagrams below ($\iota_\partial$ means `inclusion of boundary'):

\[
\begin{tikzcd}
    & S^{1} \arrow{dl}{\iota_\partial} \arrow{dr}{}& \\ 
    D^2 & & * \\
\end{tikzcd}
\qquad
\qquad
\begin{tikzcd}
    & S^{1} \arrow{dl}{} \arrow{dr}{}& \\ 
    * & & * \\
\end{tikzcd}
\]

In the second diagram (the one on the right), we've replaced $D^2$ with a homotopy equivalent space, $*$.

\begin{remark}
For this particular shape of diagram, the colimit is called a \emph{pushout}. 

Let's briefly recall how pushouts work in the category of sets, for simplicity. 
Given a diagram of sets
\begin{tikzcd}
    A  \arrow{r}{\beta} \arrow{d}{\gamma}
        & B \\
    C & \\
\end{tikzcd}
the pushout is the set $B \sqcup C / \sim$ where $b \sim c$ iff there exists some $a \in A$ such that $\beta(a) = b$ and $\gamma(a) = c$.
Colimits are a generalization of this.
\end{remark}

The colimit of the diagram on the left is $S^2$, where as the colimit of the diagram on the right is a one-point space:

\begin{center}
    \includegraphics[width=4in]{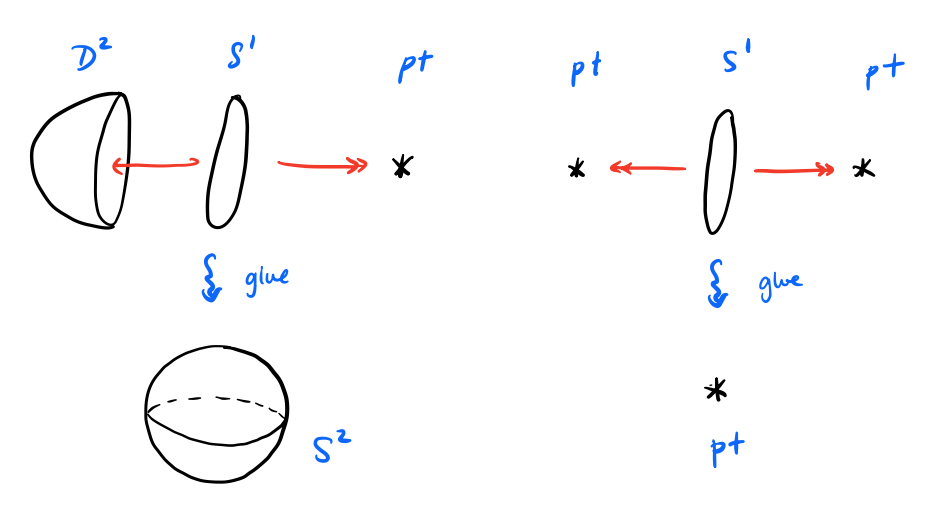}
\end{center}

These spaces are not homotopy equivalent! 
The problem is that the colimit on the right is too degenerate. 

To fix this, we thicken the pieces before we glue:

\begin{center}
    \includegraphics[width=4in]{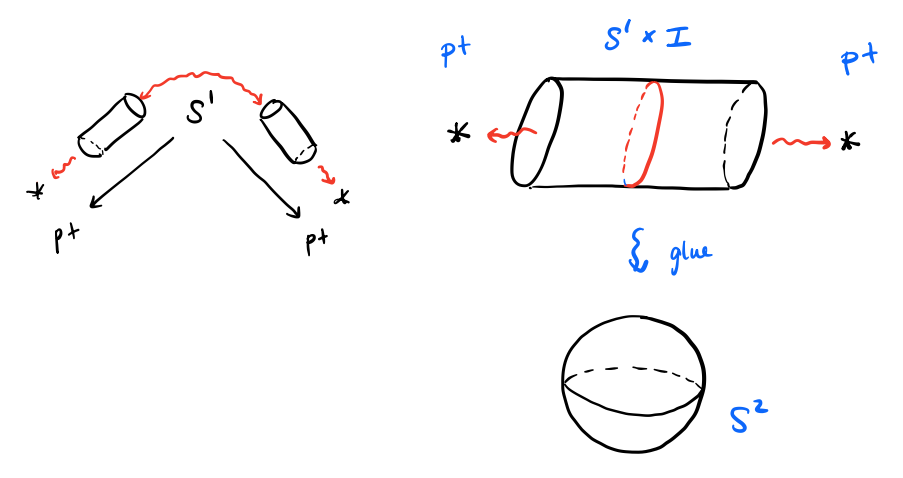}
\end{center}

This time, for each map $S^1 \to *$, we first replace $S^1$ with $S^1 \times I$, and then glue one end to $*$. Then, we glue the two pieces together.

On the right, we see that after gluing, we indeed get an $S^2$, which now agrees with the colimit of the first diagram we saw. 

This tells us that $S^2$ is the homotopy colimit (or homotopy pushout) of both diagrams.

The takeaway is that the Burnside functor approach to building $\XKh$ is meant to produce a \emph{stable homotopy type} rather than a homeomorphism class of topological spaces.
This flexibility allows us to actually obtain an invariant of the link (under Reidemeister moves of the diagram). 
So, when we build the space, we need to treat everything up to homotopy.

\section{Homological Skein Invariants for 4-manifolds}

This section serves a motivational introduction toward skein lasagna modules, which were introduced by Morrison--Walker--Wedrich in \cite{MWW-lasagna}. 
We will be brief on details; a forthcoming set of notes written with Mike Willis will focus more heavily on homological skein invariants; see \url{https://indico.ictp.it/event/10839}.

\subsection{4-dimensional exotica}

This section is based on Tye Lidman's 2015 IAS 15-minute talk \cite{lidman-postdoc-talk-IAS-4-manifolds}.

Suppose we are given two smooth 4-manifolds. How do we tell whether they are \emph{diffeomorphic}, i.e.\ smoothly equivalent? 

\begin{remark}
    Just as one can be `given' a smooth knot via a knot diagram or as a braid closure (or Gauss code, PD code, etc.), there are different ways a 4-manifold can be presented. 
    We will use Kirby diagrams, which essentially describe a handle decomposition of the smooth 4-manifold.
\end{remark}

In 3D, there is little difference among the smooth and a priori less smooth (e.g.\ PL, topological) categories.
Indeed, Reidemeister's theorem was actually proven by studying PL knot diagrams!

However, in 4D, the categories are very different.
In fact, 4D seems to be the weirdest dimension. Compare the following two theorems:
\begin{itemize}
    \item Stallings \cite{Stallings-PL-euclidean} proved that, for $d \neq 4$, if $X$ is homeomorphic to $\R^n$ 
    \note{(i.e.\  topologically equivalent to $\R^n$)}, then $X$ is in fact \emph{diffeomorphic} to $\R^n$ 
    \note{(i.e.\ smoothly equivalent to $\R^n$)}. \item Taubes \cite{Taubes-continuum} proved that there exist \emph{uncountably many} smooth manifolds homeomorphic to $\R^4$, up to diffeomorphism! 
    (See also \cite{Gompf-R4} and \cite{Freedman-Taylor}.)
\end{itemize}

\begin{notation}
Let us set some notation and terminology for the remainder of this section. 
\begin{itemize}
    \item If we consider a manifold only up to its homeomorphism class, we call it a \emph{topological manifold}. 
    If $X$ and $Y$ are homeomorphic, we write $X \homeo Y$. 
    \item If we consider a smooth manifold up to its diffeomorphism class, we call it a \emph{smooth manifold}. 
    If $X$ and $Y$ are diffeomorphism, we write $X \diffeo Y$. 

\end{itemize}
\end{notation}

Let's first restrict ourselves to \textbf{closed, connected, oriented, simply-connected 4-manifolds. }
To determine whether two such smooth manifolds, $X_1$ and $X_2$, are diffeomorphic, we can first forget their smooth structure and check if they're homeomorphic. 

Fortunately, Freedman proved that the topological 4-manifold invariant called the \emph{intersection form} classifies these manifolds. 
So, if the intersection forms are inequivalent, then $X_1$ and $X_2$ are certainly not diffeomorphic, since they aren't even homeomorphic.

So, it remains to distiguish between pairs of smooth and homeomorphic manifolds. 
In general, if we have two homeomorphic but non-diffeomorphic objects, we call them an \emph{exotic pair}. 

Here is the most famous open question about exotic manifolds:
\begin{conjecture}
    (Smooth Poincar\'e Conjecture in Dimension 4)
There are no exotic $S^4$'s. 
In other words,
\[
    X^4 \simeq S^4 
    \qquad
    \implies
    \qquad 
    X^4 \cong S^4.
\]
\end{conjecture}

\subsection{Exotic $\R^4$'s via Rasmussen's invariant}

To motivate the use of Khovanov homology in developing smooth 4-manifold invariants, we first show how Rasmussen's invariant can be used to construct exotic $\R^4$'s.

\subsubsection{Relating 4D topology and knot theory}

\begin{definition}
Let $K$ be a smooth knot in $S^3$.
The \emph{$0$-trace} of $K$, denoted $X_0(K)$, is the  manifold obtained by attaching a 4-dimensional $0$-framed 2-handle to $B^4$ along $K$.
\note{Recall that the $0$-framing is the Seifert framing.}
This is a smooth 4-manifold whose closed 3-manifold boundary is $S_0^3(K)$, the $0$-surgery of $S^3$ along $K$. 
\end{definition}

The \emph{Trace Embedding Lemma} below is the powerful lemma that will allow us to translate our knowledge of knots (e.g.\ Rasmussen's invariant from Khovanov homology)
to knowledge about 4-manifolds.

It is true in both the smooth and \emph{topologically locally flat} categories. 
Roughly speaking, a surface $F$ is topologically locally flatly embedded in a 4-manifold $W$ if 
for every point $x \in F$, there is a neighborhood $U$ of $x$ and a chart that identifies $U$ with a neighborhood of $\R^2 \subset \R^4$. 
\note{In other words, locally, it can be make to look flat.}

\begin{lemma}[see \cite{Piccirillo-conway-knot}, Lemma 1.3] (Trace embedding lemma)
A knot $K$ embedded in $S^3$ is smoothly slice (resp.\ topologically locally flatly)
\textbf{if and only if}
$X_0(K)$ embeds smoothly (resp.\ topologically locally flatly) into $S^4$.
\end{lemma}

\begin{remark}
    Note that smooth embeddings are automatically topologically locally flat embeddings. 
    Because the latter term is too long, we usually just say `topologically' or `locally flatly' embedded.

    \alert{However}, do not take this to mean `homeomorphically'. For example, for any knot $K$, there is a PL-disk in $B^4$ with boundary $K$: just cone $K$ to a point. In general, the cone point of this disk is not locally flat!
\end{remark}

\begin{proof}[Sketch of proof of Trace Embedding Lemma]

\note{The proof is similar in both categories; just replace the word `smooth' with `locally flat' throughout.}

\fbox{$\Longrightarrow$}

Suppose $K$ is slice; then there exists a slice disk $D$:

\begin{center}
    \includegraphics[width=3in]{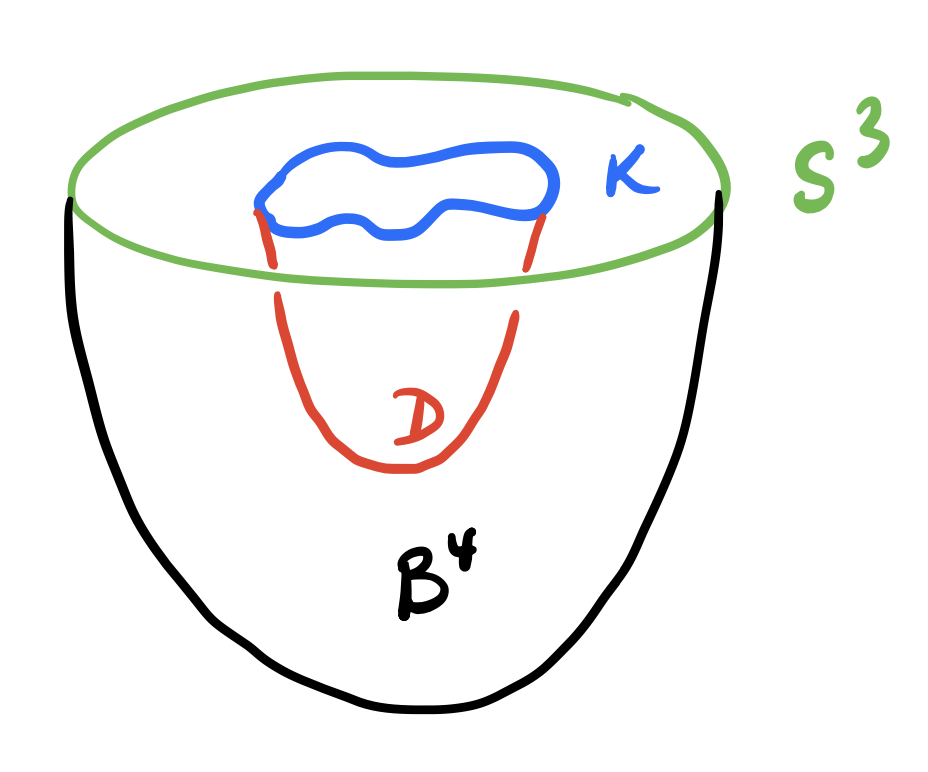}
\end{center}

Attaching a neighborhood of the disk to a $B^4$ along a neighborhood of $K$ as shown below gives us an embedded $X_0(K)$ inside $S^4$:

\begin{center}
    \includegraphics[width=2.5in]{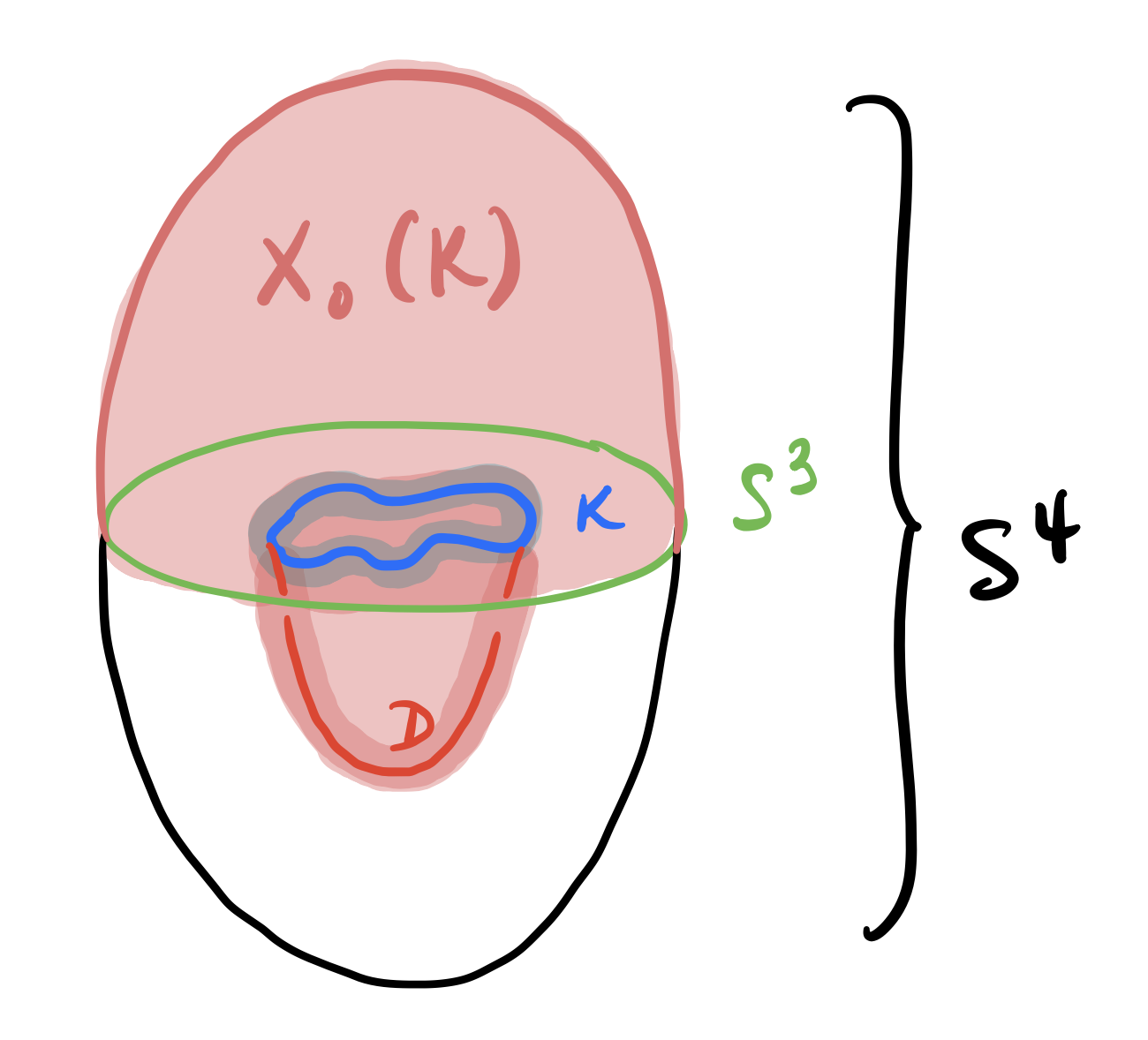}
\end{center}

\note{Note that the framing of the attached handle is indeed the surface framing $0$; this is clear if you are familiar with framings as obstruction classes. 
To see this more intuitively, you could imagine floating $D$ up to the $S^3$ boundary to get an immersed surface with boundary $K$, and then modify the self-intersections to get some (higher genus) Seifert surface for $K$.}

\fbox{$\Longleftarrow$}

Now assume that there exists an embedding $\phi: X_0(K) \into S^4$. 
Consider the cartoon below. 

\begin{center}
    \includegraphics[width=2in]{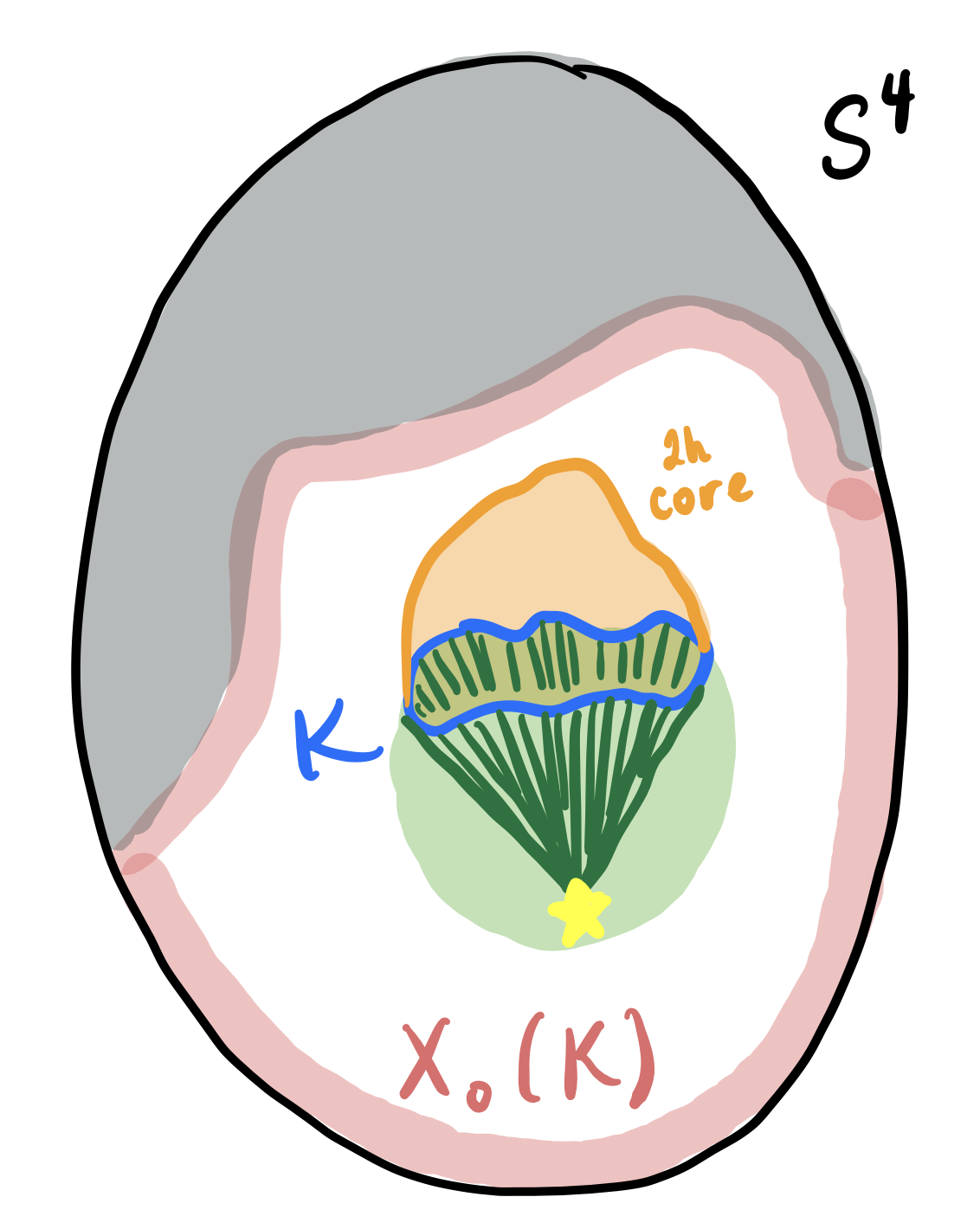}
\end{center}

\begin{itemize}
    \item The orange is the embedded image of the core of the 4D 2-handle. 
    \item The light green is the embedded image of the 4D 0-handle. This embedding might not look as nice is as it is drawn! 
    \item The dark green is the embedded image of a cone on $K$; it is smooth away from the yellow star, the cone point $p$.
\end{itemize}

Let $\nu(p) $ be a very small neighborhood of the cone point $p$. 
\note{This neighborhood looks like a standard $B^4$.}

Then, the intersection of the dark green cone and $\partial \nu(p) \cong S^3$ is a copy of the knot $K$.
Also, the dark green cone, outside of $\nu(p)$, together with the core of the 2-handle, form a slice disk for this copy of $K$. 
\end{proof}

\subsubsection{Detour: Satellites and cables}

\begin{definition}
    Let $P$ be a link in the solid torus $S^1 \times D^2$, and let $K$ be a knot in $S^3$. 
    Let $\nu(K)$ be a \emph{tubular neighborhood} of $K$ (a thickening of $K$).

    Both $S^1 \times D^2$ and $\nu(K)$ are solid tori. 
    The former comes with a canonical framing, and $\nu(K)$ has the canonical Seifert framing. 
    
    Let $\phi: S^1 \times D^2 \to \nu(K)$ be a  diffeomorphism that preserves our canonical framings. 
    
    The \emph{satellite} $P(K)$ is the image of $K$ under $\phi$. 

    In the context of the satelliting construction, we call $P$ the \emph{pattern} and $K$ the \emph{companion knot}. 

    We can also choose a different framing for the satelliting procedure by pre-composing $\phi$ with some Dehn twists. 
\end{definition}

\begin{example}
\label{eg:whitehead-double-trefoil}
Here is an example pattern $P = \Wh$ and companion $K = 3_1$:
    \begin{center}
    \includegraphics[width=4in]{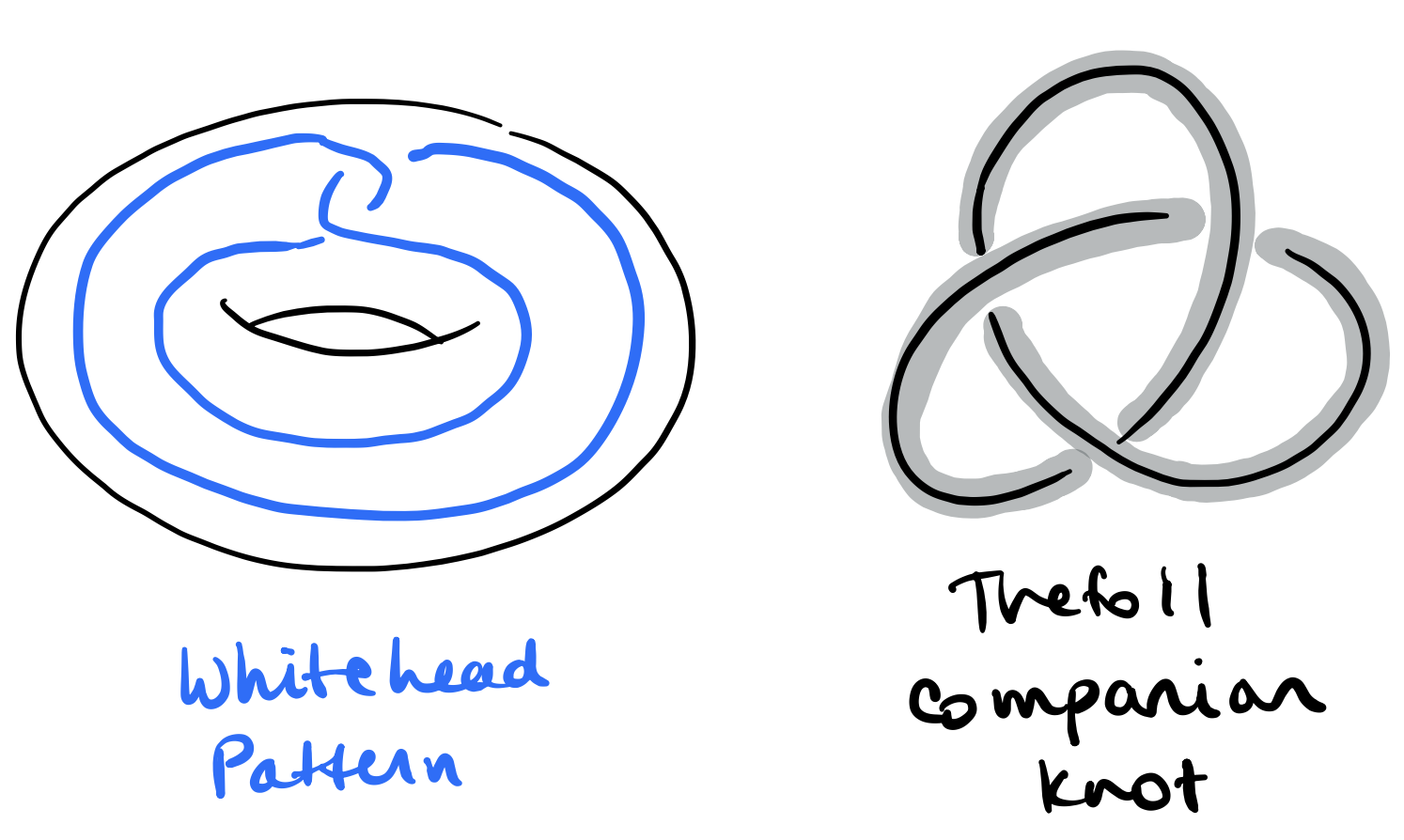}
    \end{center}

The trefoil diagram shown above has writhe $w=3$. 
It turns out that this means that the blackboard (flat-on-paper) framing is 3 more (positive Dehn twists) than the Seifert framing. 

Here is the blackboard framing:
\begin{center}
    \includegraphics[width=2in]{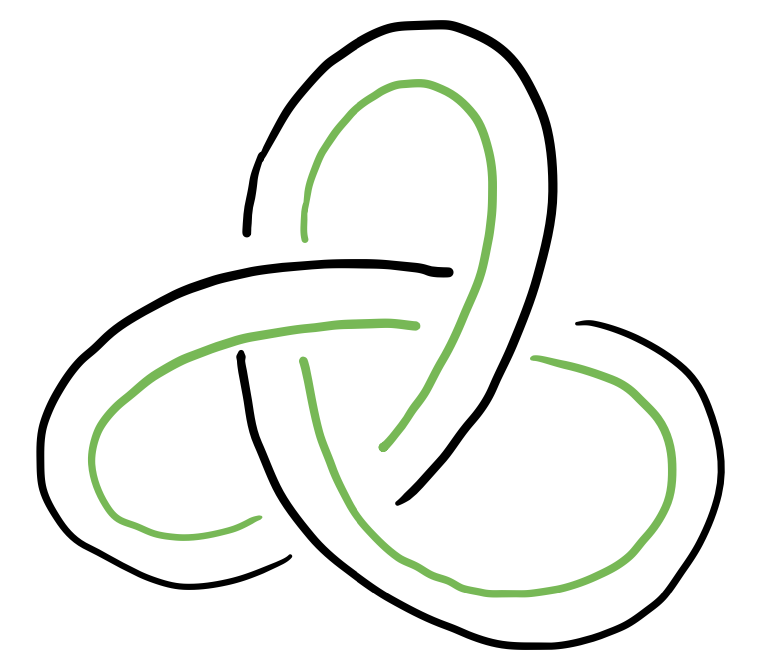}
\end{center}

Here is a Seifert surface for the trefoil:
\begin{center}
    \includegraphics[width=2in]{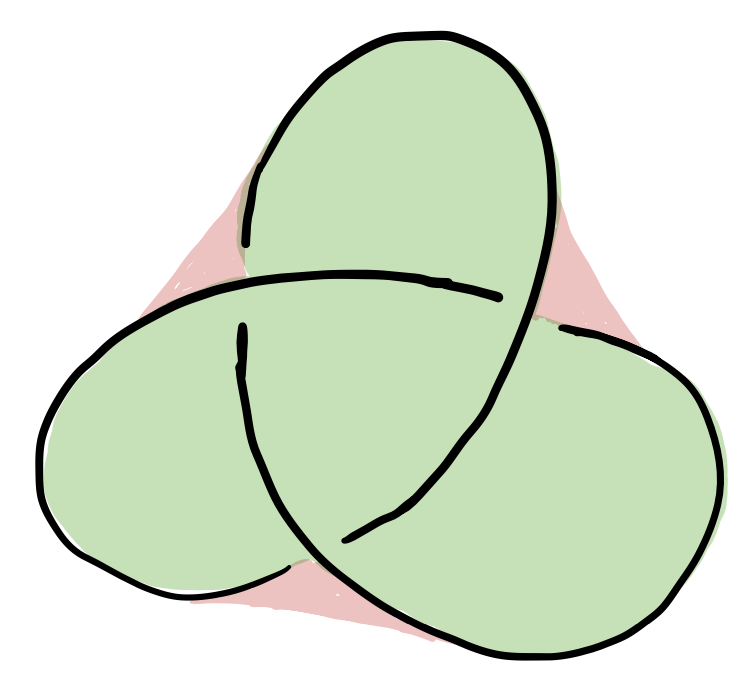}
\end{center}

Here is the Seifert-framed pushoff $K^+$ in blue: 
\begin{center}
    \includegraphics[width=2in]{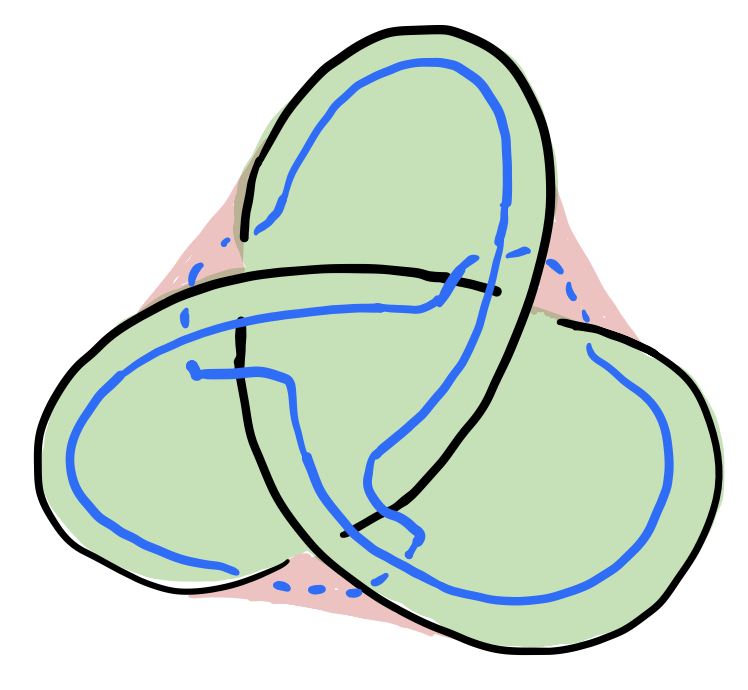}
\end{center}

Here are $K$ and $K^+$ without the Seifert surface:
\begin{center}
    \includegraphics[width=2in]{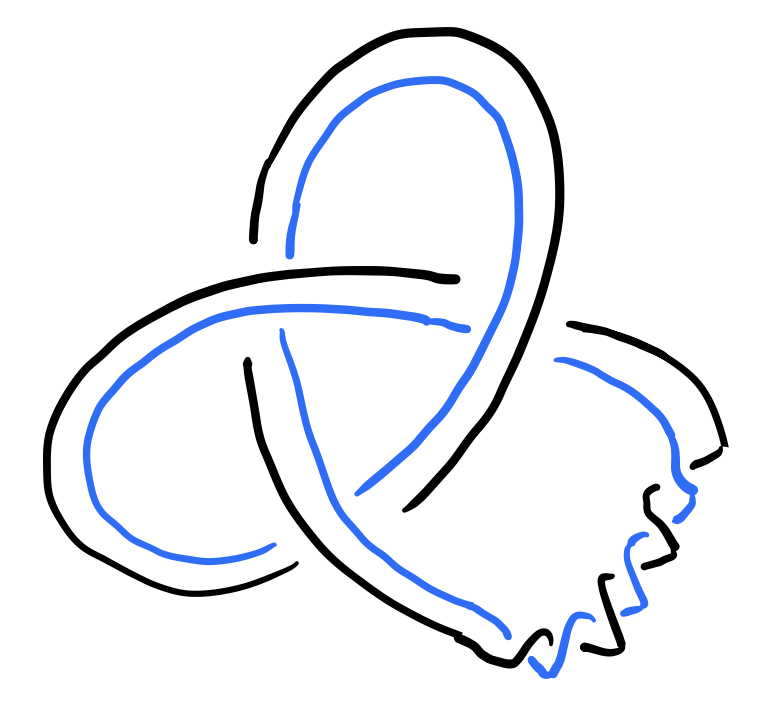}
\end{center}

We can use the following shorthand for the above diagram:
\begin{center}
    \includegraphics[width=2in]{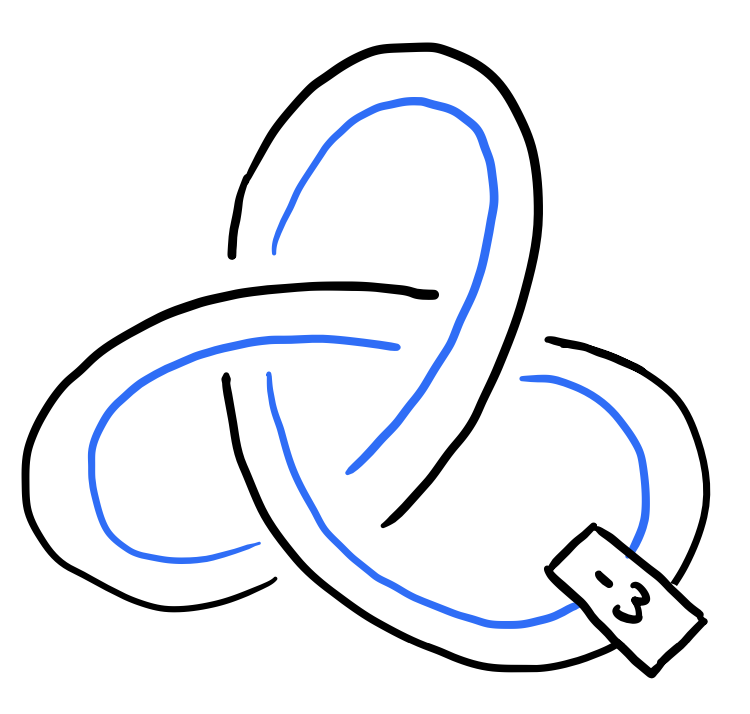}
\end{center}

So, the (Seifert-framed) satellite $\Wh(3_1)$ is the following knot:

\begin{center}
    \includegraphics[width=2in]{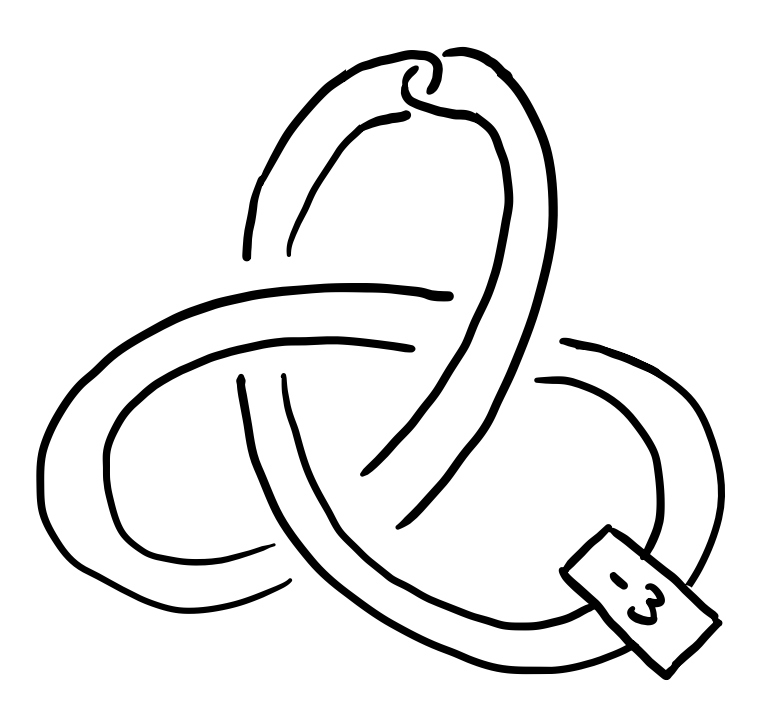}
\end{center}
\end{example}

\subsubsection{Back to exotic $\R^4$'s}

To build an exotic $\R^4$, we will make use of a knot $K$ that is 
\begin{itemize}
    \item topologically slice 
    \item smoothly \emph{not} slice. 
\end{itemize}

We can find such knots by applying the following criteria:
\begin{itemize}
    \item The Alexander polynomial of $K$ is trivial: $\Delta_K(t) \overset{\cdot}{=} 1$.
    \note{Freedman showed that if $\Delta_K(t) \overset{\cdot}{=} 1$, then $K$ is topologically slice \cite{Freedman-Quinn-book}.}
    
    \item The Rasmussen invariant of $K$ is nontrivial: $s(K) \neq 0$. 
    \note{Rasmussen showed, as we know, that $s(K) \neq 0$ means $K$ is \emph{not} smoothly slice \cite{Rasmussen-s}.}
\end{itemize}

For example, one could choose the 2-twisted
positive Whitehead double of the right-handed trefoil, $\Wh^{+2}(3_1)$. 

If you know how to compute the Alexander polynomial, you can fairly quickly convince yourself that the Whitehead double of \emph{any} knot has trivial Alexander polynomial. (Actually, any framing will work too; see Exercise \ref{ex:whitehead-alexander-trivial}.)

On the other hand, Whitehead doubling generally complicates knots, and complicated knots are less likely to be smoothly slice. 
Indeed, Hedden--Ording showed that $s(\Wh^{+2}(3_1)) \neq 0$ (Theorem 1.1 of \cite{Hedden-Ording-s-tau}).

\begin{exercise}
\label{ex:whitehead-alexander-trivial}
Let $K$ be a knot in $S^3$, and 
let $\Wh^n(K)$ be the $n$-framed Whitehead satellite of $K$.
The Alexander polynomial of $\Wh^n(K)$ is trivial
(i.e.\ $\Delta_{\Wh^n(K)}(t) \doteq 1$). 
\end{exercise}

The following can be considered a corollary of Hedden--Ording's example, or of Rasmussen's slice genus bound. 
The proof uses some facts we haven't discussed in class, though.

\begin{corollary}
There exist exotic $\R^4$'s. 

\begin{proof}
(This proof follows Shintaro Fushida-Hardy's notes from Ciprian Manolescu's course \cite{SFH-MAT283A-notes}.)

Let $K$ be a knot with $\Delta_K(t) \overset{\cdot}{=} 1$ but $s(K) \neq 0$.

Since $K$ is topologically slice, $X_0(K)$ topologically embeds in $S^4$:

\begin{center}
    \includegraphics[width=2.5in]{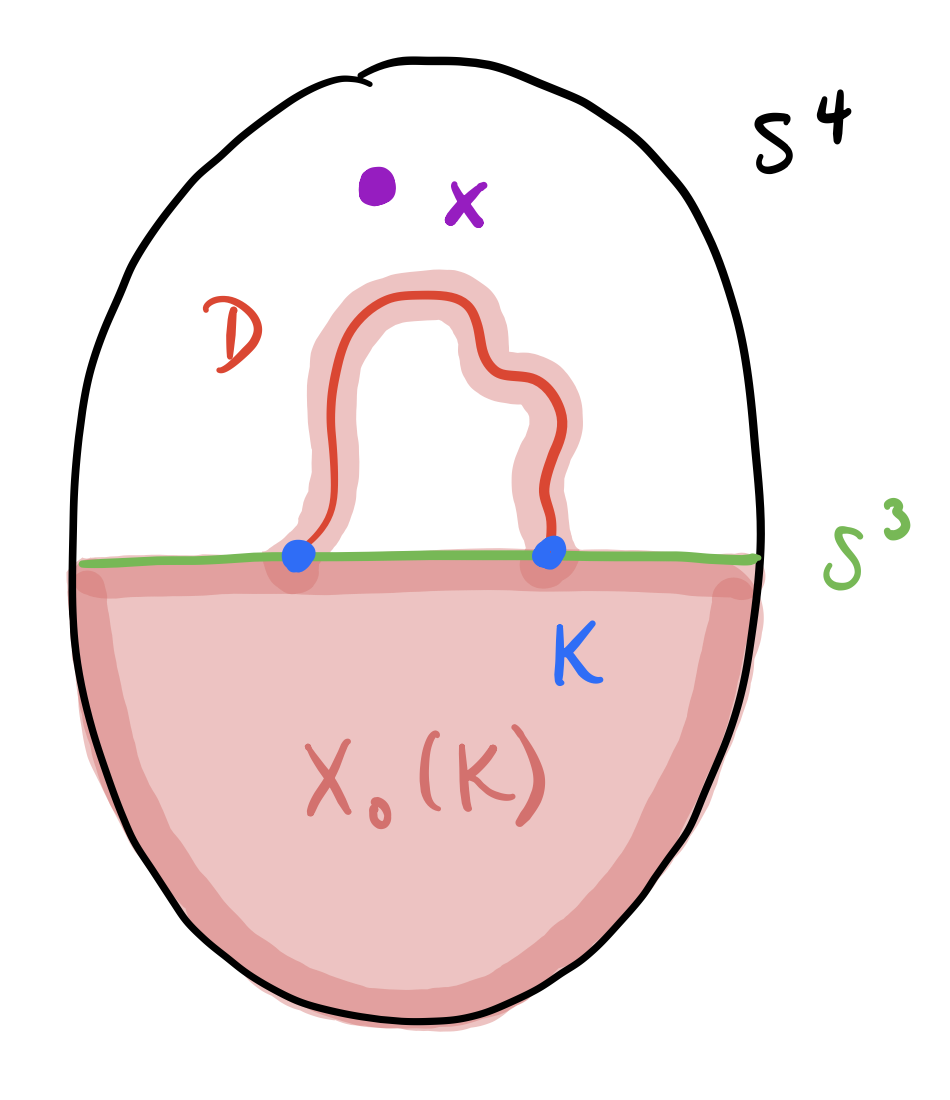}
\end{center}

Pick a point $x \in S^4 - X_0(K)$; then $S^4 - \{x\}$ is a topological $\R^4$.
We will give this space a smooth structure as follows. 

\begin{itemize}
    \item Let
    \[
        Z = S^4 - X_0(K) - \{x\}.
    \]
    Quinn showed that every \emph{open} 4-manifold admits a smooth structure \cite[Corollary 2.2.3]{quinn-ends-III}.
    Pick any such smooth structure; we now regard $Z$ as a smooth manifold.
    \item $X_0(K)$ comes with a smooth structure induced by the handle decomposition.
    \item The shared boundary $\partial Z \approx \partial X_0(K)$ is a 3-manifold, and actually has a unique smooth structure \cite{Moise-3-manifolds}; this means that there exists diffeomorphism $\phi: \partial Z \to \partial X_0(K)$.
\end{itemize}

Now build the open smooth 4-manifold
\[
    R = Z \cup_{\phi} X_0(K).
\]
Topologically, it is an $\R^4$. 
However, it cannot be the standard smooth $\R^4$ because $X_0(K)$ embeds into it smoothly. 
Therefore $R$ and $\R^4$ are an exotic pair.

\end{proof}
\end{corollary}

\subsection{Skein Lasagna modules from Khovanov homology}

For many decades (and counting), many state-of-the-art tools for detecting exotic 4-manifold pairs made use of gauge theory. 
In a way, gauge theory makes explicit use of the smooth structures on 4-manifolds by capturing the differential topology of the manifold; a priori, computing these invariants require solving partial differential equation.

On the other hand, Morrison--Walker--Wedrich's \emph{skein lasagna modules} feel radically different. 
There are no PDEs present, and the sensitivity to smooth structure seems to only rely on the fact that Khovanov homology sees smooth structure.

\begin{remark}
There are probably many viewpoints out there, and I should really ask around to see what how others would explain why Khovanov homology `sees' smooth structure. 
But right now, my point of view is that Khovanov homology `sees' smooth structure because the TQFT requires cobordisms to be decomposable into elementary cobordisms; this decomposition into smooth pieces implies that a functorial invariant like Khovanov homology really requires smooth structure. 
\end{remark}

Let us first motivate the definition of skein lasagna modules by looking at the 3-dimensional inspiration for this construction: skein modules. 
In particular, we will look at a particular example that makes use of what we have already seen in this course: the \emph{Kauffman bracket skein module}.

\begin{definition}
    Let $F$ be a surface. 
    The \emph{Kauffman bracket skein module} of $F$, denoted by $\KBSM(F)$ is the $\Z[A, A\inv]$-module 
    \begin{itemize}
        \item generated by all link diagrams drawn on $F$
        \item modulo the local relations from the Kauffman bracket polynomial:
        \[
            \langle \fullmoon \rangle = 1
            \qquad
            \langle \fullmoon \sqcup D \rangle = 
            (-A^{-2} - A^2) \langle D \rangle
            \qquad
            \langle \crossing \rangle
            = A \langle \vertres \rangle + A\inv \langle \horizres \rangle
        \]
    \end{itemize}
    \note{This is the original convention and differs  slightly from our Definition \ref{def:kauffman-bracket}.}
\end{definition}

We view $\KBSM(F)$ as a 3-manifold invariant, because it captures ``the knot theory'' in the manifold $F \times I$.

\begin{remark}
\begin{enumerate}
    \item For the simplest version of this invariant, we can evaluate $A$ at a root of unity, such as $A=1$.
    If we set $A = 1$, then $\KBSM(D^2, A=1) \cong \Z$.
    \item Since $D^2 \cong \R^2$, $\KBSM(D^2)$ captures the knot theory in $\R^3$, which is the same as the knot theory in $S^3$.
    \note{This is good, because if we want a 3-manifold invariant, we probably want it to be quite simple on the simplest of 3-manifolds.}
    \item Recall that the Kauffman bracket is not a knot invariant in our usual sense! Indeed, to recover the Jones polynomial, we actually had to multiply by a monomial determined by the writhe of the diagram. 
    But this just means that the Kauffman bracket polynomial is a \emph{framed link invariant}, which is a valid notion of `knot theory' in a 3-manifold. In fact, one could argue that this is more natural, since the framing of a framed knot tells us whether it can bound a surface embedded in the 3-manifold, which is more geometric information.
\end{enumerate}
    
\end{remark}

In this section, we will discuss Morrison--Walker--Wedrich's skein lasagna modules from $\KhR_2$, which is the Khovanov--Rozansky homology categorifying the $\essl_2$ link polynomial (a.k.a.\ the Jones polynomial). This differs slightly from $\Kh$ because it is a framed link invariant, and because the quantum degree is reversed by convention (so $\deg_q(X) = 2$ in $\KhR_2$). 
See Section 3.2 of \cite{Sullivan-Zhang-lasagna} for more details on various conventions appearing in the literature.

Let $W$ be a 4-manifold, possibly with boundary, and let $L \subset \partial W$ be a possibly empty link in the boundary of $W$. 

The skein lasagna module of the pair $(W;L)$ is a bigraded $\Q$ vector space
\[
    \lasagna(W;L)  = \bigoplus_{i,j, \alpha} \CS_{0,i,j}^{2,\alpha}(W;L)
\]
where
\begin{itemize}
    \item $2$ indicates that we are working with $\KhR_2$,
    \item $i$ is the homological grading,
    \item $j$ is the quantum grading, and
    \item $\alpha \in H_2(W,L)$ is a relative 2nd homology class.
\end{itemize}
\note{Also, $0$ is the blob grading, which we will not discuss. See \cite{MWW-lasagna}.}

This module is defined similarly to skein modules for 3-manifolds: there is a ridiculously large set of generators, but we then impose some local relations.

\mz{At this point, please clear the variable $i$ so that we can use it to index things again.}

\begin{definition}
\note{Refer to the cartoon below while reading the definition.}

\begin{center}
    \includegraphics[width=3in]{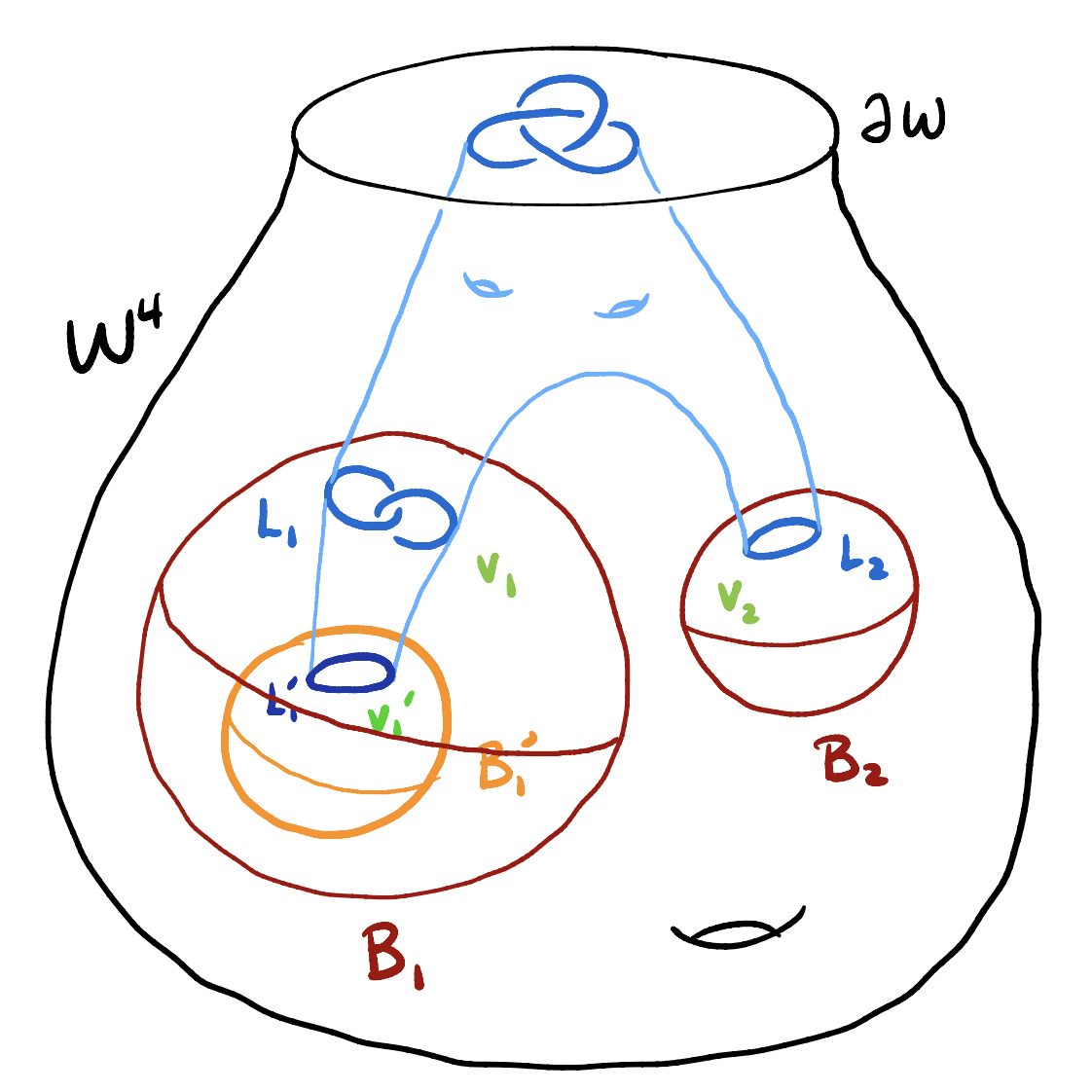}
\end{center}

    A \emph{lasagna filling} for the pair $(W,L)$ is the data
    \[
        \lasfilling = ( \Sigma, \{B_i, L_i, v_i\}_{i=1}^k)
    \]
    where
    \begin{itemize}
        \item each $B_i$ is a $B^4$ embedded in the interior of $W$
        \item $L_i$ is a link in $\partial B_i \cong S^3$ \\
        \note{We can't compute the Khovanov(-Rozansky) homology for $L \subset \partial W$ because it's some random 3-manifold, but we sure can compute Khovanov homology for the $L_i$ in these $S^3$ boundaries of 4-balls!}
        \item $v_i$ is a class in $\KhR_2(L_i)$ 
        \item $\Sigma$ is a framed, oriented, possibly disconnected surface properly embedded in $W \backslash \bigcup_i \overset{\circ}{B_i}$ with boundary $-L \sqcup \coprod_i L_i$ (with the correct framings!)
    \end{itemize}
\end{definition}

\begin{definition}
    The \emph{$\KhR_2$ skein lasagna module of the pair $(W;L)$} is 
    \[
        \lasagna(W;L)  = \Z \langle \text{lasagna fillings }\lasfilling\text{ for } (W; L) \rangle / \sim
    \]
    where the relations $\sim$ are generated by the following imposed equivalences:
    \begin{itemize}
        \item (Isotopy) $\Sigma$ is considered up to isotopy rel boundary. 
        \note{If you're concerned about isotopy of the $L_i$, note that the next relation takes care of this.}
        \item (Nesting) 
        \note{Refer to the previous cartoon.}
        Suppose we have two lasagna fillings 
        \[
            \lasfilling = ( \Sigma, \{B_i, L_i, v_i\}_{i=1}^k)
        \qquad 
        \text{and} 
        \qquad
            \lasfilling = ( \Sigma', \{B_i', L_i
            , v_i'\}_{i=1}^k)
        \]
        that are identical except within a ball $B_1$ where $B_1' \subset B_1$.
        That is, $(B_i, L_i, v_i) = (B_i', L_i', v_i')$ for all $i\neq 1$, and $\Sigma' \setminus B_1 = \Sigma$.

        Let $\Sigma'' = \Sigma' \setminus \Sigma$. 
        If we also have $\KhR_2(\Sigma'')(v_1') = v_1$, then 
        $\lasfilling' \sim \lasfilling$. 
    
        \note{In basic point-set topology, we compare open sets by looking at smaller open sets inside intersections. This is the same idea; we require also that the labels are coherent with respect to the TQFT.}
    \end{itemize}
\end{definition}

Not many examples have been computed thus far; most computations use Manolescu--Neithalath's \emph{cabled Khovanov homology} (\S \ref{sec:cabled-Kh}), to be discussed shortly.
Here is perhaps the most important example that you can compute right now.

\begin{exercise}
    Let $L \subset S^3$ be a link. Prove that 
    \[
        \lasagna(B^4; L) \cong \KhR_2(L).
    \]
\end{exercise}

\subsection{Cabled Khovanov homology}
\label{sec:cabled-Kh}

In \cite{MN22}, which is the reference for this subsection, Manolescu--Neithalath define \emph{cabled Khovanov-Rozansky homology} and show that it is isomorphic to the skein lasagna module for the pair $(W; \emptyset)$, where $W$ is a 2-handlebody (i.e.\ no 1- or 3-handles).

\begin{remark}
    There are of course many more results in \cite{MN22}. They can also handle certain types of links $L \subset \partial W$. See \cite{MWWhandles} for how to work with more links and other handles. 
\end{remark}

\begin{warning}
    We omit all grading information below. Remember that these are all graded invariants! So, for example, there will be quantum degree shifts that make all the cobordism maps grading-preserving.
\end{warning}

\alert{Be aware} that the link $\CL$ in the following theorem is \emph{not} a link in the boundary of $W$, but rather the framed link along which the 2-handles of $W$ are attached.

\begin{theorem}[see \cite{MN22}, Theorem 1.1]
Let $\CL$ be a framed $\ell$-component link in $S^3$, and let $W$ be the 4-manifold obtained by attaching 4D 2-handles along $\CL$. 
Then the cabled Khovanov homology of $\CL$ is isomorphic to the skein lasagna module of $W$. 
That is, for  $\alpha \in \Z^\ell \cong H_2(W;\Z)$, there is a graded isomorphism
\[
    \cabledKhR_{2,\alpha}(\CL) \cong \CS_0^{2,\alpha}(W; \emptyset).
\]
\end{theorem}

In this section, we give a sense of how Manolescu--Neithalath's cabled Khovanov homology \cite{MN22} is computed, via an example computed independently in \cite{Sullivan-Zhang-lasagna} and \cite{Ren-Willis-lasagna}.

Throughout, refer to the drawing below:

\begin{equation}
\label{eq:cabled-Kh-example}
\includegraphics[width=.9\textwidth]{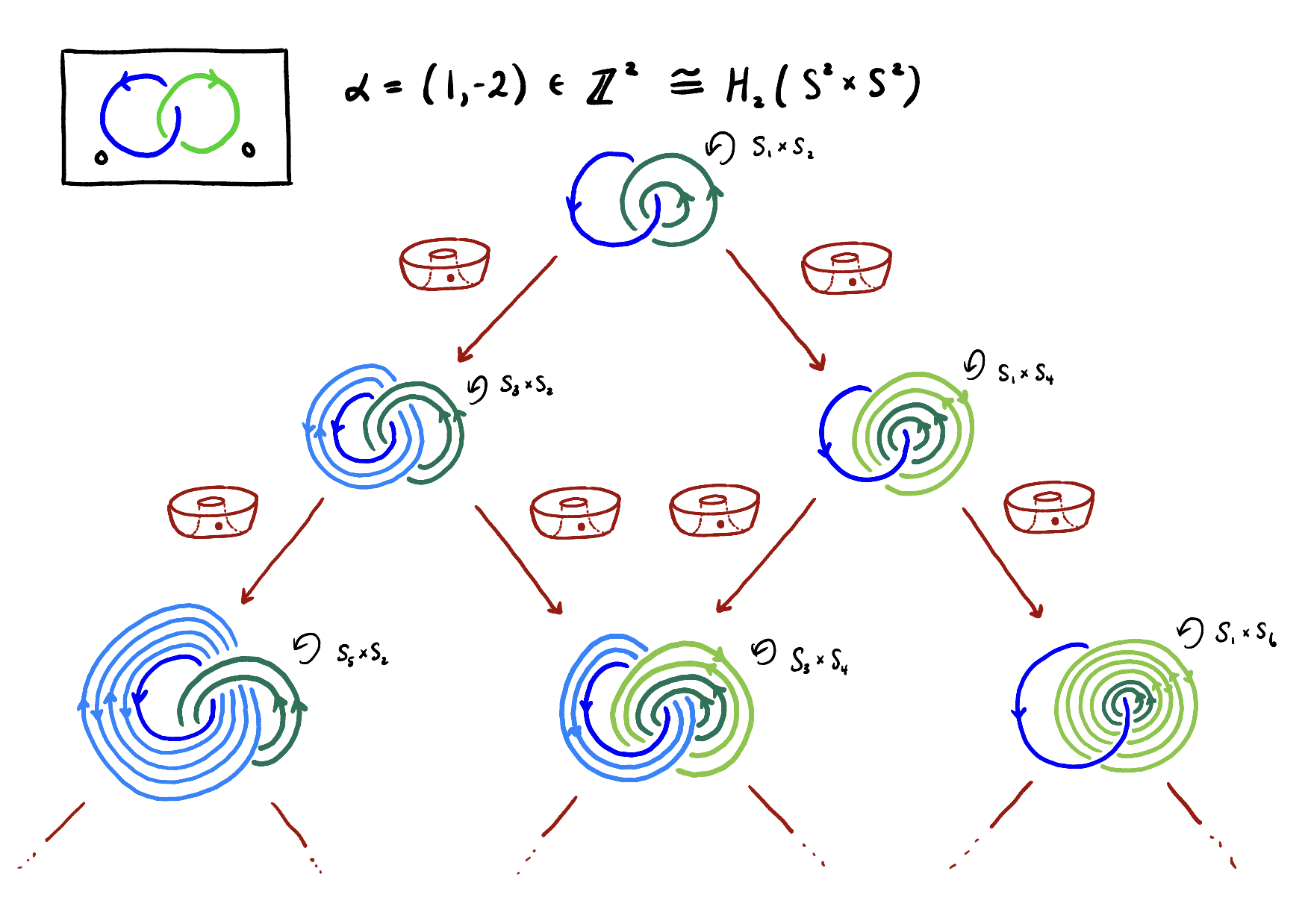}
\end{equation}

\begin{definition}
    Let $K$ be a framed, oriented knot in $S^3$, and let $c \in \Z$. The \emph{$c$-cable of $K$} is the $c$-component link obtained by satelliting $K$ with the $c$-component identity braid pattern $\widehat{\idbraid_c} \subset S^1 \times D^2$ according to the framing of $K$. 

    Let $\CL$ be a framed, oriented link in $S^3$ with $\ell$ components, indexed as 
    \[
        \CL = K_1 \cup K_2 \cup \cdots \cup K_\ell.
    \]
    Let $\alpha = (c_1, c_2, \cdots, c_\ell) \in \Z^\ell$. 
    The \emph{$\alpha$-cable of $\CL$} is the link obtained by replacing each $K_i$ with the $c_i$-cable of $K_i$. 
\end{definition}

\begin{warning}
    \alert{Be aware} that the following description is quite casual, and also only works for $\cabledKhR_2$, not the cabled $\KhR_N$ homologies for $N \geq 3$. 
\end{warning}

Below are the steps for computing $\cabledKhR_{2,\alpha}(W)$ where $W = S^2 \times S^2$ and $\alpha = (1,-2) \in H_2(W;\Z) \cong \Z^2$.

\begin{enumerate}

\item Start with the framed, oriented Hopf link $\CL$ as shown. 
The Kirby diagram fixes a choice of isomorphism between $H_2(W;\Z)$ and $\Z^2 = \Z_{\text{blue}} \oplus \Z_{\text{green}}$. 

\item Build a directed system of framed links and cobordisms, in the shape of the poset $\Z_{\geq 0}^2$, as follows.

At $(0,0)$, associate the $(1,-2)$-cabling of $\CL$. 
Given a link at $(x,y)$, 
\begin{itemize}
    \item associate the link at $(x+1,y)$ by adding two antiparallel components into the cable on the blue component.
    \item associate the link at $(x,y+1)$ by adding two antiparallel components into the cable on the green component.
\end{itemize}
The cobordism associated to the edges $(x,y) \to (x+1, y)$ and $(x,y) \to (x, y+1)$ are dotted annuli as shown in the drawing \eqref{eq:cabled-Kh-example}. 

\item Build a directed system $\mathscr{D}$ of graded vector spaces and linear maps in the same shape as above, by performing the following replacements:
\begin{itemize}
    \item The cabled link $L$ at $(x,y)$ has $c_1 = 1+2x$ blue components and $c_2 = -2+2y$ green components. Grigsby--Licata--Wehrli show that there is a $S_{c_1} \times S_{c_2}$ action on the links, induced by braiding the components with each other. 
    \note{By braiding, I mean passing hoops through each other; if you're not familiar with this, just imagine swapping components.}
    Replace the link $L$ at $(x,y)$ with the symmetrization of $\KhR_2(L)$ under the induced $S_{c_1} \times S_{c_2}$ action.
    \item Along edges, replace the dotted annuli with the corresponding induced symmetrized cobordism maps.
\end{itemize}
\note{To symmetrize in a sensible algebraic manner, we work over a field of characteristic $0$.}

\item Take the colimit of the directed system $\mathscr{D}$ to get the cabled Khovanov homology of $\mathcal{L}$ at homological level $\alpha$. 
\end{enumerate}

This description of the skein lasagna module of 4D 2-handlebodies makes some computations more tractable. 
For example, by computing the colimits of the diagonals of the drawing \eqref{eq:cabled-Kh-example} (and for all other homological levels $\alpha$), Ian and I were able to compute the $\KhR_2$ skein lasagna module of $S^2 \times S^2$:

\begin{theorem}[Sullivan-Zhang \cite{Sullivan-Zhang-lasagna}, Ren-Willis \cite{Ren-Willis-lasagna}]
    The $\KhR_2$ skein lasagna module of $S^2 \times S^2$ vanishes. 
\end{theorem}

\subsubsection{Ian's lecture}
In the last lecture, Ian discussed Ren--Willis' vanishing and non-vanishing results from \cite{Ren-Willis-lasagna}:
\begin{itemize}
    \item The vanishing criteria partly relies on Ng's Thurston-Bennequin bound. 
    \item The non-vanishing criteria rely on filtered, Lee version of skein lasagna modules. 
\end{itemize}

In particular, he gave a sketch of the first non-gauge-theoretic exotic detection result for 4-manifolds using cabled Khovanov homology (Theorem 1.1 of \cite{Ren-Willis-lasagna}). 

We will not include more details here, as a forthcoming set of notes will delve into skein lasagna modules more carefully.

\bibliographystyle{alpha}
\bibliography{main}

\end{document}